\documentclass{gtart}

\usepackage{amsmath,amssymb,amsfonts}
\usepackage{ifpdf}
\usepackage{ifthen}

\ifpdf
\usepackage[pdftex]{hyperref}
\else
\usepackage[dvips]{hyperref}
\fi

\ifpdf
\usepackage[pdftex,all,color]{xy}
\else
\usepackage[all,color]{xy}
\fi

\ifpdf
\newcommand{\pathtofigs}{figures/pdf}
\else
\newcommand{\pathtofigs}{figures/eps}
\fi

\newcommand{\arxiv}[1]{\href{http://arxiv.org/abs/#1}{\tt arXiv:\nolinkurl{#1}}}

\newtheorem{thm}[subsection]{Theorem}
\newtheorem{defn}[subsection]{Definition}
\newtheorem{prop}[subsection]{Proposition}

\newtheorem{lemma}[subsection]{Lemma}
\newtheorem{cor}[subsection]{Corollary}

\newenvironment{remark}{\noindent\textsl{Remark.}}{}

\newcommand{\Z}{\mathbb Z}

\newcommand{\CKT}{R3/CKT/}

\newcommand{\psmallmatrix}[1]{\left(\begin{smallmatrix} #1 \end{smallmatrix}\right)}

\newcommand{\mathfig}[2]{{\hspace{-3pt}\begin{array}{c}%
  \raisebox{-2.5pt}{\includegraphics[width=#1\textwidth]{\pathtofigs/#2}}%
    \end{array}\hspace{-3pt}}}

\newcommand{\reflectmathfig}[2]{{\hspace{-3pt}\begin{array}{c}%
  \raisebox{-2.5pt}{\reflectbox{\includegraphics[width=#1\textwidth]{\pathtofigs/#2}}}%
    \end{array}\hspace{-3pt}}}

\newcommand{\rotatemathfig}[3]{{\hspace{-3pt}\begin{array}{c}%
  \raisebox{-2.5pt}{\rotatebox{#2}{\includegraphics[height=#1\textwidth]{\pathtofigs/#3}}}%
    \end{array}\hspace{-3pt}}}

\def\clap#1{\hbox to 0pt{\hss#1\hss}}
\def\mathllap{\mathpalette\mathllapinternal}
\def\mathrlap{\mathpalette\mathrlapinternal}

\def\mathllapinternal#1#2{%
\llap{$\mathsurround=0pt#1{#2}$}}
\def\mathrlapinternal#1#2{%
\rlap{$\mathsurround=0pt#1{#2}$}}

\def\llbracket{\left[\!\!\left[}
\def\rrbracket{\right]\!\!\right]}

\newcommand{\s}[1]{s_{\text{{\tiny #1}}}}
\renewcommand{\sh}[2]{s_{\text{{\tiny #1}},#2}}     
\newcommand{\T}[1]{t_{\text{{\tiny #1}}}}
\newcommand{\Th}[2]{t_{\text{{\tiny #1}},#2}}
\newcommand{\h}[1]{h_{\text{{\tiny #1}}}}

\newcommand{\low}[2]{\overline{R3_{#1}} \! \!  \ ^{#2}}

\newcommand{\cw}{\circlearrowright}
\newcommand{\ccw}{\circlearrowleft}

\newcommand{\OO}{{\mathcal{O} \To \mathcal{O}}}
\newcommand{\OP}{{\mathcal{O} \To \mathcal{P}}}
\newcommand{\PO}{{\mathcal{P} \To \mathcal{O}}}
\newcommand{\PP}{{\mathcal{P} \To \mathcal{P}}}

\newcommand{\Hom}[3]{\operatorname{Hom}_{#1}\left(#2,#3\right)}
\newcommand{\Ob}[1]{\operatorname{Ob}(#1)}

\newcommand{\Cob}[1]{{\mathcal Cob}\left(\su{#1}\right)}
\newcommand{\Kob}[1]{{\mathbf{Kob}}\left(\su{#1}\right)}
\newcommand{\Kom}[1]{\operatorname{Kom}\left(#1\right)}
\newcommand{\Mat}[1]{\operatorname{Mat}\left(#1\right)}
\newcommand{\su}[1]{\mathfrak{su}_{#1}}
\newcommand{\Ortang}{\mathbf{Ortang}}

\newcommand{\Kh}[1]{\llbracket#1\rrbracket}
\newcommand{\cone}[3]{C\left(#1 \xrightarrow{#2} #3\right)}
\newcommand{\vcone}[3]{C\left(
    \def\objectstyle{\scriptstyle}
    \def\labelstyle{\scriptstyle}
    \vcenter{\xymatrix{#1 \ar[d]^{#2} \\ #3}}
    \right)}

\newcommand{\shift}[1]{\left[#1\right]}

\newcommand{\htpy}{\simeq}
\newcommand{\iso}{\cong}
\newcommand{\To}{\rightarrow}

\newcommand{\compose}{\circ}
\newcommand{\directSum}{\oplus}
\newcommand{\DirectSum}{\bigoplus}
\newcommand{\tensor}{\otimes}

\newcommand{\Id}{\boldsymbol{1}}

\newcommand{\IsoTo}{\overset{\iso}{\To}}

\newcommand{\directSumStack}[2]{{\begin{matrix}#1 \\ \DirectSum \\#2\end{matrix}}}

\newcommand{\directSumStackThree}[3]{{\begin{matrix}#1 \\ \DirectSum \\#2 \\ \DirectSum \\#3\end{matrix}}}

\ifpdf
\usepackage[pdftex]{graphicx}
\else
\usepackage[dvips]{graphicx}
\fi 

\title{Functoriality for the $\su{3}$ Khovanov homology}
\author{David Clark}

\address{
   Department of Mathematics\\
   University of California, San Diego\\
   La Jolla, CA 92093-0112
}

\email{dclark@math.ucsd.edu}
\urladdr{http://www.math.ucsd.edu/\~{}dclark}

\date{                                 
  First edition: \today.
  This edition: \today.
}

\primaryclass{57M25} \secondaryclass{57M27; 57Q45}
\keywords{
  Khovanov homology,
  categorification,
  link cobordism,
  spider,
  quantum knot invariants.
}

\begin{document}

\begin{abstract}
We prove that Morrison and Nieh's categorification of the $\su{3}$ quantum knot invariant \cite{MorrN:sl(3)} is functorial with respect to tangle cobordisms. This is in contrast to the categorified $\su{2}$ theory \cite{Khov00_MR1740682,BN05_MR2174270}, which was not functorial as originally defined \cite{Jac04_MR2113903,ClarMW:Fix}.

We use methods of Bar-Natan \cite{BN07_MR2320156} to construct explicit chain maps for each variation of the third Reidemeister move. Then, to show functoriality, we modify arguments used by Clark, Morrison, and Walker \cite{ClarMW:Fix} to show that induced chain maps are invariant under Carter and Saito's movie moves \cite{CS93_MR1238875,CS97_MR1445361}.

\end{abstract}

\maketitle                          

\section{Introduction}                  
\label{sec:introduction}

\subsection{The $\su{3}$ link invariant and its categorification}
\label{ssec:intro-categorification}

Khovanov first categorified the $\su{3}$ link invariant in \cite{Khov04_MR2100691}; it was later generalized by MacKaay and Vaz in \cite{MacV07_MR2336253}. Independently, in \cite{MorrN:sl(3)}, Morrison and Nieh give a local geometric construction in the spirit of Bar-Natan \cite{BN05_MR2174270}, using the language of planar algebras and canopoleis\footnote{We find this to be a pleasing plural form of \emph{canopolis}, and surely the purest from the standpoint of Greek etymology (cf. \emph{metropolis}, \emph{metropoleis} \cite{wikt:metropolis}). By way of analogy,\linebreak formulae\,:\,formulas\,:\,:\,canopoleis\,:\,canopolises.}.
Indeed, the $\su{3}$ quantum link invariant can be thought of as a map of planar algebras, defined on generators by
\begin{align*}
 \mathfig{0.06}{webs/positive_crossing}&  \mapsto  \phantom{- q^{-3}} \mathllap{q^2}  \mathfig{0.06}{webs/two_strand_identity} - \phantom{q^{-2}} \mathllap{ q^3} \mathfig{0.06}{webs/upwards_H_diagram} \\
 \mathfig{0.06}{webs/negative_crossing}&  \mapsto - q^{-3} \mathfig{0.06}{webs/upwards_H_diagram} + q^{-2} \mathfig{0.06}{webs/two_strand_identity}
\end{align*}
and subject to the relations of Kuperberg's $\su{3}$ spider \cite{Kup96_MR1403861}
\begin{align}
\label{eq:spider}
   \mathfig{0.045}{webs/clockwise_circle} & = q^2 + 1 + q^{-2}\\
   \mathfig{0.045}{webs/bubble} & = q\;\mathfig{0.009}{webs/tall_strand} + q^{-1}\;\mathfig{0.009}{webs/tall_strand} \\
   \mathfig{0.08}{webs/oriented_square}
   & = \mathfig{0.08}{webs/two_strands_horizontal} + \mathfig{0.08}{webs/two_strands_vertical}
\end{align}
which will reduce a $\Z[q,q^{-1}]$-linear combination of trivalent graphs (``webs") to a polynomial.

Morrison and Nieh use a technique similar to Bar-Natan's \cite{BN05_MR2174270} to categorify this map of planar algebras. The new source category (technically a \emph{canopolis}) $\Ortang$ is that of oriented tangles and their cobordisms, and the target category $\Kob{3}$ consists of formal complexes of webs with chain maps given by seamed cobordisms (``foams"). In \cite{MorrN:sl(3)}, it is shown that this categorified map, which we'll call $Kh(\su{3})$ (technically a \emph{canopolis morphism}), is well-defined on objects, i.e., isotopy of a tangle does not change the homotopy type of the image complex. Put yet another way, ``$Kh(\su{3})$ is a link invariant."

\subsection{Main result}
\label{ssec:main-result}

What's not shown in \cite{MorrN:sl(3)} is whether $Kh(\su{3})$ is truly functorial, i.e., that it is also well-defined on morphisms (up-to-isotopy tangle cobordisms), and thus an honest map of canopoleis. Conveniently, we can view a tangle cobordism in 4-space as a sequence of tangle diagrams called a ``movie".\footnote{There is some subtlety here about being able to assume that such cobordisms are in general position; this is addressed carefully in \cite{ClarMW:Fix}.} Further, any cobordism admits a movie presentation such that the tangles in subsequent frames differ by either a single Reidemeister move or a single Morse move (the birth or death of a circle, or the splicing of two strands). This partitioning of a cobordism $C$ into simple combinatorial steps gives us an obvious way to attempt a definition of a chain map $Kh(\su{3})(C)$.

Thanks to Carter and Saito \cite{CS97_MR1445361, CS93_MR1238875} (and also Roseman \cite{Rose98_MR1634466}), there is also a way to view isotopies of tangle cobordisms in this movie presentation context: two tangle cobordisms are isotopic if and only if they are related by a sequence of the movie moves\footnote{In an oriented theory, like the one in this paper, one must consider all possible orientations of these moves, in addition to the usual variations resulting from reflections and crossing changes.} in Figure \ref{fig:movie-moves}. Thus $Kh(\su{3})$ is only well-defined if it yields homotopic chain maps when applied to the cobordism on each side of every movie move.

\begin{figure}[ht]
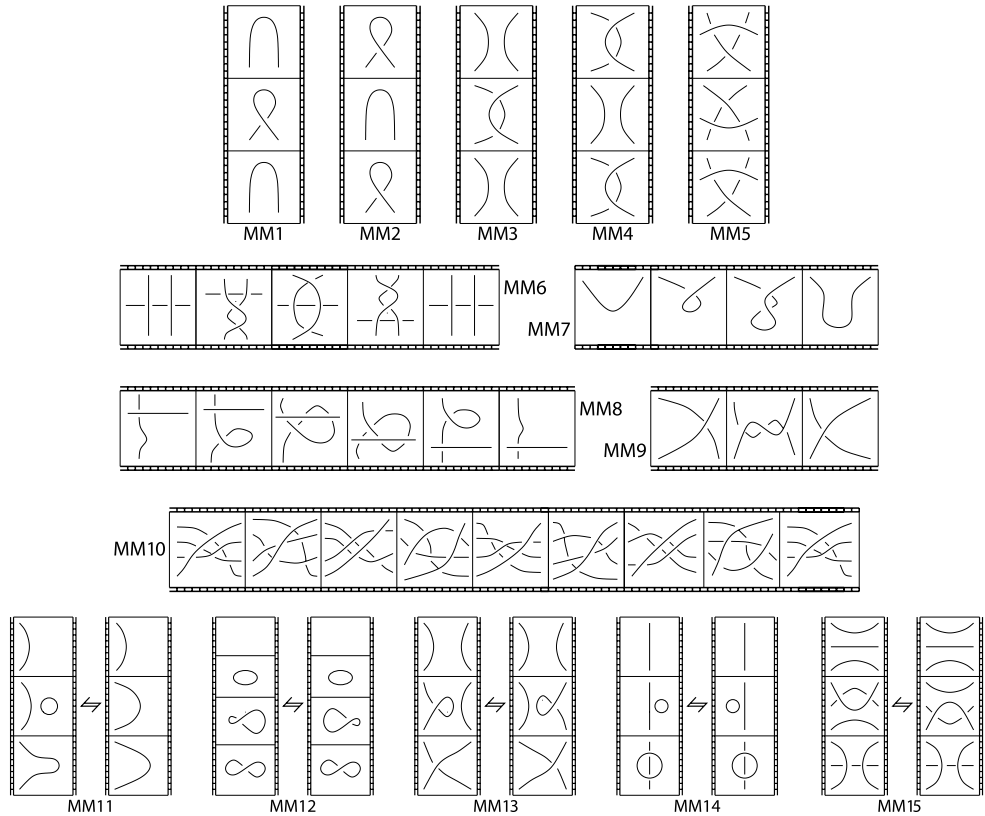

$$\mathfig{0.55}{movie_moves/MM1-5}$$
$$\mathfig{0.77}{movie_moves/MM6-10}$$
$$\mathfig{0.98}{movie_moves/MM11-15}$$
\caption[Carter and Saito's unoriented movie moves]{Carter and Saito's unoriented movie moves, numbered according to Bar-Natan \cite{BN05_MR2174270}. Note that first ten moves are circular, and so should be paired with the constant movie of the first frame.}
\label{fig:movie-moves}
\end{figure}

This turned out not to be the case in for the categorified $\su{2}$ invariant \cite{Khov00_MR1740682,BN05_MR2174270}, as first documented by Jacobsson \cite{Jac04_MR2113903}: certain movie moves changed the sign of the induced chain map. This issue was resolved in \cite{ClarMW:Fix} with a modified construction designed to incorporate a previously neglected piece of representation theory: the fact that the fundamental representation of $\su{2}$ is antisymmetrically self-dual, and the source of the sign anomaly.

Such an issue does not exist for $\su{3}$, which is not self-dual at all. This, along with some experimental evidence, led Morrison and Nieh to conjecture that their theory is in fact functorial. In this paper, we'll prove it as a theorem.

\begin{thm}
\label{thm:1}
$Kh(\su{3}): \Ortang \rightarrow \Kob{3}$ is a canopolis morphism; in particular, oriented tangle cobordisms induce well-defined (up to homotopy) chain maps in $\Kob{3}$.
\end{thm}

It would be a natural next step to extend this theory to \textbf{WebCob}, the category of knotted webs and seamed cobordisms in four-space. For objects, well-definition relies on two additional Reidemeister-type moves: sliding a strand past a vertex (R4), and flipping a vertex upside-down\footnote{This is equivalent to changing the cyclic ordering of the edges around the vertex}(R5). Unfortunately, the complexes associated to each side of these moves are not quite chain homotopy equivalent: there are extra grading shifts in the way. (Somewhat incidentally, we treat the R4 case fully in Section \ref{sec:R3maps}.) An extension to knotted webs will thus require a renormalized skein theory, and our webs will probably need to carry a framing. This makes morphisms more complicated, since it's not clear exactly what a ``framed" seamed cobordism should be, or what the corresponding movie move list might look like. We hope to address this in a future paper.

The author would like to thank Justin Roberts and Scott Morrison for many useful discussions. Also, an additional thanks to Scott Morrison and Ari Nieh for allowing me access to their wonderful foam diagrams, and to Scott Carter and Masahico Saito for letting me reuse their movie move diagrams from \cite{CS93_MR1238875}. 

In Section \ref{sec:the-su3-theory} we'll review the $\su{3}$ theory of Morrison and Nieh. Much of the work comes in Section \ref{sec:reidemeister_maps}, when we explicitly define the induced maps for oriented Reidemeister moves. Finally, in Section \ref{sec:checking_movie_moves} we'll look at the induced maps on each side of the movie moves, and see that in each case they are homotopic.        
\section{The $\su{3}$ theory}           
\label{sec:the-su3-theory}

\subsection{Planar algebras and canopoleis}
\label{ssec:planar-algebras}

We'll give a brief recap of the construction of Morrison and Nieh here, and refer the reader to \cite{BN05_MR2174270}, \cite{MorrN:sl(3)}, and \cite{Web07_MR2308960} for more technical details regarding planar algebras and canopoleis.

Recall that an \emph{oriented planar arc diagram} is, colloquially, just an oriented crossingless tangle in a disk with (possibly) some smaller disks removed, and with the remaining holes given some ordering. Two such diagrams can be composed whenever the outer boundary of one diagram matches one of the inner boundaries of the other: we just shrink the first diagram and paste it into the second, giving a new planar arc diagram. More generally, let $P$ be a planar arc diagram with $n$ holes; we'll label each from 1 up to $n$, and think of the outer boundary of $P$ as the ``$0$th hole." If $\mathcal{Q}_i$ is the set of planar arc diagrams that match the boundary of the $i$th hole of $P$, then $P$ defines an operation $P: \mathcal{Q}_1 \times ... \times \mathcal{Q}_n \rightarrow \mathcal{Q}_0$. See Figure \ref{fig:operad} for an example.

\begin{figure}[ht]
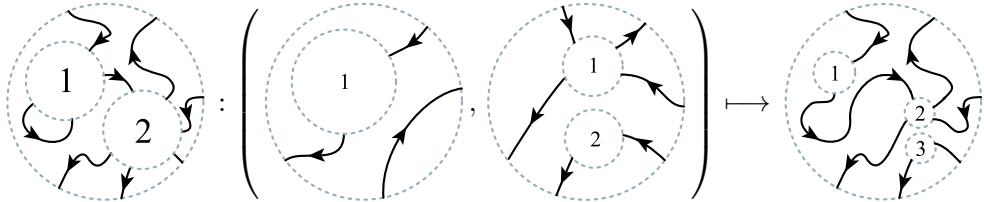

\begin{align*}
\mathfig{0.2}{webs/operad1}: \left( \mathfig{0.2}{webs/operad3}, \mathfig{0.2}{webs/operad2} \right) \longmapsto \mathfig{0.2}{webs/operad5}
\end{align*}
\caption[Composition in the oriented planar arc diagram operad]{Composition in the oriented planar arc diagram operad.}
\label{fig:operad}
\end{figure}

This operation on oriented planar arc diagrams gives them the structure of a \emph{colored operad}, where the coloring just refers to the labels (incoming and outgoing strands) on the disk boundaries. Such an operad can act on a collection of objects in some monoidal category $\mathcal{C}$: we associate to each color $s_i$ an object $\mathcal{P}(s_i)$, and to each collection of composable colors $s_1,...,s_n, s_0$ we associate the space of maps $\Hom{}{\mathcal{P}(s_1) \times ... \times \mathcal{P}(s_n)}{\mathcal{P}(s_0)}$. Of course, a properly colored planar arc diagram $P$ specifies a map $\mathcal{P}(s_1) \times ... \times \mathcal{P}(s_n) \rightarrow \mathcal{P}(s_0)$.

\begin{defn}
A \emph{planar algebra in} $\mathcal{C}$ is a collection $(\mathcal{P}(s_i)) \in \Ob{\mathcal{C}}$ that admits the above action of the operad of oriented planar arc diagrams.
\end{defn}

\begin{figure}[ht]
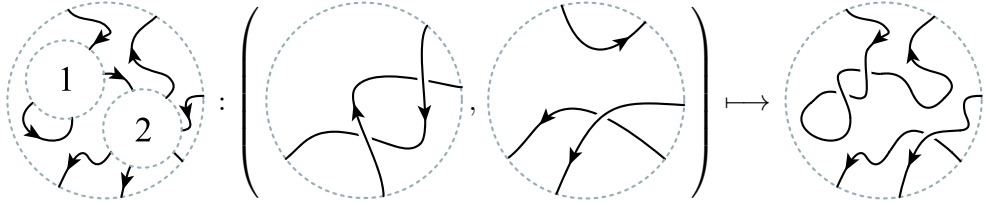

\begin{align*}
\mathfig{0.2}{webs/operad1} : \left( \mathfig{0.2}{webs/tangle_disk1}, \mathfig{0.2}{webs/tangle_disk2} \right) \longmapsto \mathfig{0.2}{webs/planar_tangle}
\end{align*}
\caption[A planar algebra]{A planar algebra: the operad action on oriented tangle diagrams.}
\label{fig:planar-tangle}
\end{figure}

Practically speaking, this structure gives us an associative way of ``multiplying" elements of our collection, in a planar fashion. As an easy example, the set of oriented tangle diagrams forms a planar algebra in the category of sets (see Figure \ref{fig:planar-tangle}), with generating set $\left\{ \mathfig{0.06}{webs/positive_disk},\mathfig{0.06}{webs/positive_disk}\right\}$; we shall, of course, consider these diagrams up to Reidemeister equivalence. Similarly, Kuperberg's $\su{3}$ spider forms a planar algebra in the category of $\Z[q,q^{-1}]$-modules, where we quotient by the $\su{3}$ spider relations (Equations \ref{eq:spider}); the spider is generated, as a planar algebra, by $\left\{ \mathfig{0.06}{webs/sink_disk},\mathfig{0.06}{webs/source_disk}\right\}$.   We can thus view the $\su{3}$ quantum link invariant as a map of planar algebras, which is convenient for both computational efficiency and organizing philosophy.

The goal, then, is to categorify this local picture of a quantum invariant: to do this, we invoke the notion of a canopolis, first appearing in \cite{BN05_MR2174270}.

\begin{defn}
A \emph{canopolis} is a planar algebra in some (monoidal\footnote{Here the monoidal structure is just given by cartesian product.}) category of categories $(\mathcal{C}(s_i))$. In particular, both the collection of objects and the collection of morphisms form planar algebras.
\end{defn}

Now a planar arc diagram $P$ will define a \emph{functor} $P:\mathcal{C}(s_1) \times ... \times \mathcal{C}(s_n) \rightarrow \mathcal{C}(s_0)$. We can view each category $\mathcal{C}(s)$ as a ``can" (rather than just a disc) with a specified label $s$ that can be plugged into a cylinder with a matching label in $P \times [0,1]$: objects will live on the tops and bottoms of cans, and morphisms will live inside cans. Further, the fact that $P$ defines a functor guarantees that planar operations commute with the usual composition of morphisms within their categories. Thus, we can build a ``city of cans" by composing vertically or horizontally in any order. It will also be useful to talk about maps between canopoleis.

\begin{defn}
A \emph{canopolis morphism} $\mathbf{C} \rightarrow \mathbf{C'}$ is a collection of functors $$(\mathcal{C}(s_i)) \rightarrow (\mathcal{C}'(s_i))$$ that commute with all planar algebra operations.
\end{defn}

Our first example of a canopolis will be the categorification of the set of oriented tangles. Let $\Ortang(s)$ be the category of tangle cobordisms with fixed boundary denoted by $s \in \mathcal{S}$, where $\mathcal{S}$ indexes the set of strand intersections with the boundary circle and their orientations, up to cyclic permutation. Then we define $\Ortang$ to be the canopolis in the category $\bigcup_{s \in \mathcal{S}} \Ortang(s)$.
(Note that we need more than just the cobordisms between two individual crossings to generate all possible tangle cobordisms.)
Here we want all morphisms considered up to four-dimensional isotopy; when viewing a generic morphism as a movie of tangle diagrams, this means we mod out by the movie moves.

\subsection{Categorifying the $\su{3}$ spider}
\label{ssec:categorifying}

A more interesting example, because it involves relations, is the categorification of the $\su{3}$ spider. Let $\Cob{3}_s$ be the category of cans with fixed boundary as above: the objects are $\su{3}$ webs and the morphisms are seamed cobordisms (or \emph{foams}), which are just CW-complexes modeled on $``Y" \times [0,1]$, plus some additional data.

\begin{defn}[from \cite{MorrN:sl(3)}]
Given two webs $D_1$ and $D_2$ drawn in a disc, both with boundary
$\partial$, a \emph{seamed cobordism} from $D_1$ to $D_2$ is a
2-dimensional CW-complex $F$ with
\begin{itemize}
\item exactly three oriented $2$-cells meeting along each oriented singular $1$-cell, such that the orientations on the $2$-cells all induce the same orientation on the seam;
\item a cyclic ordering on those three $2$-cells;
\item and an identification of the boundary of $F$ with $D_1 \cup D_2
\cup (\partial \times [0,1])$ such that
\begin{itemize}
\item the orientations on the sheets induce the orientations on the
edges of $D_1$, and the opposite orientations on the edges of $D_2$,
\item and the cyclic orderings around the singular seams agree with the
cyclic orderings around a vertex in $D_1$ or $D_2$ given by its
embedding in the disc; the anticlockwise ordering for ``inwards" vertices, the clockwise ordering for ``outwards" vertices.
\end{itemize}
\end{itemize}
Note that a foam is an abstract space; while its boundary is identified with lines on the surface of the can $D^2 \times [0,1]$ (and thus it is picturesque to view foams as in Figure \ref{fig:canopolis-action}), the foam doesn't literally live in the can.
\end{defn}

\begin{figure}[ht]
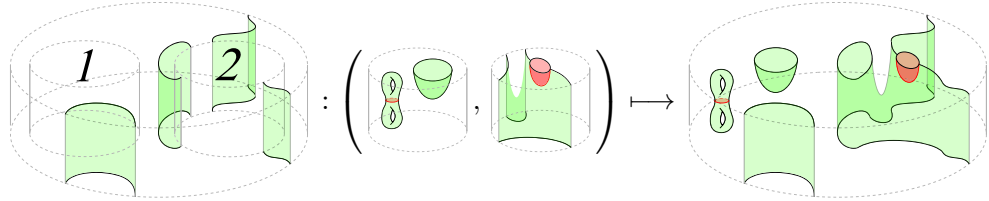

\begin{align*}
\mathfig{0.3}{foams/canopolis1b} : \left( \mathfig{0.1}{foams/canopolis2b}, \mathfig{0.1}{foams/canopolis3b} \right) \longmapsto \mathfig{0.3}{foams/canopolis4b}
\end{align*}
\caption[Planar composition of foams in $\Cob{3}$.]{Planar composition of foams in $\Cob{3}$. It is convenient to view these foams in cans, though really they are not embedded there.}
\label{fig:canopolis-action}
\end{figure}

We define $\Cob{3}$ to be the canopolis in the category $\bigcup_{s \in \mathcal{S}} \Cob{3}_s$ (see Figure \ref{fig:canopolis-action}), where, for $R$ a ring in which 2 and 3 are invertible\footnote{In this paper, we'll assume $R=\Z[\frac{1}{2}, \frac{1}{3}]$.}, we allow formal $R$-linear combinations of morphisms, and where we impose the following local relations on foams:

\begin{itemize}
\item``Closed foam" relations:
\begin{align}
\label{eq:closed_foams}%
\mathfig{0.075}{cobordisms/sphere} & = 0 &
\mathfig{0.075}{cobordisms/torus} & = 3 \\
\notag
\mathfig{0.15}{cobordisms/double_torus} & = 0 &
\mathfig{0.16}{cobordisms/triple_torus_disc} & = 0
\end{align}

\item The ``neck cutting" relation:
\begin{equation}
\label{eq:neck_cutting}
 \rotatemathfig{0.04}{90}{cobordisms/cylinder} = \frac{1}{3}   \rotatemathfig{0.04}{90}{cobordisms/neck_cutting_left}
      - \frac{1}{9} \rotatemathfig{0.04}{90}{cobordisms/neck_cutting_middle}
      + \frac{1}{3} \rotatemathfig{0.04}{90}{cobordisms/neck_cutting_right}
\end{equation}

\item The ``airlock" relation:
\begin{equation}
\label{eq:airlock}
    \rotatemathfig{0.05}{90}{cobordisms/airlock} = -
    \rotatemathfig{0.05}{90}{cobordisms/cup_cap}
\end{equation}

\item The ``tube" relation
\begin{equation}
\label{eq:tube_relation}
    \mathfig{0.1}{cobordisms/tube_relation/tube} =
    \frac{1}{2}
    \mathfig{0.1}{cobordisms/tube_relation/two_bubbles_lower_kiss} +
    \frac{1}{2}
    \mathfig{0.1}{cobordisms/tube_relation/two_bubbles_upper_kiss}
\end{equation}

\item The ``three rocket" relation:
\begin{equation}
\label{eq:rocket_relation}
  \mathfig{0.1}{cobordisms/rocket_relation/rocket_z}
 + \mathfig{0.1}{cobordisms/rocket_relation/rocket_x}
 + \mathfig{0.1}{cobordisms/rocket_relation/rocket_y} = 0
\end{equation}

\item The ``seam-swap" relation: reversing the cyclic order of the three 2-cells
attached to a closed singular seam is equivalent to multiplication by $-1$.

\item The ``sheet relations" (which can be derived from the relations above):
\begin{align}
  \mathfig{0.05}{cobordisms/sheet_algebra/blister} & = 0 &
 \label{eq:choking-torus-multiplication}%
   \mathfig{0.065}{cobordisms/sheet_algebra/two_handle_discs} & =
   -3 \mathfig{0.06}{cobordisms/sheet_algebra/handle} \\
 \label{eq:torus-choking-torus}%
 \mathfig{0.065}{cobordisms/sheet_algebra/handle_handle_disc} & = 0 &
  \mathfig{0.065}{cobordisms/sheet_algebra/two_handles} & = 0
\end{align}
The first of these four is the extremely useful ``blister relation."
\end{itemize}

\begin{remark}
\begin{enumerate}
    \item $\Cob{3}$ is generated, as a canopolis, by the cup, cap, saddle, zip, and unzip morphisms below.
        \begin{equation*}
            \mathfig{0.08}{foams/birth} \quad
            \mathfig{0.08}{foams/death} \quad
            \mathfig{0.08}{cobordisms/saddle2} \quad
            \mathfig{0.08}{cobordisms/zip} \quad \mathfig{0.08}{cobordisms/unzip}
        \end{equation*}
    \item As a consequence of the local relations, all closed foams in $\Cob{3}$ can be evaluated to scalars. (Lemma 3.3 in \cite{MorrN:sl(3)})
\end{enumerate}
\end{remark}

As it turns out, $\Cob{3}$ will benefit from slightly more structure. First we'll make it into a graded canopolis by endowing web diagrams with formal grading shifts given by powers of $q$. Further, define the grading of a morphism $C$ from $q^{m_1}D_1$ to $q^{m_2}D_2$ by
\begin{align}
\label{eq:grading}
\deg C = 2\chi(C) - B + \frac{V}{2} + m_2 - m_1
\end{align}
where $B$ is the number of boundary points on $D_i$ and $V$ is the total number of trivalent vertices in the webs $D_1$ and $D_2$. Note that the local relations above are degree homogeneous, and that degree is additive under canopolis composition.

Second, we form $\Mat{\Cob{3}}$ by introducing formal direct sums of objects, and allowing matrices of morphisms between these direct sums. Morrison and Nieh prove that the graded decategorification of $\Mat{\Cob{3}}$  is, in fact, Kuperberg's $\su{3}$ spider.

Finally, for the \emph{coup de grace}, we arrive at the canopolis $\Kom{\Mat{\Cob{3}}}$ by considering chain complexes (up to chain homotopy equivalence) with objects and morphisms in $\Mat{\Cob{3}}$. We'll have to be slightly more explicit about the action of planar arc diagrams now that cans will be associated with complexes and chain maps, rather than just objects and morphisms in $\Cob{3}$. However, the rule is simple: apply the usual construction for tensor product of complexes, but use the planar arc diagram to ``multiply" objects and morphisms instead of $\otimes$. (See Appendix \ref{ssec:planar-comps} for details.) For convenience, let's make the abbreviation $\Kob{3} := \Kom{\Mat{\Cob{3}}}$.


\subsection{A link homology}
\label{ssec:link-homology}

Having defined the relevant canopoleis, Morrison and Nieh proceed to construct a link homology that categorifies the $\su{3}$ quantum link invariant, i.e., a map $Kh(\su{3}): \Ob{\Ortang} \rightarrow \Ob{\Kob{3}}$. Such a map is easily defined on objects in $\Ortang$ by the following categorified skein relations:
\begin{equation*}
\xymatrix@R-1mm{
 \mathfig{0.06}{webs/positive_crossing} \ar@{|->}[r] & \quad \; & \Bigg( \bullet \ar[r]
    & q^2\phantom{-} \mathfig{0.06}{webs/two_strand_identity} \ar[r]^{\mathfig{0.06}{cobordisms/zip}} & q^3 \mathfig{0.06}{webs/upwards_H_diagram} \ar[r] & \bullet \Bigg) \\
 \mathfig{0.06}{webs/negative_crossing} \ar@{|->}[r] & \Bigg( \bullet \ar[r] &
    q^{-3} \mathfig{0.06}{webs/upwards_H_diagram} \ar[r]^{\mathfig{0.06}{cobordisms/unzip}} & q^{-2} \mathfig{0.06}{webs/two_strand_identity} \ar[r] & \bullet \Bigg) &
}
\end{equation*}
The homological heights here are $-2$, $-1$, $0$, $1$, and $2$; the webs with $q^{\pm2}$ shifts lie at height 0 in each case. (Let's also establish the following nomenclature for the webs in this picture: we'll call the ones with $q^{\pm2}$ shifts the \emph{smoothly-resolved} webs for these crossings, and the ones with $q^{\pm3}$ shifts the \emph{I-resolved} webs for these crossings.)

These crossings will compose under planar operations to make larger tangles, as will the associated complexes. One important subtlety is that planar composition of complexes is independent of the order of composition, up to chain isomorphism, but this isomorphism is not the obvious permutation: one needs to sprinkle some minus signs into the permutation to make it a chain map. The upshot is that (1) for well-definition of complexes, all crossings in a tangle diagram must be equipped with an ordering (of course, this is equivalent to the ordering of holes in a planar arc diagram), and (2) there are (slightly) nontrivial chain maps that will reorder the crossings. See Appendix \ref{ssec:sign-conventions} for details.

To complete our map on objects, we need only check that the map $Kh({\su{3}})$ is invariant under isotopy of tangles, i.e., that Reidemeister moves applied to the source tangle do not change the homotopy type of the resulting complex. This is essentially done in \cite{MorrN:sl(3)} by constructing a chain homotopy equivalence for each version of the oriented Reidemeister moves, though we will provide additional details in Section \ref{sec:reidemeister_maps}.

\subsection{A canopolis morphism?}

There is a natural way to define $Kh(\su{3})$ on morphisms (tangle cobordisms) as well. Since a surface in 4-space can be presented by a movie, we can view a cobordism as a sequence of tangles diagrams that, at each stage, differ by a Reidemeister or Morse move. Thus we need only define chain maps for these six generating moves. This is easy: Morse moves induce the obvious gluing of \mbox{0-,} \mbox{1-,} or 2-handles into a foam, and Reidmeister moves already have chain maps defined for the link homology. Again, the heart of the issue here is whether these induced maps are well-defined, i.e., invariant under the movie moves. If they are, then $Kh(\su{3})$ is a canopolis morphism.

\section{Reidemeister maps}             
\label{sec:reidemeister_maps}

\subsection{The Reidemeister one and two maps}

We'll take these (more or less) directly from their definition in \cite{MorrN:sl(3)}, where they are derived and proven to be homotopy equivalences. The Reidemeister one maps are shown in Figures \ref{fig:r1a} and \ref{fig:r1b} for the positive (R1a) and negative (R1b) twist, respectively.

\begin{figure}[h]
$$\mathfig{0.5}{R1+2/R1a/R1a_map}$$
\caption[A homotopy equivalence for $R1$a: the positive twist]{A homotopy equivalence for $R1$a: the positive twist.}\label{fig:r1a}
\end{figure}

\begin{figure}[h]
$$\mathfig{0.5}{R1+2/R1b/R1b_map}$$
\caption[A homotopy equivalence for $R1$b: the negative twist]{A homotopy equivalence for $R1$b: the negative twist.}\label{fig:r1b}
\end{figure}

The Reidemeister two maps come in two flavors, parallel or antiparallel, and the maps are given in Figures \ref{fig:r2a} and \ref{fig:r2b}. Note that changing which strand moves over top does not change our maps, except that we will always use the following ordering convention: the negative crossing is 1, and the positive crossing is 2.

\begin{figure}[h]
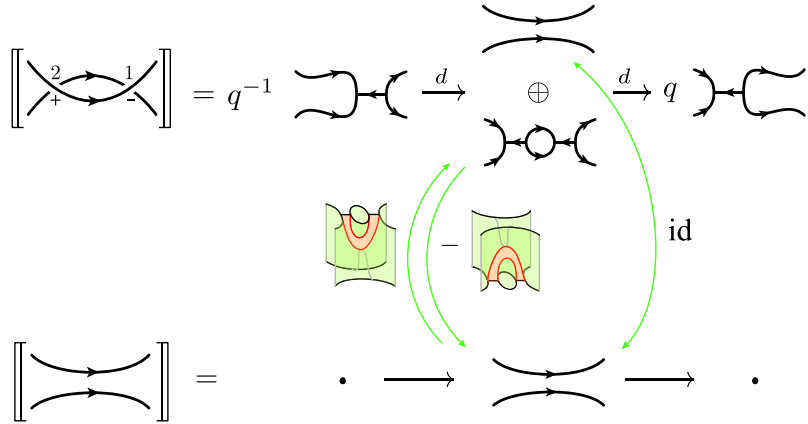

$$\mathfig{0.8}{R1+2/R2a/R2a_map}$$
\caption[A homotopy equivalence for $R2$a: parallel strands]{A homotopy equivalence for $R2$a: parallel strands.}\label{fig:r2a}
\end{figure}

\begin{figure}[h]
$$\mathfig{0.8}{R1+2/R2b/R2b_map}$$
\caption[A homotopy equivalence for $R2$b: antiparallel strands]{A homotopy equivalence for $R2$b: antiparallel strands.}\label{fig:r2b}
\end{figure}

It's also worth noting that our antiparallel map (Figure \ref{fig:r2b}) is $-1$ times the original map in \cite{MorrN:sl(3)}; we're free to multiply any of these maps by a scalar, and some brief experimentation confirms that this particular scalar, in this particular place, is needed for functoriality.

\subsection{The Reidemeister three maps}
\label{sec:R3maps}

Here we have some work to do: in order to compute movie move maps, we'll need to know the $R3$ maps explicitly, for every flavor of the move. The ``categorified Kauffman trick" (CKT) (first used by Bar-Natan in \cite{BN05_MR2174270} and then by Morrison and Nieh in \cite{MorrN:sl(3)}) provides an efficient method for computing the $R3$ maps.

There are eight different versions of the oriented Reidemeister three move, and we'll use them all for the movie move calculations. To use the CKT here, we'll first need to look at some smaller complexes: the ``before" and ``after" complexes of the move that slides a strand past a trivalent vertex. There are eight variations of this move: the vertex can be a sink or a source, and the moving strand can lie on top or below the vertex strands and can be oriented in two possible ways. For convenience, let's name them based on whether the vertex strands point I(n) or O(ut), the crossing strand is A(bove) or B(elow) the vertex strands, and the crossing strand is oriented L(eft) or R(ight). The following lemmas from \cite{MorrN:sl(3)} provide homotopy equivalences\footnote{Note that these maps include both homological and $q$-grading shifts, and so are not completely honest homotopy equivalences of the two sides of the move.} for two of the variations using Bar-Natan's simplification algorithm \cite{BN07_MR2320156}.

\begin{lemma}[IBL variation: Lemma 4.4 in \cite{MorrN:sl(3)}]
\label{lem:IBL}
The complex
{
\newcommand{\pd}[1]{\mathfig{0.08}{strand_past_vertex/IBL/#1}}
\newcommand{\pa}[1]{\mathfig{0.1}{strand_past_vertex/IBL/#1}}
\begin{equation*}
\Kh{\pd{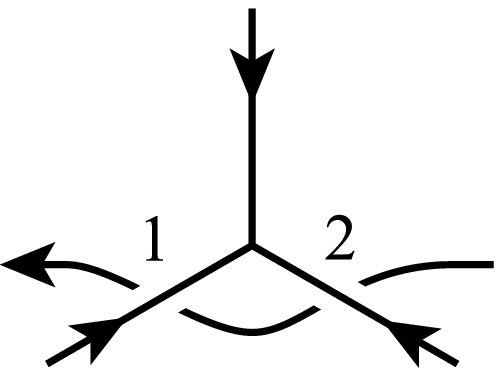}} =
 \left(
  \xymatrix@C+=12mm{
  q^4 \pd{initial1}
    \ar[r]^{\psmallmatrix{\text{z} \\ \text{z}}} &
  \directSumStack{q^5 \pd{initial2}}{q^5 \pd{initial3}}
    \ar[r]^{\psmallmatrix{-\text{z} & \text{z}}} &
  q^6 \pd{initial4}
  }
 \right)
\end{equation*}
is homotopy equivalent to the complex
\begin{equation*}
q^8 \Kh{\pd{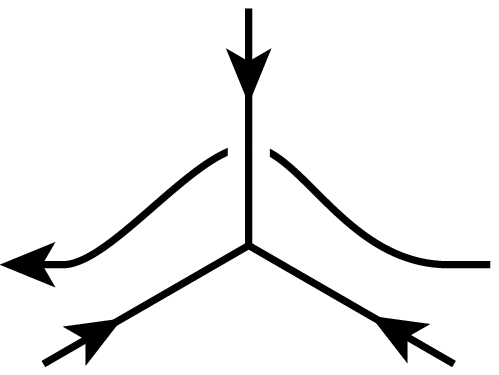}}\shift{+2} =
 \left(
  \xymatrix{
    q^5 \pd{final1} \ar[r]^{\text{u}} &
    q^6 \pd{final2}
  }
 \right)
\end{equation*}
}
via the simplifying map $\s{IBL}$, which separates by homological height into
\begin{align*}
 \sh{IBL}{0} & = \begin{pmatrix} 0 \end{pmatrix} &
 \sh{IBL}{1} & = \begin{pmatrix} - \text{z} \compose \text{d} & 1  \end{pmatrix} &
 \sh{IBL}{2} & = \begin{pmatrix} r \end{pmatrix}.
\end{align*}
Here $\text{d}$ is the debubbling map, $\text{z}$ is the zip map, $\text{u}$ is the unzip map, and $r$ is the ``downward-open
half barrel" cobordism.
\end{lemma}

\begin{remark} For our movie move calculations, we'll also need the inverse (unsimplifying) map $\T{IBL}$, given by
\begin{align*}
 \Th{IBL}{1} & = \begin{pmatrix} -\text{b} \compose \text{u} \\ 1 \end{pmatrix} &
 \Th{IBL}{2} & = \begin{pmatrix} -\bar{r} \end{pmatrix},
\end{align*}
where $\text{b}$ is the bubbling map and $\bar{r}$ is the ``upward-open half barrel."
\end{remark}


These maps are shown in Figure \ref{fig:dbrr}, and their origin is discussed in Appendix \ref{sec:simplifying}.

\begin{figure}[ht]
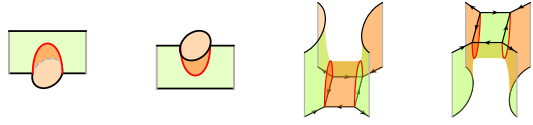

    \begin{equation*}
        \mathfig{0.08}{foams/debubble} \quad \quad
        \mathfig{0.08}{foams/bubble} \quad \quad
        \mathfig{0.08}{cobordisms/rocket_relation/rocket_x_lower} \quad \quad
        \mathfig{0.08}{cobordisms/rocket_relation/rocket_x_upper}
    \end{equation*}
    \caption[The maps $\text{d}$, $\text{b}$, $r$, and $\bar{r}$]{The maps $\text{d}$, $\text{b}$, $r$, and $\bar{r}$.}
    \label{fig:dbrr}
\end{figure}

\begin{lemma}[OBL variation: Lemma 4.5 in \cite{MorrN:sl(3)}]
\label{lem:OBL}
The complex
{
\newcommand{\pd}[1]{\mathfig{0.08}{strand_past_vertex/OBL/#1}}
\newcommand{\pa}[1]{\mathfig{0.1}{strand_past_vertex/OBL/#1}}
\begin{equation*}
\Kh{\pd{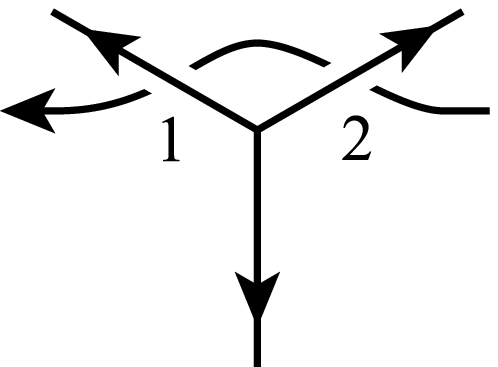}} =
 \left(
  \xymatrix@C+=12mm{
  q^4 \pd{initial1}
    \ar[r]^{\psmallmatrix{\text{z} \\ \text{z}}} &
  \directSumStack{q^5 \pd{initial2}}{q^5 \pd{initial3}}
    \ar[r]^{\psmallmatrix{\text{z} & -\text{z}}} &
  q^6 \pd{initial4}
  }
 \right)
\end{equation*}
is homotopy equivalent to the complex
\begin{equation*}
q^8 \Kh{\pd{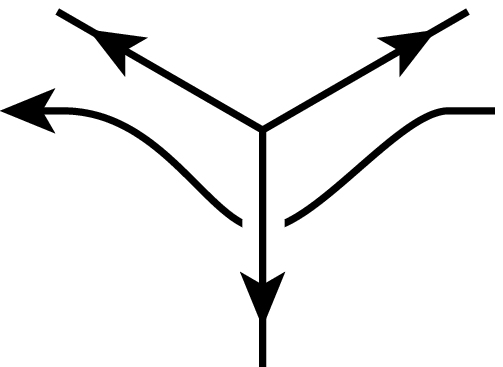}}\shift{+2} =
 \left(
  \xymatrix{
    q^5 \pd{final1} \ar[r]^{\text{u}} &
    q^6 \pd{final2}
  }
 \right)
\end{equation*}
}
via the simplifying map $\s{OBL}$, given by
\begin{align*}
 \sh{OBL}{0} & = \begin{pmatrix} 0 \end{pmatrix} &
 \sh{OBL}{1} & = \begin{pmatrix} \text{z} \compose \text{d} & -1 \end{pmatrix} &
 \sh{OBL}{2} & = \begin{pmatrix} r \end{pmatrix}.
\end{align*}
\end{lemma}


\begin{remark} Again we'll need the inverse map $\T{OBL}$, which is given by
\begin{align*}
 \Th{OBL}{1} & = \begin{pmatrix} \text{b} \compose \text{u} \\ -1 \end{pmatrix} &
 \Th{OBL}{2} & = \begin{pmatrix} -\bar{r} \end{pmatrix}.
\end{align*}
\end{remark}

Explicit homotopy equivalences for the other six variations of the ``strand-past-vertex" move are given below.

\begin{lemma}
\label{lem:spv}
\begin{enumerate}
\item IBR: The complex
{
\newcommand{\pd}[1]{\mathfig{0.08}{strand_past_vertex/IBR/#1}}
\newcommand{\pa}[1]{\mathfig{0.1}{strand_past_vertex/IBR/#1}}
\begin{equation*}
\label{eq:IBR}
\Kh{\pd{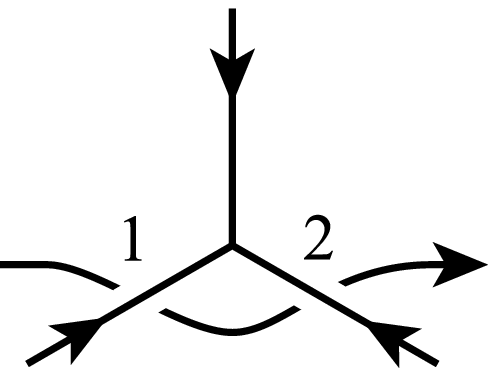}} =
 \left(
  \xymatrix@C+=12mm{
  q^{-6} \pd{initial1}
    \ar[r]^{\psmallmatrix{\text{u} \\ -\text{u}}} &
  \directSumStack{q^{-5} \pd{initial2}}{q^{-5} \pd{initial3}}
    \ar[r]^{\psmallmatrix{\text{u} & \text{u}}} &
  q^{-4} \pd{initial4}
  }
 \right)
\end{equation*}
is homotopy equivalent to the complex
\begin{equation*}
q^{-8} \Kh{\pd{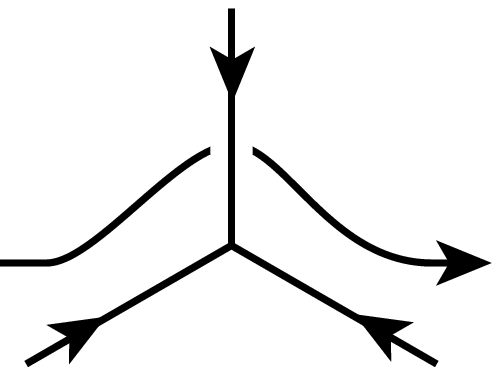}}\shift{-2} =
 \left(
  \xymatrix{
    q^{-6} \pd{final1} \ar[r]^{\text{u}} &
    q^{-5} \pd{final2}
  }
 \right)
\end{equation*}
}
via the simplifying map $\s{IBR}$, given by
\begin{align*}
 \sh{IBR}{-2} & = \begin{pmatrix} r \end{pmatrix} &
 \sh{IBR}{-1} & = \begin{pmatrix} -\text{z} \compose \text{d} & 1 \end{pmatrix} &
 \sh{IBR}{0} & = \begin{pmatrix} 0 \end{pmatrix}.
\end{align*}

\noindent The inverse map $\T{IBR}$ is given by
\begin{align*}
 \Th{IBR}{-2} & = \begin{pmatrix} -\bar{r} \end{pmatrix} &
 \Th{IBR}{-1} & = \begin{pmatrix} -\text{b} \compose \text{u} \\ 1 \end{pmatrix}.
\end{align*}


\item OBR: The complex
{
\newcommand{\pd}[1]{\mathfig{0.08}{strand_past_vertex/OBR/#1}}
\newcommand{\pa}[1]{\mathfig{0.1}{strand_past_vertex/OBR/#1}}
\begin{equation*}
\label{eq:OBR}
\Kh{\pd{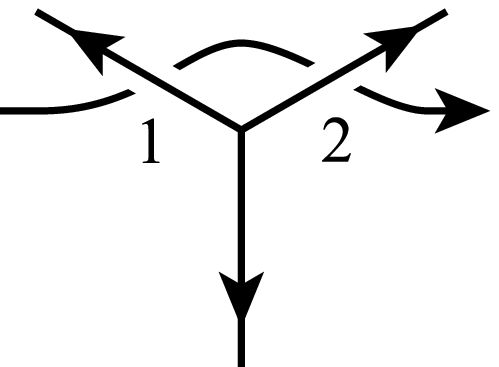}} =
 \left(
  \xymatrix@C+=12mm{
  q^{-6} \pd{initial1}
    \ar[r]^{\psmallmatrix{-\text{u} \\ \text{u}}} &
  \directSumStack{q^{-5} \pd{initial2}}{q^{-5} \pd{initial3}}
    \ar[r]^{\psmallmatrix{\text{u} & \text{u}}} &
  q^{-4} \pd{initial4}
  }
 \right)
\end{equation*}
is homotopy equivalent to the complex
\begin{equation*}
q^{-8} \Kh{\pd{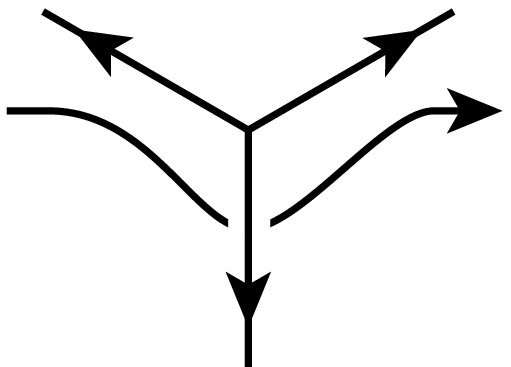}}\shift{-2} =
 \left(
  \xymatrix{
    q^{-6} \pd{final1} \ar[r]^{\text{u}} &
    q^{-5} \pd{final2}
  }
 \right)
\end{equation*}
}
via the simplifying map $\s{OBR}$, given by
\begin{align*}
 \sh{OBR}{-2} & = \begin{pmatrix} r \end{pmatrix} &
 \sh{OBR}{-1} & = \begin{pmatrix}  \text{z} \compose \text{d} & -1 \end{pmatrix} &
 \sh{OBR}{0} & = \begin{pmatrix} 0 \end{pmatrix}.
\end{align*}

\noindent The inverse map $\T{OBR}$ is given by
\begin{align*}
 \Th{OBR}{-2} & = \begin{pmatrix} -\bar{r} \end{pmatrix} &
 \Th{OBR}{-1} & = \begin{pmatrix} \text{b} \compose \text{u} \\ -1 \end{pmatrix}.
\end{align*}


\item IAR: The complex
{
\newcommand{\pd}[1]{\mathfig{0.08}{strand_past_vertex/IAR/#1}}
\newcommand{\pa}[1]{\mathfig{0.1}{strand_past_vertex/IAR/#1}}
\begin{equation*}
\label{eq:IAR}
\Kh{\pd{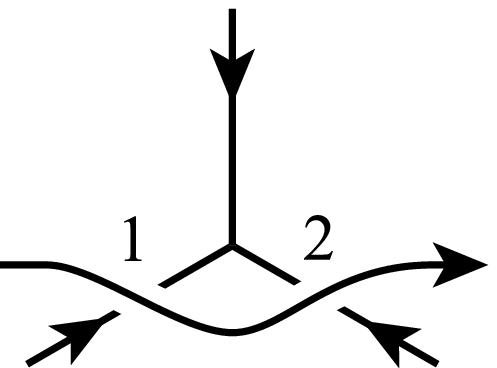}} =
 \left(
  \xymatrix@C+=12mm{
  q^4 \pd{initial1}
    \ar[r]^{\psmallmatrix{\text{z} \\ \text{z}}} &
  \directSumStack{q^5 \pd{initial2}}{q^5 \pd{initial3}}
    \ar[r]^{\psmallmatrix{\text{z} & -\text{z}}} &
  q^6 \pd{initial4}
  }
 \right)
\end{equation*}
is homotopy equivalent to the complex
\begin{equation*}
q^8 \Kh{\pd{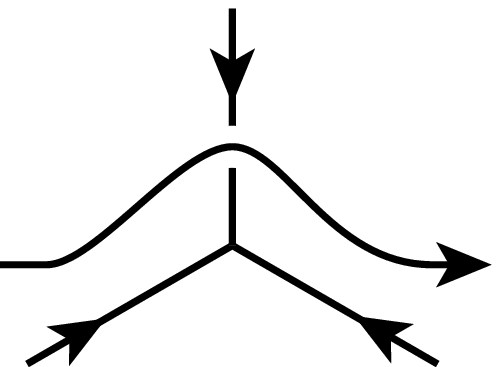}}\shift{+2} =
 \left(
  \xymatrix{
    q^5 \pd{final1} \ar[r]^{\text{u}} &
    q^6 \pd{final2}
  }
 \right)
\end{equation*}
}
via the simplifying map $\s{IAR}$, given by
\begin{align*}
 \sh{IAR}{0} & = \begin{pmatrix} 0 \end{pmatrix} &
 \sh{IAR}{1} & = \begin{pmatrix} \text{z} \compose \text{d} & -1 \end{pmatrix} &
 \sh{IAR}{2} & = \begin{pmatrix} r \end{pmatrix}.
\end{align*}

\noindent The inverse map $\T{IAR}$ is given by
\begin{align*}
 \Th{IAR}{1} & = \begin{pmatrix} \text{b} \compose \text{u} \\ -1 \end{pmatrix} &
 \Th{IAR}{2} & = \begin{pmatrix} -\bar{r} \end{pmatrix}.
\end{align*}


\item OAR: The complex
{
\newcommand{\pd}[1]{\mathfig{0.08}{strand_past_vertex/OAR/#1}}
\newcommand{\pa}[1]{\mathfig{0.1}{strand_past_vertex/OAR/#1}}
\begin{equation*}
\label{eq:OAR}
\Kh{\pd{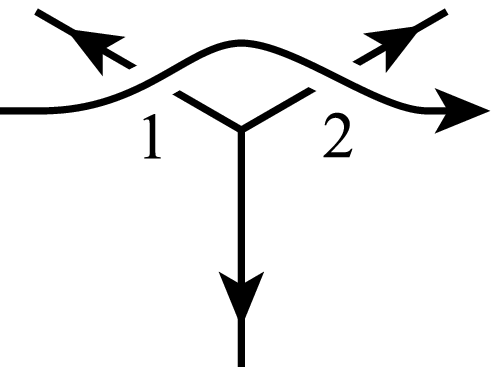}} =
 \left(
  \xymatrix@C+=12mm{
  q^4 \pd{initial1}
    \ar[r]^{\psmallmatrix{\text{z} \\ \text{z}}} &
  \directSumStack{q^5 \pd{initial2}}{q^5 \pd{initial3}}
    \ar[r]^{\psmallmatrix{-\text{z} & \text{z}}} &
  q^6 \pd{initial4}
  }
 \right)
\end{equation*}
is homotopy equivalent to the complex
\begin{equation*}
q^8 \Kh{\pd{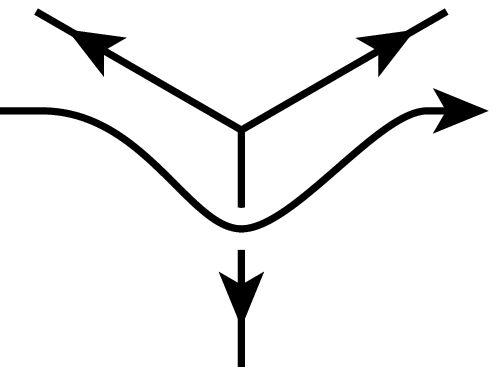}}\shift{+2} =
 \left(
  \xymatrix{
    q^5 \pd{final1} \ar[r]^{\text{u}} &
    q^6 \pd{final2}
  }
 \right)
\end{equation*}
}
via the simplifying map $\s{OAR}$, given by
\begin{align*}
 \sh{OAR}{0} & = \begin{pmatrix} 0 \end{pmatrix} &
 \sh{OAR}{1} & = \begin{pmatrix} - \text{z} \compose \text{d} & 1  \end{pmatrix} &
 \sh{OAR}{2} & = \begin{pmatrix} r \end{pmatrix}.
\end{align*}

\noindent The inverse map $\T{OAR}$ is given by
\begin{align*}
 \Th{OAR}{1} & = \begin{pmatrix} -\text{b} \compose \text{u} \\ 1 \end{pmatrix} &
 \Th{OAR}{2} & = \begin{pmatrix} -\bar{r} \end{pmatrix}.
\end{align*}


\item IAL: The complex
{
\newcommand{\pd}[1]{\mathfig{0.08}{strand_past_vertex/IAL/#1}}
\newcommand{\pa}[1]{\mathfig{0.1}{strand_past_vertex/IAL/#1}}
\begin{equation*}
\label{eq:IAL}
\Kh{\pd{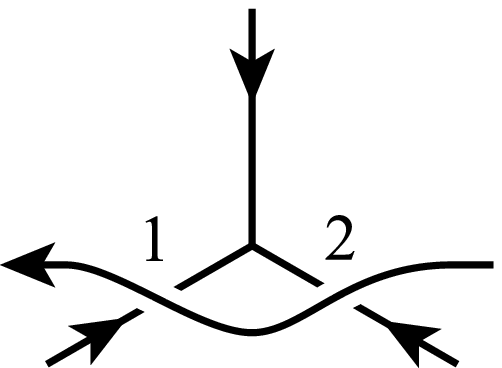}} =
 \left(
  \xymatrix@C+=12mm{
  q^{-6} \pd{initial1}
    \ar[r]^{\psmallmatrix{-\text{u} \\ \text{u}}} &
  \directSumStack{q^{-5} \pd{initial2}}{q^{-5} \pd{initial3}}
    \ar[r]^{\psmallmatrix{\text{u} & \text{u}}} &
  q^{-4} \pd{initial4}
  }
 \right)
\end{equation*}
is homotopy equivalent to the complex
\begin{equation*}
q^{-8} \Kh{\pd{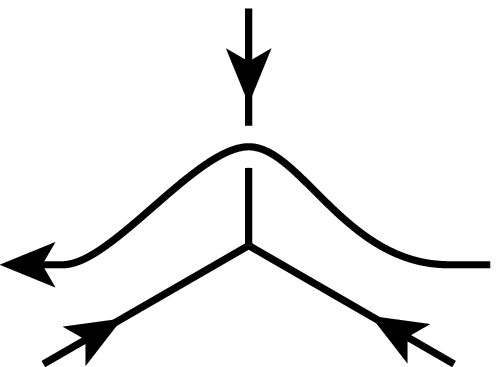}}\shift{-2} =
 \left(
  \xymatrix{
    q^{-6} \pd{final1} \ar[r]^{\text{u}} &
    q^{-5} \pd{final2}
  }
 \right)
\end{equation*}
}
via the simplifying map $\s{IAL}$, given by
\begin{align*}
 \sh{IAL}{-2} & = \begin{pmatrix} r \end{pmatrix} &
 \sh{IAL}{-1} & = \begin{pmatrix}  \text{z} \compose \text{d} & -1 \end{pmatrix} &
 \sh{IAL}{0} & = \begin{pmatrix} 0 \end{pmatrix}.
\end{align*}

\noindent The inverse map $\T{IAL}$ is given by
\begin{align*}
 \Th{IAL}{-2} & = \begin{pmatrix} -\bar{r} \end{pmatrix} &
 \Th{IAL}{-1} & = \begin{pmatrix} \text{b} \compose \text{u} \\ -1 \end{pmatrix}.
\end{align*}


\item OAL: The complex
{
\newcommand{\pd}[1]{\mathfig{0.08}{strand_past_vertex/OAL/#1}}
\newcommand{\pa}[1]{\mathfig{0.1}{strand_past_vertex/OAL/#1}}
\begin{equation*}
\label{eq:OAL}
\Kh{\pd{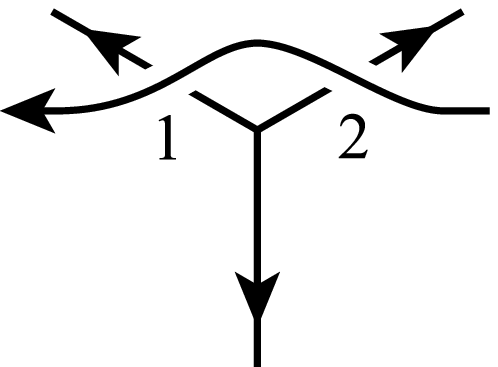}} =
 \left(
  \xymatrix@C+=12mm{
  q^{-6} \pd{initial1}
    \ar[r]^{\psmallmatrix{\text{u} \\ -\text{u}}} &
  \directSumStack{q^{-5} \pd{initial2}}{q^{-5} \pd{initial3}}
    \ar[r]^{\psmallmatrix{\text{u} & \text{u}}} &
  q^{-4} \pd{initial4}
  }
 \right)
\end{equation*}
is homotopy equivalent to the complex
\begin{equation*}
q^{-8} \Kh{\pd{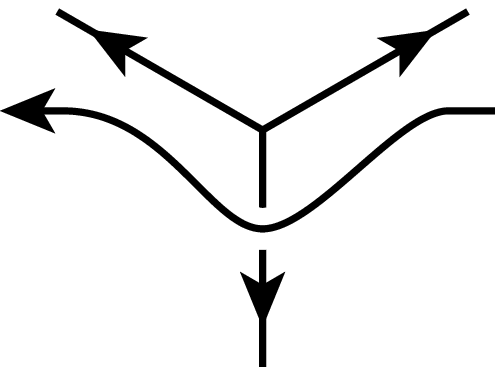}}\shift{-2} =
 \left(
  \xymatrix{
    q^{-6} \pd{final1} \ar[r]^{\text{u}} &
    q^{-5} \pd{final2}
  }
 \right)
\end{equation*}
}
via the simplifying map $\s{OAL}$, given by
\begin{align*}
 \sh{OAL}{-2} & = \begin{pmatrix} r \end{pmatrix} &
 \sh{OAL}{-1} & = \begin{pmatrix} -\text{z} \compose \text{d} & 1 \end{pmatrix} &
 \sh{OAL}{0} & = \begin{pmatrix} 0 \end{pmatrix}.
\end{align*}

\noindent The inverse map $\T{OAL}$ is given by
\begin{align*}
 \Th{OAL}{-2} & = \begin{pmatrix} -\bar{r} \end{pmatrix} &
 \Th{OAL}{-1} & = \begin{pmatrix} -\text{b} \compose \text{u} \\ 1 \end{pmatrix}.
\end{align*}

\end{enumerate}

\end{lemma}

\begin{proof}

See Appendix \ref{ssec:SPV-moves}.

\end{proof}

\begin{remark}
Notice that, modulo orientations, reflections, and rotations, the above complexes are very similar. This leads us to make the following observations, which we'll use when computing the $R3$ moves:
\begin{enumerate}
\item For moves IBL, OBL, IAR, and OAR, the \emph{lowest} homological component of the $s$ map, which originates at the doubly smoothly-resolved object, is zero; the highest component is the half-barrel $r$.
\item For moves IBR, OBR, IAL, and OAL, the \emph{highest} homological component of the $s$ map, which originates at the doubly I-resolved object, is zero; the lowest component is $r$.
\end{enumerate}
\end{remark}

Before explicitly computing the eight $R3$ maps, we'll need some basic results from homological algebra.

\begin{defn}
Given a chain map $f : A^{\bullet} \rightarrow B^{\bullet}$, the cone over $f$ is $C(f)^{\bullet} = A^{\bullet + 1} \oplus B^{\bullet}$, with differential
$$ d_{C(f)} =
\begin{pmatrix}
d_A & 0 \\
f & -d_B \end{pmatrix}$$
\end{defn}

\begin{defn}
\label{def:simple}
A map $r:B^\bullet \rightarrow C^\bullet$ is a \emph{strong deformation retract}\footnote{In \cite{MorrN:sl(3)} this is called a \emph{simple homotopy equivalence}.} with inverse $i$ if
\begin{itemize}
\item $\Id_B - i r = d_B h + h d_B$;
\item $\Id_C = r i$; and
\item $h i = r h = 0$,
\end{itemize}
where $h:B^\bullet \rightarrow B^{\bullet - 1}$.
\end{defn}

\begin{remark}
Each $s$ map above is a strong deformation retract with inverse $t$; see Appendix \ref{sec:simplifying}.
\end{remark}

The following two lemmas about cones were first presented and proven in \cite{BN05_MR2174270}; we'll refer to them as ``cone-reducing" lemmas:

\begin{lemma} \label{lem:f-to-rf} If $f : A^{\bullet} \rightarrow B^{\bullet}$ is a chain map, $r : B^{\bullet} \rightarrow C^{\bullet}$ is a strong deformation retract, and $i : C^{\bullet} \rightarrow B^{\bullet}$ is the inverse of $r$ via the homotopy $h$, then the cone $C(rf)$ is homotopic to the cone $C(f)$, via
$$
\xymatrix{ C(f)^{\bullet} = A^{\bullet + 1} \oplus B^{\bullet} \ar@/_/[rr]_{\psmallmatrix{1 & 0 \\ 0 & r}} & & A^{\bullet + 1} \oplus C^{\bullet} =  C(rf)^{\bullet} \ar@/_/[ll]_{\psmallmatrix{1 & 0 \\ -hf & i}}}
$$
\end{lemma}

\begin{lemma} \label{lem:f-to-fi} If $f : B^{\bullet} \rightarrow A^{\bullet}$ is a chain map, $r : B^{\bullet} \rightarrow C^{\bullet}$ is a strong deformation retract, and $i : C^{\bullet} \rightarrow B^{\bullet}$ is the inverse of $r$ via the homotopy $h$, then the cone $C(fi)$ is homotopic to the cone $C(f)$, via
$$
\xymatrix{ C(f)^{\bullet} = B^{\bullet + 1} \oplus A^{\bullet} \ar@/_/[rr]_{\psmallmatrix{r & 0 \\ fh & 1}} & & C^{\bullet + 1} \oplus A^{\bullet} =  C(fi)^{\bullet} \ar@/_/[ll]_{\psmallmatrix{i & 0 \\ 0 & 1}}}
$$
\end{lemma}

We'll also need these two, which are proven analogously:

\begin{lemma} \label{lem:f-to-fr} If $f : C^{\bullet} \rightarrow A^{\bullet}$ is a chain map, $r : B^{\bullet} \rightarrow C^{\bullet}$ is a strong deformation retract, and $i : C^{\bullet} \rightarrow B^{\bullet}$ is the inverse of $r$ via the homotopy $h$, then the cone $C(fr)$ is homotopic to the cone $C(f)$, via

$$
\xymatrix{ C(f)^{\bullet} = C^{\bullet + 1} \oplus A^{\bullet} \ar@/_/[rr]_{\psmallmatrix{i & 0 \\ 0 & 1}} & & B^{\bullet + 1} \oplus A^{\bullet} =  C(fr)^{\bullet} \ar@/_/[ll]_{\psmallmatrix{r & 0 \\ 0 & 1}}}
$$
\end{lemma}

\begin{lemma} \label{lem:f-to-if} If $f : A^{\bullet} \rightarrow C^{\bullet}$ is a chain map, $r : B^{\bullet} \rightarrow C^{\bullet}$ is a strong deformation retract, and $i : C^{\bullet} \rightarrow B^{\bullet}$ is the inverse of $r$ via the homotopy $h$, then the cone $C(if)$ is homotopic to the cone $C(f)$, via

$$
\xymatrix{ C(f)^{\bullet} = A^{\bullet + 1} \oplus C^{\bullet} \ar@/_/[rr]_{\psmallmatrix{1 & 0 \\ 0 & i}} & & A^{\bullet + 1} \oplus B^{\bullet} =  C(if)^{\bullet} \ar@/_/[ll]_{\psmallmatrix{1 & 0 \\ 0 & r}}}
$$
\end{lemma}

Now we're ready to attack the $R3$ complexes themselves. First, let's name the eight variations, six of which are braidlike and two of which are starlike. As in \cite{ClarMW:Fix}, we'll label the braidlike moves by circling anticlockwise around the tangle boundary and recording the height of each outgoing strand ($h$ for high, $m$ for middle, and $l$ for low). The starlike moves are labeled either clockwise or anticlockwise, depending on which way we have to circle to see the low, middle, and then high outgoing strands. We also need to pick a time direction for each move, and will use the convention that the ``before" diagram has a crossing to the right of the low strand, while the ``after" diagram has a crossing to the left. All of these labels and conventions are shown in Figure \ref{fig:R3-variations}.

\begin{figure}[ht]
\newcommand{\RIII}[1]{\xymatrix@C+20pt{%
  \mathfig{0.15}{R3/variations/R3_#1_1} \ar@<0.5ex>[r]^{R3_{#1}} & \mathfig{0.15}{R3/variations/R3_#1_2} \ar@<0.5ex>[l]^{R3_{#1}^{-1}}%
}}%
\begin{align*}
\RIII{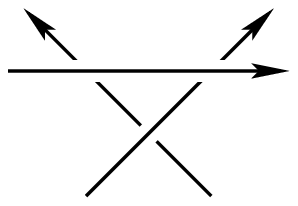} & & \RIII{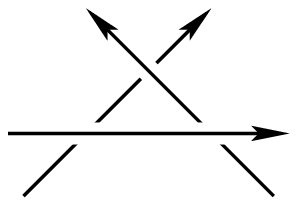} & \\
\RIII{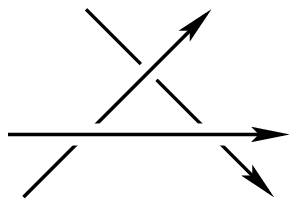} & & \RIII{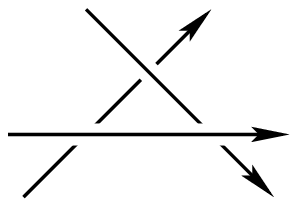} & \\
\RIII{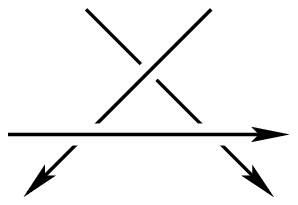} & & \RIII{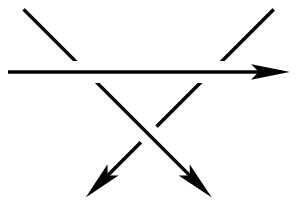} & \\
\xymatrix@C+20pt{%
  \mathfig{0.15}{R3/variations/R3_anti_1} \ar@<0.5ex>[r]^{R3_\circlearrowleft} & \mathfig{0.15}{R3/variations/R3_anti_2} \ar@<0.5ex>[l]^{R3_\circlearrowleft^{-1}}%
} & &
\xymatrix@C+20pt{%
  \mathfig{0.15}{R3/variations/R3_clock_1} \ar@<0.5ex>[r]^{R3_\circlearrowright} & \mathfig{0.15}{R3/variations/R3_clock_2} \ar@<0.5ex>[l]^{R3_\circlearrowright^{-1}} %
}
 &
\end{align*}
\caption[The eight variations of the $R3$ move]{The eight variations of the $R3$ move.}
\label{fig:R3-variations}
\end{figure}

The CKT works by decomposing the nine-object ``before" and ``after" complexes of an $R3$ move as cones over the local differential for a particular crossing. For the time being, let's take this to be the highest crossing. It's important at this point to introduce another set of conventions: the way in which we order crossings. For variations $hml$, $lmh$, $mlh$, and $\cw$, we'll use the following ordering: in the initial tangle the crossings will be ordered `middle', `low', `high', while in the final tangle they will be ordered `low', `middle', `high'. For variations $lhm$, $mhl$, $hlm$, $\ccw$, we'll instead use the inverse ordering: the crossings of the initial tangle will be ordered `low', `middle', `high', and of the final tangle, `middle', `low', `high'. \footnote{This ordering convention is more cumbersome still than the one used in \cite{ClarMW:Fix}, and even worse must be altered when we resolve the low crossing instead of the high. It is a necessary evil, though, as the alert reader may notice as we work though the CKT for the different variations.}

Consider, for example, the initial complex of the $hml$ move:
$$ \Kh{\mathfig{0.1}{\CKT hml/hml_initial}} \iso
    \cone{\mathfig{0.08}{\CKT hml/cone_source_1}}{\text{z}^{\text{above}}}{\mathfig{0.08}{\CKT hml/I_braiding_1}},$$
where $\text{z}^{\text{above}}$ is the zip differential for the high crossing. Morrison and Nieh used this decomposition, as well as the one for the final complex of the $hml$ move, to show that the two complexes were homotopy equivalent. We will restate their argument from \cite{MorrN:sl(3)} here, while fleshing out some more details to give us an explicit map.

\begin{remark}
The initial and final tangles in the Propositions below may be rotated relative to their definitions in Figure \ref{fig:R3-variations}, for convenience.
\end{remark}

First we need an easy lemma.

\begin{lemma}
\label{lem:IBL-OBL-equal}%
The two compositions
\begin{align*}
\xymatrix{
 \mathfig{0.08}{\CKT hml/cone_source_1} \ar[r]^{\text{z}} &
 \mathfig{0.08}{\CKT hml/I_braiding_1} \ar[r]^{\s{IBL}} &
 \mathfig{0.08}{\CKT hml/I_braiding_2}
}%
\intertext{and}%
\xymatrix{
 \mathfig{0.08}{\CKT hml/cone_source_2} \ar[r]^{\text{z}} &
 \mathfig{0.08}{\CKT hml/I_braiding_3} \ar[r]^{\s{OBL}} &
 \mathfig{0.08}{\CKT hml/I_braiding_2}
},
\end{align*}
using the maps defined in Lemmas \ref{lem:OBL} and \ref{lem:IBL}, are equal.
\end{lemma}
\begin{proof}
This is a straightforward, object-by-object comparison; there is only foam isotopy involved---no foam relations are necessary.
\end{proof}

\begin{prop}[The $hml$ variation of $R3$]
\label{prop:R3-v1}
$$\Kh{\mathfig{0.1}{\CKT hml/hml_initial}}
\iso \vcone{\mathfig{0.08}{\CKT hml/cone_source_1}}{\text{z}^{\text{above}}}{\mathfig{0.08}{\CKT hml/I_braiding_1}}
 \xrightarrow[f_{\text{hml}}]{\htpy} \vcone{\mathfig{0.08}{\CKT hml/cone_source_2}}{\text{z}^{\text{below}}}{\mathfig{0.08}{\CKT hml/I_braiding_3}} \iso \Kh{\mathfig{0.1}{\CKT hml/hml_final}}$$
is a homotopy equivalence via the map
$$f_{\text{hml}} = \begin{pmatrix} 1 & 0 \\ -\h{OBL} \compose \text{z} & \T{OBL} \compose \s{IBL} \end{pmatrix}.$$
The homotopy inverse of this map is given by
$$g_{\text{hml}} = \begin{pmatrix} 1 & 0 \\ -\h{IBL} \compose \text{z} & \T{IBL} \compose \s{OBL} \end{pmatrix}.$$
The maps $\h{OBL}$ and $\h{IBL}$ are just the homotopies for the simplifications of the OBL and IBL complexes; we won't compute them explicitly.
\end{prop}

\begin{proof}
We shall follow through the composition one piece at a time.
\begin{align*}
\cone{\mathfig{0.08}{\CKT hml/cone_source_1}}{\text{z}}{\mathfig{0.08}{\CKT hml/I_braiding_1}}
        & \htpy \cone{\mathfig{0.08}{\CKT hml/cone_source_1}}{\mathrlap{\s{IBL}\compose \text{z}} \phantom{\s{OBL}\compose \text{z}}}{\mathfig{0.08}{\CKT hml/I_braiding_2}} \\
        & = \cone{\mathfig{0.08}{\CKT hml/cone_source_2}}{\s{OBL} \compose \text{z}}{\mathfig{0.08}{\CKT hml/I_braiding_2}} \\
        & \htpy \cone{\mathfig{0.08}{\CKT hml/cone_source_2}}{\mathrlap{\quad \; \text{z}} \phantom{\s{OBL}\compose \text{z}}}{\mathfig{0.08}{\CKT hml/I_braiding_3}}
\end{align*}
The homotopy equivalences on the first and last lines follow from Lemma \ref{lem:f-to-rf}, and are given by the matrices $\begin{pmatrix} 1 & 0 \\ 0 & \s{IBL} \end{pmatrix}$ and $\begin{pmatrix} 1 & 0 \\ -\h{OBL} \compose \text{z} & \T{IBL} \end{pmatrix}$. Equality on the second line is exactly Lemma \ref{lem:IBL-OBL-equal}.
\end{proof}

We'll now determine the homotopy equivalences for the other seven $R3$ variations. The techniques for $lhm$, $mhl$, and $lmh$ are essentially the same as for $hml$, and we will omit the details of the proofs. The other four moves will require some modification.

\begin{prop}
\label{prop:R3-v2,3,4}
\begin{enumerate}
\item The $lhm$ variation.
$$\Kh{\mathfig{0.1}{\CKT lhm/lhm_initial}}
\iso \vcone{\mathfig{0.08}{\CKT lhm/cone_source_1}}{\text{z}^{\text{above}}}{\mathfig{0.08}{\CKT lhm/I_braiding_1}}
 \xrightarrow[f_{\text{lhm}}]{\htpy} \vcone{\mathfig{0.08}{\CKT lhm/cone_source_2}}{\text{z}^{\text{below}}}{\mathfig{0.08}{\CKT lhm/I_braiding_3}} \iso \Kh{\mathfig{0.1}{\CKT lhm/lhm_final}}$$
is a homotopy equivalence via the map
$$f_{\text{lhm}} = \begin{pmatrix} 1 & 0 \\ -\h{IBR} \compose \text{z} & \T{IBR} \compose \s{OBR} \end{pmatrix}.$$
The homotopy inverse of this map is given by
$$g_{\text{lhm}} = \begin{pmatrix} 1 & 0 \\ -\h{OBR} \compose \text{z} & \T{OBR} \compose \s{IBR} \end{pmatrix}.$$

\item The $mhl$ variation.
$$\Kh{\mathfig{0.1}{\CKT mhl/mhl_initial}}
\iso \vcone{\mathfig{0.08}{\CKT mhl/I_braiding_1}}{\text{u}^{\text{above}}}{\mathfig{0.08}{\CKT mhl/cone_source_1}}
 \xrightarrow[f_{\text{mhl}}]{\htpy} \vcone{\mathfig{0.08}{\CKT mhl/I_braiding_3}}{\text{u}^{\text{below}}}{\mathfig{0.08}{\CKT mhl/cone_source_2}} \iso \Kh{\mathfig{0.1}{\CKT mhl/mhl_final}}$$
is a homotopy equivalence via the map
$$f_{\text{mhl}} = \begin{pmatrix} \T{OBL} \compose \s{IBL} & 0 \\ \text{u} \compose \h{IBL} & 1 \end{pmatrix}.$$
The homotopy inverse of this map is given by
$$g_{\text{mhl}} = \begin{pmatrix} \T{IBL} \compose \s{OBL} & 0 \\ \text{u} \compose \h{OBL} & 1 \end{pmatrix}.$$

\item The $lmh$ variation.
$$\Kh{\mathfig{0.1}{\CKT lmh/lmh_initial}}
\iso \vcone{\mathfig{0.08}{\CKT lmh/I_braiding_1}}{\text{u}^{\text{above}}}{\mathfig{0.08}{\CKT lmh/cone_source_1}}
 \xrightarrow[f_{\text{lmh}}]{\htpy} \vcone{\mathfig{0.08}{\CKT lmh/I_braiding_3}}{\text{u}^{\text{below}}}{\mathfig{0.08}{\CKT lmh/cone_source_2}} \iso \Kh{\mathfig{0.1}{\CKT lmh/lmh_final}}$$
is a homotopy equivalence via the map
$$f_{\text{lmh}} = \begin{pmatrix} \T{IBR} \compose \s{OBR} & 0 \\ \text{u} \compose \h{OBR} & 1 \end{pmatrix}.$$
The homotopy inverse of this map is given by
$$g_{\text{lmh}} = \begin{pmatrix} \T{OBR} \compose \s{IBR} & 0 \\ \text{u} \compose \h{IBR} & 1 \end{pmatrix}.$$
\end{enumerate}
\end{prop}

\begin{remark} Notice that, for the $mhl$ and $lmh$ variations, the interesting partial homotopy equivalences occur at the source complexes in the cones, rather than the target complexes. Thus the relevant cone-reducing lemma here is Lemma \ref{lem:f-to-fi}.
\end{remark}

For the $hlm$, $mlh$, $\ccw$, and $\cw$ variations of $R3$, the high crossing resolves parallel to the low strand, rather than perpendicular to it, which makes the cone slightly more complicated. Consider, for example, the complexes in the $\cw$ move. The initial complex is
$$ \Kh{\mathfig{0.1}{\CKT cw/cw_initial}} \iso
    \cone{\mathfig{0.08}{\CKT cw/cone_source_1}}{\text{z}^{\text{above}}}{\mathfig{0.08}{\CKT cw/H_braiding_1}},$$
with final complex
$$ \Kh{\mathfig{0.1}{\CKT cw/cw_final}} \iso
    \cone{\mathfig{0.08}{\CKT cw/cone_source_3}}{\text{z}^{\text{below}}}{\mathfig{0.08}{\CKT cw/H_braiding_3}}.$$
Here neither the source nor target complexes in the cones are the same; instead, the source complexes are related by two $R2$ moves, and the target complexes by a different sequence of strand-past-vertex moves. This will ultimately make our string of homotopy equivalences longer, but the idea is essentially the same. As before, we'll give the gory details in only one of the cases.

First, however, we'll need a statement analogous to Lemma \ref{lem:IBL-OBL-equal}.

\begin{lemma}
\label{lem:hbraid-equal}
The two compositions
\begin{align*}
\xymatrix{
 \mathfig{0.08}{\CKT cw/cone_source_2} \ar[r]^{\rho_1} &
 \mathfig{0.08}{\CKT cw/cone_source_1} \ar[r]^{\text{z}} &
 \mathfig{0.08}{\CKT cw/H_braiding_1} \ar[r]^{\T{IBR}} &
 \mathfig{0.08}{\CKT cw/H_braiding_2}
}%
\intertext{and}%
\xymatrix{
 \mathfig{0.08}{\CKT cw/cone_source_2} \ar[r]^{\rho_2} &
 \mathfig{0.08}{\CKT cw/cone_source_3} \ar[r]^{\sigma} &
 \mathfig{0.08}{\CKT cw/cone_source_3r} \ar[r]^{\text{z}} &
 \mathfig{0.08}{\CKT cw/H_braiding_3} \ar[r]^{\T{OBR}} &
 \mathfig{0.08}{\CKT cw/H_braiding_2}
},
\end{align*}
are equal. Here, $\rho_1$ and $\rho_2$ are the $R2$ tuck maps on the bottom two strands and on the top two strands, respectively, and $\sigma$ is the obvious crossing-re\-ordering map.
\end{lemma}
\begin{proof}
Again, as in the proof of Lemma \ref{lem:IBL-OBL-equal}, this is a straightforward exercise in foam isotopy. Just remember that $\sigma$ is the identity on all objects except the doubly I-resolved ones, on which it acts by $-1$.
\end{proof}

\begin{prop}[The $\cw$ version of $R3$]
\label{prop:R3-v5}
$$\Kh{\mathfig{0.1}{\CKT cw/cw_initial}}
\iso \vcone{\mathfig{0.08}{\CKT cw/cone_source_1}}{\text{z}^{\text{above}}}{\mathfig{0.08}{\CKT cw/H_braiding_1}}
 \xrightarrow[f_{\cw}]{\htpy} \vcone{\mathfig{0.08}{\CKT cw/cone_source_3r}}{\text{z}^{\text{below}}}{\mathfig{0.08}{\CKT cw/H_braiding_3}} \iso \Kh{\mathfig{0.1}{\CKT cw/cw_final}}$$
is a homotopy equivalence via the map
$$f_{\cw} = \begin{pmatrix} \sigma \compose \rho_2 \compose \rho_1^{-1} & 0 \\ \s{OBR} \compose \T{IBR} \compose \text{z} \compose h_1 & \s{OBR} \compose \T{IBR} \end{pmatrix}.$$
The homotopy inverse of this map is given by
$$g_{\cw} = \begin{pmatrix} \rho_1 \compose \rho_2^{-1} \compose \sigma & 0 \\ \s{IBR} \compose \T{OBR} \compose \text{z} \compose h_2 & \s{IBR} \compose \T{OBR} \end{pmatrix}.$$
Here, the maps $h_1$ and $h_2$ are homotopies for the $R2$ equivalences, and we won't compute them explicitly.
\end{prop}

\begin{proof}
This time we have the following composition:
\begin{align*}
\cone{\mathfig{0.08}{\CKT cw/cone_source_1}}{\; \text{z} \;}{\mathfig{0.08}{\CKT cw/H_braiding_1}}
        & \htpy \cone{\mathfig{0.08}{\CKT cw/cone_source_1}}{\mathrlap{\quad \T{IBR} \compose \text{z}} \phantom{\T{OBR} \compose \text{z} \compose \sigma \compose \rho_2}}{\mathfig{0.08}{\CKT cw/H_braiding_2}} \\
        & \htpy \cone{\mathfig{0.08}{\CKT cw/cone_source_2}}{\mathrlap{\; \; \T{IBR} \compose \text{z} \compose \rho_1} \phantom{\T{OBR} \compose \text{z} \compose \sigma \compose \rho_2}}{\mathfig{0.08}{\CKT cw/H_braiding_2}} \\
        & = \cone{\mathfig{0.08}{\CKT cw/cone_source_2}}{\T{OBR} \compose \text{z} \compose \sigma \compose \rho_2}{\mathfig{0.08}{\CKT cw/H_braiding_2}} \\
        & \htpy \cone{\mathfig{0.08}{\CKT cw/cone_source_3}}{\mathrlap{\; \; \T{OBR} \compose \text{z} \compose \sigma} \phantom{\T{OBR} \compose \text{z} \compose \sigma \compose \rho_2}}{\mathfig{0.08}{\CKT cw/H_braiding_2}} \\
        & \htpy \cone{\mathfig{0.08}{\CKT cw/cone_source_3}}{\mathrlap{\qquad \text{z} \compose \sigma} \phantom{\T{OBR} \compose \text{z} \compose \sigma \compose \rho_2}}{\mathfig{0.08}{\CKT cw/H_braiding_3}} \\
        & \htpy \cone{\mathfig{0.08}{\CKT cw/cone_source_3r}}{\mathrlap{\qquad \; \text{z}} \phantom{\T{OBR} \compose \text{z} \compose \sigma \compose \rho_2}}{\mathfig{0.08}{\CKT cw/H_braiding_3}}
\end{align*}
The homotopy equivalences on the first and fifth lines follow from Lemma \ref{lem:f-to-if}, and are given by the matrices $\begin{pmatrix} 1 & 0 \\ 0 & \T{IBR} \end{pmatrix}$ and $\begin{pmatrix} 1 & 0 \\ 0 & \s{OBR} \end{pmatrix}$. The equivalences on the second and fourth lines come from Lemma \ref{lem:f-to-fi}, and are given by the matrices $\begin{pmatrix} \rho_1^{-1} & 0 \\ \T{IBR} \compose \text{z} \compose h_1 & 1 \end{pmatrix}$ and $\begin{pmatrix} \rho_2 & 0 \\ 0 & 1 \end{pmatrix}$. Equality on the third line is just Lemma \ref{lem:hbraid-equal}, and the equivalence (in fact, isomorphism) on the last line is given by the matrix $\begin{pmatrix} \sigma & 0 \\ 0 & 1 \end{pmatrix}$. Note that without this crossing reordering our $R2$ moves would fail to be consistent in the cone.
\end{proof}

The explicit maps for the remaining three $R3$ variations are computed in much the same way, and we will omit the details.

\begin{prop}
\label{prop:R3-v6,7,8}
\begin{enumerate}
\item The $hlm$ variation.
$$\Kh{\mathfig{0.1}{\CKT hlm/hlm_initial}}
\iso \vcone{\mathfig{0.08}{\CKT hlm/cone_source_1}}{\text{z}^{\text{above}}}{\mathfig{0.08}{\CKT hlm/H_braiding_1}}
 \xrightarrow[f_{hlm}]{\htpy} \vcone{\mathfig{0.08}{\CKT hlm/cone_source_3}}{\text{z}^{\text{below}}}{\mathfig{0.08}{\CKT hlm/H_braiding_3}} \iso \Kh{\mathfig{0.1}{\CKT hlm/hlm_final}}$$
is a homotopy equivalence via the map
$$f_{hlm} = \begin{pmatrix} \sigma \compose \rho_1 \compose \rho_2^{-1} & 0 \\ \s{IBL} \compose \T{OBL} \compose \text{z} \compose h_2 & \s{IBL} \compose \T{OBL} \end{pmatrix}.$$
The homotopy inverse of this map is given by
$$g_{hlm} = \begin{pmatrix} \rho_2 \compose \rho_1^{-1} \compose \sigma & 0 \\ \s{OBL} \compose \T{IBL} \compose \text{z} \compose h_1 & \s{OBL} \compose \T{IBL} \end{pmatrix}.$$

\item The $mlh$ variation.
$$\Kh{\mathfig{0.1}{\CKT mlh/mlh_initial}}
\iso \vcone{\mathfig{0.08}{\CKT mlh/H_braiding_1}}{\text{u}^{\text{above}}}{\mathfig{0.08}{\CKT mlh/cone_source_1}}
 \xrightarrow[f_{mlh}]{\htpy} \vcone{\mathfig{0.08}{\CKT mlh/H_braiding_3}}{\text{u}^{\text{below}}}{\mathfig{0.08}{\CKT mlh/cone_source_3}} \iso \Kh{\mathfig{0.1}{\CKT mlh/mlh_final}}$$
is a homotopy equivalence via the map
$$f_{mlh} = \begin{pmatrix} \s{IBL} \compose \T{OBL} & 0 \\ -h_2 \compose \text{u} \compose \s{IBL} \compose \T{OBL} & \sigma \compose \rho_1 \compose \rho_2^{-1} \end{pmatrix}.$$
The homotopy inverse of this map is given by
$$g_{mlh} = \begin{pmatrix} \s{OBL} \compose \T{IBL} & 0 \\ -h_1 \compose \text{u} \compose \s{OBL} \compose \T{IBL} & \rho_2 \compose \rho_1^{-1} \compose \sigma \end{pmatrix}.$$

\item The $\ccw$ variation.
$$\Kh{\mathfig{0.1}{\CKT ccw/ccw_initial}}
\iso \vcone{\mathfig{0.08}{\CKT ccw/H_braiding_1}}{\text{u}^{\text{above}}}{\mathfig{0.08}{\CKT ccw/cone_source_1}}
 \xrightarrow[f_{\ccw}]{\htpy} \vcone{\mathfig{0.08}{\CKT ccw/H_braiding_3}}{\text{u}^{\text{below}}}{\mathfig{0.08}{\CKT ccw/cone_source_3}} \iso \Kh{\mathfig{0.1}{\CKT ccw/ccw_final}}$$
is a homotopy equivalence via the map
$$f_{\ccw} = \begin{pmatrix} \s{OBR} \compose \T{IBR} & 0 \\ -h_1 \compose \text{u} \compose \s{OBR} \compose \T{IBR} & \sigma \compose \rho_2 \compose \rho_1^{-1} \end{pmatrix}.$$
The homotopy inverse of this map is given by
$$g_{\ccw} = \begin{pmatrix} \s{IBR} \compose \T{OBR} & 0 \\ -h_2 \compose \text{u} \compose \s{IBR} \compose \T{OBR} & \rho_1 \compose \rho_2^{-1} \compose \sigma \end{pmatrix}.$$

\end{enumerate}
\end{prop}

\begin{remark} The proof for the $hlm$ variation uses Lemmas \ref{lem:f-to-if} and \ref{lem:f-to-fi}, as in the $\cw$ case above, while the $\ccw$ and $mlh$ variations require Lemmas \ref{lem:f-to-fr} and \ref{lem:f-to-rf}.
\end{remark}

Having determined the $R3$ maps explicitly, we see they are various and complicated. However, to compute with movie moves we only need a small list of facts that apply to all eight $R3$ variations. The following lemmas provide this distillation, and apply to both the $R3$ moves and their inverses.

First, and briefly, we'll reintroduce some notation from \cite{ClarMW:Fix}. As we've seen, the CKT essentially separates each $R3$ ``before" and ``after" complex into two smaller complexes, which we'll call ``layers," whose diagrams differ by the resolution of a single crossing. We'll denote layers that look like (a rotated version of) either $\mathfig{0.06}{R3/CKT/O_layer1}$ or $\mathfig{0.06}{R3/CKT/O_layer2}$ by $\mathcal{O}$, since the strands involved in the crossing appear to be $\mathcal{O}$rthogonal to the uninvolved strand. In contrast, we'll denote layers that look like (a rotated version of) either $\mathfig{0.06}{R3/CKT/P_layer1}$ or $\mathfig{0.06}{R3/CKT/P_layer2}$ by $\mathcal{P}$, for ``$\mathcal{P}$arallel to the uninvolved strand." Notice that the homological ordering of the layers may be either $\OP$ or $\PO$, depending on the crossing signs and orientations for each $R3$ move:
\begin{itemize}
\item the $hml$, $lhm$, $mlh$, and $\ccw$ variations are ordered $\OP$
\item the $mhl$, $lmh$, $hlm$, and $\cw$ variations are ordered $\PO$
\end{itemize}
This allows us to decompose each $R3$ map as
$$R3_\star = R3_\star^\OO + R3_\star^\OP + R3_\star^\PO + R3_\star^\PP,$$
where $\star$ is one of the eight variations and $R3_\star^{a \To b}$ is the component from the $a$ layer to the $b$ layer. Of course, we've already performed this decomposition in the propositions above. For example, the matrix for $f_{\ccw}$ (from Proposition \ref{prop:R3-v6,7,8}, part $(3)$) can be written, using this notation, as
$$\begin{pmatrix} R3_\ccw^\OO & R3_\ccw^\OP \\ & \\ R3_\ccw^\PO & R3_\ccw^\PP \end{pmatrix}.$$

Let's also name a morphism that will arise frequently:
$$R = \mathfig{0.15}{foams/R3starlike_OO}.$$


\begin{lemma}
\label{lem:R3-OP}
If the layers of $R3_\star$ are arranged as $\OP$, then the map from the parallel layer to the orthogonal layer, $R3_\star^\PO$, is zero. Otherwise, if the layers are arranged as $\PO$, then the map $R3_\star^\PO$ is zero. (That is, the diagonal map pointing backwards in homological height is always zero.)
\end{lemma}

\begin{lemma}
\label{lem:R3-OO}
The map between the orthogonal layers, $R3_\star^\OO$, is the identity chain map when $\star = hml, lhm, mhl,$ or $lmh$. When $\star = hlm$ or $mlh$, the map $R3_\star^\OO$, restricted to the two homologically extreme objects, is $-id$. When $\star = \cw $ or $\ccw$, the restriction of $R3_\star^\OO$ to extreme objects is the appropriate rotation of the morphism $-R$.
\end{lemma}

\begin{lemma}
\label{lem:R3-PP}
The map between the parallel layers, $R3_\star^\PP$, kills the doubly smoothly-resolved object (which resides at either the highest or lowest homological height) when $\star = hml, lhm, mhl,$ or $lmh$, and kills both extreme objects when $\star = hlm, mlh, \cw,$ or $\ccw$. Further, for each variation, in the middle homological height there is a pair of objects (one in the source complex and one in the target complex) that have the same unoriented diagram; the component of the $R3_\star^\PP$ map between these objects is $-id$ for $\star = hml, lhm, mhl$ or $lmh$, the identity for $\star = hlm$ or $mlh$, and the appropriate rotation of $R$ for $\star = \cw$ or $\ccw$. Every other entry of the $R3_\star^\PP$ map in the middle homological height is some multiple of a foam that looks locally a cup, a cap, or one of the following
\begin{align*}
\mathfig{0.1}{foams/bubble} \hspace{.5in} & \mathfig{0.1}{foams/debubble}
\end{align*}
near all circles or bigons in either the source or target.
\end{lemma}

These lemmas follow easily by observation of our CKT complexes and direct calculation of the maps therein, involving only foam isotopy. It turns out we will need one more (rather obscure) piece of data for movie move 6, which is again easily computed: it concerns the $\OO$ map for the $\cw$ $R3$ variation.

\begin{lemma}
\label{lem:cw-mid}
The $R3_\cw^\OO$ map acts on the middle homological height objects in the following way:
$$R3_\cw^\OO: \begin{pmatrix} \mathfig{0.1}{R3/cw_mid/cw_mid_bigonup} \\ \mathfig{0.1}{R3/cw_mid/cw_mid_hex} \end{pmatrix} \xrightarrow{\begin{pmatrix} -T \; \; -1 \\ \phantom{-}0 \; \; -T' \end{pmatrix}} \begin{pmatrix} \mathfig{0.1}{R3/cw_mid/cw_mid_hex} \\ \mathfig{0.1}{R3/cw_mid/cw_mid_bigondown} \end{pmatrix}$$
where
$$T = \ \mathfig{.2}{foams/R3starlike_OOmid},$$
and $T'$ is the appropriate rotation/reflection.
\end{lemma}

For our movie move calculations, it will also be convenient to have the analogous lemmas when we determine the $\mathcal{O}$ and $\mathcal{P}$ layers by resolving the lowest crossing, rather than the highest. The CKT works just as well in this context, this time using the four versions of the strand-past-vertex move we haven't seen so far (computed in Lemma \ref{lem:spv}). The maps look very similar to the ones we've worked out above; this time, however, the $hml, lmh, hlm,$ and $mlh$ variations will more closely resemble the CKT from Proposition \ref{prop:R3-v1}, while the $mhl, lhm, \cw,$ and $\ccw$ variations will take after Proposition \ref{prop:R3-v5}. It's also worth noting that a different set of crossing ordering conventions will become much more convenient here. For variations $hml, lmh, hlm,$ and $mlh$, we'll use the following ordering: in the initial tangle the crossings will be ordered `middle', `high', `low', while in the final tangle they will be ordered `high', `middle', `low'. For variations $mhl, lhm, \cw,$ and $\ccw$, we'll instead use the inverse ordering: the crossings of the initial tangle will be ordered `high', `middle', `low', and of the final tangle, `middle', `high', `low'. We'll demonstrate two of these $R3$ homotopy equivalences below (denoting them with a bar), and leave the rest as an exercise to the reader.


\begin{prop}
\label{prop:R3L-v1,2}
\begin{enumerate}

\item The $hlm$ variation.
$$\Kh{\mathfig{0.1}{\CKT hlm_L/hlm_L_initial}}
\iso \vcone{\mathfig{0.08}{\CKT hlm_L/I_braiding_1}}{\text{u}^{\text{above}}}{\mathfig{0.08}{\CKT hlm_L/cone_source_1}}
 \xrightarrow[\overline{f_{\text{hlm}}}]{\htpy} \vcone{\mathfig{0.08}{\CKT hlm_L/I_braiding_3}}{\text{u}^{\text{below}}}{\mathfig{0.08}{\CKT hlm_L/cone_source_2}} \iso \Kh{\mathfig{0.1}{\CKT hlm_L/hlm_L_final}}$$
is a homotopy equivalence via the map
$$\overline{f_{\text{hlm}}} = \begin{pmatrix} \T{OAR} \compose \s{IAR} & 0 \\ \text{u} \compose \h{IAR} & 1 \end{pmatrix}.$$
The homotopy inverse of this map is given by
$$\overline{g_{\text{hlm}}} = \begin{pmatrix} \T{IAR} \compose \s{OAR} & 0 \\ \text{u} \compose \h{OAR} & 1 \end{pmatrix}.$$

\item The $lhm$ variation.
$$\Kh{\mathfig{0.1}{\CKT lhm_L/lhm_L_initial}}
\iso \vcone{\mathfig{0.08}{\CKT lhm_L/H_braiding_1}}{\text{u}^{\text{above}}}{\mathfig{0.08}{\CKT lhm_L/cone_source_1}}
 \xrightarrow[\overline{f_{lhm}}]{\htpy} \vcone{\mathfig{0.08}{\CKT lhm_L/H_braiding_3}}{\text{u}^{\text{below}}}{\mathfig{0.08}{\CKT lhm_L/cone_source_3}} \iso \Kh{\mathfig{0.1}{\CKT lhm_L/lhm_L_final}}$$
is a homotopy equivalence via the map
$$\overline{f_{lhm}} = \begin{pmatrix} \s{IAR} \compose \T{OAR} & 0 \\ -h_1 \compose \text{u} \compose \s{IAR} \compose \T{OAR} & \rho_2 \compose \rho_1^{-1} \compose \sigma \end{pmatrix}.$$
The homotopy inverse of this map is given by
$$\overline{g_{lhm}} = \begin{pmatrix} \s{OAR} \compose \T{IAR} & 0 \\ -h_2 \compose \text{u} \compose \s{OAR} \compose \T{IAR} & \sigma \compose \rho_1 \compose \rho_2^{-1} \end{pmatrix}.$$

\end{enumerate}
\end{prop}

The following lemmas are analogous to Lemmas \ref{lem:R3-OP}, \ref{lem:R3-OO}, and \ref{lem:R3-PP}: they give the summary information we need about the $R3$ maps (and their inverses) obtained by resolving the \emph{lowest} crossing, which we'll denote $\overline{R3_\star}$.

\begin{lemma}
\label{lem:R3L-OP}
If the layers of $R3_\star$ are arranged as $\OP$, then the map from the parallel layer to the orthogonal layer, $\low{\star}{\PO}$, is zero. Otherwise, if the layers are arranged as $\PO$, then the map $\low{\star}{\PO}$ is zero. (That is, the diagonal map pointing backwards in homological height is always zero.)
\end{lemma}

\begin{lemma}
\label{lem:R3L-OO}
The map between the orthogonal layers, $\low{\star}{\OO}$, is the identity chain map when $\star = hml, lmh, hlm,$ or $mlh$. When $\star = mhl$ or $lhm$, the map $\low{\star}{\OO}$, restricted to the two homologically extreme objects, is $-id$. When $\star = \cw $ or $\ccw$, the restriction of $\low{\star}{\OO}$ to extreme objects is the appropriate rotation of the morphism $-R$
\end{lemma}

\begin{lemma}
\label{lem:R3L-PP}
The map between the parallel layers, $\low{\star}{\PP}$ kills the doubly smoothly-resolved object (which resides at either the highest or lowest homological height) when $\star = hml, lmh, hlm,$ or $mlh$, and kills both extreme objects when $\star = mhl, lhm, \cw,$ or $\ccw$. Further, for each variation, in the middle homological height there is a pair of objects (one in the source complex and one in the target complex) that have the same unoriented diagram; the component of the $\low{\star}{\PP}$ map between these objects is $-id$ for $\star = hml, lmh, hlm$ or $mlh$, the identity for $\star = mhl$ or $lhm$, and the appropriate rotation of the morphism $R$ for $\star = \cw$ or $\ccw$. Every other entry of the $\low{\star}{\PP}$ map in the middle homological height is some multiple of a foam that looks locally a cup, a cap, or one of the following
\begin{align*}
\mathfig{0.1}{foams/bubble} \hspace{.5in} & \mathfig{0.1}{foams/debubble}
\end{align*}
near all circles or bigons in either the source or target.
\end{lemma}


\section{Checking movie moves}
\label{sec:checking_movie_moves}

Here we will prove Theorem \ref{thm:1}, which asserts functoriality for the theory. This requires showing that chain maps in $\Kob{3}$, induced by link cobordisms, are well-defined, i.e., they are invariant under changing a cobordism presentation by a movie move. The overall strategy will be very similar to the one used in \cite{ClarMW:Fix}.

\subsection{Duality and homotopy isolation}
\label{ssec:duality}
We begin by stating a result about duality with respect to Hom-sets. For tangles $P$ and $Q$, denote any gluing of them by $P \bullet Q$, and let $\overline{Q}$ denote the reflection of $Q$. The following proposition was first presented in \cite{ClarMW:Fix}:

\begin{prop}
\label{prop:duality}
Given oriented tangles $P$, $Q$ and $R$, there is an isomorphism
between the spaces of chain maps up to homotopy
$$F : \Hom{Kh}{\Kh{P \bullet Q}}{\Kh{R}} \IsoTo \Hom{Kh}{\Kh{P}}{\Kh{R \bullet \overline{Q}}}.$$
\end{prop}

While this result was originally proven in the context of Khovanov's $\su{2}$ theory, it clearly holds for the $\su{3}$ case without any changes to the statement or proof (for whose details we refer the interested reader to \cite{ClarMW:Fix}). It's important to note, however, that the proof assumes the theory is already invariant under MM9, the ninth movie move. As such, invariance under MM9, shown in Section \ref{sssec:MM9}, must take place with complete independence of the material in this section.

From this we get an easy corollary, also given in \cite{ClarMW:Fix}, which will be highly useful during our movie move checks. In particular, note that it applies to (the first and last frames of) every movie move.

\begin{cor}
\label{cor:post-duality}
Let $T_1$ and $T_2$ be tangles with $k$ endpoints such that
$\overline{T_1} \bullet T_2$ is an unlink with $m$ components. Then the space of
chain maps modulo chain homotopy from $\Kh{T_1}$ to $\Kh{T_2}$ in
grading $m-k$ is 1-dimensional, and all chain maps of grading higher
than $m-k$ are chain homotopic to zero.
\end{cor}

The second component of machinery we'll need is the ``homotopy isolation" idea from \cite{ClarMW:Fix}.

\begin{defn}
Let $C^{\bullet}$ and $D^{\bullet}$ be complexes in a graded additive category, with $A$ a direct summand in some $C^i$. We'll say $A$ is $C$-$D$ \emph{homotopically isolated} if, for any grading zero homotopy $h:C^{\bullet} \rightarrow D^{\bullet - 1}$, the restriction of $dh + hd$ to $A$ is zero.
\end{defn}

\begin{lemma}
\label{lem:hom-iso1}
Let $f,g:C^\bullet \rightarrow D^\bullet$ be chain maps, and say $f \htpy \alpha g$ are homotopic for some scalar $\alpha$. If $f$ and $g$ agree and are nonzero on a $C$-$D$ homotopically isolated object $A$ in $C^\bullet$, then we have that $f \htpy g$ are homotopic.
\end{lemma}

By Corollary \ref{cor:post-duality}, we know that any movie move (except for MM9) changes the induced map in $\Kob{3}$ by at most a scalar. We'll show this scalar is always 1 by computing with homotopically isolated objects, which have a convenient description in the $\su{3}$ web case.

\begin{lemma}
\label{lem:hom-iso2}
Let $\Kh{T_1}$ and $\Kh{T_2}$ be the complexes for two tangle diagrams, and let $D_1$ be a web appearing as a direct summand somewhere in $\Kh{T_1}$. Then
\begin{enumerate}
\item $A$ is $\Kh{T_1}$-$\Kh{T_1}$ homotopically isolated if $D_1$ contains no cycles (as a graph) and is not connected by differentials to webs containing cycles;
\item $A$ is $\Kh{T_1}$-$\Kh{T_2}$ homotopically isolated if $\Kh{T_1}$ and $\Kh{T_2}$ contain only acyclic webs.
\end{enumerate}
\end{lemma}

To prove this, we'll first need a more general result about Hom-sets of foams.

\begin{defn}
The \emph{bare grading} of a morphism $C$ between webs $D_1$ and $D_2$ is given by $$\deg'(C) = 2\chi(C) - B + \frac{V}{2},$$
where $B$ is the number of boundary points on $h$ and $V$ is the total number trivalent vertices in $D_1$ and $D_2$.\footnote{Here we've simply neglected the contributions from formal shifts of boundary webs present in the usual $\Kob{3}$ grading.}
\end{defn}

\begin{prop}
\label{prop:no-positive-grading}
If $D_1$ and $D_2$ are acyclic then there are no nonzero morphisms with positive bare grading between them.
\end{prop}

\begin{proof}
Let $C:D_1 \rightarrow D_2$. Our first task is to remove the closed seams from $C$, producing a new foam called $\widetilde{C}$. Begin by performing neck cutting on each sheet incident on a closed seam. (If there are $k$ closed seams to begin with, there will be $3k$ operations). This will produce $3^{3k}$ terms, in which all the original closed seams are sequestered in closed foams. We can evaluate each one of these closed foams to a scalar by the remark in Section \ref{ssec:categorifying}, leaving us with a new presentation of $C = \sum C_i$, a degree-homogeneous linear combination. Here the $C_i$ may still have closed seams, but only of the variety appearing in the neck cutting relation: locally, they will all look like ``choking handles" $\left(\mathfig{0.1}{cobordisms/handle_disc_bdy_left}\right)$. At this point, we can perform neck cutting once again (to remove unwanted tubes connecting sheets) so that each $C_i$ has the following pieces:
\begin{itemize}
\item 0-cells given by trivalent vertices and boundary points.
\item 1-cells given by seams, boundary lines, and edges in $D_i$.
\item 2-cells given either by discs in choking handles, or by sheets that intersect $D_1 \cup D_2$ nontrivially and that may have handles or choking handles.
\end{itemize}

Now pick any $C_i \neq 0$ and consider the foam $\widetilde{C}$ obtained from $C_i$ by removing all handles and choking handles. Since these pieces have bare grading $-1$ and $0$, respectively, we have that $\deg'(C) = \deg'(C_i) \leq \deg'(\widetilde{C})$. Also note that $\widetilde{C}$ has no closed seams and no handles, and can thus be decomposed as follows:
\begin{itemize}
\item 0-cells given by trivalent vertices and boundary points.
\item 1-cells given by seams, boundary lines, and edges in $D_i$.
\item 2-cells given by genus zero sheets that intersect $D_1 \cup D_2$ nontrivially.
\end{itemize}
It suffices to prove the result for this much simpler foam $\widetilde{C}$.

Let $F$ be the number of seams in $\widetilde{C}$, and let $S$ be the number of sheets. We now claim that $\chi(\widetilde{C}) = S-2F$, which we can see as follows. Imagine building $\widetilde{C}$ out of its $S$ disjoint sheets. We'll then add seams, joining together three sheets at each seam. Each of these operations will reduce the Euler characteristic by $2$, giving the formula. Thus, since $F=\frac{V}{2}$, $\deg'(\widetilde{C}) = 2S - 3F - B$.

Next we'll show that, if $D_1$ and $D_1$ are acyclic, this formula can be modified to $\deg'(\widetilde{C}) = 2S -4F - N$, where $N = \pi_0 (D_1) + \pi_0 (D_2)$, i.e., the total number of connected components of the boundary webs. Assuming $D_1$ and $D_2$ are acyclic just means that each of their components is a tree. This says that $B = V_i + 2n_i$, where $n_i$ is the number of connected components and $V_i$ is the number of trivalent vertices in $D_i$. Adding these two equations we get
\begin{align*}
    2B &= V_1 + V_2 + 2(n_1 + n_2)\\
    &= V + 2N\\
    \Rightarrow B &= \frac{V}{2} + N = F + N.
\end{align*}
Thus we have that $\deg'(\widetilde{C}) = 2S - 3F - B = 2S - 4F -N$.

Remember, our goal is to show that $\deg'(\widetilde{C}) \leq 0$, which is now equivalent to proving
\begin{align}
\label{eq:seams}
S \leq 2F + \frac{N}{2}.
\end{align}
We'll get there by considering the boundary 1-cells of the sheets in $\widetilde{C}$, which, as mentioned before, consist of seams, boundary lines, and edges in $D_i$. It's an easy observation that the total number of edges, $E$, in $D_1$ and $D_2$ is just $E = 2V + N$. These segments, as well as the boundary lines, can serve as part of the boundary for a single sheet. Each seam, however, will serve as a boundary component for three distinct sheets (from our acyclicity assumption).

For an example, let's say that every sheet in $\widetilde{C}$ were a $d$-gon. Then we would have that $S d = E + B + 3F$: each side of each sheet corresponds to a graph edge, a boundary line, or one third of a seam. Let's make a more general assumption: that every sheet in $\widetilde{C}$ has at least four sides. If this is the case, then
\begin{align*}
    4S &\leq E + B + 3F\\
    &\leq (2V + N) + (F + N) + 3F\\
    &\leq 8F + 2N.
\end{align*}

This would give us Equation \ref{eq:seams}. Note that there cannot be three-sided sheets (or in fact any odd-sided sheets, by acyclicity); unfortunately, there can be bigons. However, we observe that bigons in $\widetilde{C}$ must have one edge in one of the $D_i$ and the other edge a seam intersecting the same $D_i$ twice. In other words, this bigon must be part of a zip or unzip morphism. Thus we can factor $\widetilde{C}$ into a stack of zips, unzips, and a foam in which each sheet has at least four sides. Since the zip and unzip morphisms have bare degree $-1$, we have our result.
\end{proof}

\begin{proof}[Proof of Lemma \ref{lem:hom-iso2}]
A degree zero homotopy is a morphism $$h: q^{m}D_1 \rightarrow q^{m-1}D_2.$$ (Here, we could have $D_2$ in either $\Kh{T_1}$ or $\Kh{T_2}$, depending on which part of the lemma we're trying to prove.) Thus, by Equation \ref{eq:grading}, $\deg'(h) = 1$. And by Proposition \ref{prop:no-positive-grading}, $h$ must be the zero map.
\end{proof}

The first part of this lemma is well-suited for the reversible movie moves (MM6-10) and the second part for the those involving Morse moves. It's an easy observation that every web in the initial (and final) complex $C$ of movie moves 6, 7, and 8 is $C$-$C$ homotopically isolated, and every web in the initial and final complexes $C$ and $D$ of 11, 13, and 15 is $C$-$D$ homotopically isolated. This means we can compare induced chain maps simply by applying them to a single object of our choice. Movie moves 12 and 14, unfortunately, do not contain homotopically isolated objects, so we'll need to compute the induced maps on all objects; luckily the complexes are small, and this is not a great burden. We'll handle movie moves 9 and 10 with different techniques: the former because of the paragraph after Proposition \ref{prop:duality}, and the latter just to illustrate something fun.

Keep in mind that, as always, all crossings in these moves must be ordered, and they may need to be reordered to be consistent with the conventions we've defined for the Reidemeister maps. (Recall the discussion about planar compositions of complexes in Section \ref{ssec:link-homology} and Appendix \ref{ssec:sign-conventions}) However, the chain maps induced by crossing reorderings are trivial in every movie move except MM6 and MM9. The signs appearing in MM6 are particularly nasty, but we will show some sample calculations. 
\subsection{MM1-5}
\label{sssec:MM1-5}
The first five movie moves are trivial; they simply say that a
Reidemeister move followed by its inverse is the identity.

\subsection{MM6-10}
Movie moves 6 through 10 involve no Morse moves, and so are
reversible. We only need to check one time direction, and in all cases we'll be comparing the map induced by the movie shown to the identity map (induced by the constant movie).\footnote{We apologize to the thorough reader of \cite{ClarMW:Fix}, for whom much of the prose and organization of the MM6 and MM8 calculations may induce d\'{e}j\`{a} vu.}

\subsubsection*{MM6}
\label{sssec:MM6}
$$\mathfig{0.5}{movie_moves/MM6}$$

There are 24 variations of MM6. To see this, we'll first make use of rotational symmetry to require that the 'horizontal' strand (the one
not involved in either $R2$ move) points from left to right. There are then sixteen possibilities for the initial frame of the movie move; these come from four choices of height orderings and four choices of orientations.
The horizontal strand can either lie entirely above or entirely below the two vertical strands ('non-interleaved'), or it may pass under one and over the
other ('interleaved', 'ascending' or 'descending').
The two vertical strands may be either parallel or anti-parallel. When they are parallel, they may point up or down, and when they
are anti-parallel they may have a clockwise or anti-clockwise orientation. All of these variations are displayed in Figure \ref{fig:MM6-variations}.

\begin{figure}[ht]
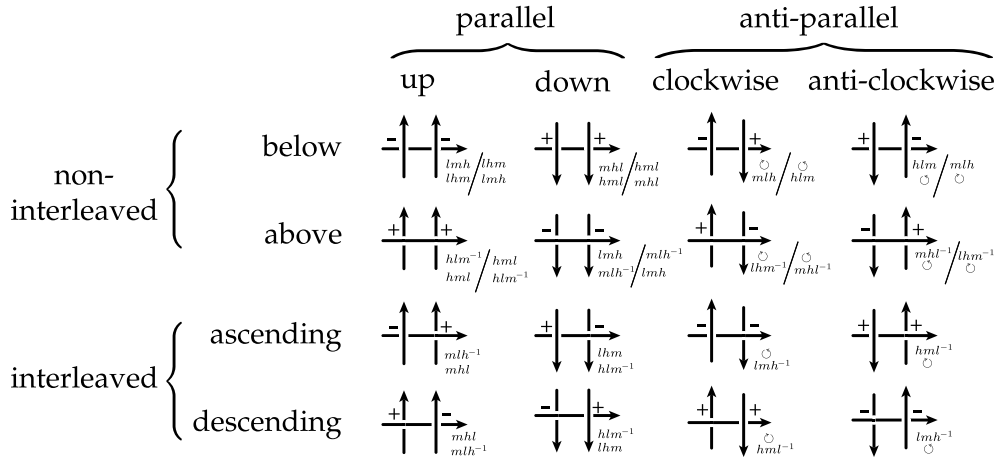

$$\mathfig{0.99}{movie_moves/MM6/table_of_variations}$$
\caption[16 variations for the initial frame of MM6]{16 variations for the initial frame of MM6.}\label{fig:MM6-variations}
\end{figure}

Further, the eight variations in which the strands are 'non-interleaved' (the first two rows of Figure \ref{fig:MM6-variations}) each have two sub-variations,
which we don't see until the second frame of the movie. Of the two vertical strands, either one can pass above the other during the $R2$ moves; in Figure \ref{fig:MM6-variations}, the 'left passing above the right' sub-variation is listed to the left of the slash. In the 'interleaved' variations, there is no choice here.

We will thus treat four major cases,
\begin{itemize}
\item non-interleaved, parallel variations,
\item non-interleaved, anti-parallel variations,
\item interleaved, parallel variations and
\item interleaved, anti-parallel variations.
\end{itemize}

\paragraph{\textbf{Non-interleaved parallel variations}}
There are four possible initial frames that are 'non-interleaved' and have parallel vertical strands. Each of these
initial frames has two possible sub-variations, depending on the relative heights of the vertical strands during the $R2$ moves.
For each of the four initial frames, we will treat uniformly the sub-variations in which the upper $R2$-induced
crossing is negative and the lower one is positive, and then indicate how to treat the other four sub-variations.

Recall that our lemmas encapsulating the details of the $R3$ variations require that we
separate the initial and final complexes into layers $\mathcal{O}$ and $\mathcal{P}$
by resolving a crossing. Maneuvering through the pair of $R3$s in
this movie move is most efficiently managed by resolving the
$R2$-induced crossings: the upper one for the first $R3$, and the
lower one for the second $R3$. Notice that since the upper crossing
is negative, the first $R3$ will have homological ordering $\OP$,
while the second $R3$ will have ordering $\PO$. Since the horizontal strand could be either above or below the vertical ones, these two crossing could be either the high or
low crossings in their respective $R3$ moves. Luckily, we have lemmas that deal with either case, so we needn't treat them separately.

\begin{figure}[h]
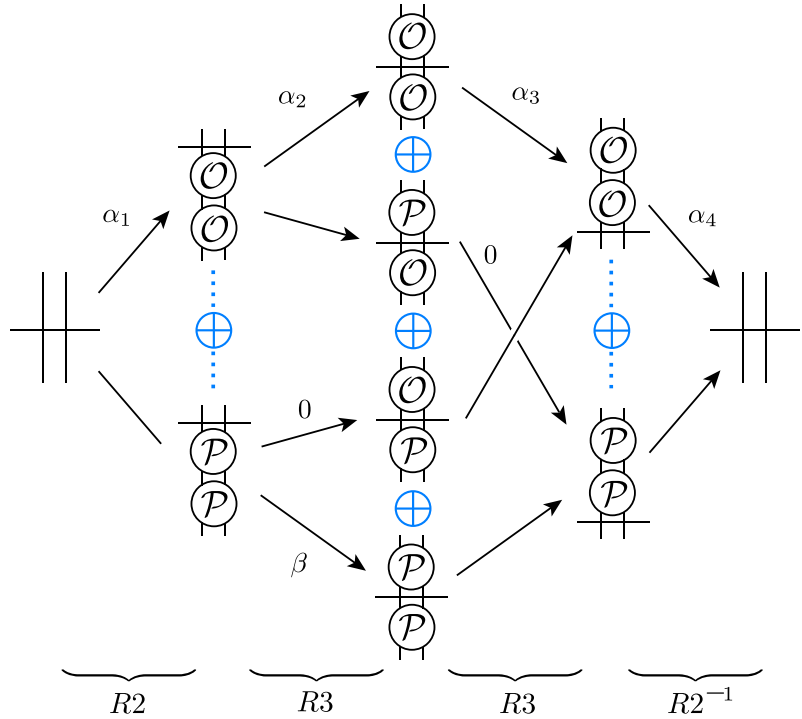

\begin{align*}
\mathfig{0.8}{movie_moves/MM6/MM6_comps}
\end{align*}
\caption[MM6 maps for the non-interleaved variations]{MM6 maps for the non-interleaved variations.}
\label{fig:MM6-comps}
\end{figure}

Our ``bundle" of maps for this subcase is given in Figure \ref{fig:MM6-comps}, where $\mathcal{O}$s and $\mathcal{P}$s describe
whether the indicated crossing resolution has strands orthogonal
or parallel to the horizontal strand. For example
$\mathfig{0.06}{movie_moves/MM6/MM6_PO}$ is our notation for
$\mathfig{0.06}{movie_moves/MM6/MM6_POstrands1}$ (or, thinking ahead, $\mathfig{0.06}{movie_moves/MM6/MM6_POstrands2}$ if the vertical strands are anti-parallel). Also, we've
cheated slightly with this diagram: the fourth column should
contain two additional summands, those with mixed
$\mathcal{O}$s and $\mathcal{P}$s. However, while there are
non-zero maps into these summands, the $R2^{-1}$ maps out are always
zero. Thus we needn't excessively complicate things with their
presence.

We're left with a sum of four compositions. The two middle
compositions are both zero, as each contains a leg (labelled with
``$0$") that's zero by Lemma \ref{lem:R3-OP}. The top
composition ($\alpha_{i}$'s) is just the identity: $\alpha_1$ and $\alpha_4$ are components of $R2a$ moves,
and $\alpha_2$ and $\alpha_3$ are each the identity, by Lemma
\ref{lem:R3-OO}. (Each map is a component of the $\OO$ map; when the horizontal strand lies below,
the $R3$ variations are $lmh,lhm,mhl$ and $hml$, which are exactly the four for which the $\OO$ part of the $R3$ map is the identity, and when the
horizontal strand lies above, the $R3$ variations are $hml,hlm,lmh$ and $mlh$, which are exactly the four for which the $\OO$ part of the $\overline{R3}$ map is the identity.)

The bottom composition is slightly more mysterious, but we see that the map $\beta$ sends doubly smoothly-resolved objects to zero by Lemma \ref{lem:R3-PP}. Thus, if we choose a doubly smoothly-resolved object to begin with, it will map to the doubly smoothly-resolved object in $\mathfig{0.06}{movie_moves/MM6/MM6_PP}$, and thereafter to zero. Further, as mentioned before, any initial object here is homotopically isolated, so the computation with this particular one suffices. Note also that, with this choice, the top composition involves only objects with smoothly resolved crossings, so we needn't worry about extra signs from crossing reorderings.

The other four sub-variations, in which the signs of the
$R2$-induced crossings are reversed, are proven analogously: note
that the objects in Figure \ref{fig:MM6-comps} will then have all
$\mathcal{O}$s and $\mathcal{P}$s swapped.

\paragraph{\textbf{Non-interleaved anti-parallel variations}}
First consider those cases in which the left vertical strand is oriented downward and
the right upward. Again we'll be referring to Figure \ref{fig:MM6-comps}. Consider
the object $\mathfig{0.1}{movie_moves/MM6/MM6f1_trace}$ which, since
the two signs of the initial crossings now differ, has homologically extreme height.

The composition $\alpha_4 \compose \alpha_3 \compose \alpha_2 \compose \alpha_1$ now specifies to

{
\begin{align*}
\newcommand{\pd}[1]{\mathfig{0.1}{movie_moves/MM6/#1}}
\newcommand{\pa}[1]{\mathfig{0.05}{movie_moves/MM6/#1}}
\xymatrix@R-7.5mm{
    \pd{MM6f1} \ar[r]^{R_2b} & \pd{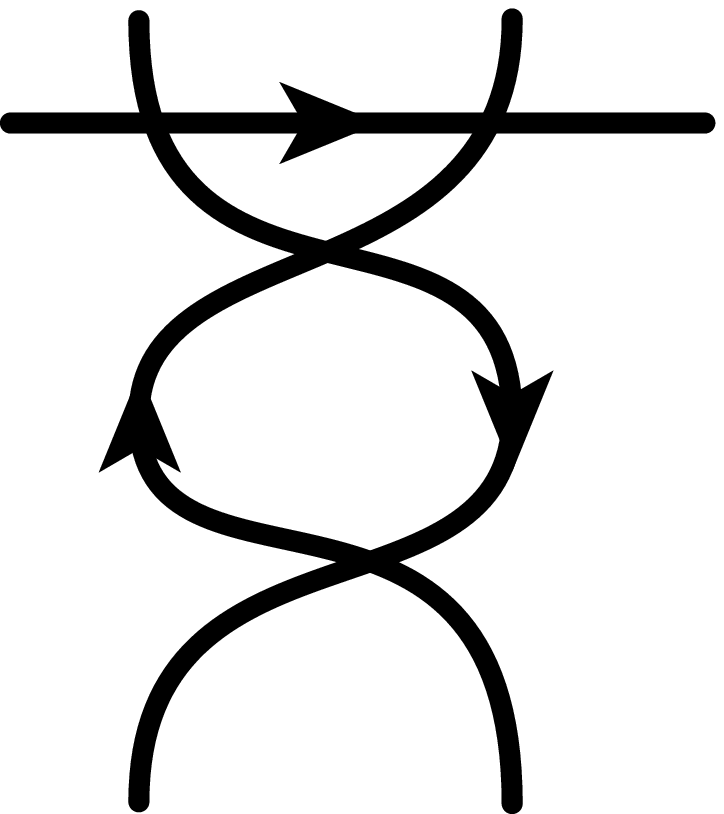} \ar[r]^{R_3} & \pd{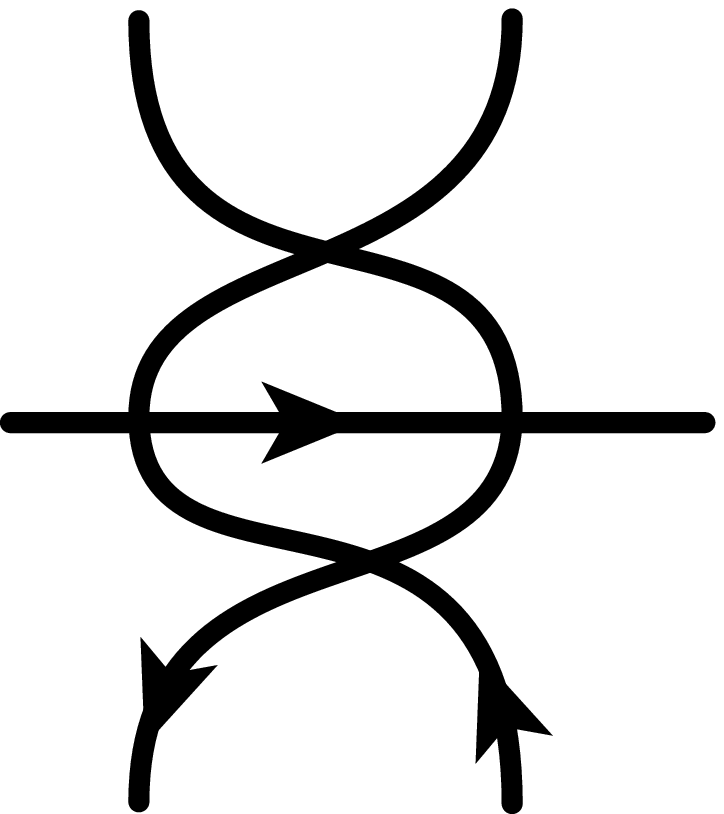} \ar[r]^{R_3} & \pd{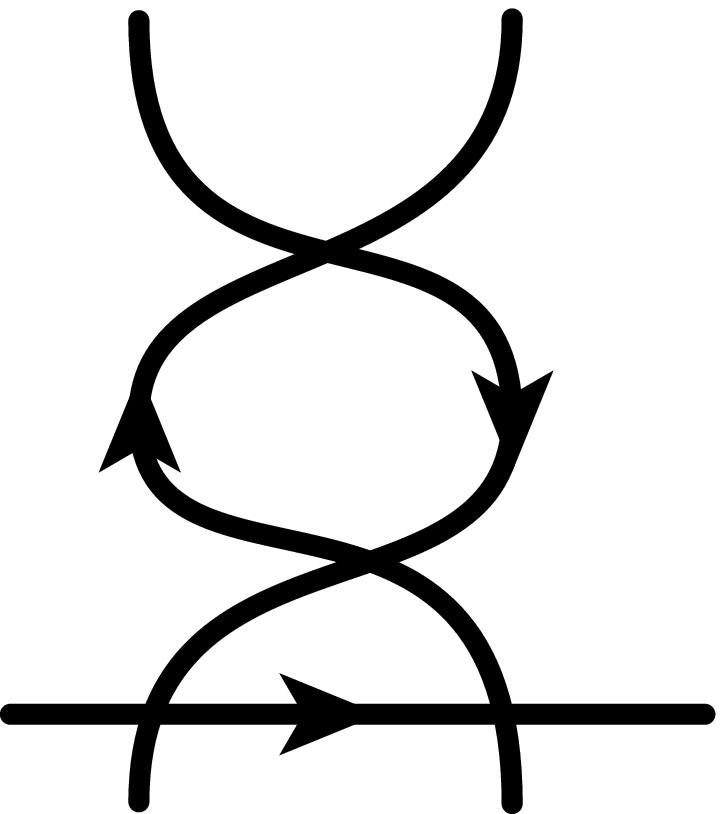} \ar[r]^{{R_2b}^{-1}} & \pd{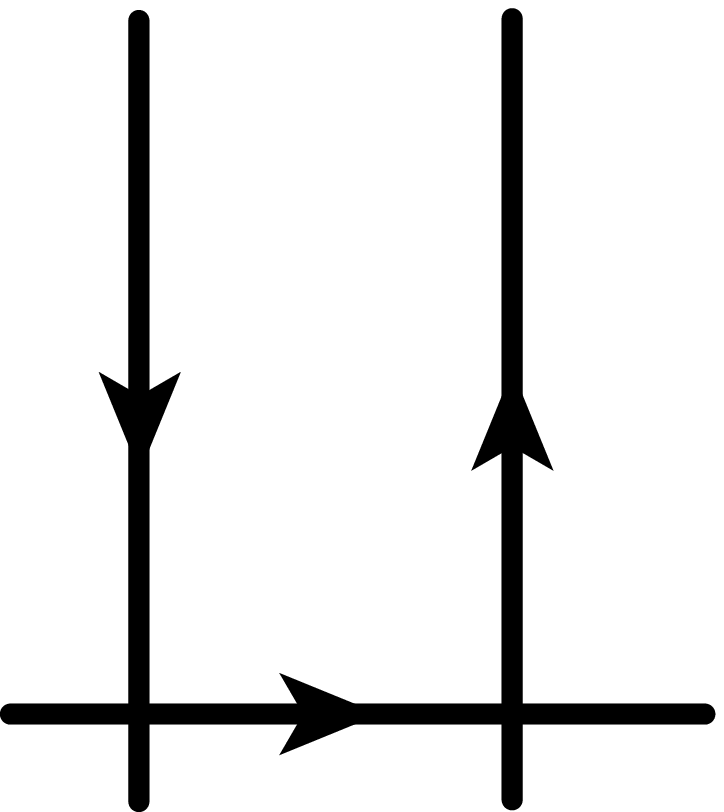} \\
    \\
    \pd{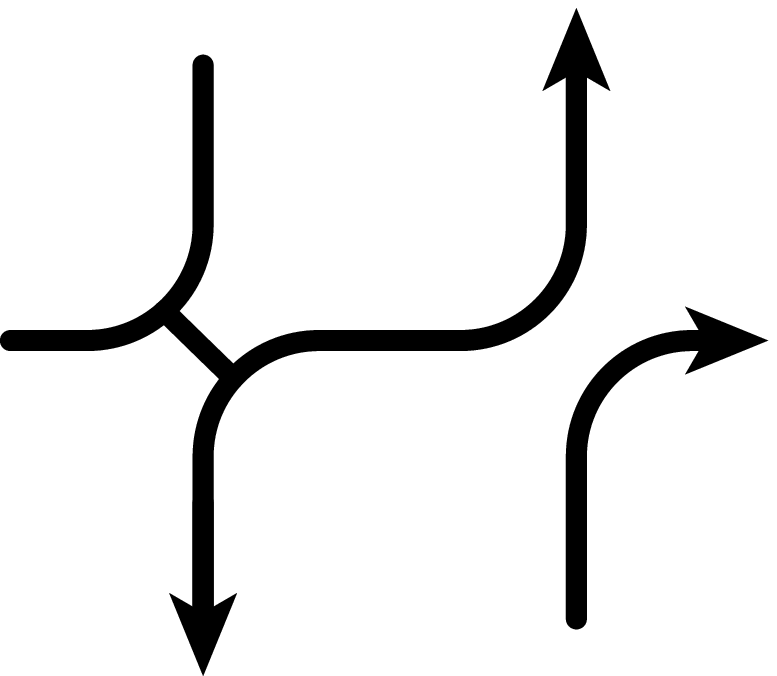} \ar@{|->}[r]^{\bar{r}} & \pd{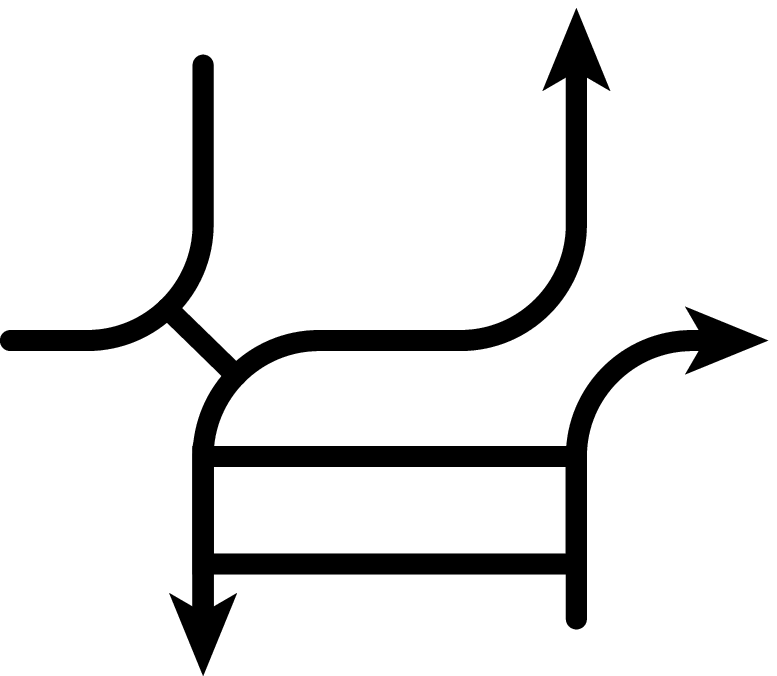} \ar@{|->}[r]^{-1} & \pd{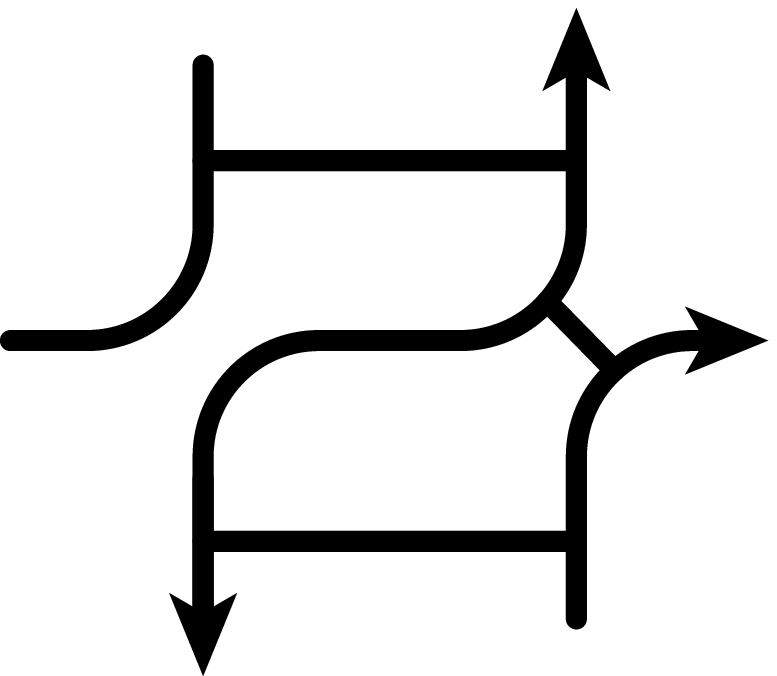} \ar@{|->}[r]^{-R} & \pd{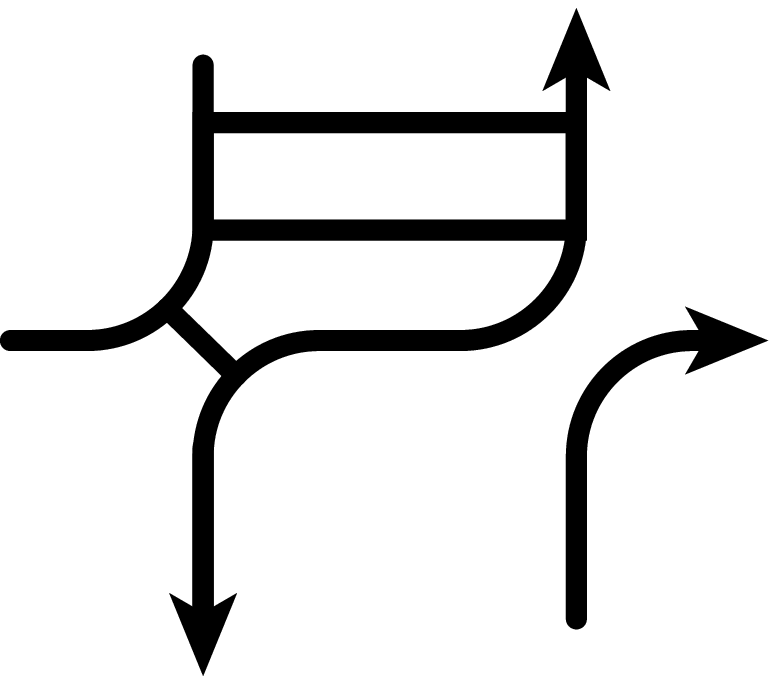} \ar@{|->}[r]^{-r} & \pd{MM6f1_trace} \\
}
\end{align*}
where $\bar{r}$ and $r$ are the upward- and downward-opening half barrels appearing in the the $R2$ (and strand-past-vertex) chain maps, and $R$ is the (appropriately rotated) morphism from Lemma \ref{lem:R3-OO}, shown again below
$$R = \ \mathfig{0.15}{foams/R3starlike_OO}.$$
This composition reduces via an airlock relation to $id$.

Of course, starting with an extreme object also guarantees this $\alpha$ composition is the only one we need to worry about, as $\beta=0$ from Lemma \ref{lem:R3-PP}. There's no crossing reordering sign here, either, but we do need to check. Below is an example calculation for one of the variations (we leave the others as an exercise), giving a total sign of $(-1)^2$. Recall that a crossing reordering map $\sigma_{ij}$ only gives a sign when mapping an object in which the crossings labeled $i$ and $j$ are both I-resolved. The unlabeled maps have already been described above.
{
\begin{align*}
\newcommand{\pd}[1]{\mathfig{0.055}{movie_moves/MM6/reordering1/#1}}
\newcommand{\pa}[1]{\mathfig{0.055}{movie_moves/MM6/#1}}
\xymatrix@R-7.5mm@C-2.5mm{
    \pd{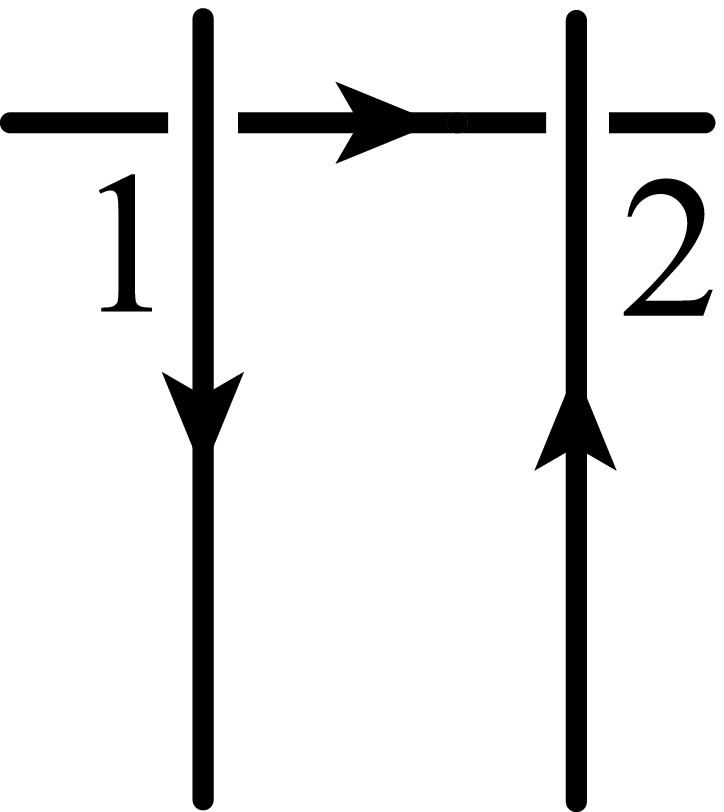} \ar[r]^{R_2b} & \pd{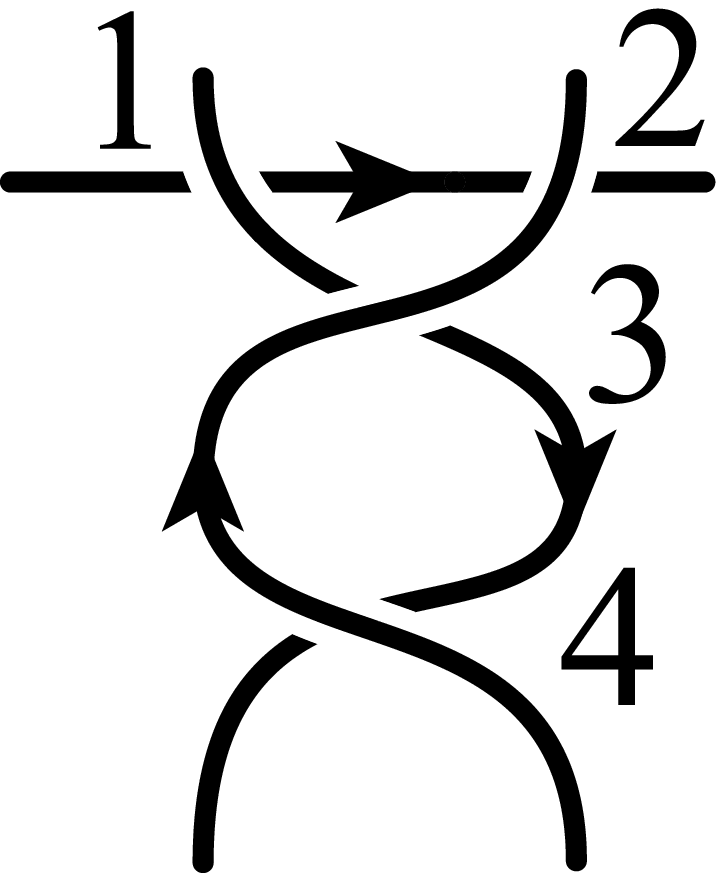} \ar[r]^{R_3} & \pd{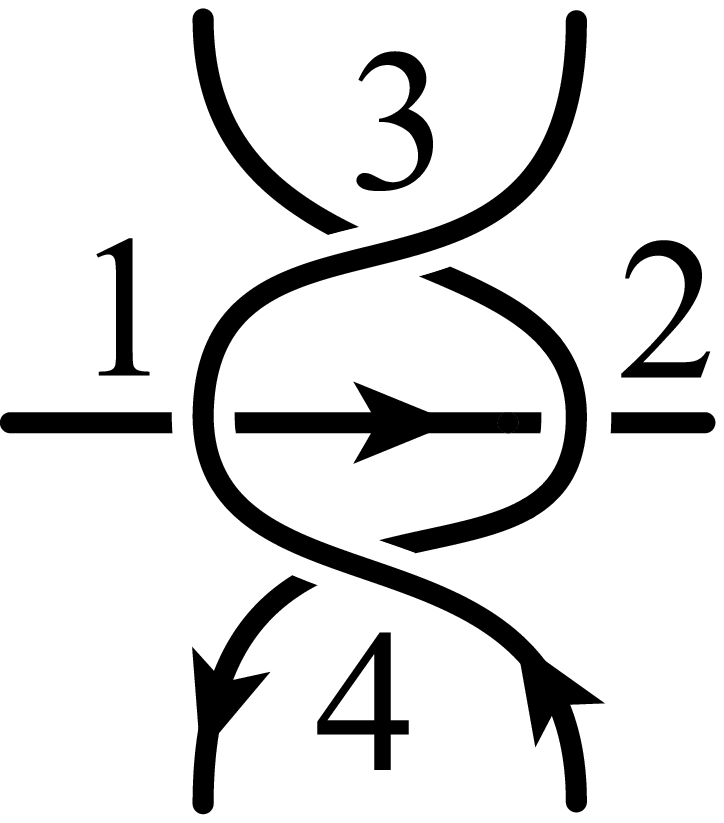} \ar[r]^{\sigma_{34}} & \pd{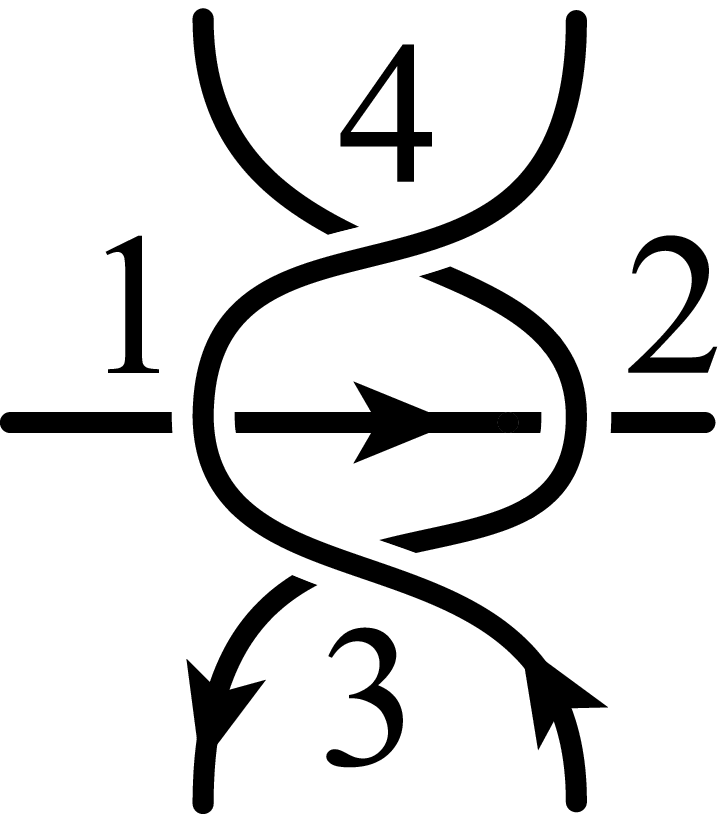} \ar[r]^{\sigma_{12}} & \pd{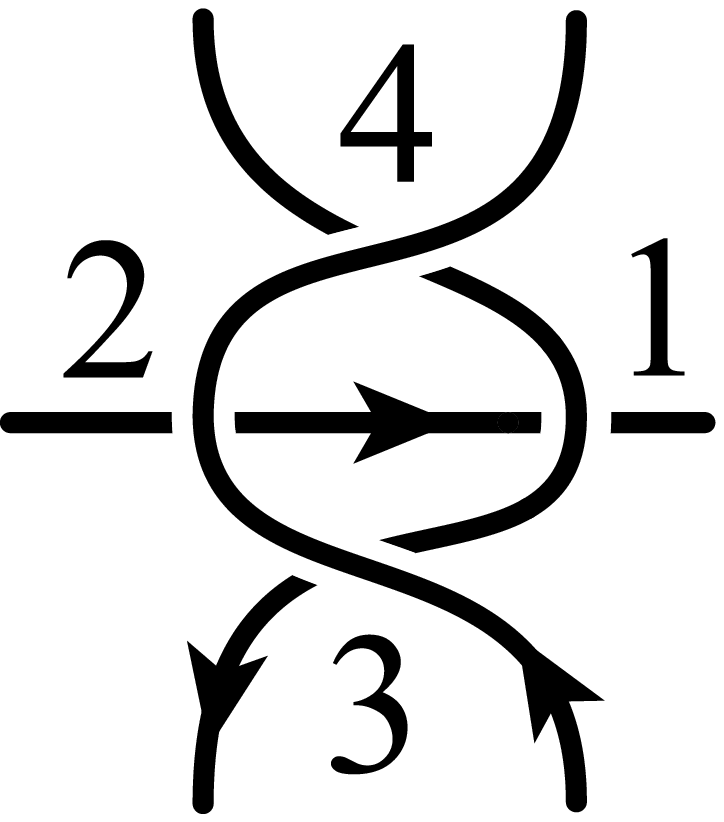}  \\ & & & \qquad \ar[r]^{R_3} & \pd{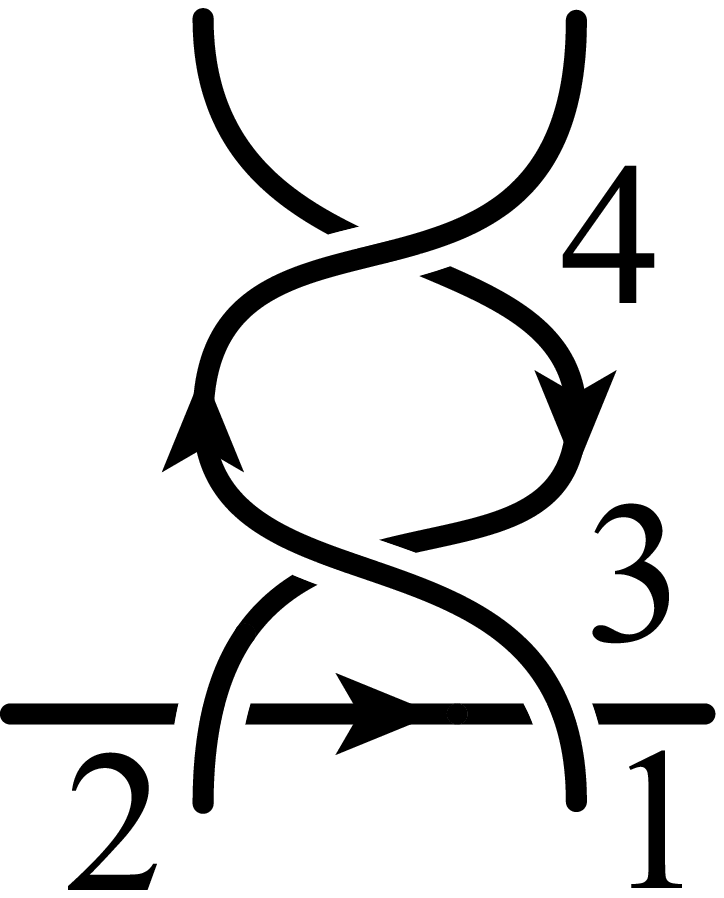} \ar[r]^{\sigma_{34}} & \pd{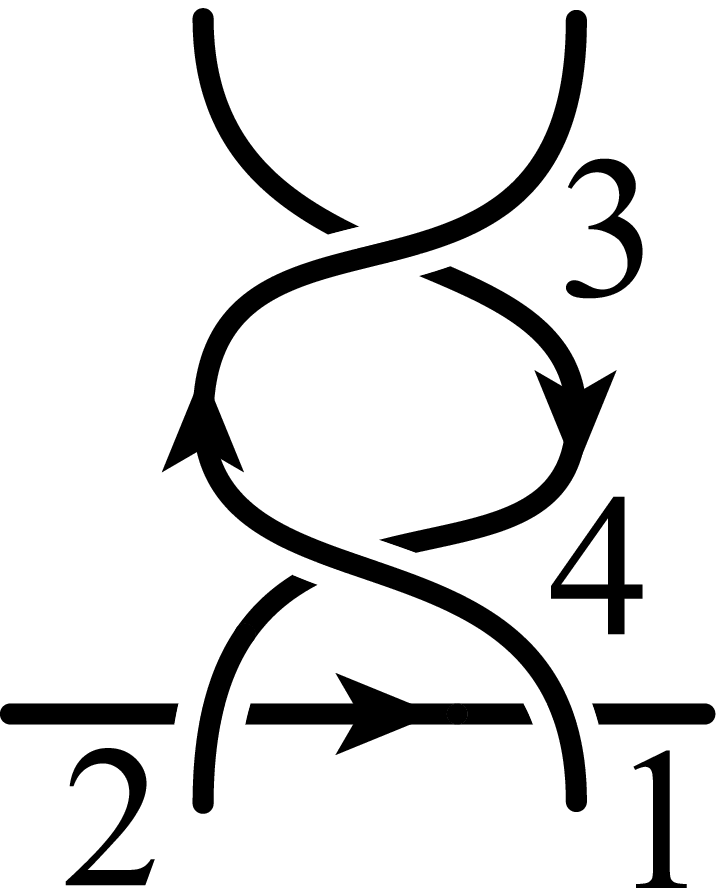}
    \ar[r]^{{R_2b}^{-1}} & \pd{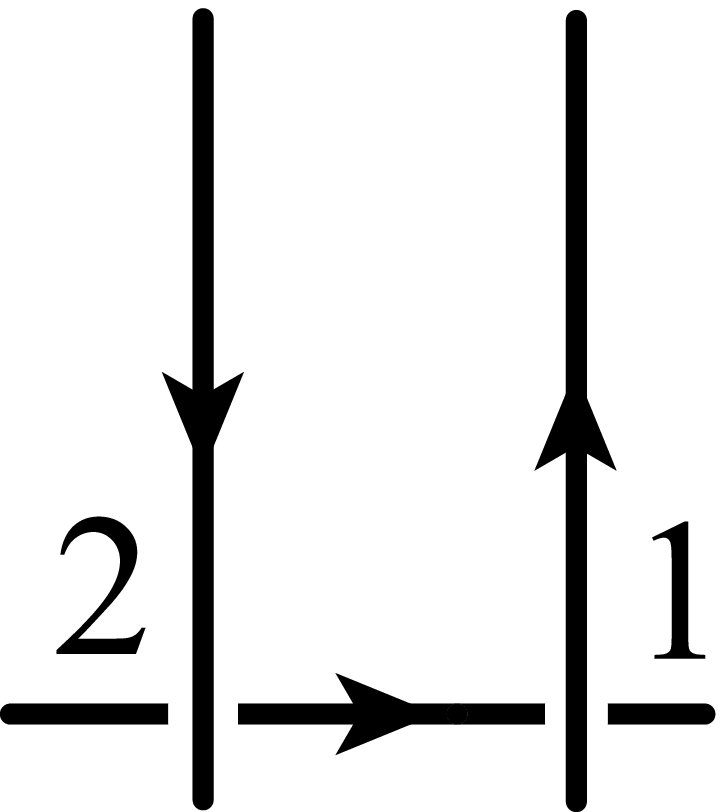} \ar[r]^{\sigma_{12}} & \pd{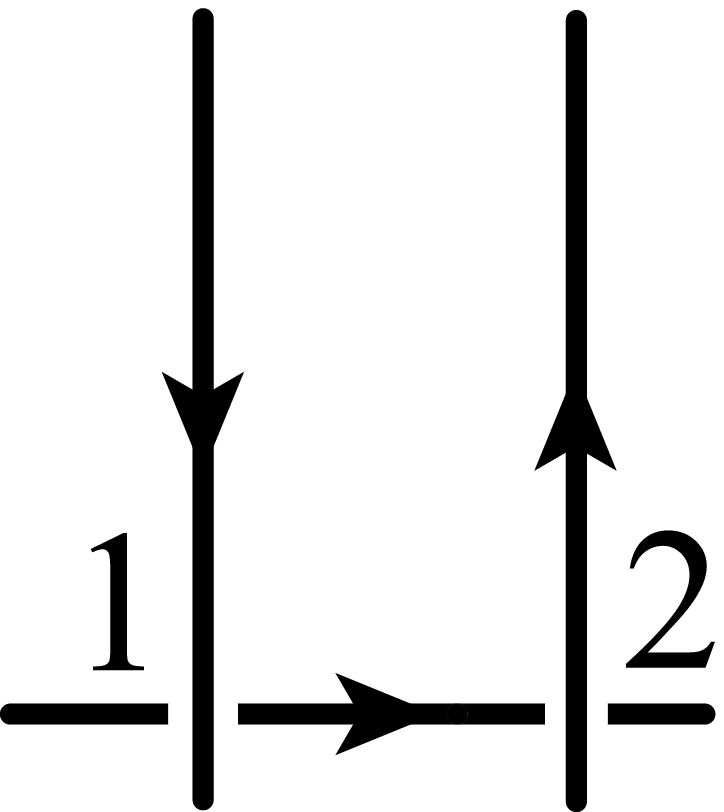}
    \\
    \pa{MM6f1_trace} \ar@{|->}[r]^{} & \pa{MM6f2_trace} \ar@{|->}[r]^{} & \pa{MM6f3_trace} \ar@{|->}[r]^{-1} & \pa{MM6f3_trace} \ar@{|->}[r]^{1} & \pa{MM6f3_trace} \\ & & & \qquad \ar@{|->}[r]^{} & \pa{MM6f4_trace} \ar@{|->}[r]^{-1} & \pa{MM6f4_trace} \ar@{|->}[r]^{} & \pa{MM6f1_trace} \ar@{|->}[r]^{1} & \pa{MM6f1_trace}
}
\end{align*}
}

The argument for the case in which the left vertical strand is oriented upward, and the right downward, is essentially the same.

}

\paragraph{\textbf{Interleaved variations}}
There are eight variations, and essentially two distinct computations will cover them all. Start with $hml^{-1}/\circlearrowright$, $\circlearrowleft/lmh^{-1}$, $mlh^{-1}/mhl$, and $lhm/hlm^{-1}$: we'll show the calculation for the first, and explain the necessary alterations for the other three versions.

{
\newcommand{\pd}[1]{\mathfig{0.1}{movie_moves/MM6/#1}}
\newcommand{\pa}[1]{\mathfig{0.05}{movie_moves/MM6/#1}}

\begin{align*}%
\xymatrix@R-7.5mm@C-1.5mm{
    \pd{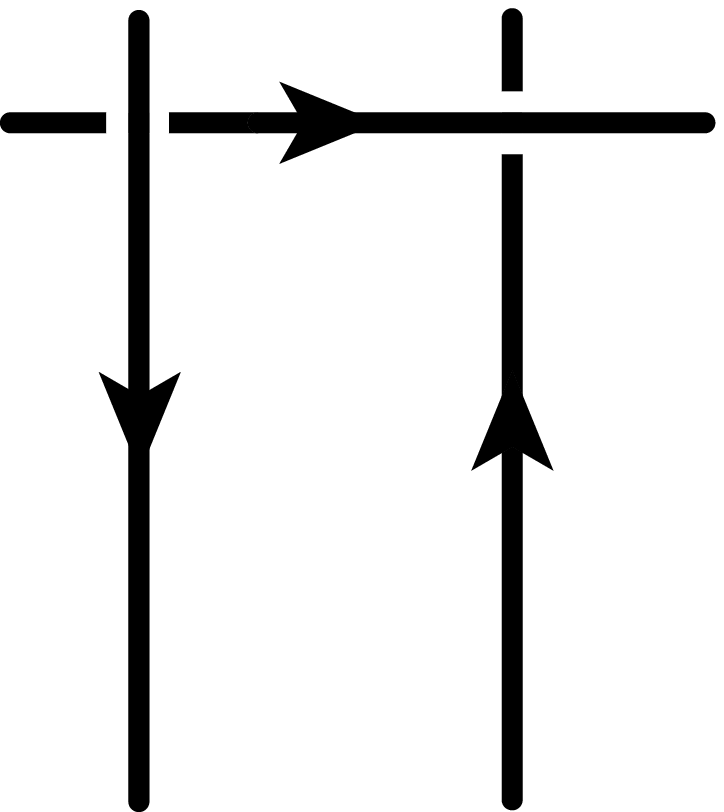} \ar[r]^{R2b} & \pd{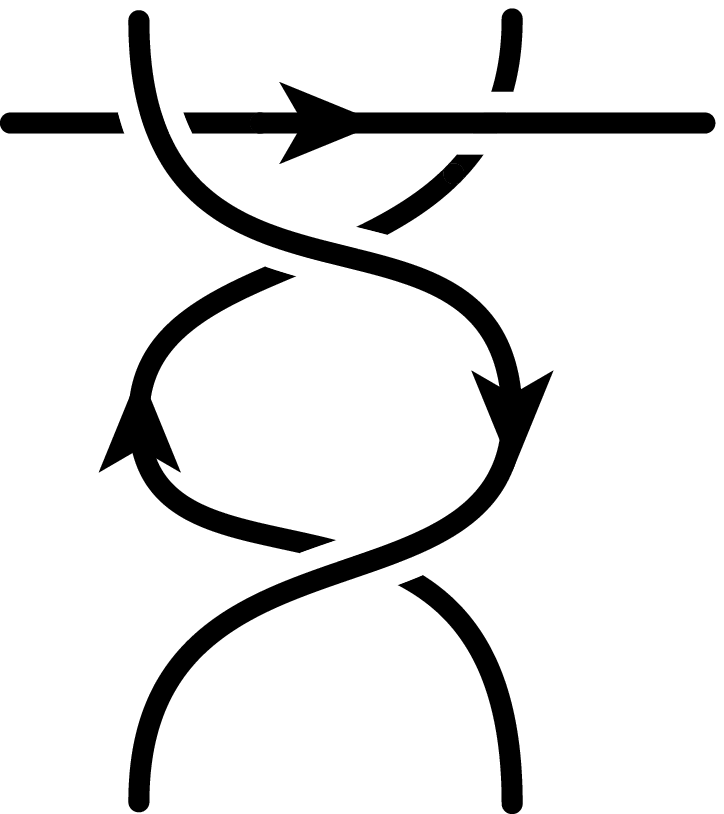} \ar[r]^{R3_{hml}^{-1}} & \pd{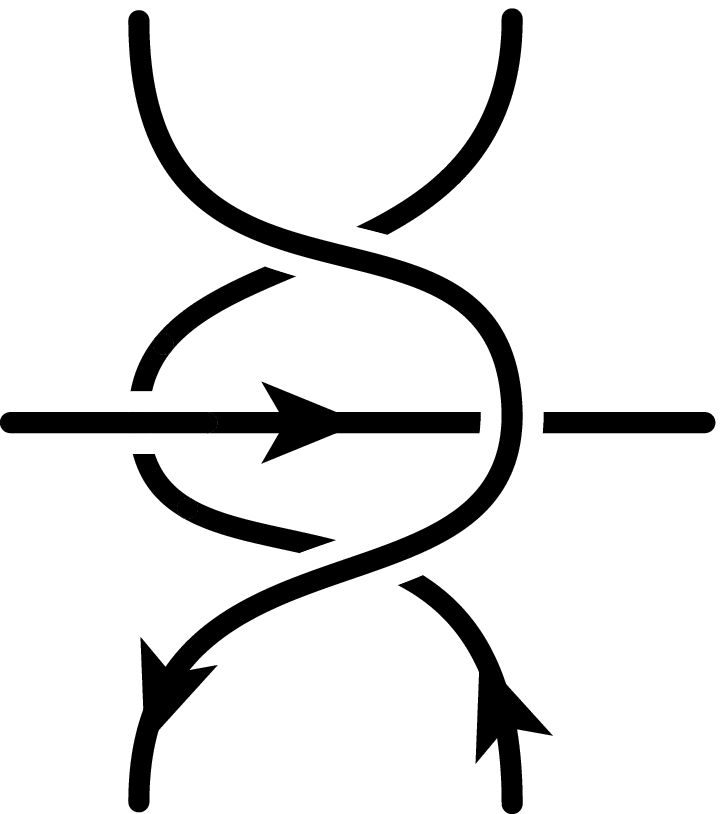} \ar[r]^{R3_{\circlearrowright}} & \pd{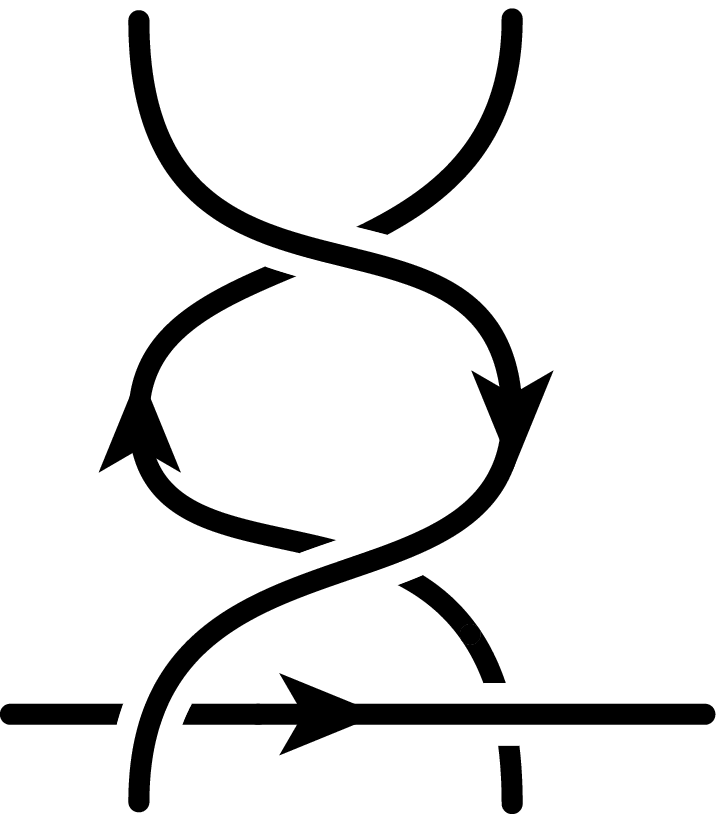} \ar[r]^{{R2b}^{-1}} & \pd{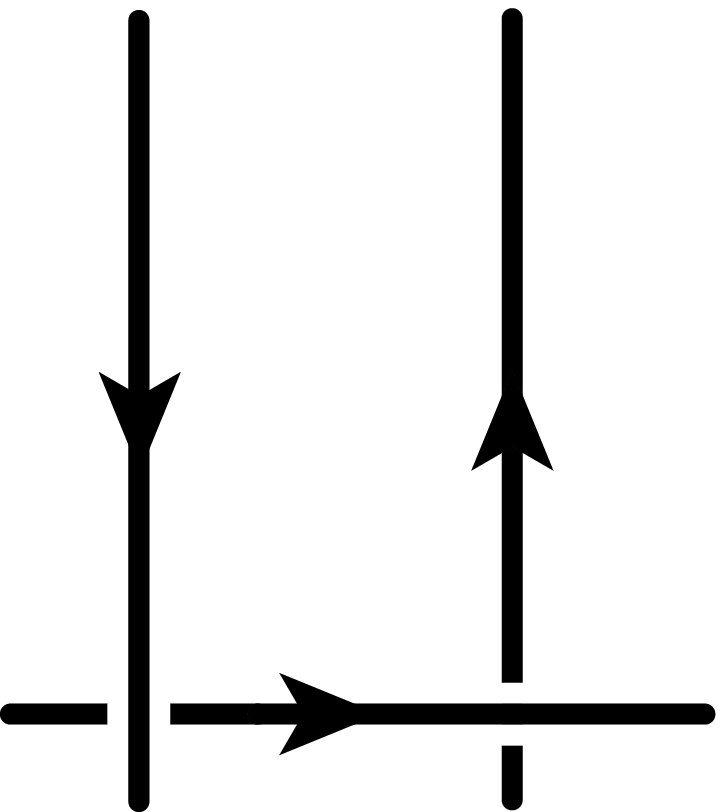} \\
    \\
    \pd{MM6f1_trace} \ar@{|->}[dr] \ar@{|->}[r] & \pd{MM6f2_trace} \ar@{|->}[dr]^{\PP} \ar@{|->}[r]^{\PP} \ar@{}[d]|{\directSum} & \pd{MM6f3_trace} \ar@{|->}[r]^{\OO} \ar@{}[d]|{\directSum} & \pd{MM6f4_trace} \ar@{|->}[r] \ar@{}[d]|{\directSum} & \pd{MM6f1_trace} \\
    & \pd{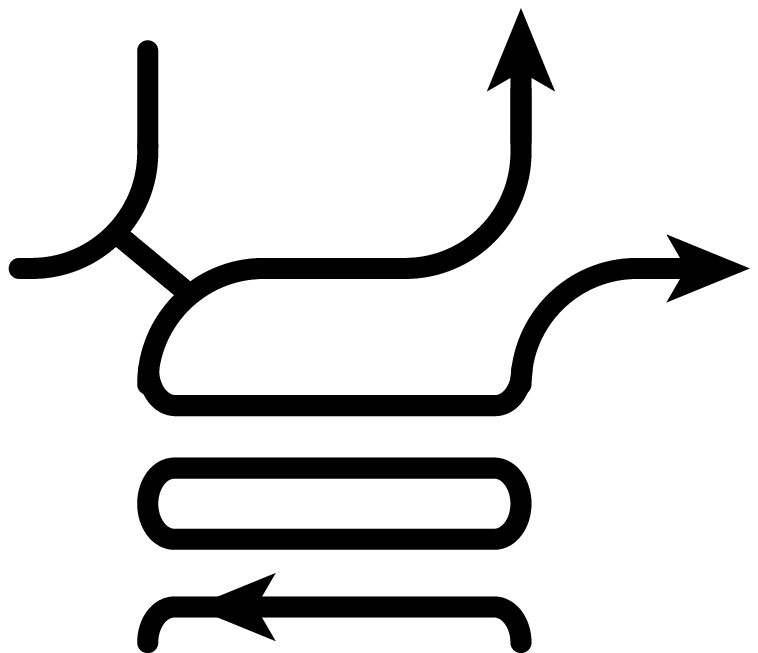} \ar@{|->}[dr]^{\PP} & \pd{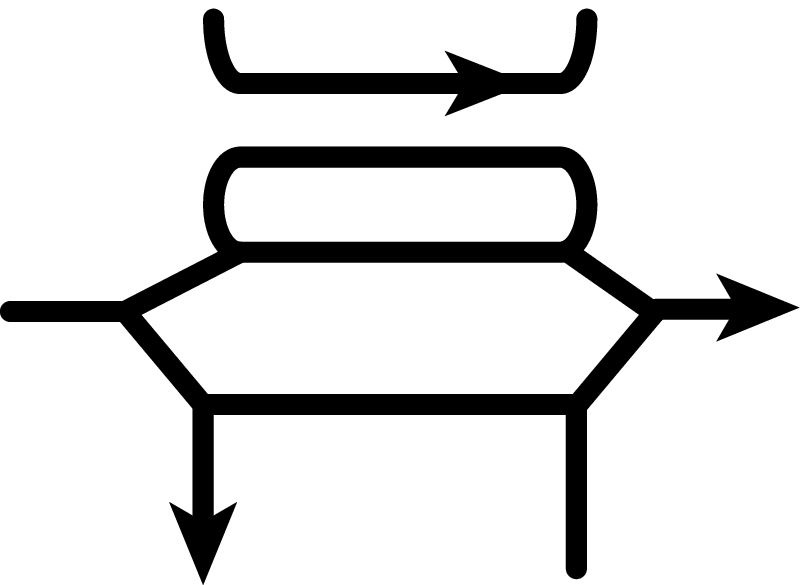} \ar@{|->}[r]^{\OO} \ar@{|->}[dr]^{\OO} & \pd{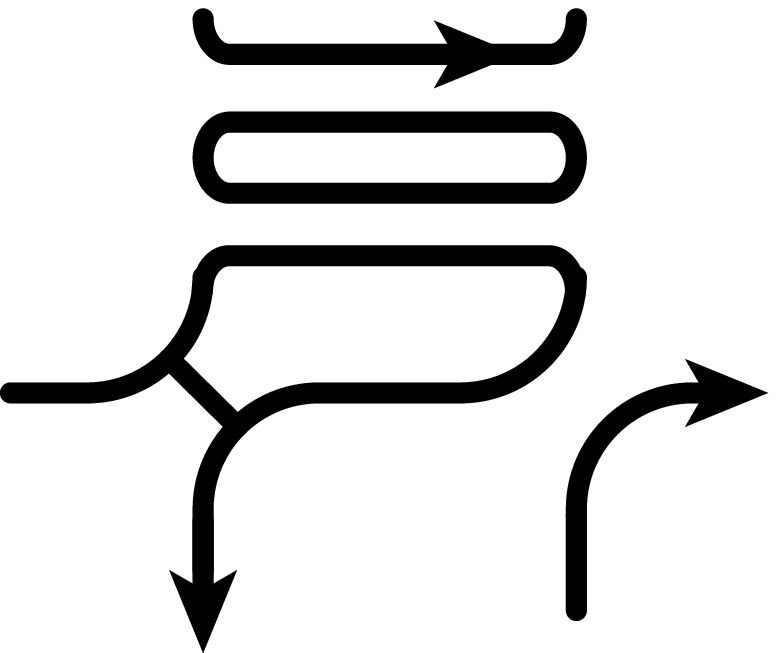} \ar@{|->}[ur] \ar@{}[d]|{\directSum} & \\
    & & 0 & \pd{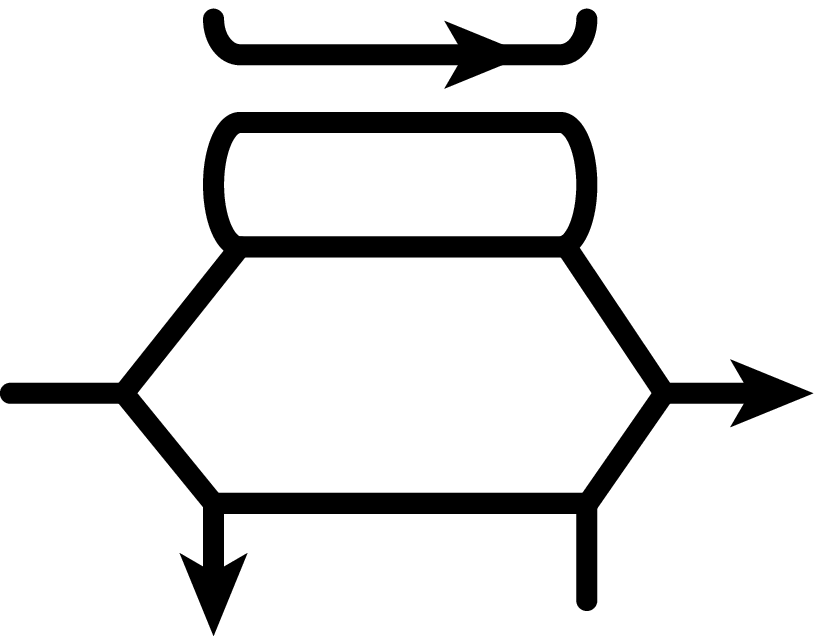} \ar@{|->}[r] & 0
}
\end{align*}
}

Notice that our first $R3$ map is ordered $\OP$ and the second $\PO$, each with the high crossing resolved, and that the maps for these moves are labeled by their source and target layers; in particular, the initial $\mathcal{O}$ layer for the second move and the final $\mathcal{P}$ layer for the first move coincide.

Lemma \ref{lem:R3-OP} tells us there are only four compositions we need to keep track of here. The first map into the second row has a doubly smoothly-resolved target in the initial $\mathcal{P}$ layer of $R3_{hml}^{-1}$, which thereafter maps to zero by Lemma \ref{lem:R3-PP}. The composition including the rest of the second row contains a blister, and thus is the zero map; this is because the second map has a bubble from Lemma \ref{lem:R3-PP}, the third map unzips the bubbled bigon by Lemma \ref{lem:cw-mid}, and the fourth map, an $R2b$ untuck, caps it off. The composition terminating at zero in the third row also uses Lemma \ref{lem:cw-mid}.

Thus we're left with only the first row, which is easily seen to be the same composition we saw in the non-interleaved anti-parallel case: the identity.

The calculations for the $\circlearrowleft^{-1}/lmh^{-1}$, $mlh^{-1}/mhl^{-1}$, and $lhm^{-1}/hlm$ variations are very similar. For $\circlearrowleft^{-1}/lmh^{-1}$, the initial object will have an I-resolved left crossing and a smoothly-resolved right crossing, and we'll resolve each $R3$ move into layers using the low crossing. Thus we'll need to compute using the $\overline{R3}$ maps. The $mlh^{-1}/mhl$ and $lhm/hlm^{-1}$ variations are even easier: we start with the doubly smoothly-resolved object in each case, and resolve into layers using the high crossings or the low crossings, respectively. Also, in each of these three variations, there is no need for an analogy of the obscure Lemma \ref{lem:cw-mid}. This is because the corresponding $\OO$ map originating in the second row always has just one component, the identity, by Lemma \ref{lem:R3-OO}. Crossing reordering maps are trivial in all four of these variations.

The computations for $\circlearrowright/hml^{-1}$, $lmh^{-1}/\circlearrowleft^{-1}$, $mhl^{-1}/mlh^{-1}$, and $hlm/lhm^{-1}$ are somewhat different; again, we'll explicitly show the first.
{
\newcommand{\pd}[1]{\mathfig{0.1}{movie_moves/MM6/#1}}
\newcommand{\pa}[1]{\mathfig{0.05}{movie_moves/MM6/#1}}
\begin{align*}%
\xymatrix@R-7.5mm@C-1.5mm{
    \pd{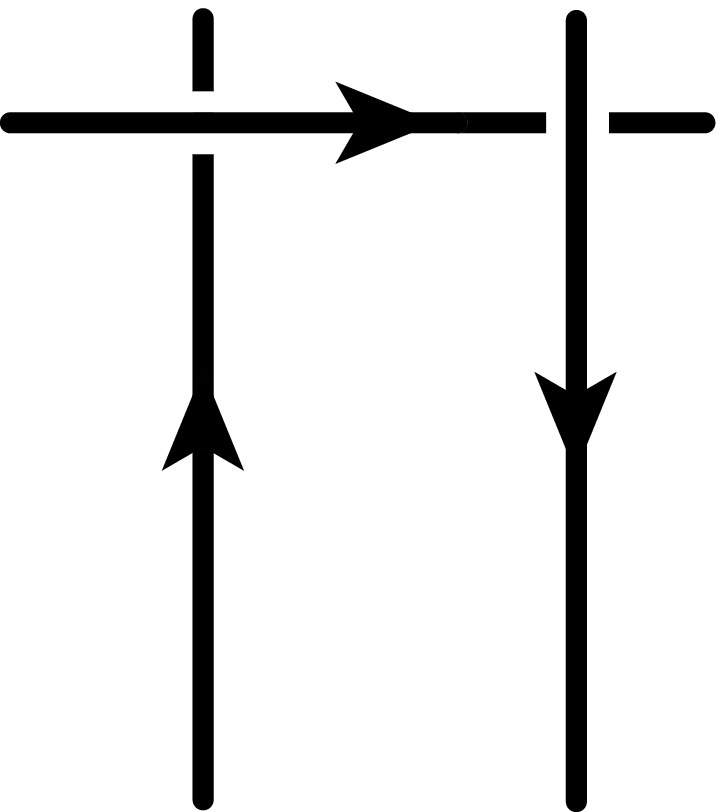} \ar[r]^{R2b} & \pd{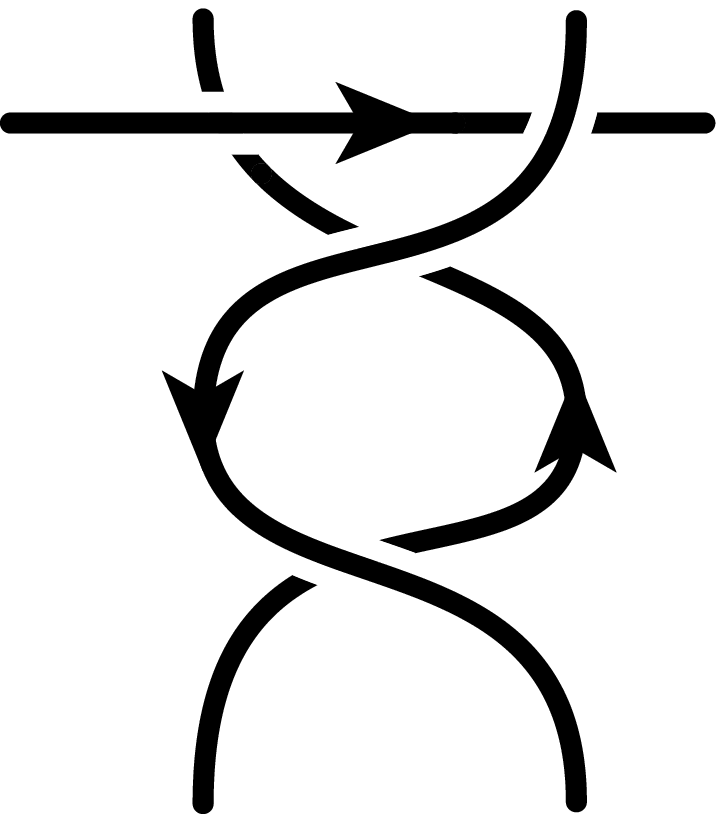} \ar[r]^{R3_{\circlearrowright}} & \pd{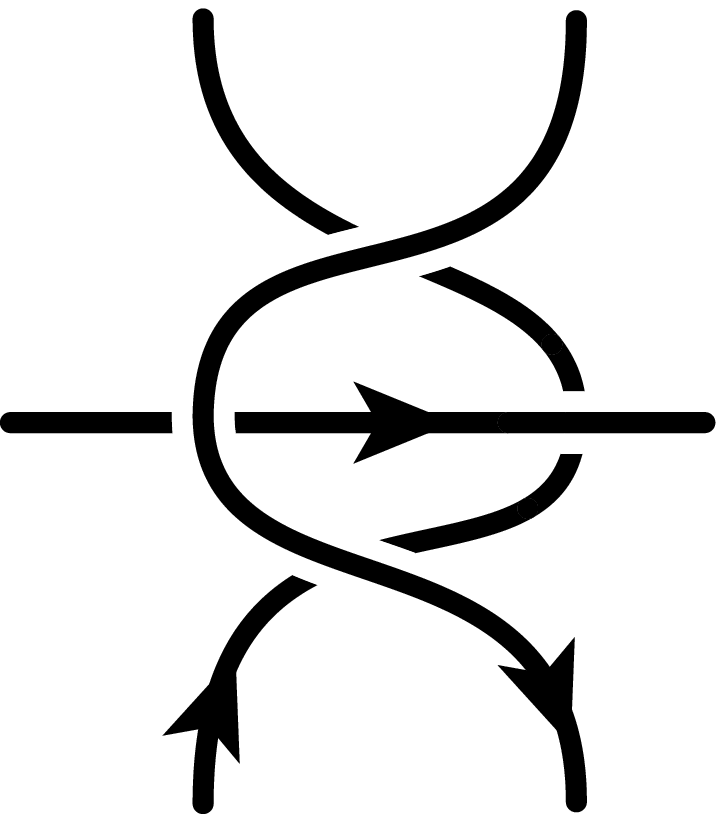} \ar[r]^{\low{hml}{-1}} & \pd{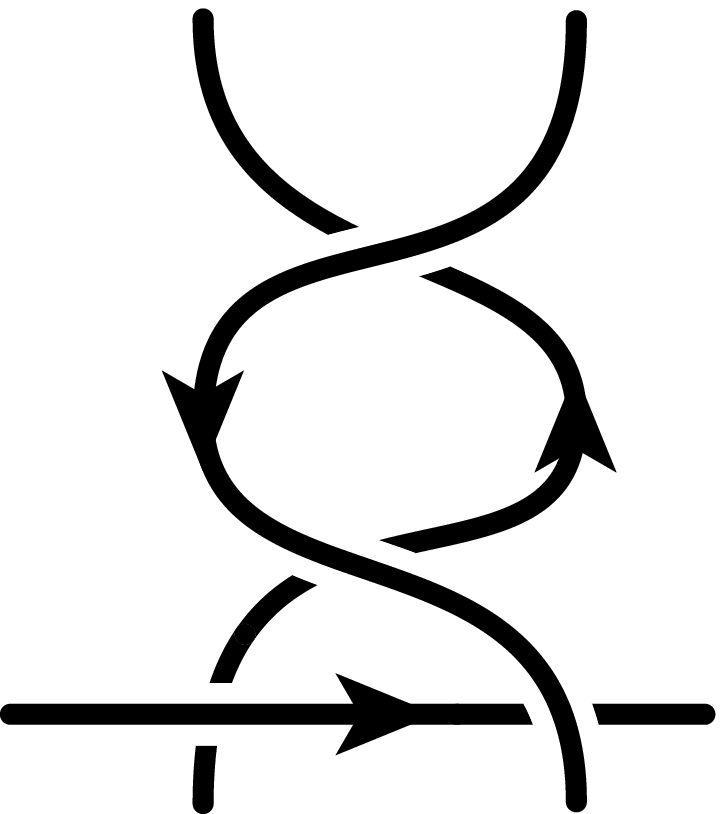} \ar[r]^{{R2b}^{-1}} & \pd{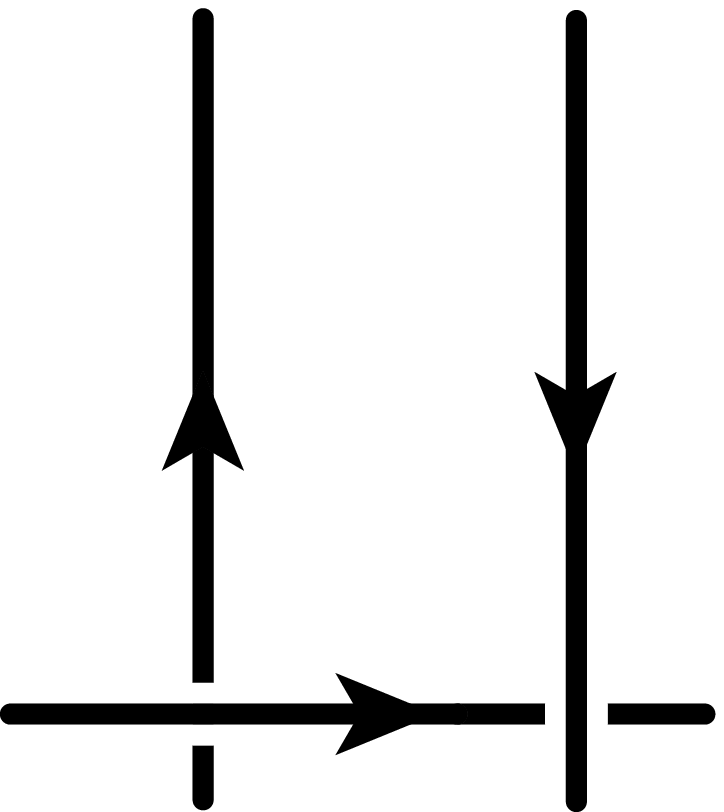} \\
    \\
    \pd{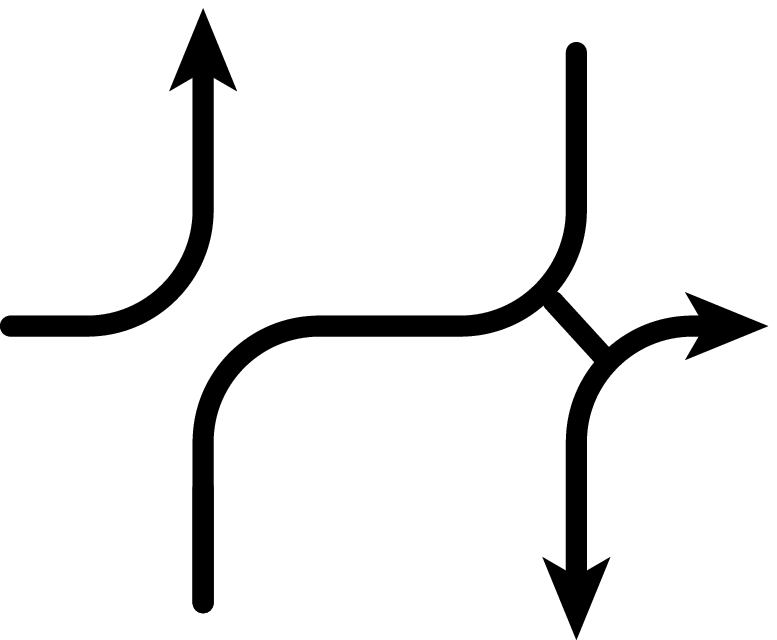} \ar@{|->}[r] \ar@{|->}[dr] & \pd{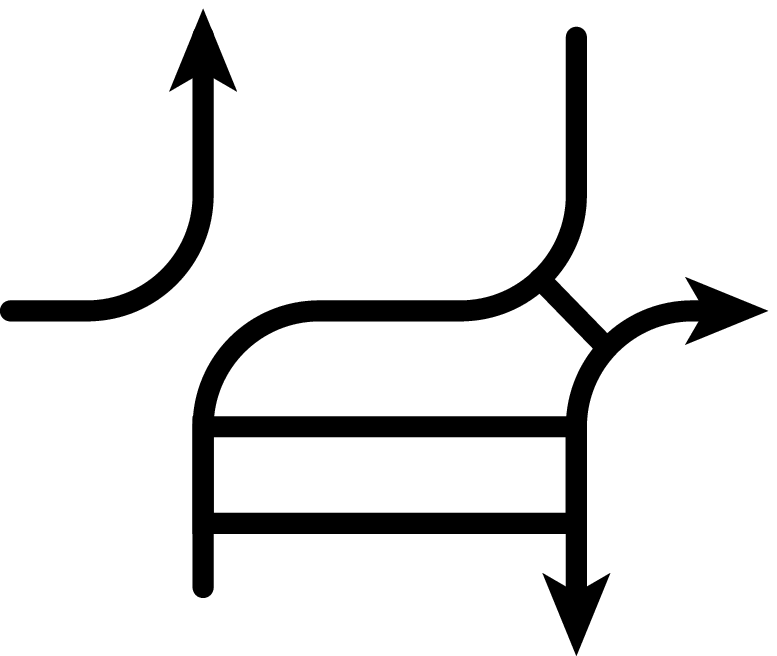} \ar@{|->}[r]^{\OO} \ar@{}[d]|{\directSum} & \pd{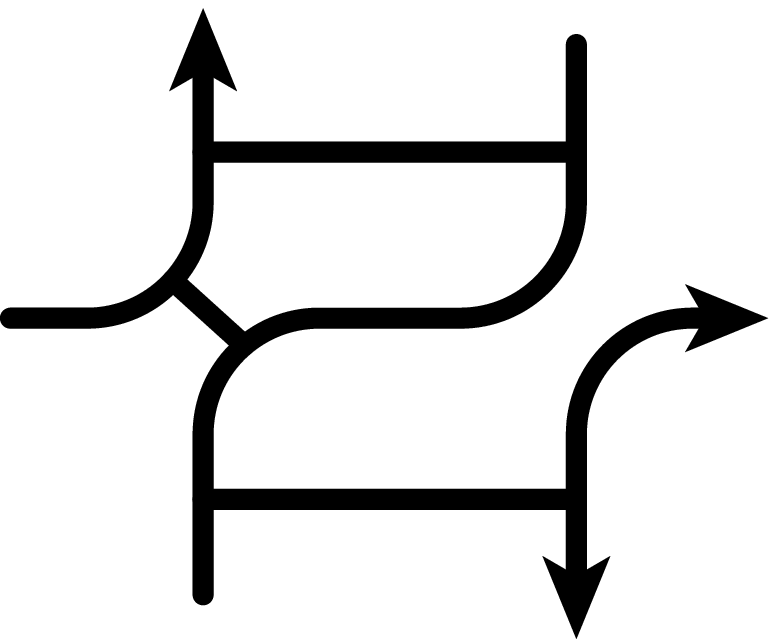} \ar@{|->}[r]^{\OO} \ar@{|->}[dr]^{\OP} \ar@{}[d]|{\directSum} & \pd{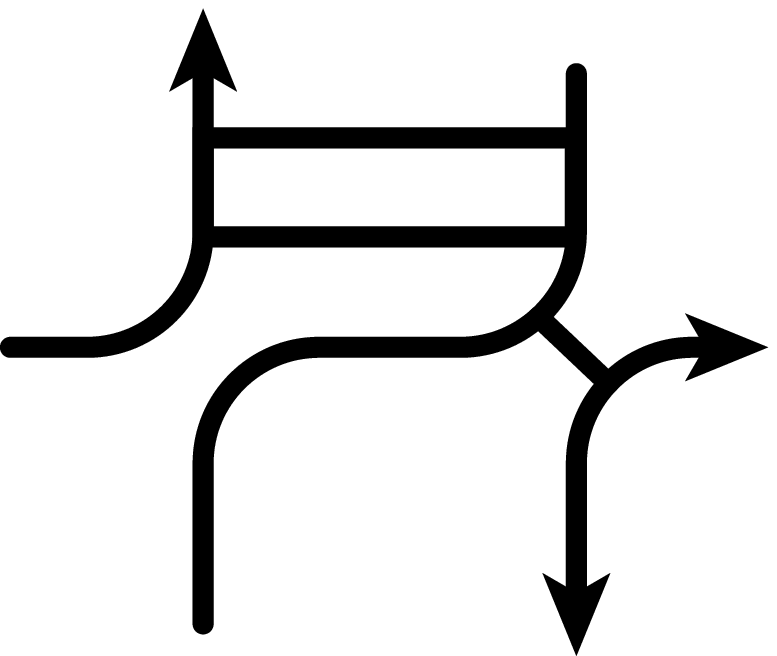} \ar@{|->}[r] \ar@{}[d]|{\directSum} & \pd{MM6intbf1_trace} \ar@{}[d]|{\directSum} \\
     & \pd{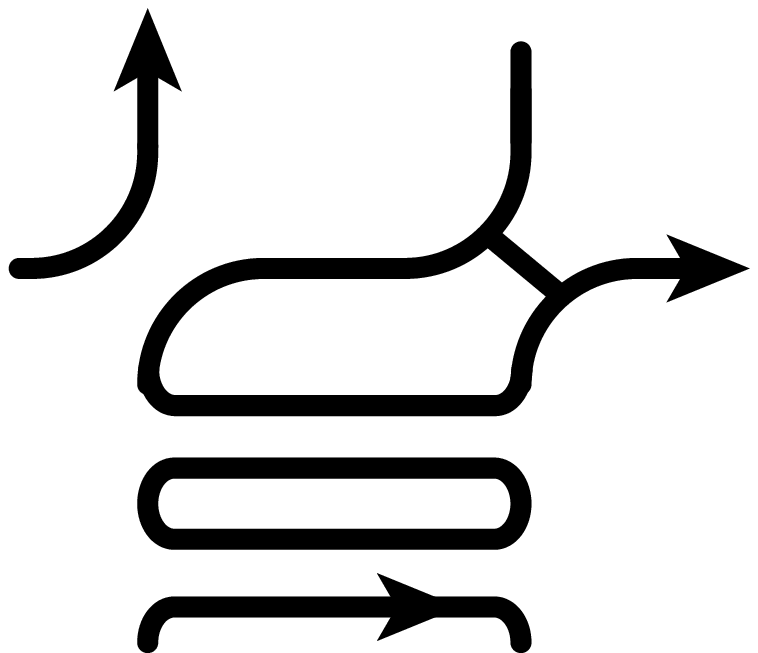} \ar@{|->}[r]^{\OO} & \pd{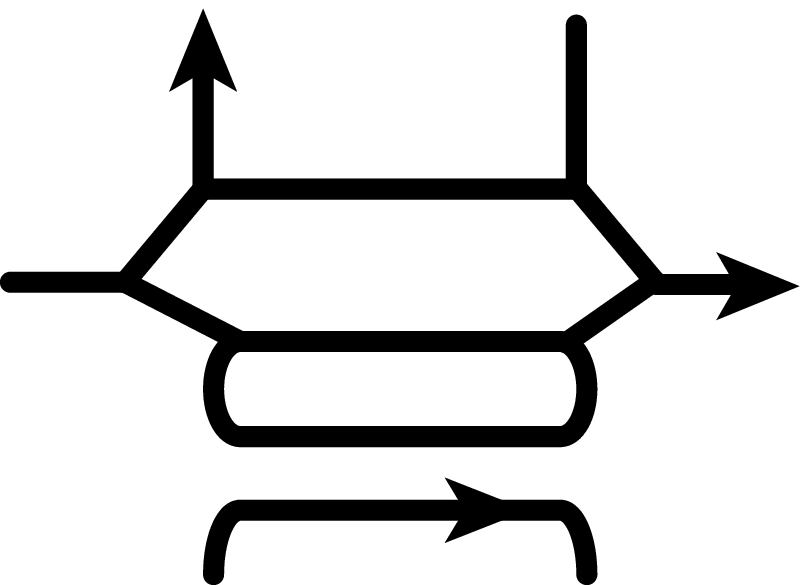} \ar@{|->}[r]^{\PP} & \pd{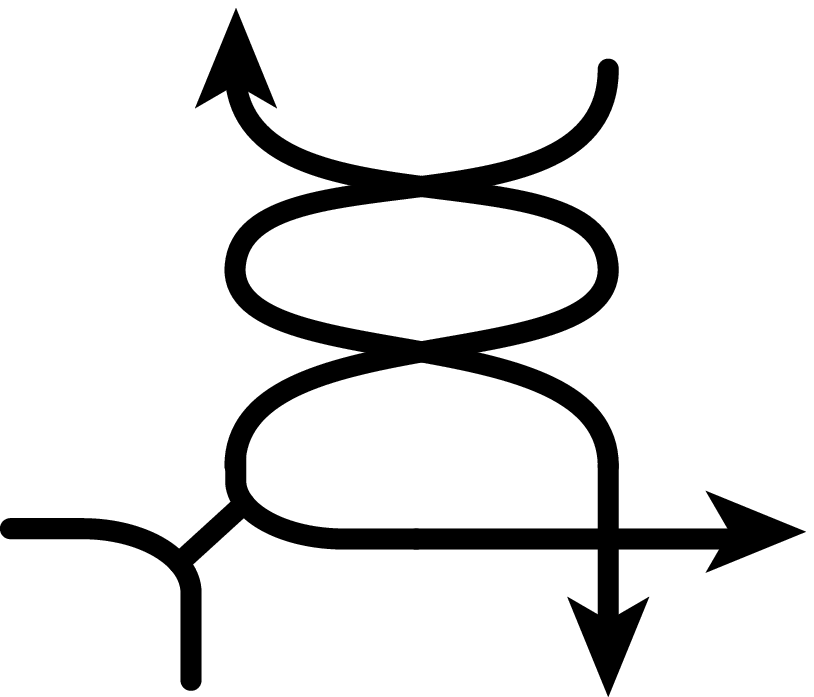} \ar@{|->}[r] & \pd{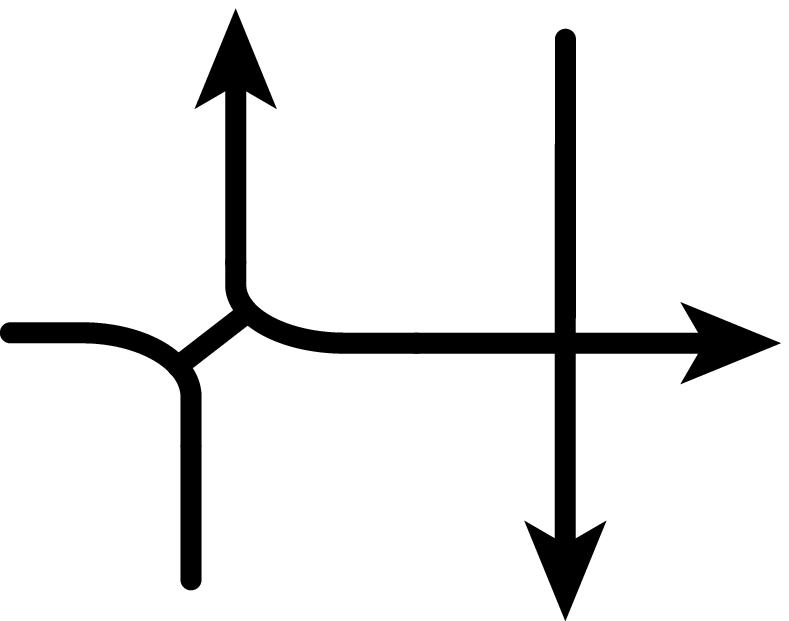}
}
\end{align*}
}
Now our first $R3$ map is ordered $\PO$ with the high crossing resolved, and the second is ordered $\OP$ with the low crossing resolved. Again, we'll keep track of the layers to which objects belong by referring to the labels on the maps.

By Lemma \ref{lem:R3-OP}, we have three compositions to consider. Two of them factor through the second row, and thus map to a complex with the left crossing I-resolved; since our map is a multiple of the identity, these compositions must sum to zero. (Note that the $\OO$ map on the second row comes from Lemma \ref{lem:cw-mid}.) So we're left with the first row. Using Lemma \ref{lem:R3-OO} for the first $R3$, Lemma \ref{lem:R3L-OO} for the second $R3$ (where our map comes from resolving the low crossing), the $R2b$ map definitions, and an application of the airlock relation, we get the map $(-r) \compose id \compose (-R) \compose \bar{r} = -id$. We'll also get a crossing reordering sign here ($(-1)^5$, shown below), giving us the identity on the nose.
{
\begin{align*}
\newcommand{\pd}[1]{\mathfig{0.055}{movie_moves/MM6/reordering2/#1}}
\newcommand{\pa}[1]{\mathfig{0.055}{movie_moves/MM6/#1}}
\xymatrix@R-7.5mm@C-2.5mm{
    \pd{MM6f1} \ar[r]^{R_2b} & \pd{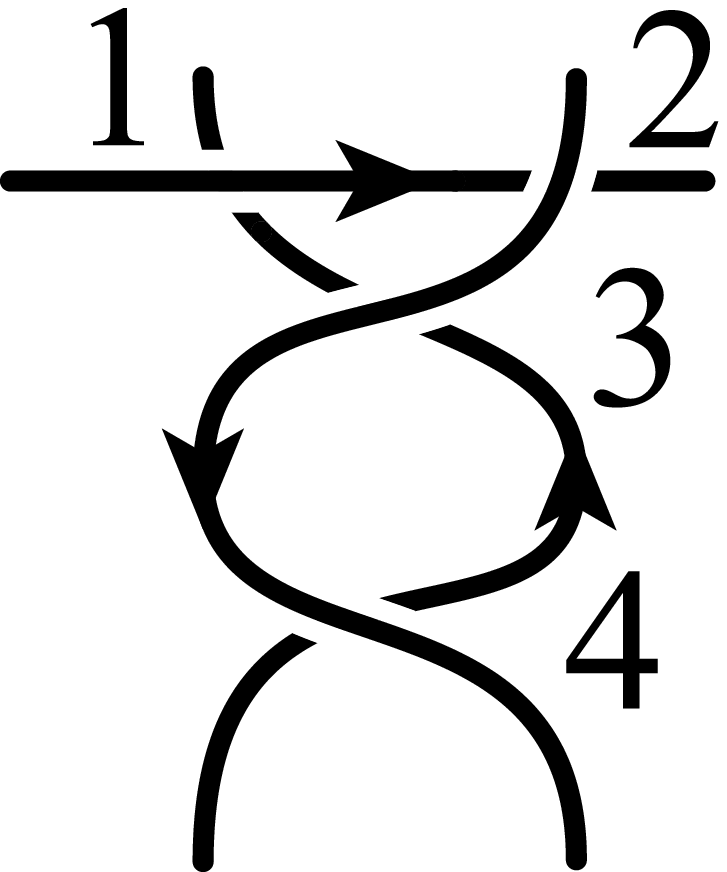} \ar[r]^{\sigma_{23}} & \pd{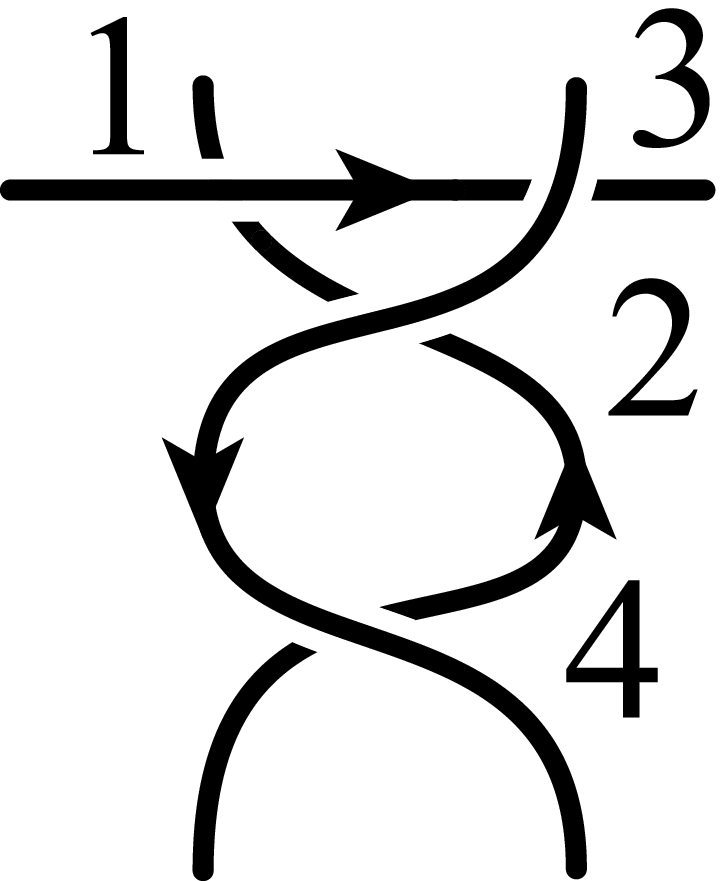} \ar[r]^{\sigma_{12}} & \pd{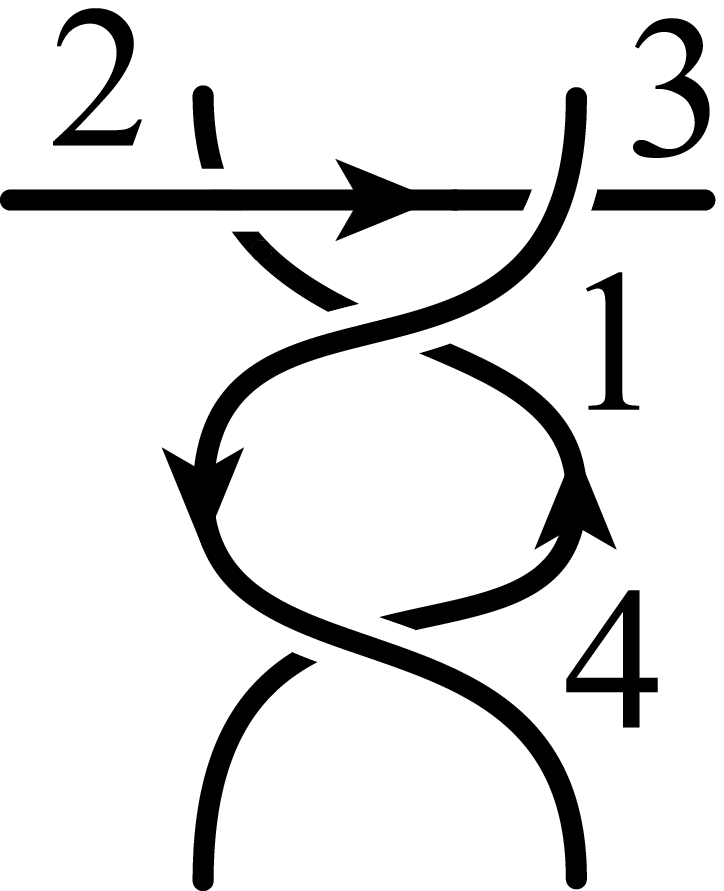} \ar[r]^{R_3} & \pd{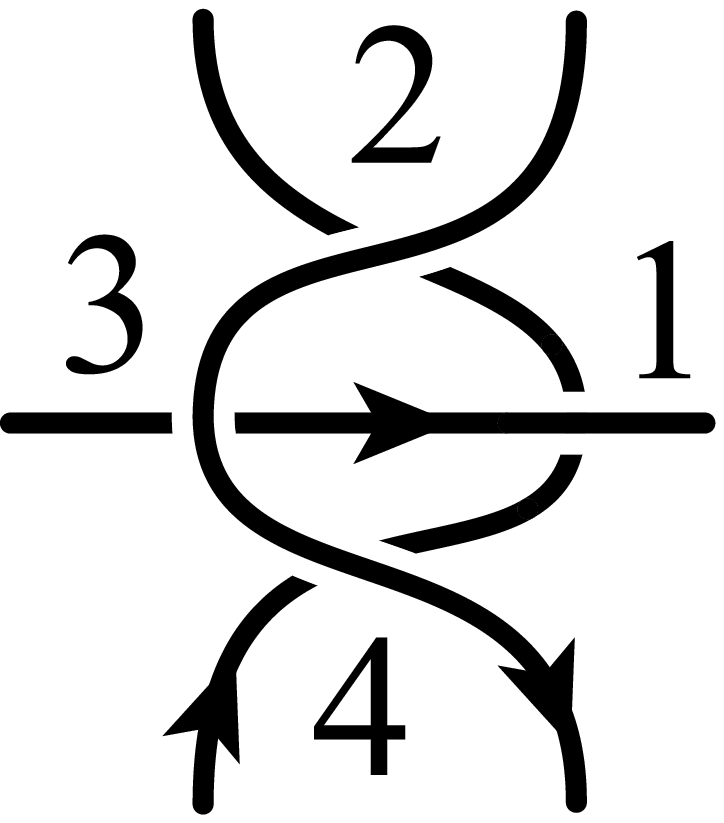} \ar[r]^{\sigma_{23}} & \pd{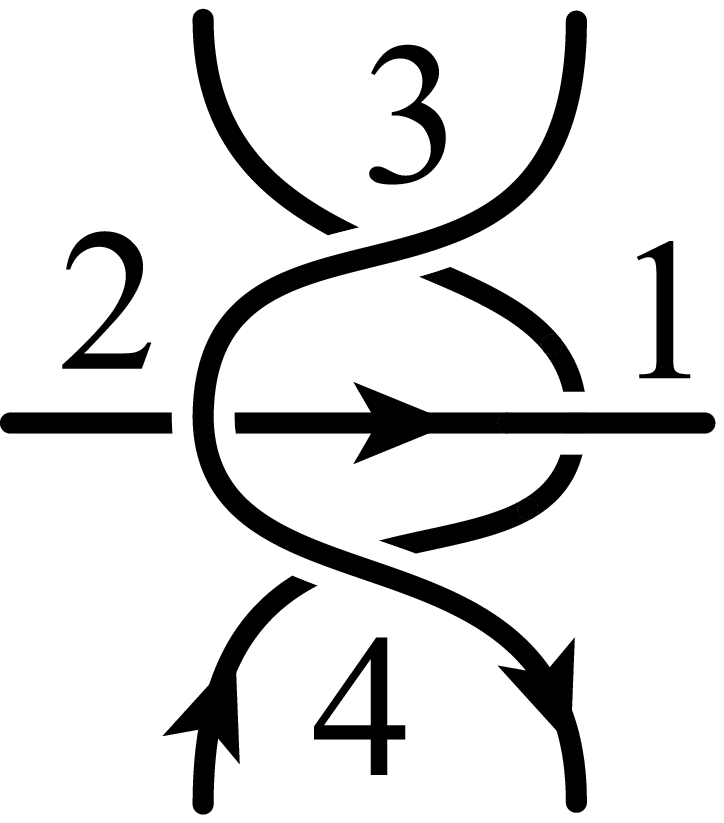} \ar[r]^{\sigma_{12}} & \pd{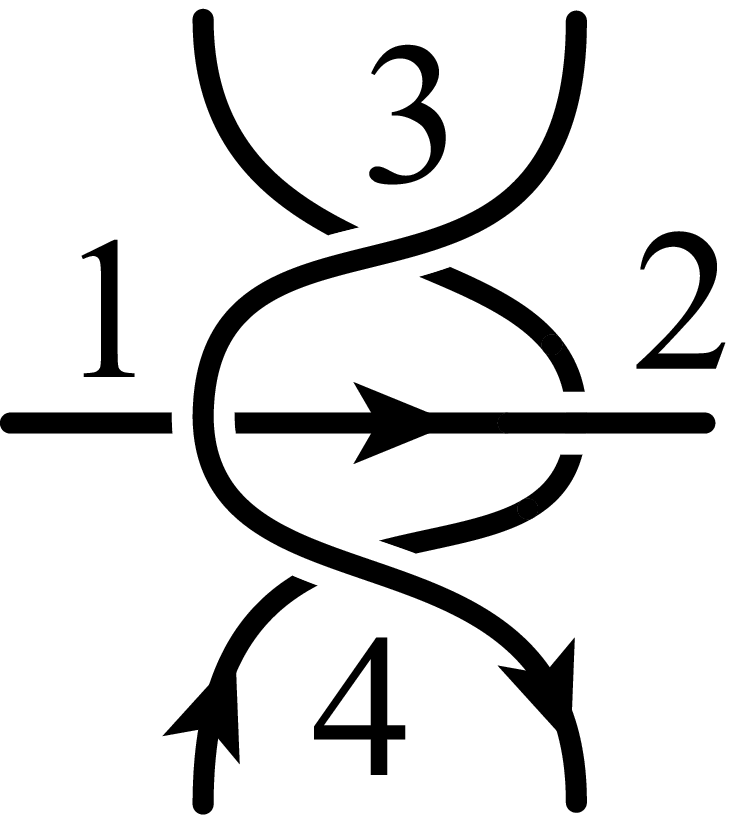}
    \ar[r]^{\sigma_{34}} & \pd{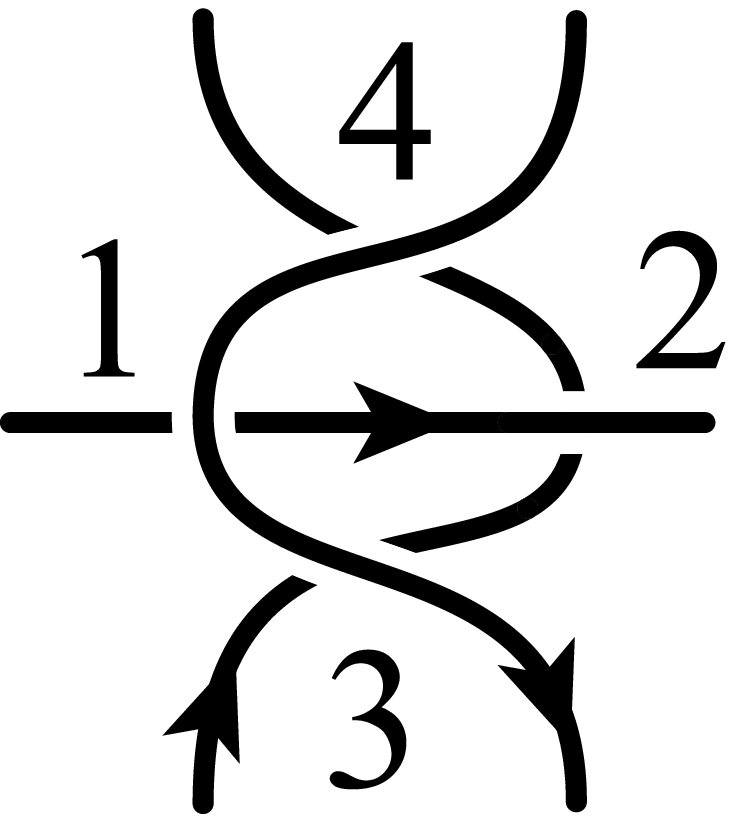} \\ \qquad \; \ar[r]^{\sigma_{23}} & \pd{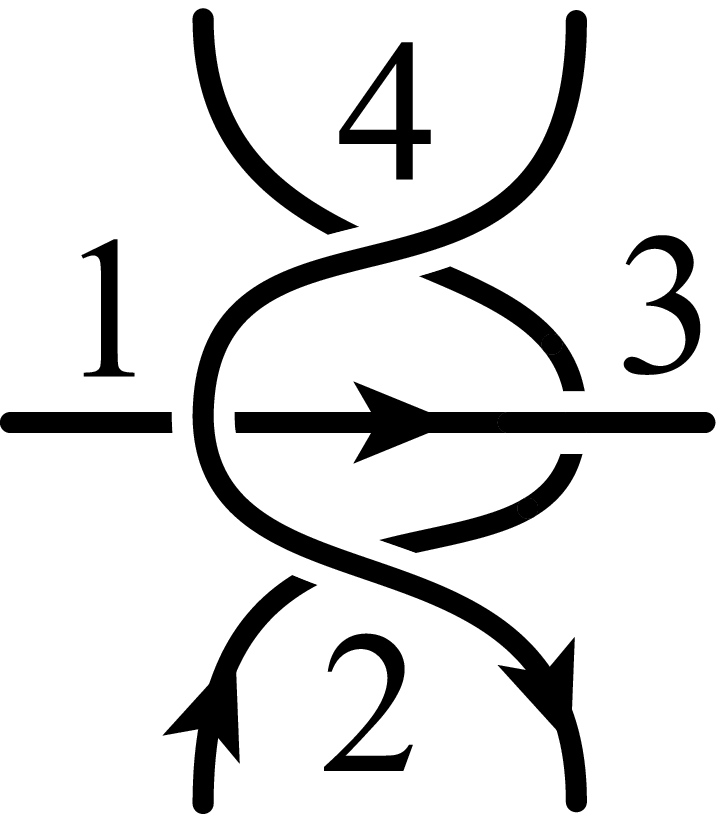} \ar[r]^{R_3} & \pd{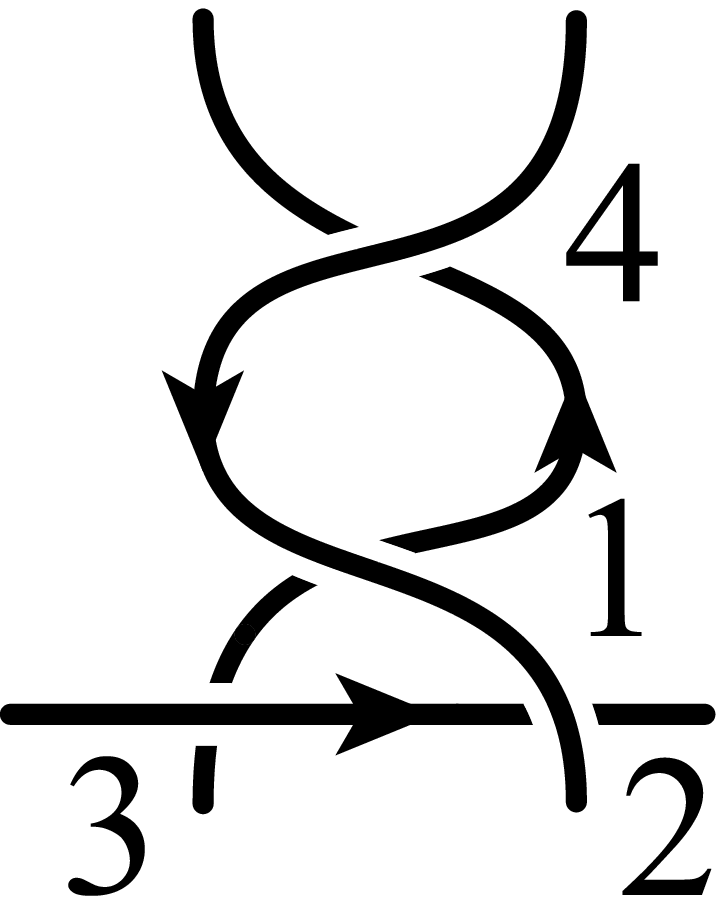} \ar[r]^{\sigma_{12}} & \pd{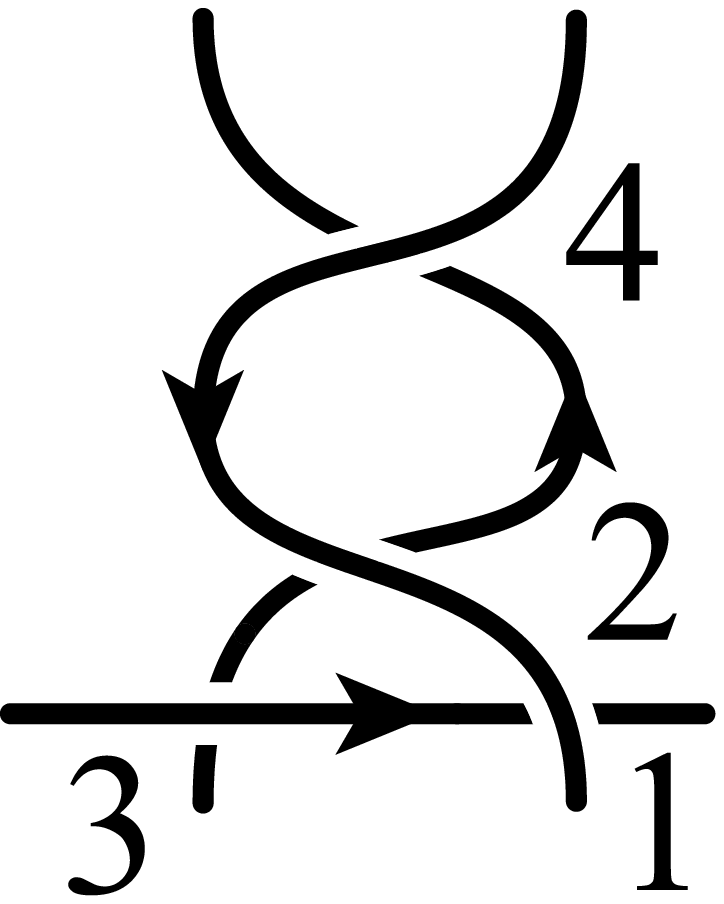} \ar[r]^{\sigma_{23}} & \pd{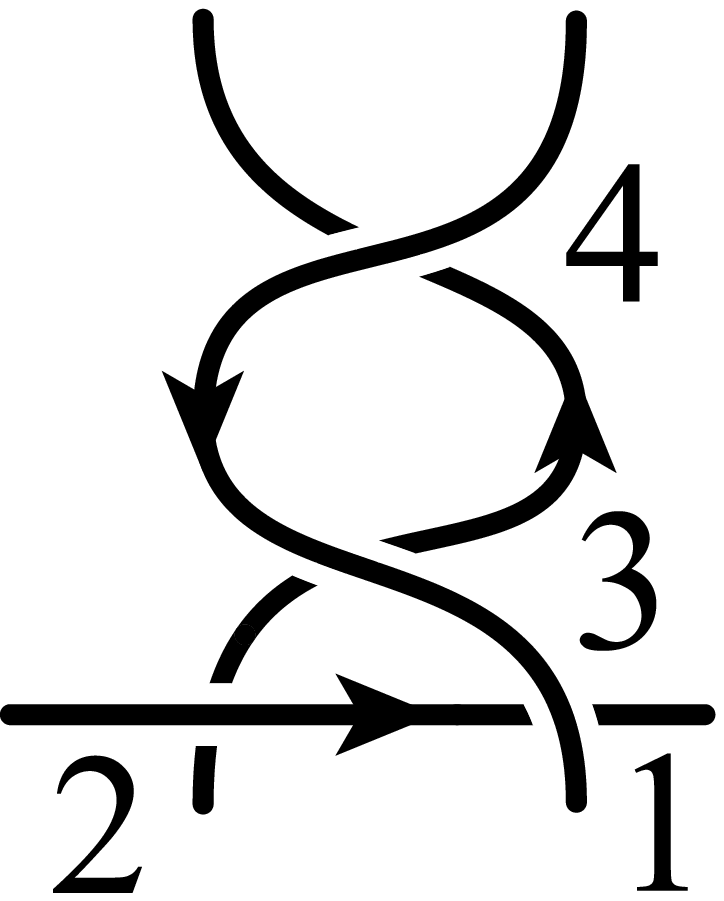} \ar[r]^{\sigma_{34}} & \pd{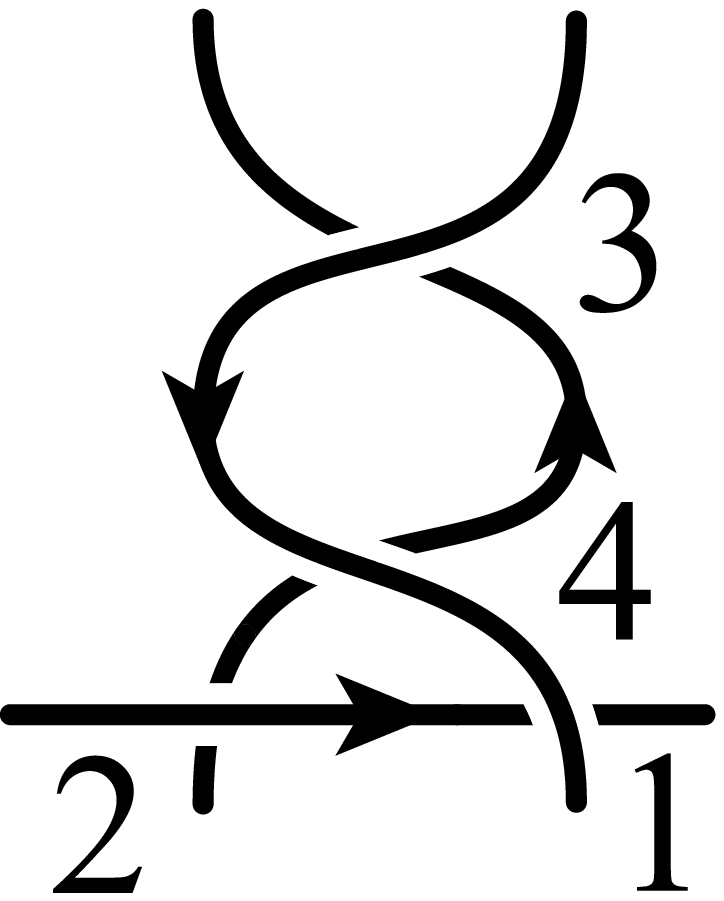} \ar[r]^{R_2b^{-1}} & \pd{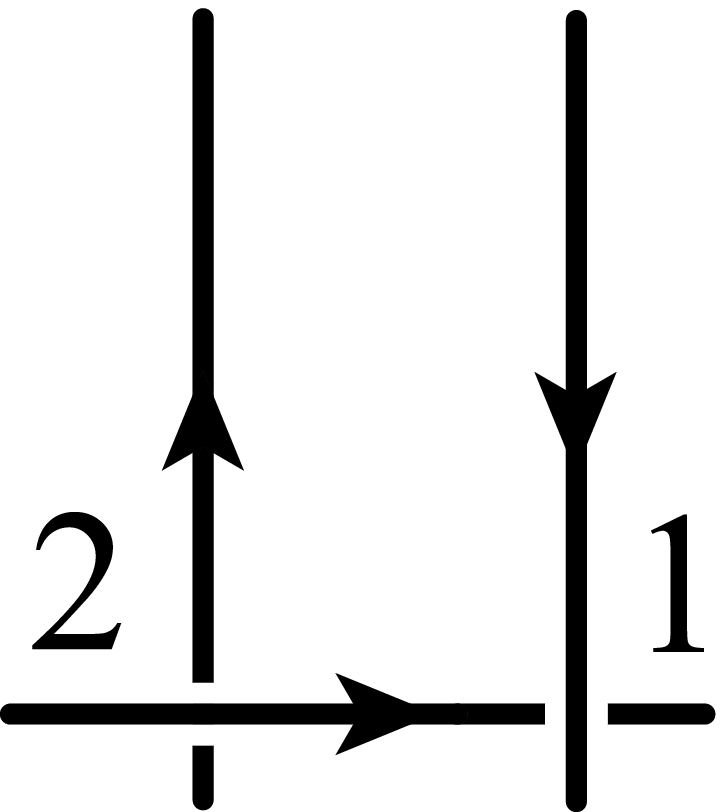} \ar[r]^{\sigma_{12}} & \pd{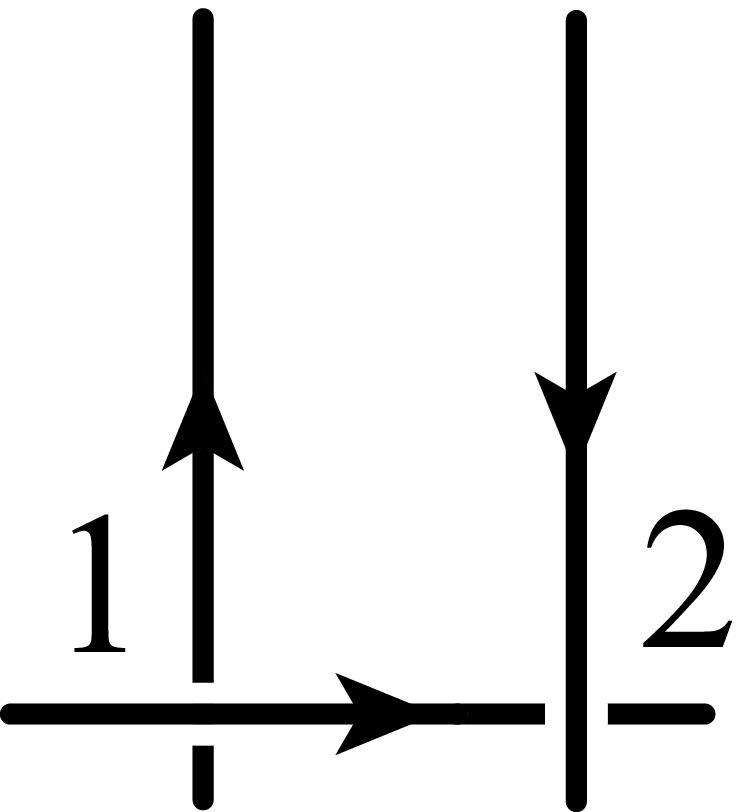}
    \\
    \pa{MM6intbf1_trace} \ar@{|->}[r]^{} & \pa{MM6intbf2_trace} \ar@{|->}[r]^{-1} & \pa{MM6intbf2_trace} \ar@{|->}[r]^{1} & \pa{MM6intbf2_trace} \ar@{|->}[r]^{} & \pa{MM6intbf3_trace} \ar@{|->}[r]^{-1} & \pa{MM6intbf3_trace} \ar@{|->}[r]^{1} & \pa{MM6intbf3_trace} \ar@{|->}[r]^{-1} & \pa{MM6intbf3_trace} \\ \qquad \; \ar@{|->}[r]^{1} & \pa{MM6intbf3_trace} \ar@{|->}[r]^{} & \pa{MM6intbf4_trace} \ar@{|->}[r]^{-1} & \pa{MM6intbf4_trace} \ar@{|->}[r]^{1} & \pa{MM6intbf4_trace} \ar@{|->}[r]^{-1} & \pa{MM6intbf4_trace} \ar@{|->}[r]^{} & \pa{MM6intbf1_trace} \ar@{|->}[r]^{1} & \pa{MM6intbf1_trace}
}
\end{align*}
}

There are a few modifications necessary for $lmh^{-1}/\circlearrowleft$, $mhl/mlh^{-1}$, and for $hlm^{-1}/lhm$. In the $lmh^{-1}/\circlearrowleft$ case, we start with the object with smoothly-resolved left crossing and I-resolved right crossing, and resolve the first $R3$ on low and the second on high; a crossing ordering sign will appear here. For each of $mhl/mlh^{-1}$ and $hlm^{-1}/lhm$ our initial object will be the doubly smoothly-resolved one; $hlm^{-1}$ and $mlh^{-1}$ should be resolved on low, while $mhl$ and $lhm$ should be resolved on high. In all three cases, the identity will result.

\subsubsection*{MM7}
\label{sssec:MM7}
$$\mathfig{0.4}{movie_moves/MM7}$$

There are only four variations of MM7, depending on the
orientation of the strand, and whether the leading crossing is
positive or negative. It's easy to check that reversing orientations
has little effect on the two subsequent calculations.

When the leading crossing is positive, we get
\begin{align*}
\newcommand{\pd}[1]{\mathfig{0.09}{movie_moves/MM7/#1}}
\newcommand{\pa}[1]{\mathfig{0.05}{movie_moves/MM7/#1}}
\xymatrix@R-7.5mm@C+5mm{
    \pd{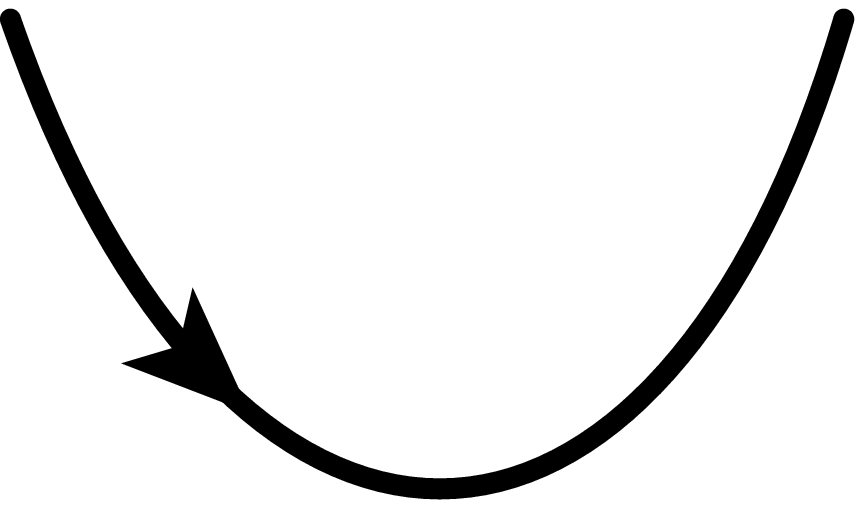} \ar[r]^{R_1a} & \pd{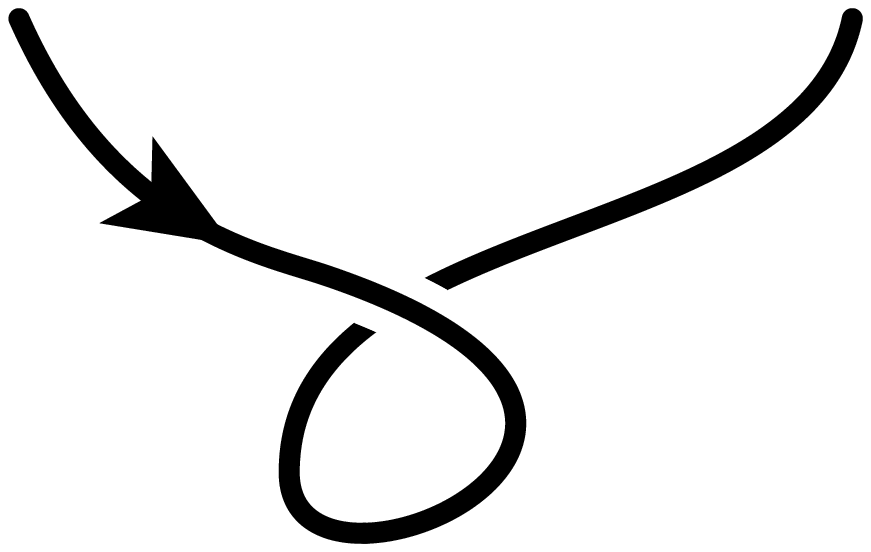} \ar[r]^{R_1b} & \pd{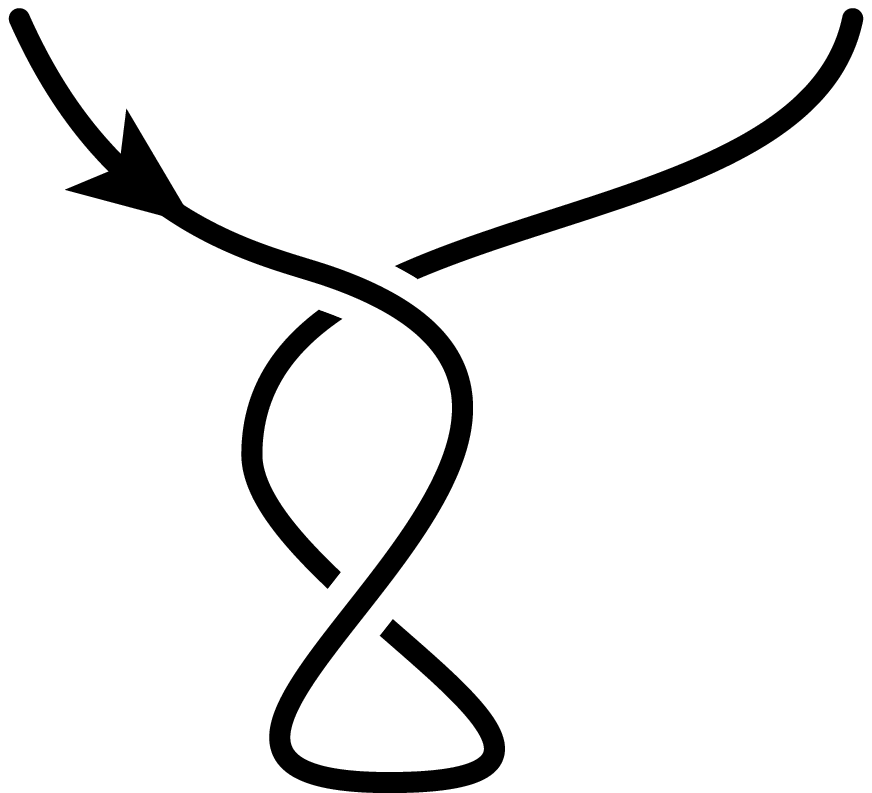} \ar[r]^{{R_2b}^{-1}} & \pd{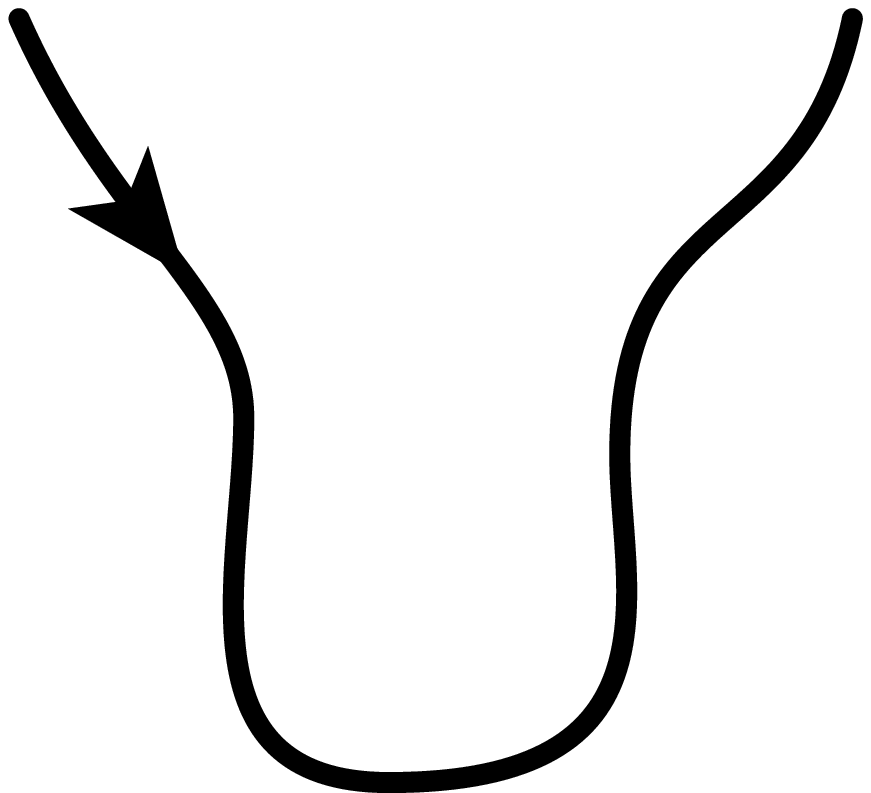} \\
    \\
    \pd{MM7f1} \ar@{|->}[r]^{\text{saddle}} & %
    \pd{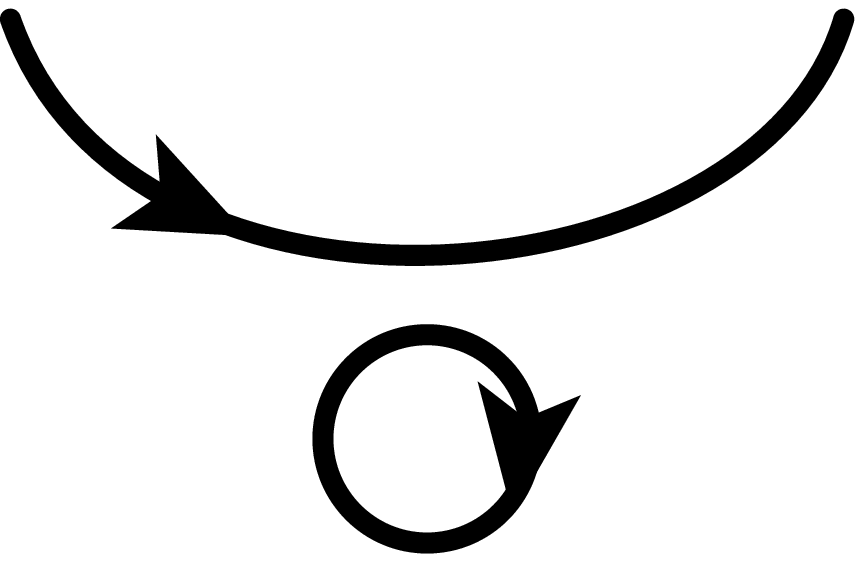} \ar@{|->}[r]^{\text{cup}} & %
    \pd{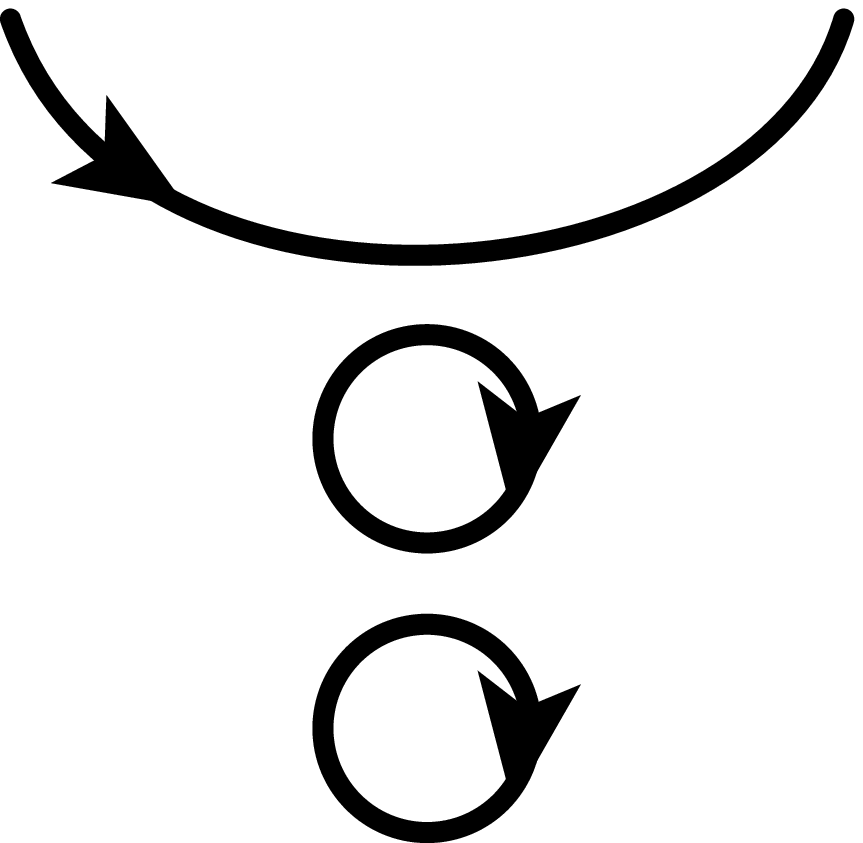} \ar@{|->}[r]^{\mathfig{0.075}{movie_moves/MM7/untuck_map}} & %
    \pd{MM7f4}, %
}
\end{align*}
while a negative crossing results in
\begin{align*}
\newcommand{\pd}[1]{\mathfig{0.1}{movie_moves/MM7/#1}}
\newcommand{\pa}[1]{\mathfig{0.05}{movie_moves/MM7/#1}}
\xymatrix@R-7.5mm@C+7mm{
    \pd{MM7f1} \ar[r]^{R_1b} & \pd{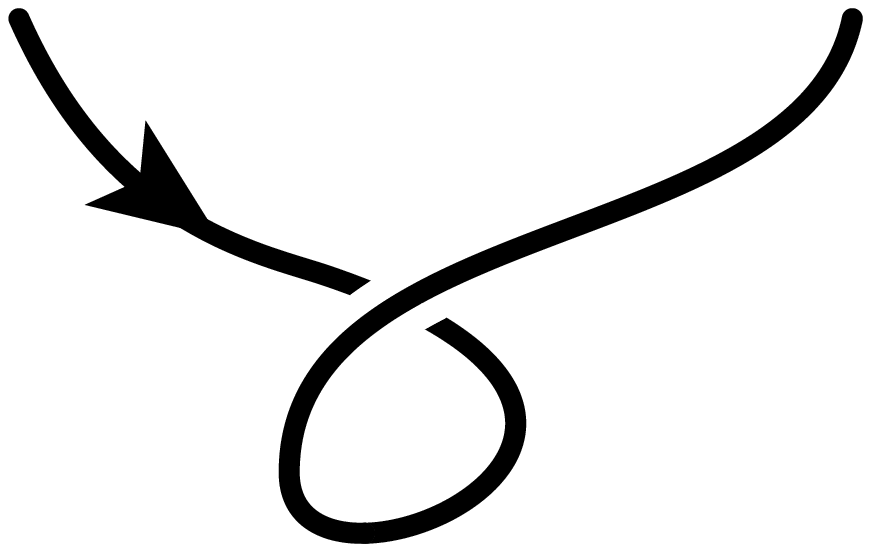} \ar[r]^{R_1a} & \pd{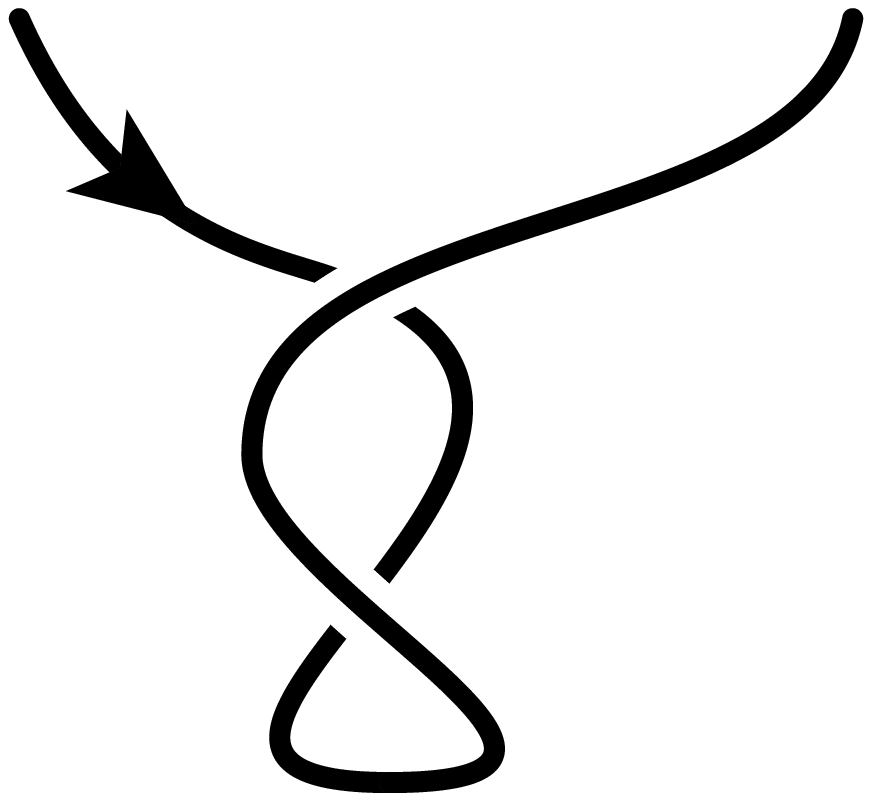} \ar[r]^{{R_2b}^{-1}} & \pd{MM7f4} \\
    \\
    \pd{MM7f1} \ar@{|->}[r]^{\text{cup}} & %
    \pd{MM7f2_trace} \ar@{|->}[r]^{\text{saddle}} & %
    \pd{MM7f3_trace} \ar@{|->}[r]^{\mathfig{0.065}{movie_moves/MM7/untuck_map}} & %
    \pd{MM7f4}. %
}
\end{align*}
Either composition is easily seen to be the identity morphism.

\subsubsection*{MM8}
\label{sssec:MM8}
$$\mathfig{0.5}{movie_moves/MM8}$$

This is the only movie move involving all three Reidemeister moves. First let's note some symmetries. By a rotation of the whole diagram, we can
assume the $R1$ move happens on the horizontal strand, beginning on the
right. Moreover, we can assume that the horizontal strand is oriented
right to left (otherwise, we can obtain this condition by a $\pi$
rotation of its time reversal).

There are then sixteen variations, depending on whether the vertical strand
lies above or below the horizontal strand, its orientation, the sign of
the crossing introduced by the first Reidemeister move in the first
frame, and finally whether the first Reidemeister move introduces a twist on the left or right side. Figure \ref{fig:MM8-comps} shows all the maps involved.
\begin{figure}[ht]
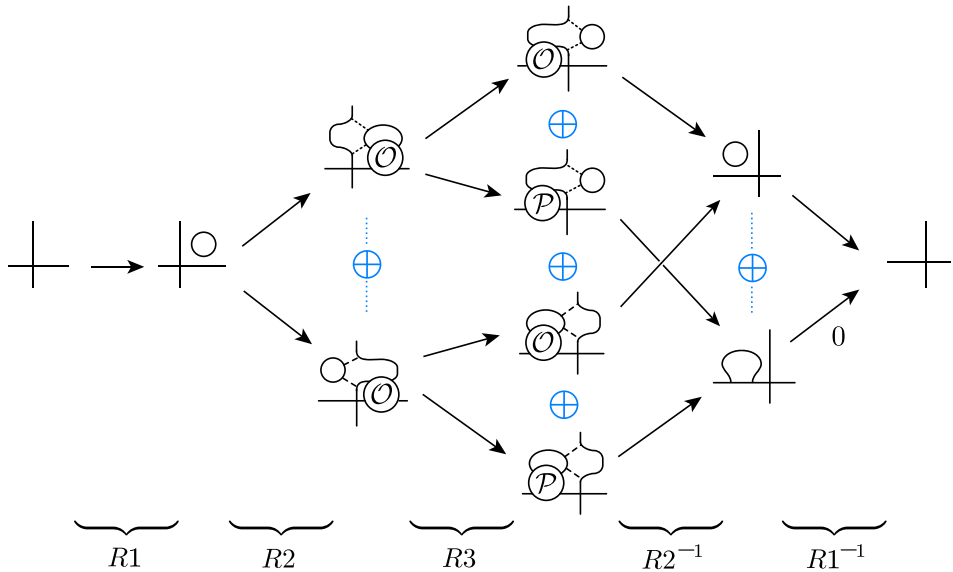

$$\mathfig{0.95}{movie_moves/MM8/MM8_compositions}$$
\caption[Possible compositions for MM8]{Possible compositions for MM8.}
\label{fig:MM8-comps}
\end{figure}
The dotted/dashed lines are contingent upon crossing signs and orientations: either all of the dotted lines will appear, and none of the dashed, or vice versa. Note that the crossing introduced by the $R1$ move is always either the low or high crossing in the $R3$ move, so we will denote its resolution with either $\mathcal{O}$ or $\mathcal{P}$ as we did in the computation for MM6. We can also observe that any map factoring through the resolution $\mathfig{0.08}{movie_moves/MM8/killstate}$ must be zero, since this object maps to zero under $R1$. Thus we need only concern ourselves with the other two compositions in Figure \ref{fig:MM8-comps}.

Consider the case of a negative twist, but ignore whether the twist appears on the left or right side of the horizontal strand, as this barely changes any of the calculations. Our computation will work regardless of whether the vertical strand is above or below the horizontal strand. If above, we'll see the $hml, lmh, mhl,$ and $lhm$ variations of $R3$ and use Lemma \ref{lem:R3-OO}; if below, the relevant $R3$ moves are $hml, lmh, hlm,$ and $mlh$ and we can apply Lemma \ref{lem:R3L-OO}. Either way, all $R3$ map components we'll encounter are just the identity. The two compositions when the vertical strand is oriented downward are as follows:

{
\newcommand{\pd}[1]{\mathfig{0.09}{movie_moves/MM8/#1}}%
\newcommand{\pa}[1]{\mathfig{0.045}{movie_moves/MM8/#1}}%
\begin{align*}%
\xymatrix@R-7.5mm@C-1.5mm{
    \pd{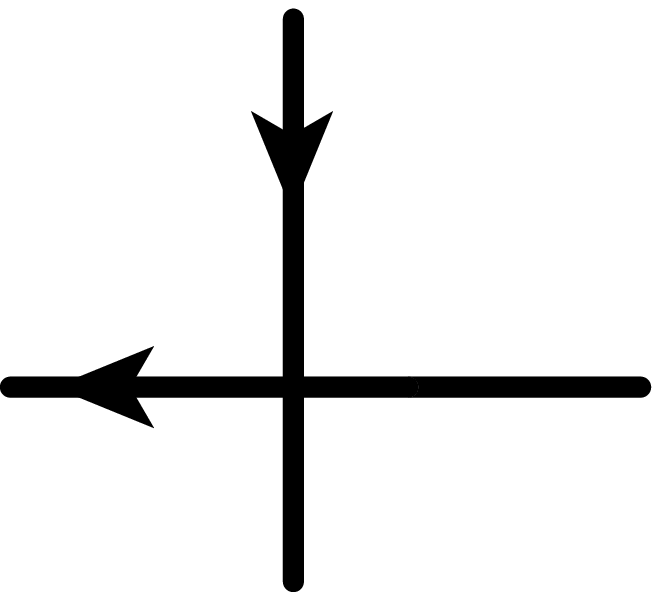} \ar[r]^{R_1b} & \pd{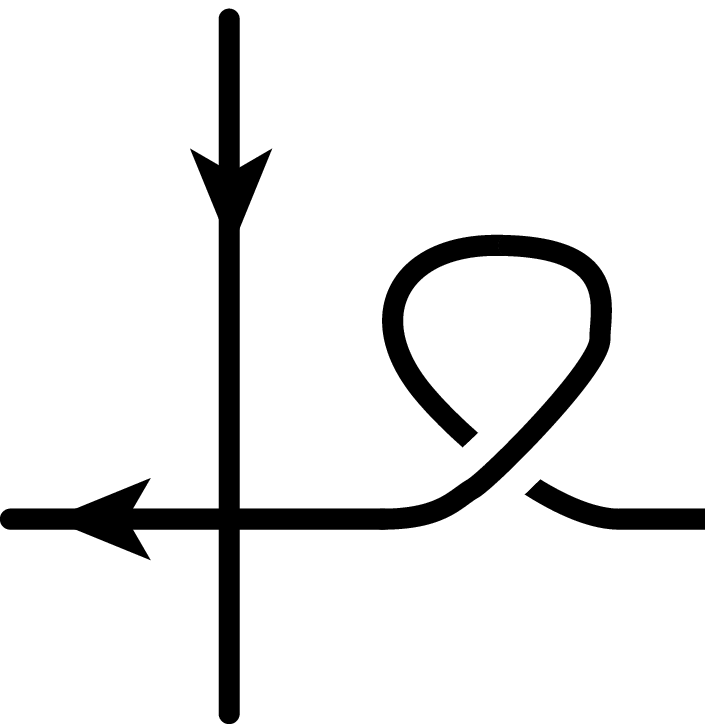} \ar[r]^{R_2b} & \pd{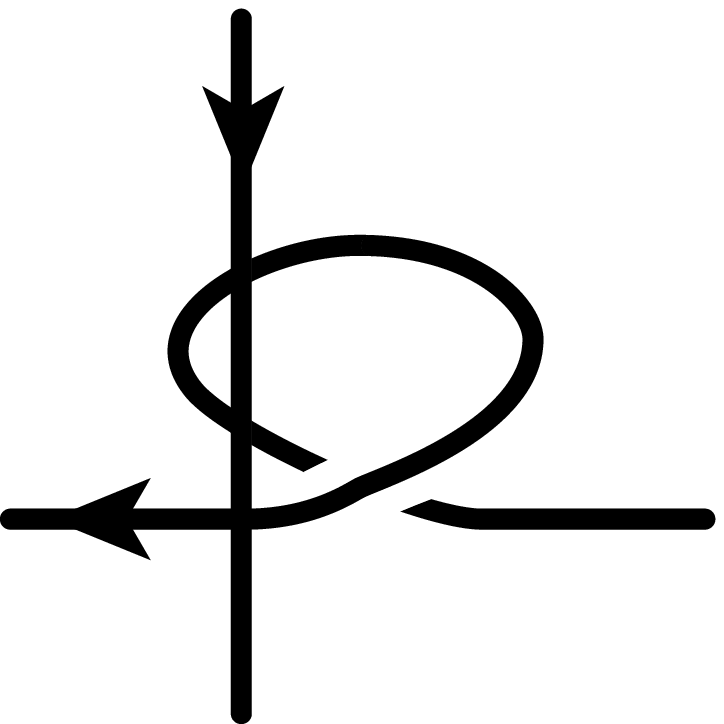} \ar[r]^{R_3} & \pd{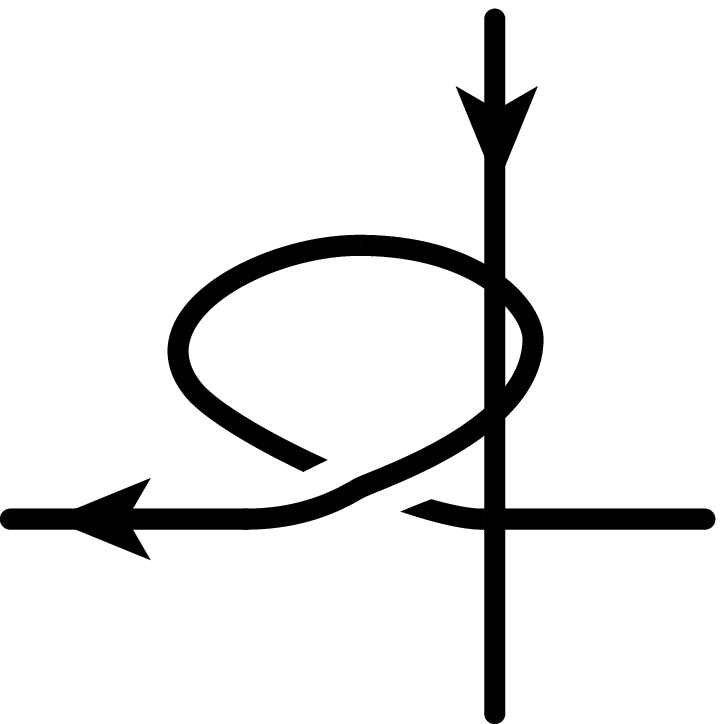} \ar[r]^{{R_2a}^{-1}} & \pd{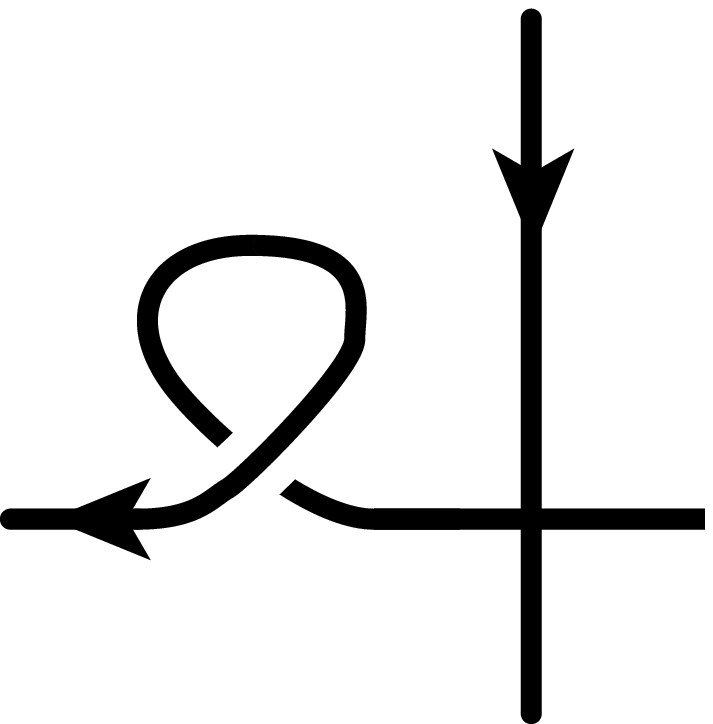} \ar[r]^{{R_1b}^{-1}} & \pd{MM8v2f1} \\
    \\
    \pd{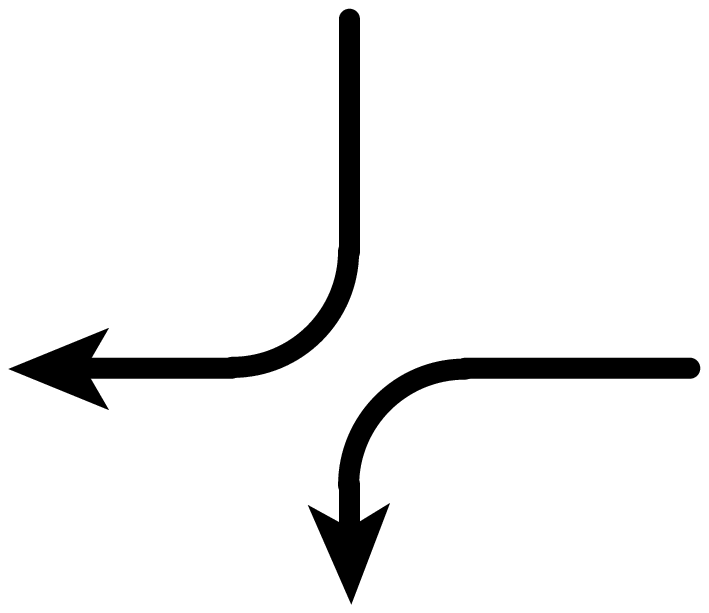} \ar@{|->}[r]^{\pa{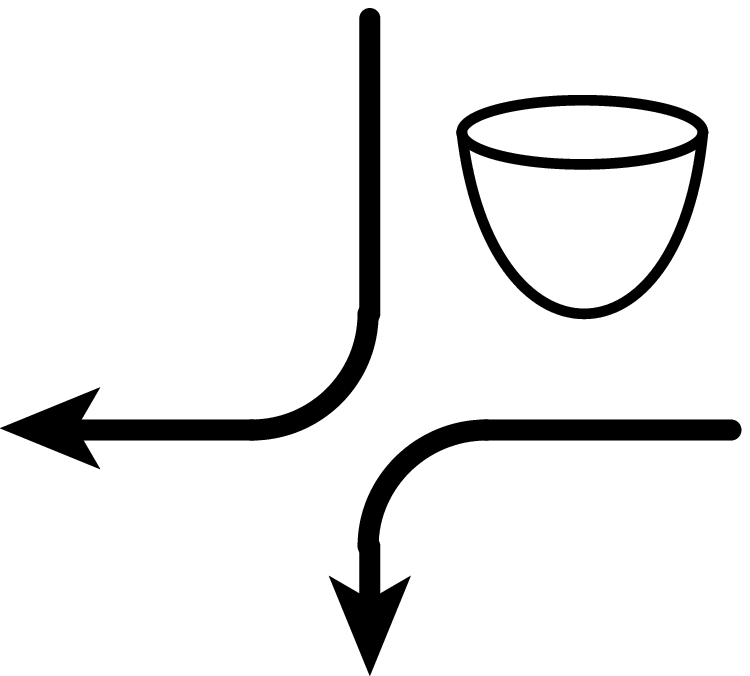}} & %
    \pd{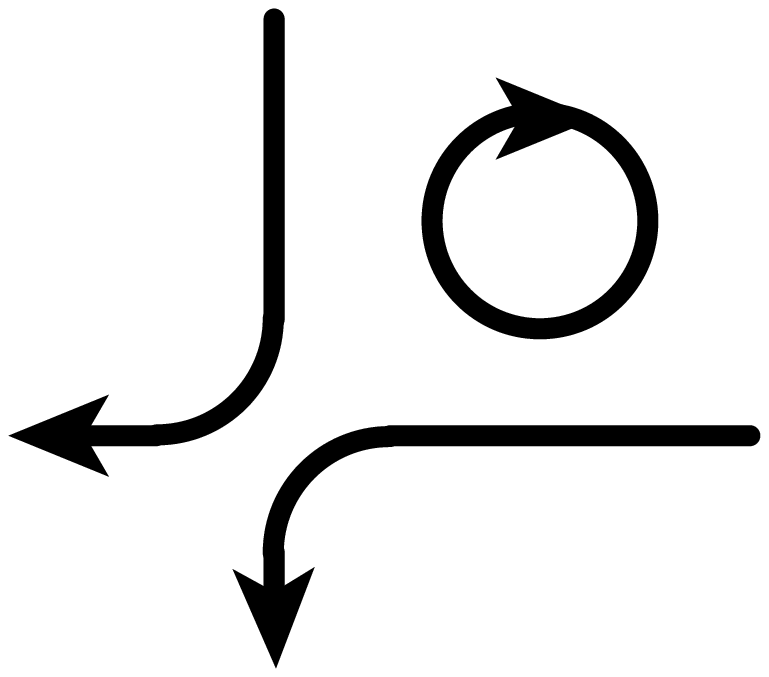} \ar@{|->}[r]^{\bar{r}} \ar@{|->}[dr]_{\beta} & %
    \pd{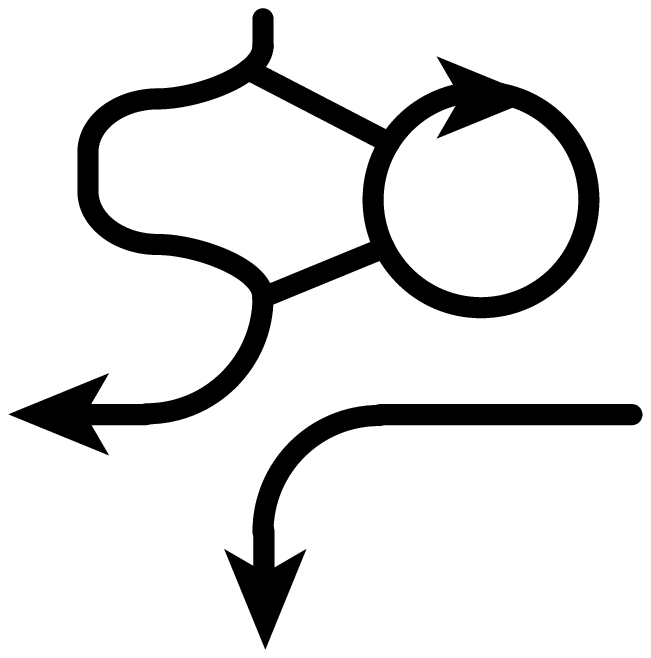} \ar@{|->}[r]^{1} \ar@{}[d]|{\directSum} & %
    \pd{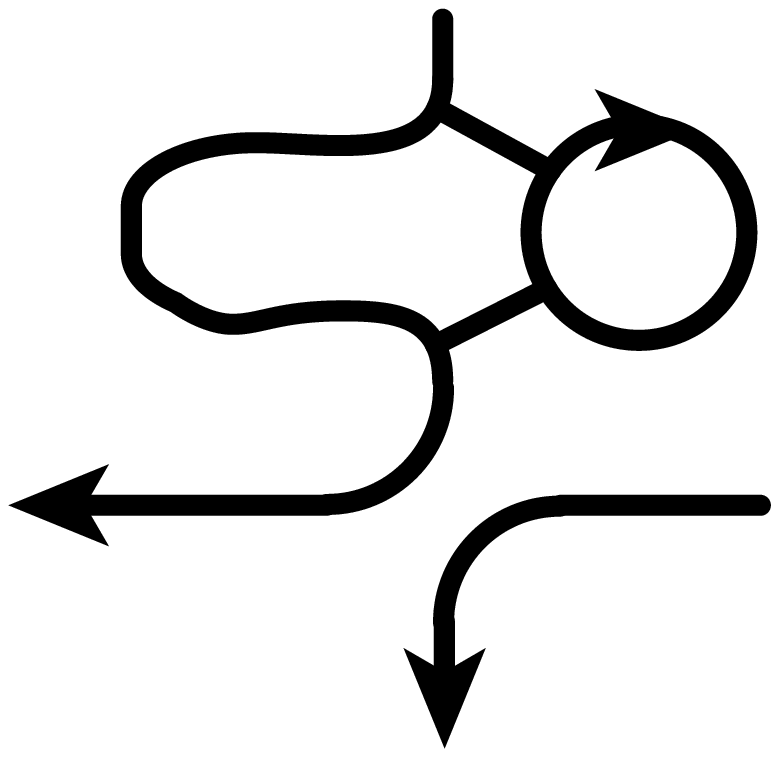} \ar@{|->}[r]^{-\text{u} \compose \text{d}} \ar@{}[d]|{\directSum} & %
    \pd{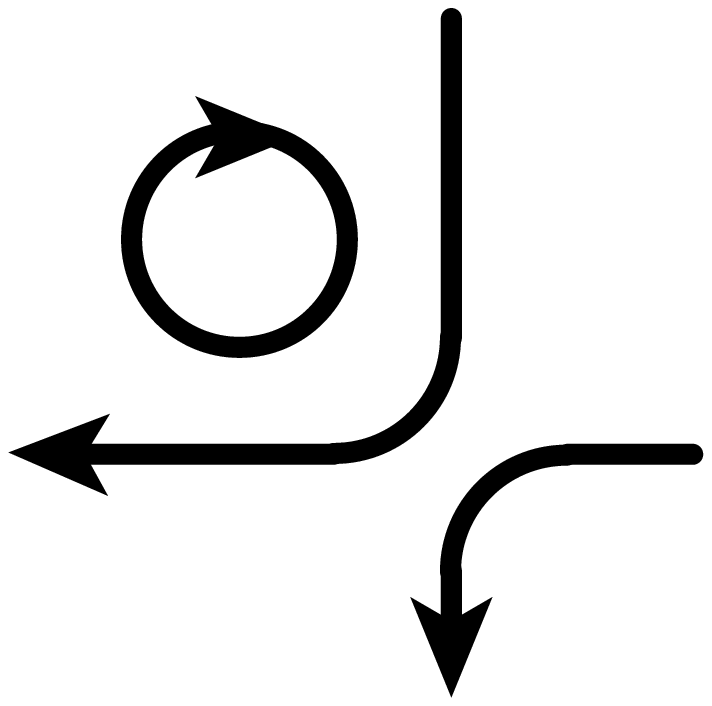} \ar@{|->}[r]^{s} & %
    \pd{MM8v2f1_trace} \\
     & &
     \pd{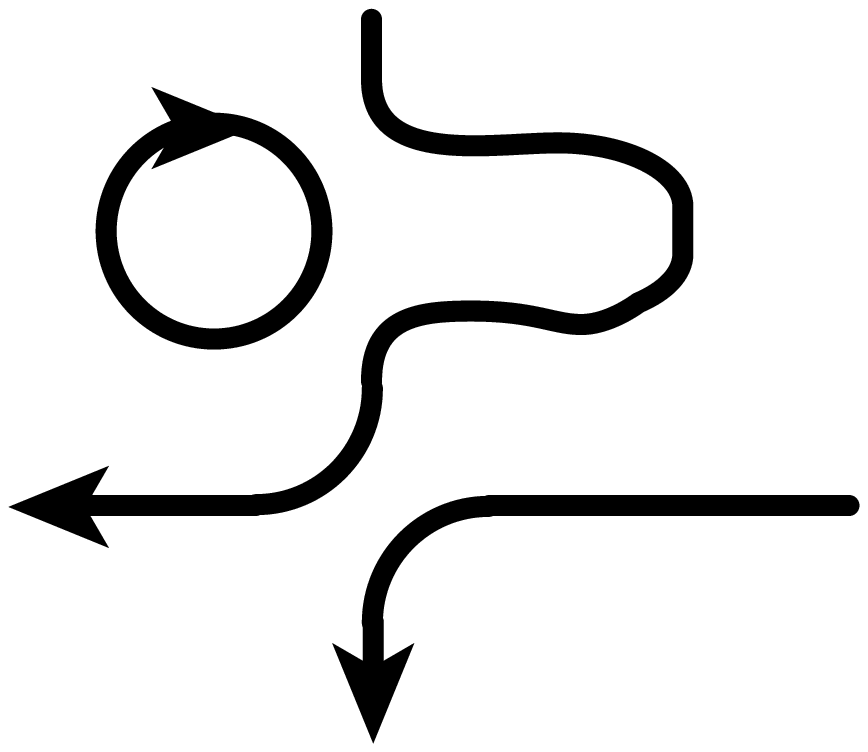} \ar@{|->}[r]^{1} &
     \pd{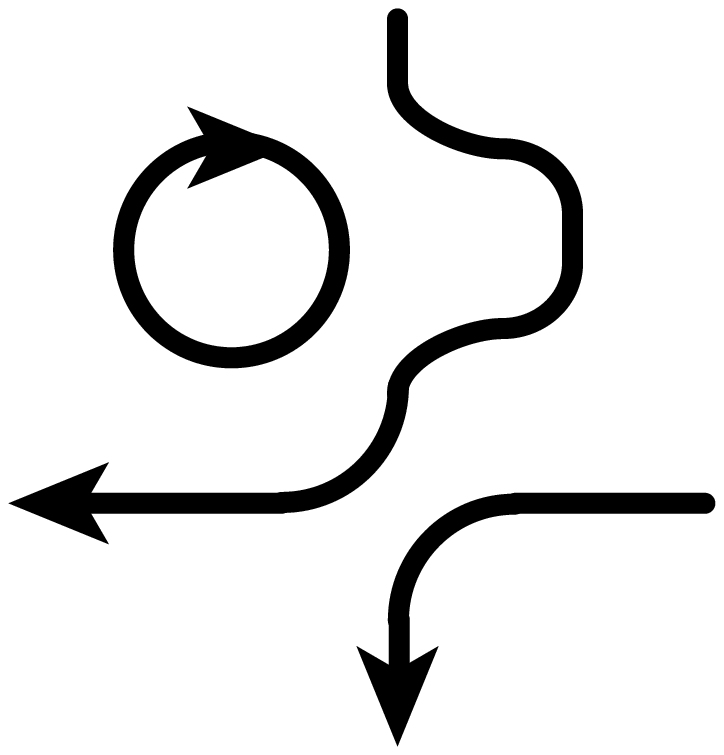} \ar@{|->}[ur]_{1} & &
}
\end{align*}
}

\noindent Here $\beta$ is the seamless component of the $R2b$ map, and $s$ is the saddle from $R1b$. The bottom map in this composition is just the identity (there are clearly no crossing reordering issues here), while the top map contains a blister, and is thus zero.

If the vertical strand is oriented upward, we'll see the following:
{
\newcommand{\pd}[1]{\mathfig{0.09}{movie_moves/MM8/#1}}%
\newcommand{\pa}[1]{\mathfig{0.045}{movie_moves/MM8/#1}}%
\begin{align*}%
\xymatrix@R-7.5mm@C-1.5mm{
    \pd{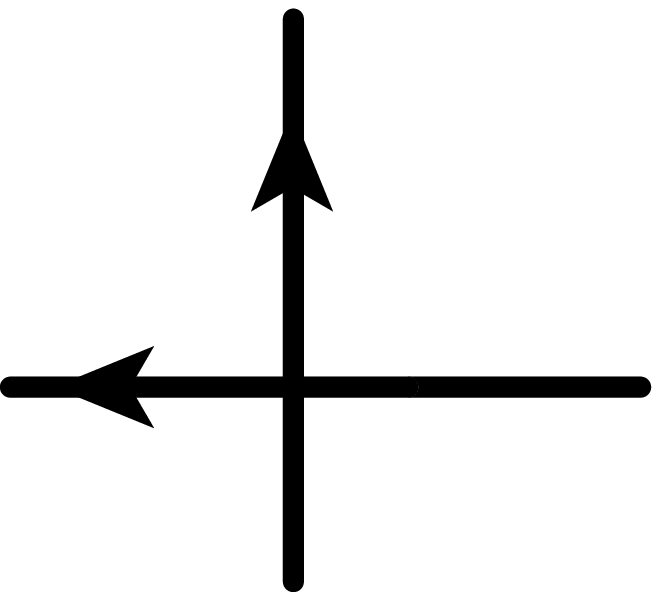} \ar[r]^{R_1b} & \pd{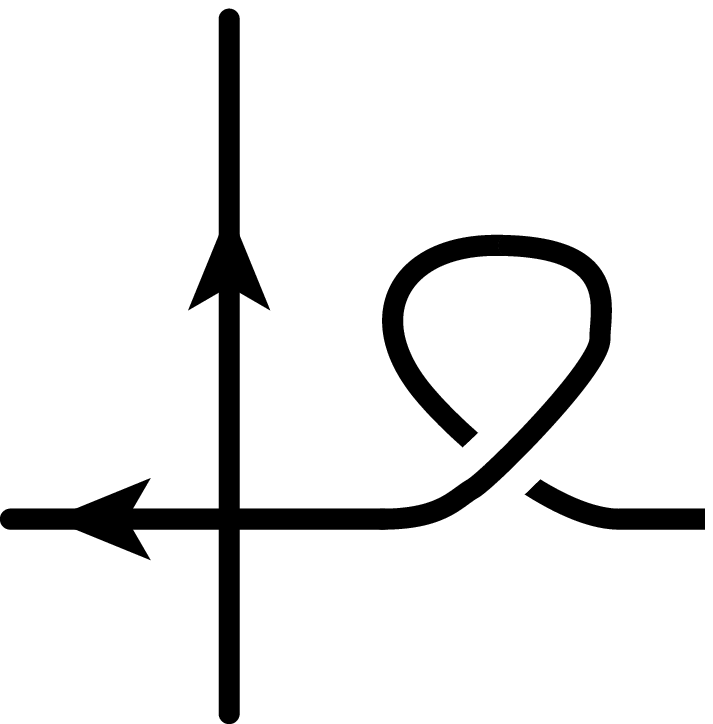} \ar[r]^{R_2a} & \pd{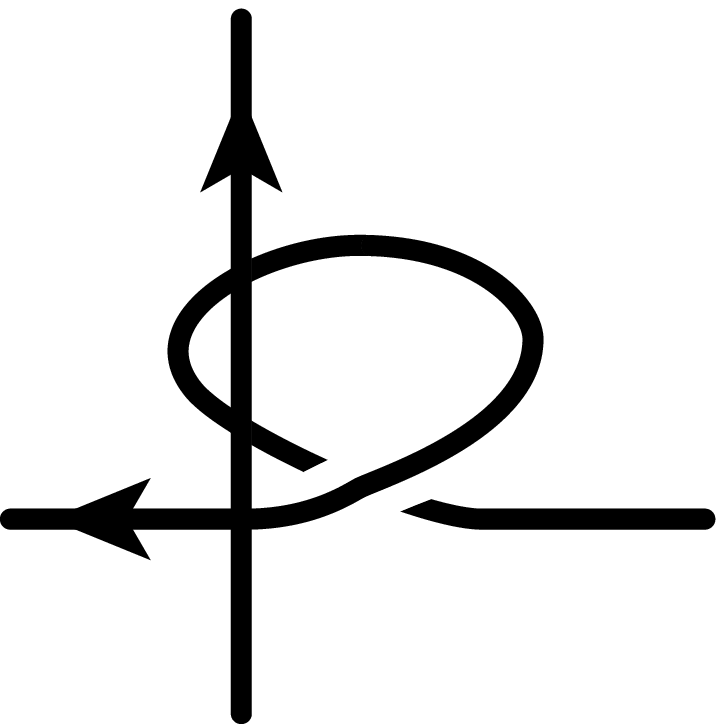} \ar[r]^{R_3} & \pd{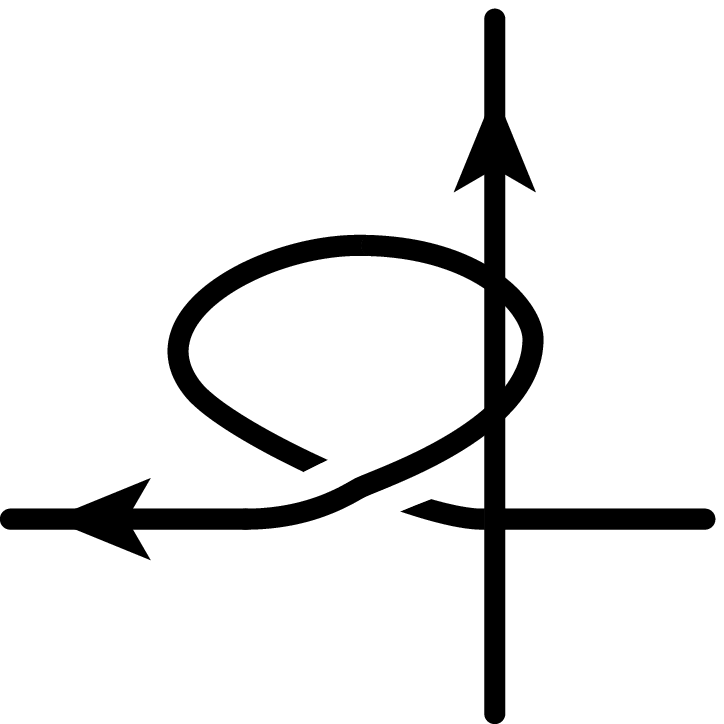} \ar[r]^{{R_2b}^{-1}} & \pd{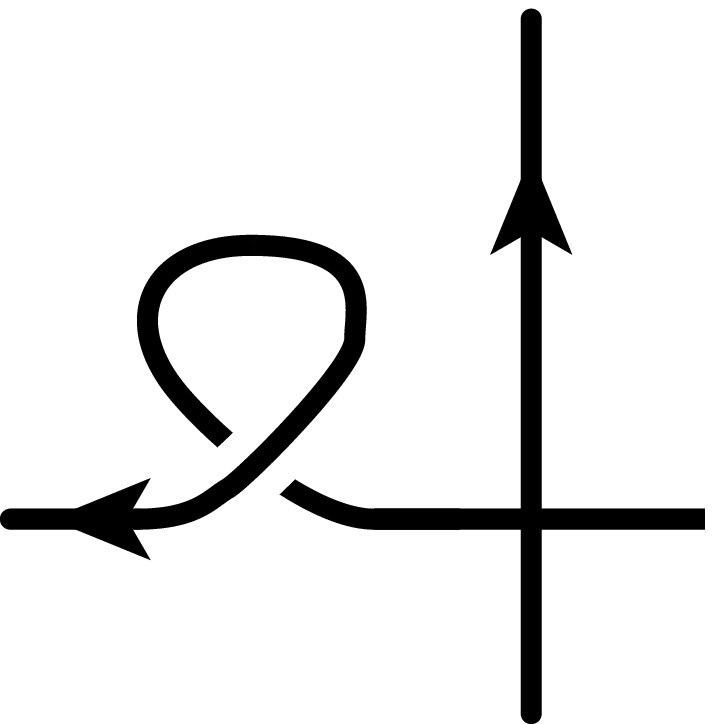} \ar[r]^{{R_1b}^{-1}} & \pd{MM8v1f1} \\
    \\
    \pd{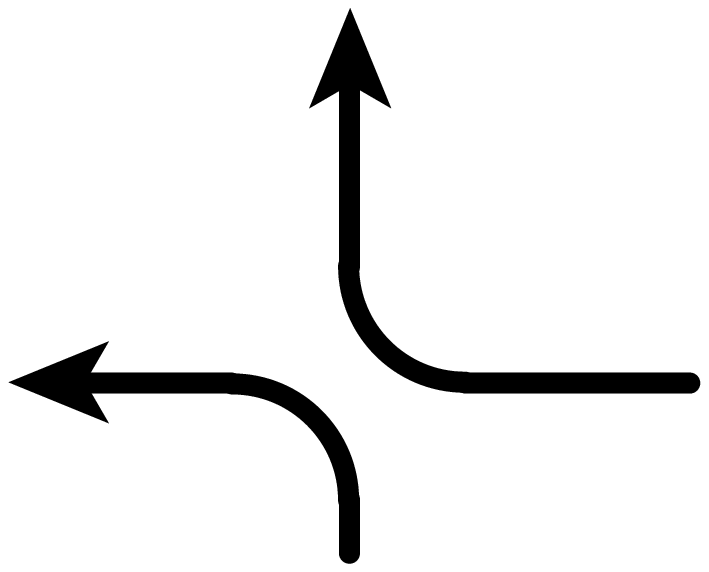} \ar@{|->}[r]^{\pa{MM8v2_map1}} & %
    \pd{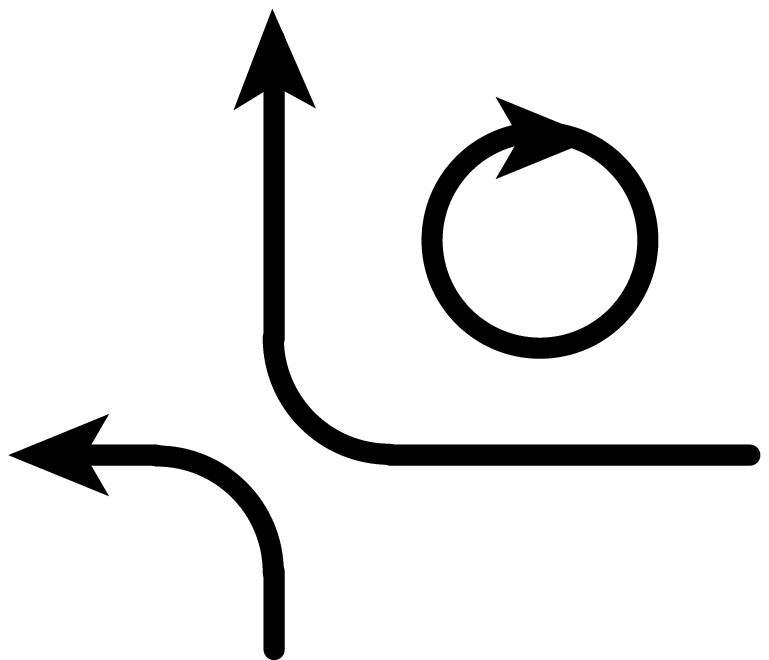} \ar@{|->}[r]^{1} \ar@{|->}[dr]_{\text{b} \compose \text{z}} & %
    \pd{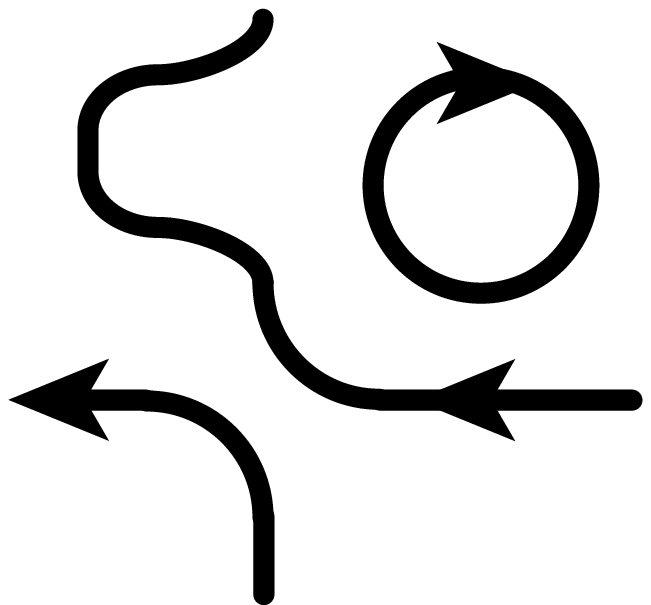} \ar@{|->}[r]^{1} \ar@{}[d]|{\directSum} & %
    \pd{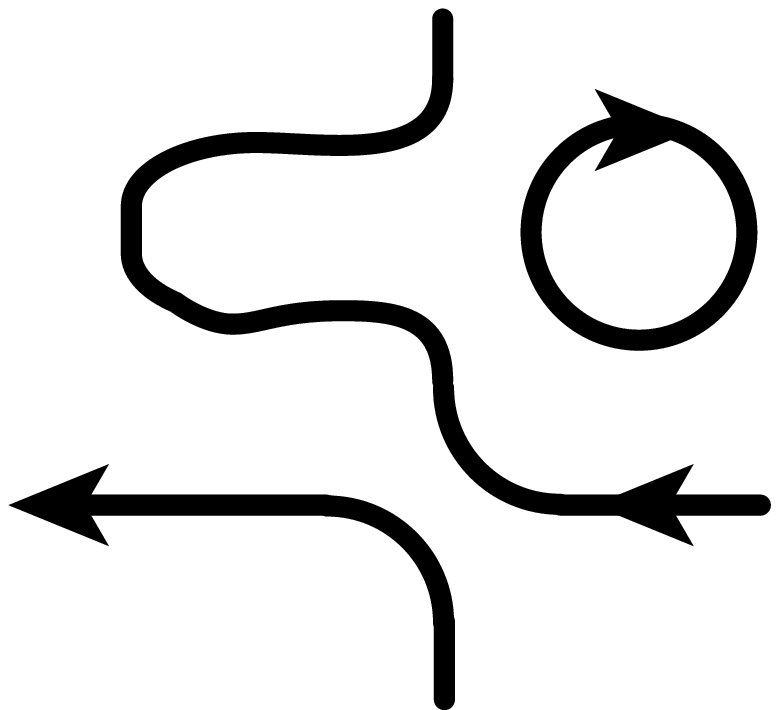} \ar@{|->}[r]^{\beta} \ar@{}[d]|{\directSum} & %
    \pd{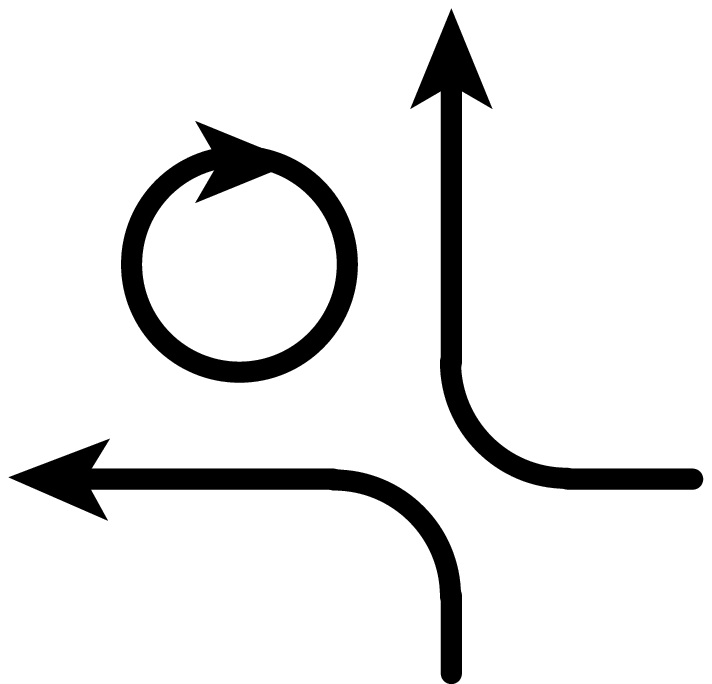} \ar@{|->}[r]^{s} & %
    \pd{MM8v1f1_trace} \\
     & &
     \pd{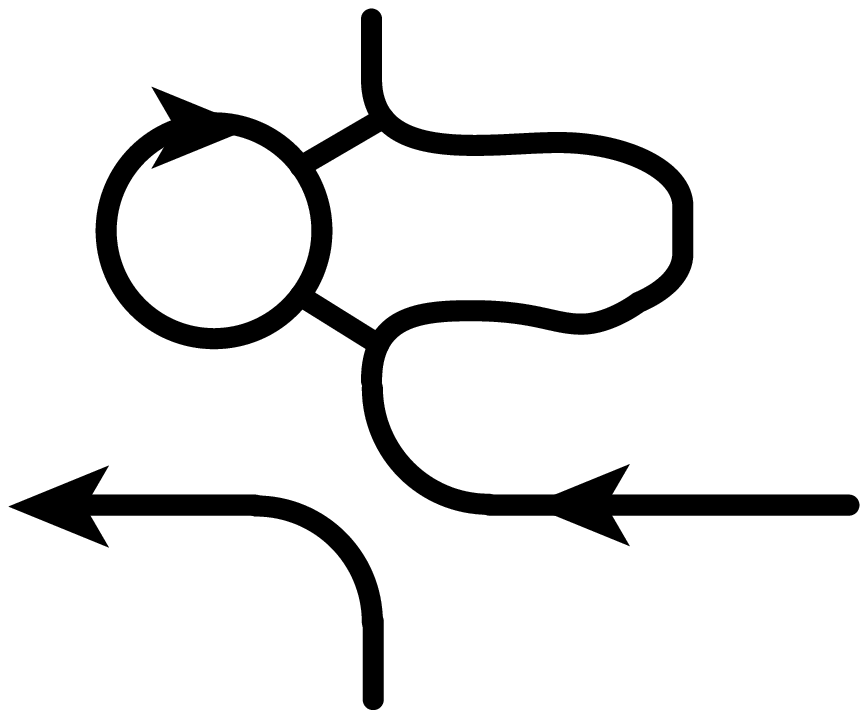} \ar@{|->}[r]^{1} &
     \pd{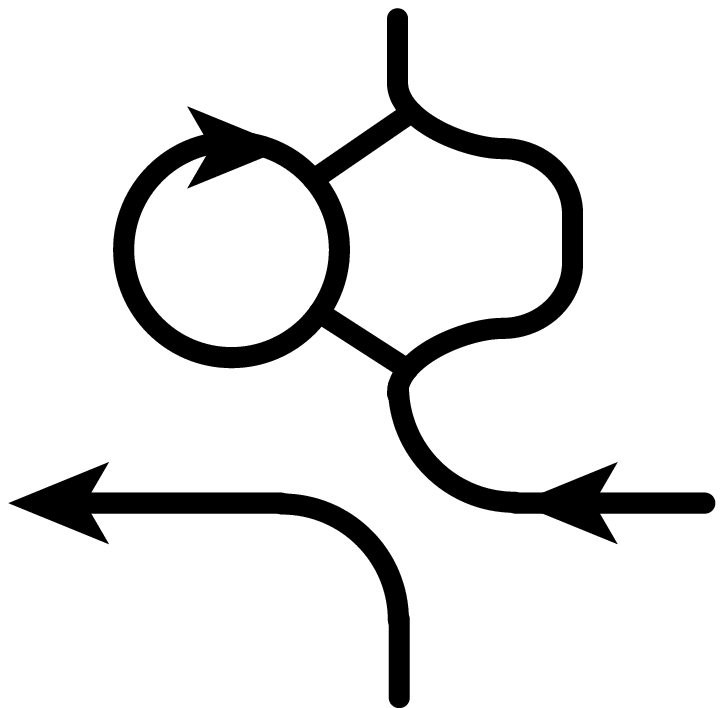} \ar@{|->}[ur]_{-r} & &
}
\end{align*}
}
This time the top map contains a sphere, and is thus zero. The bottom map takes some patience to see, but modulo the airlock relation its just the identity. Crossing reordering maps act trivially here.

The calculations for the positive twist case are almost identical.

\subsubsection*{MM9}
\label{sssec:MM9}
$$\mathfig{0.3}{movie_moves/MM9}$$

Remember that the proof of Proposition \ref{prop:duality} in \cite{ClarMW:Fix} assumed invariance under this movie move. Thus, since we don't have access to Corollary \ref{cor:post-duality}, we can't know in advance that the space of chain maps between the first and last frames is one dimensional. As such, we'll have to calculate the map on every object in the initial tangle complex, checking it's the identity on each of these objects.

There are four variations of MM9: after fixing the orientation of one
strand, we have two possible orientations for the other strand, and a choice of sign for the initial crossing. We'll do the calculations for both types of crossings with a given orientation. It's easy to see that changing orientation
essentially interchanges the maps in these two cases.

With a positive initial crossing, we have
{
\newcommand{\pd}[1]{\mathfig{0.09}{movie_moves/MM9/positive/#1}}
\newcommand{\pa}[1]{\mathfig{0.04}{movie_moves/MM9/positive/#1}}
\newcommand{\pdt}[1]{\mathfig{0.09}{movie_moves/MM9/#1}}
\begin{align*}
\xymatrix@R-7.5mm@C+12mm{
    \pd{MM9f1} \ar[r]^{R_2a} & \pd{MM9f2} \ar[r]^{\sigma_{12}} & \pd{MM9f2renumbered} \ar[r]^{{R_2a}^{-1}} & \pd{MM9f3}, \\
}
\intertext{where $\sigma$ is the necessary crossing reordering map. The components of the chain map are given by:}
\xymatrix@R-7.5mm@C+12mm{
    \pd{MM9f1_trace} \ar@{|->}[r]^{1} \ar@{|->}[dr] & %
    \pd{MM9f2_trace1} \ar@{|->}[r]^{1} \ar@{}[d]|{\directSum} & %
    \pd{MM9f2_trace1} \ar@{|->}[r]^{1} \ar@{}[d]|{\directSum} & %
    \pd{MM9f3_trace} \\ %
    & \pd{MM9f2_trace2} \ar@{|->}[r]^{1} & %
      \pd{MM9f2_trace2} \ar@{|->}[ur]^{0} & \\\\%
    \pd{MM9f1_dtrace} \ar@{|->}[r]^{\text{b} \compose \text{z}} \ar@{|->}[dr] & %
    \pd{MM9f2_dtrace1} \ar@{|->}[r]^{-1} \ar@{}[d]|{\directSum} & %
    \pd{MM9f2_dtrace1} \ar@{|->}[r]^{-\text{u} \compose \text{d}} \ar@{}[d]|{\directSum} & %
    \pd{MM9f3_dtrace} \\ %
    & \pd{MM9f2_dtrace2} \ar@{|->}[r]^{1} & %
      \pd{MM9f2_dtrace2} \ar@{|->}[ur]^{0} & %
}
\end{align*}
}
and the composition is just the identity.

With a negative crossing, we have
{
\newcommand{\pd}[1]{\mathfig{0.1}{movie_moves/MM9/negative/#1}}
\newcommand{\pa}[1]{\mathfig{0.05}{movie_moves/MM9/negative/#1}}
\newcommand{\pdt}[1]{\mathfig{0.1}{movie_moves/MM9/#1}}
\begin{align*}
\xymatrix@R-7.5mm@C+12mm{
    \pd{MM9f1} \ar[r]^{R_2b} & \pd{MM9f2} \ar[r]^{\sigma_{23}} & \pd{MM9f2renumbered} \ar[r]^{{R_2b}^{-1}} & \pd{MM9f3}, \\
    }
\intertext{with the components of the chain map given by}
\xymatrix@R-7.5mm@C+12mm{
    \pd{MM9f1_trace} \ar@{|->}[r]^{\pdt{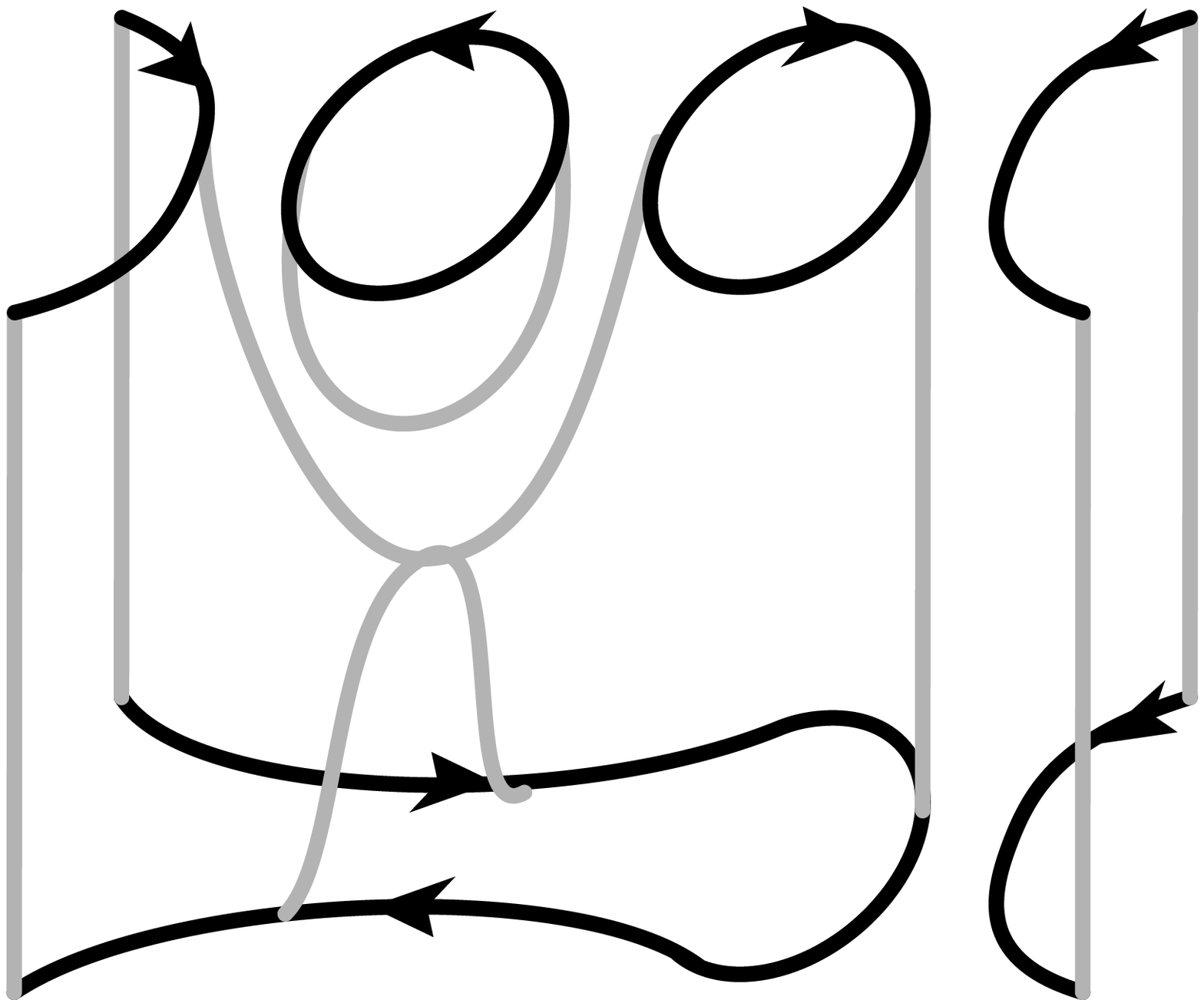}} \ar@{|->}[dr] & %
    \pd{MM9f2_trace1} \ar@{|->}[r]^{1} \ar@{}[d]|{\directSum} & %
    \pd{MM9f2_trace1} \ar@{|->}[r]^{\pdt{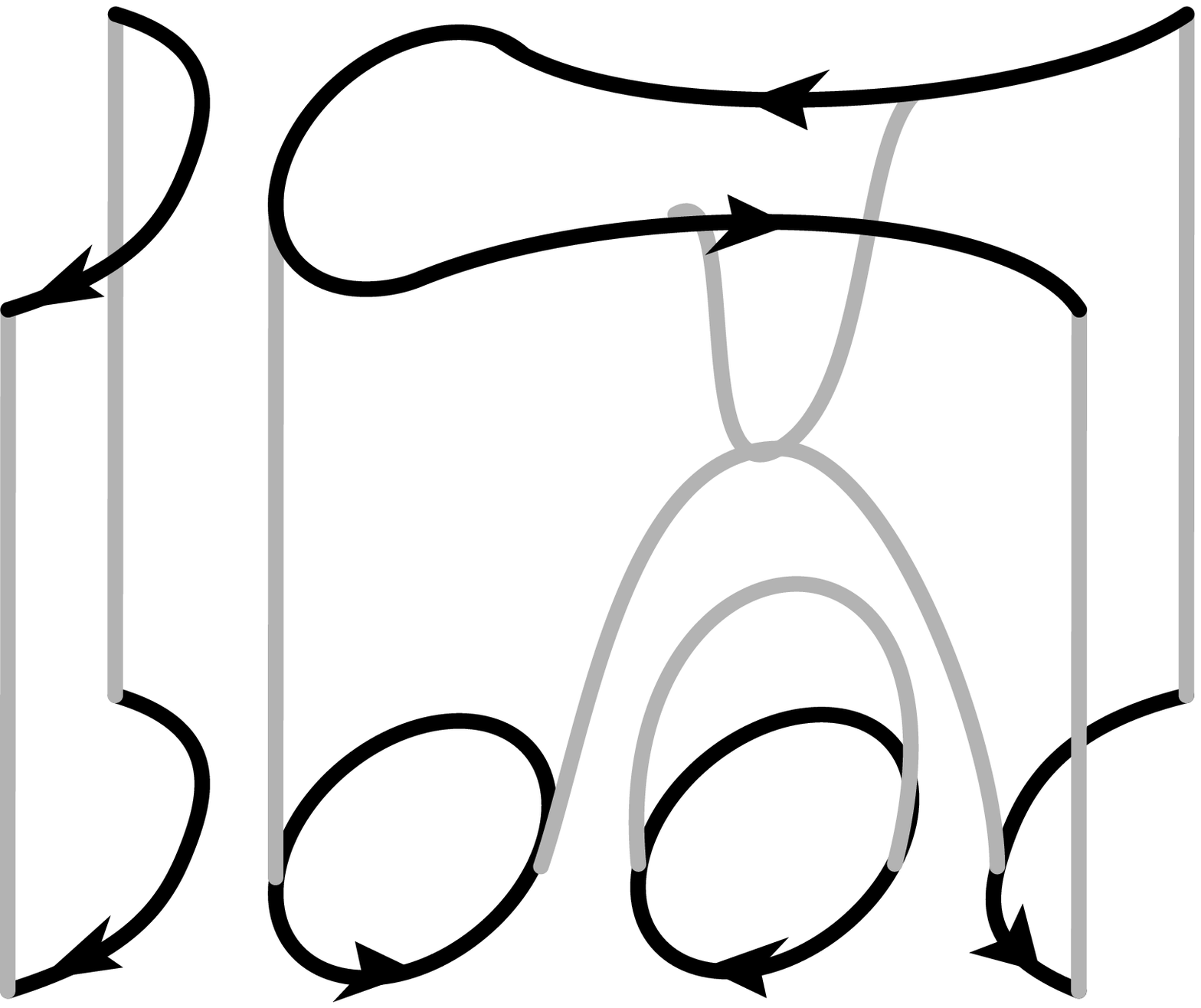}} \ar@{}[d]|{\directSum} & %
    \pd{MM9f3_trace} \\ %
    & \pd{MM9f2_trace2} \ar@{|->}[r]^{1} & %
      \pd{MM9f2_trace2} \ar@{|->}[ur]^{0} & \\\\%
    \pd{MM9f1_dtrace} \ar@{|->}[r]^{\bar{r}} \ar@{|->}[dr] & %
    \pd{MM9f2_dtrace1} \ar@{|->}[r]^{-1} \ar@{}[d]|{\directSum} & %
    \pd{MM9f2_dtrace1} \ar@{|->}[r]^{-r} \ar@{}[d]|{\directSum} & %
    \pd{MM9f3_dtrace}. \\ %
    & \pd{MM9f2_dtrace2} \ar@{|->}[r]^{1} & %
       \pd{MM9f2_dtrace2} \ar@{|->}[ur]^{0} & %
}
\end{align*}
}
Here, again, our composition is the identity.

\subsubsection*{MM10}
\label{sssec:MM10}
$$\mathfig{0.9}{movie_moves/MM10}$$

This move has the most frames and the most crossings, in addition to forty-eight variations: assuming the highest strand is oriented to the right, we have $3!$ height orderings and $2^3$ orientation possibilities for the other three strands. Various shortcuts have been successfully employed in \cite{BN05_MR2174270} and \cite{ClarMW:Fix} for the $\su{2}$ theory; however, we will build on the technique of the latter to give a completely computation-free proof of invariance under MM10.

Let's first establish that one particular variation induces the identity map. To do this, consider the non-generic projection in Figure \ref{fig:3-cell1}: a cusp over a crossing. Decomposing the space of projections of smooth tangles (with our specific boundary data) into strata of ``genericness", we can view this projection as a 3-cell in the dual complex (where a $k$-cell corresponds to a codimension $k$ stratum). Here, 0-cells correspond to generic immersions, 1-cells correspond to Reidemeister moves, and 2-cells correspond to movie moves. The 3-cell in question, shown in Figure \ref{fig:3-cell-diagram1}, is bounded by 2-cells representing MM10, MM6, and MM8, as well as five 2-cells corresponding to the ``zeroth movie move" (two simultaneous but distant Reidemeister moves). Since we've already shown that MM6 and MM8 give the identity, we get this variation of MM10 for free.

\begin{figure}[h]
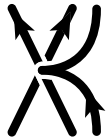

$$\mathfig{0.1}{movie_moves/MM10/3-cell1}$$
\caption[A non-generic projection]{A non-generic projection corresponding to a 3-cell
involving MM10, MM6, and MM8.}\label{fig:3-cell1}
\end{figure}

\begin{figure}[h]
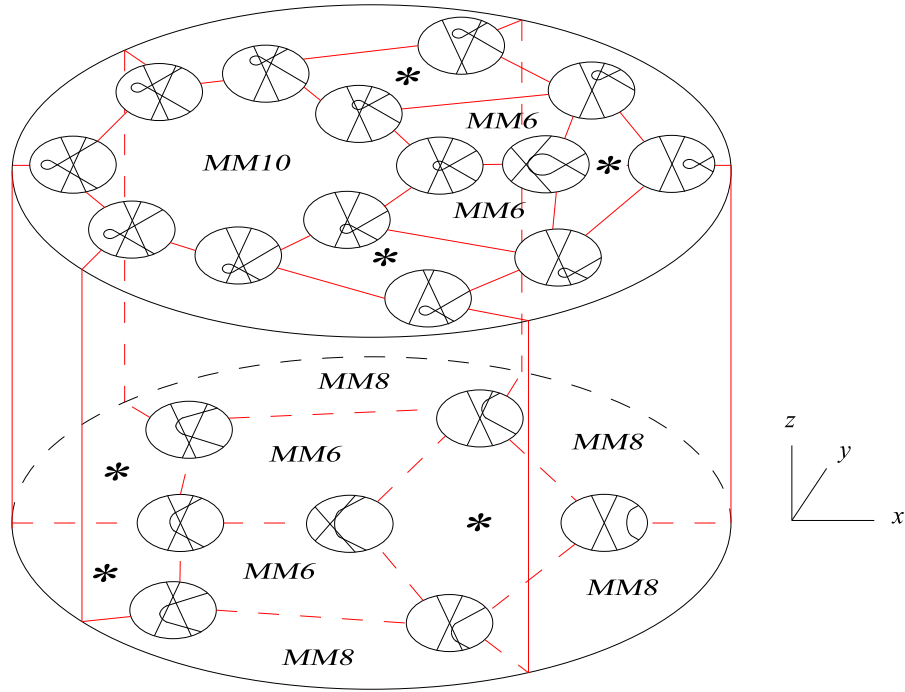

$$\mathfig{0.9}{movie_moves/MM10/3-cell_diagram1}$$
\caption[The 3-cell for the singularity in Figure
\ref{fig:3-cell1}]{The 3-cell for the singularity in Figure
\ref{fig:3-cell1}. The 0-cells here are the
generic tangle projections neighboring this singularity, achieved
by straightening ($z$ direction) and translating ($x$ and $y$ directions) the kink. The 2-cells marked with an asterisk correspond to distant Reidemeister
moves.}\label{fig:3-cell-diagram1}
\end{figure}

To check the remaining variations, we'll just repeat the argument in \cite{ClarMW:Fix}: the projection in Figure \ref{fig:3-cell2} has a dual 3-cell bounded by two MM10 2-cells, four MM6 2-cells, and six
distant Reidemeister move 2-cells. Having proved invariance for MM6, we see that invariance for either of the two MM10 variations present follows from invariance of the other. It's then straightforward to show that, with proper choices of strand orientations, invariance under the MM10 variation discussed above propagates (one variation at a time) to the other forty-seven variations. See \cite{ClarMW:Fix} for details.
\begin{figure}[h]
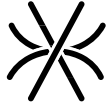

$$\mathfig{0.1}{movie_moves/MM10/3-cell2}$$
\caption[A non-generic projection]{A non-generic projection corresponding to a 3-cell
involving MM10 and MM6.}\label{fig:3-cell2}
\end{figure}
%

\subsection{MM11-15}

The final five movie moves involve Morse moves, and so aren't reversible; we'll have to compute the map for each movie (left and right), and see that they coincide.

\subsubsection*{MM11}
$$\mathfig{0.2}{movie_moves/MM11}$$
This is easy: every complex involved consists of a single object, and the cobordisms on either side are clearly isotopic, in either time direction.

\subsubsection*{MM12}
$$\mathfig{0.2}{movie_moves/MM12}$$
We can't use a homotopy isolation argument here, but a brute-force computation of all components is not difficult. Keep in mind there are two variations: the twist could be either positive or negative. Treating the positive twist, in the forward time direction (i.e., reading down) we'll see on the left

{
\newcommand{\pd}[1]{\mathfig{0.1}{movie_moves/MM12/#1}}
\newcommand{\pds}[1]{\mathfig{0.06}{movie_moves/MM12/#1}}
\newcommand{\birth}{\mathfig{0.03}{foams/birth}}
\newcommand{\death}{\mathfig{0.03}{foams/death}}
\begin{align*}
\xymatrix@R-5mm@C+3mm{
    \emptyset \ar[r]^{\birth} & \pds{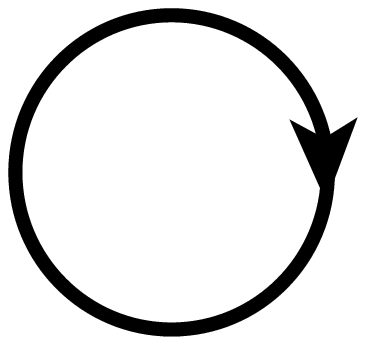} \ar[r]^{R_1a} & \pd{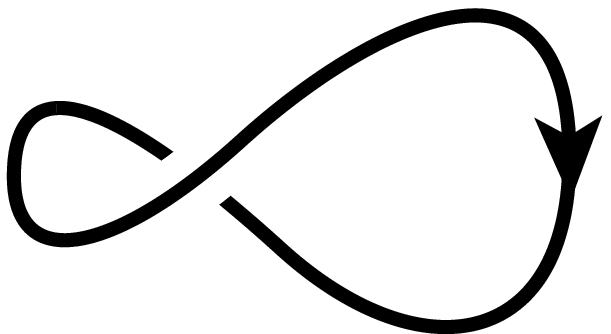} \\
     \emptyset
     \ar@{|->}[r]^{\birth} & %
     \pds{MM12f2} \ar@{|->}[r]^{\text{s}} & %
     \pd{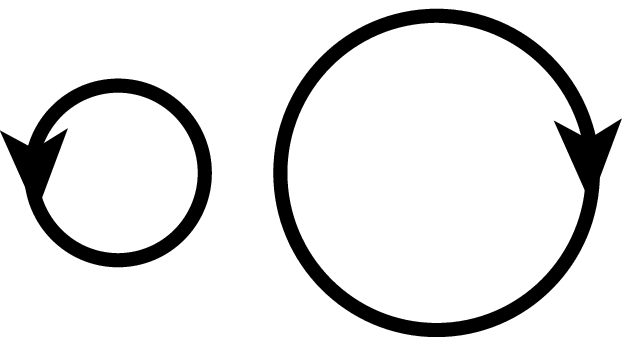}, \\ %
}
\end{align*}
while on the right we have
\begin{align*}
\xymatrix@R-5mm@C+3mm{
    \emptyset
    \ar[r]^{\birth} & \pds{MM12f2} \ar[r]^{R_1a} & \pd{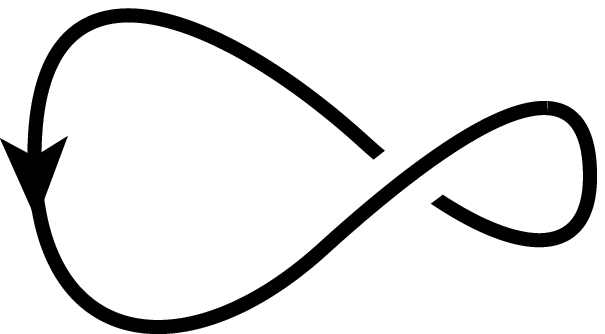} \\
     \emptyset \ar@{|->}[r]^{\birth} & %
     \pds{MM12f2} \ar@{|->}[r]^{\text{s}} & %
     \pd{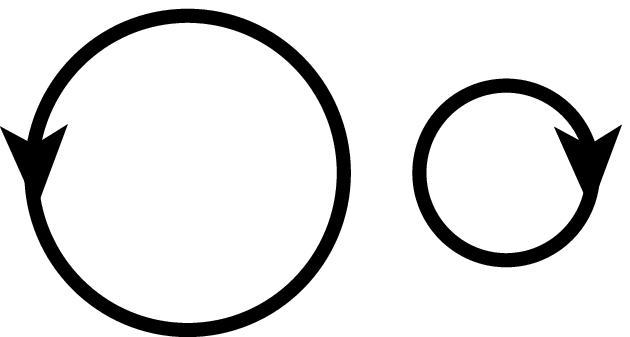}. \\ %
}
\end{align*}
where $\text{s}$ is a saddle. Either way, the morphism is a bent tube.

In reverse time, reading up, we have on the left
\begin{align*}
\xymatrix@R-5mm{
    \emptyset & \pds{MM12f2} \ar[l]_{\death} & \pd{MM12f3a} \ar[l]_{{R_1a}^{-1}}  \\
     \emptyset & %
     \pds{MM12f2} \ar@{|->}[l]_{\death} & %
     \pd{MM12f3a_trace} \ar@{}[d]|{\directSum} \ar@{|->}[l]_{\reflectmathfig{0.05}{movie_moves/MM12/cylinder_and_disc_annihilation}} \\ %
     & 0 & \pd{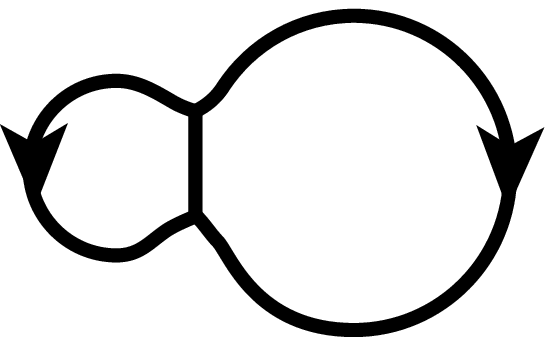}, \ar@{|->}[l]_0 \\
}
\end{align*}
while on the right we have
\begin{align*}
\xymatrix@R-5mm{
    \emptyset & \pds{MM12f2} \ar[l]_{\death} & \pd{MM12f3b} \ar[l]_{{R_1a}^{-1}}  \\
     \emptyset & %
     \pds{MM12f2} \ar@{|->}[l]_{\death} & %
     \pd{MM12f3b_trace} \ar@{}[d]|{\directSum} \ar@{|->}[l]_{\mathfig{0.05}{movie_moves/MM12/cylinder_and_disc_annihilation}} \\ %
     & 0 & \pd{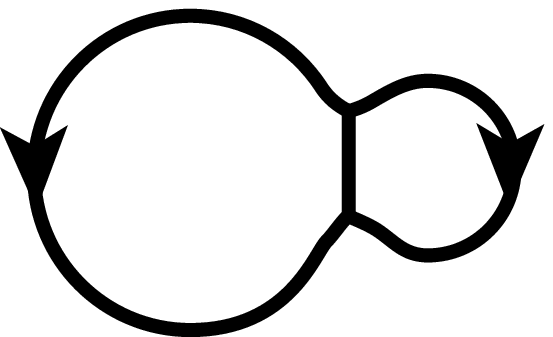}. \ar@{|->}[l]_0 \\
}
\end{align*}
}%
These maps are each just a pair of discs.

The mirror image movie move, in which the twist is negative, is similar: the morphisms appearing will be the same, but swapped with respect to the forward and reverse directions.

\subsubsection*{MM13}

$$\mathfig{0.2}{movie_moves/MM13}$$

The saddle in this move restricts the possible orientations we can see, and by symmetry we can assume that both strands are oriented upward. Let's consider the case of a positive twist. In the forward time direction, we see on the left
{
\newcommand{\pd}[1]{\mathfig{0.08}{movie_moves/MM13/#1}}
\newcommand{\pds}[1]{\mathfig{0.06}{movie_moves/MM13/#1}}
\begin{align*}
\xymatrix@R-5mm@C+10mm{
    \pd{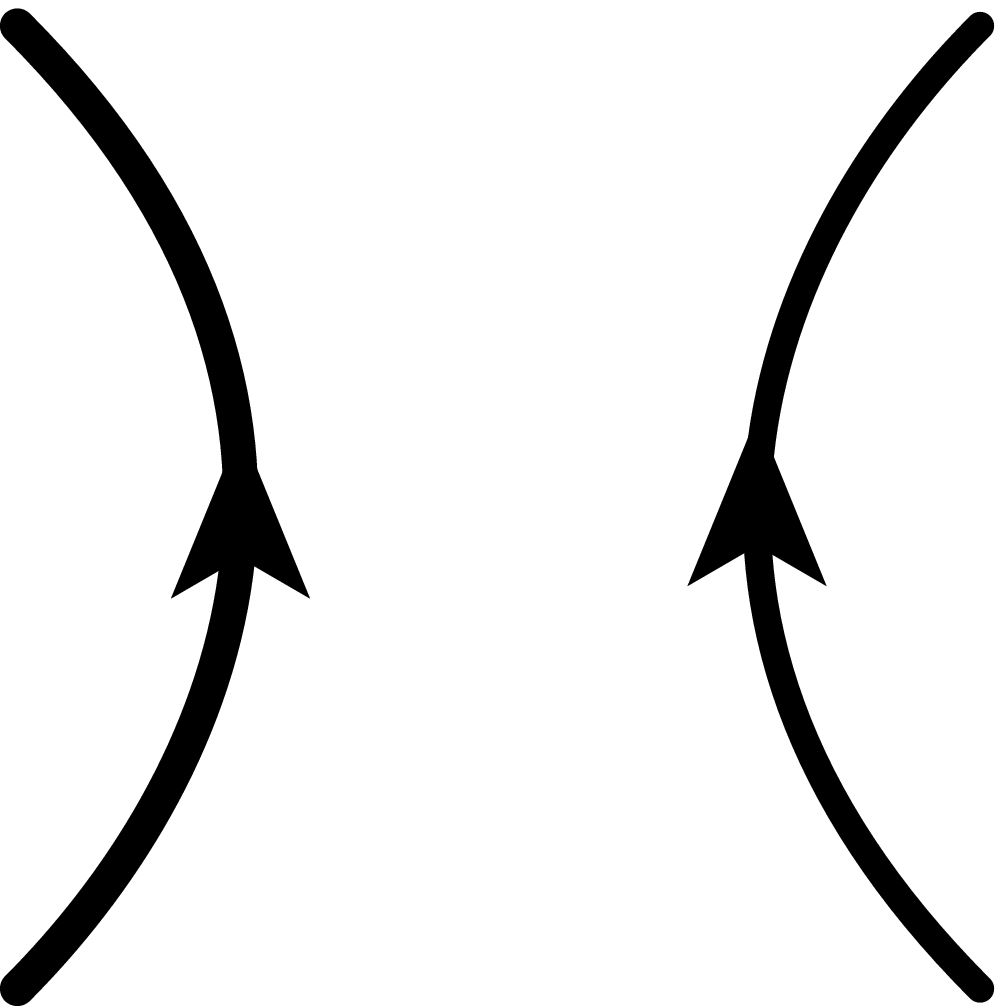} \ar@{|->}[r]^{R_1a} & \pd{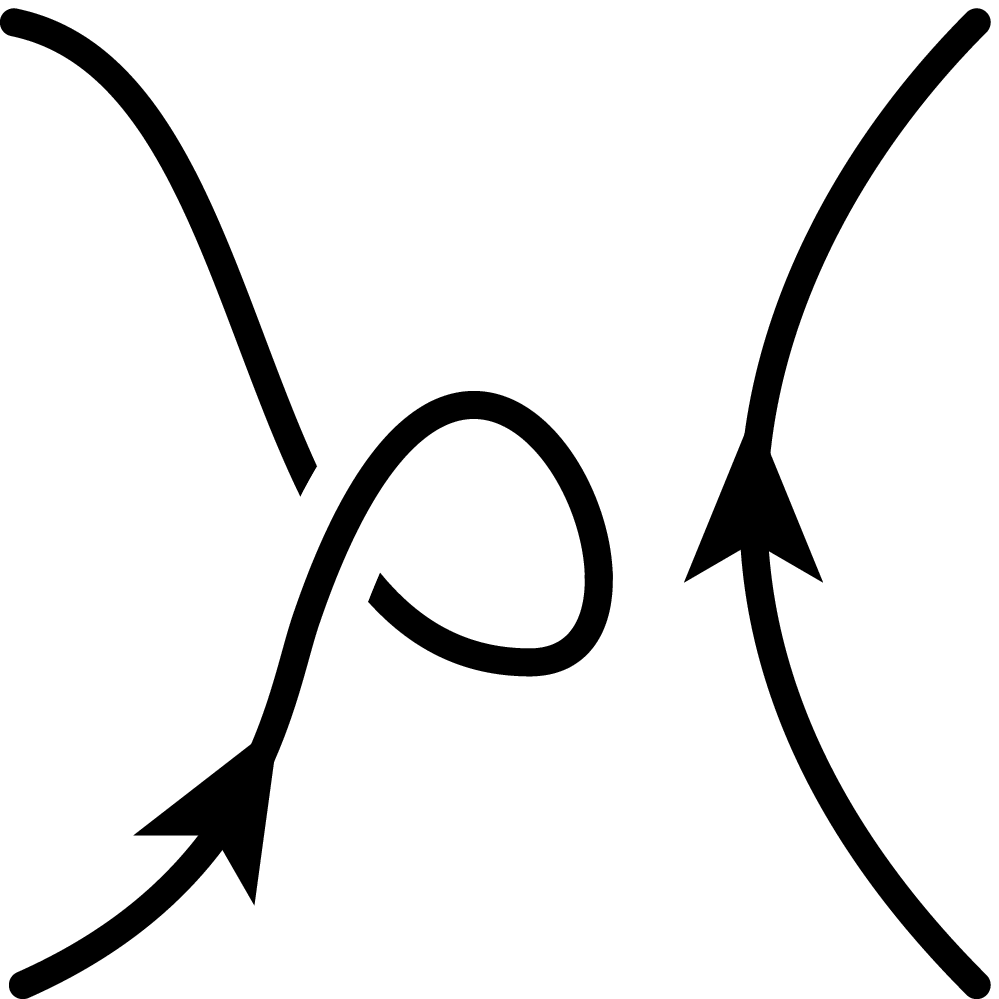} \ar@{|->}[r]^{\text{saddle}} & \pd{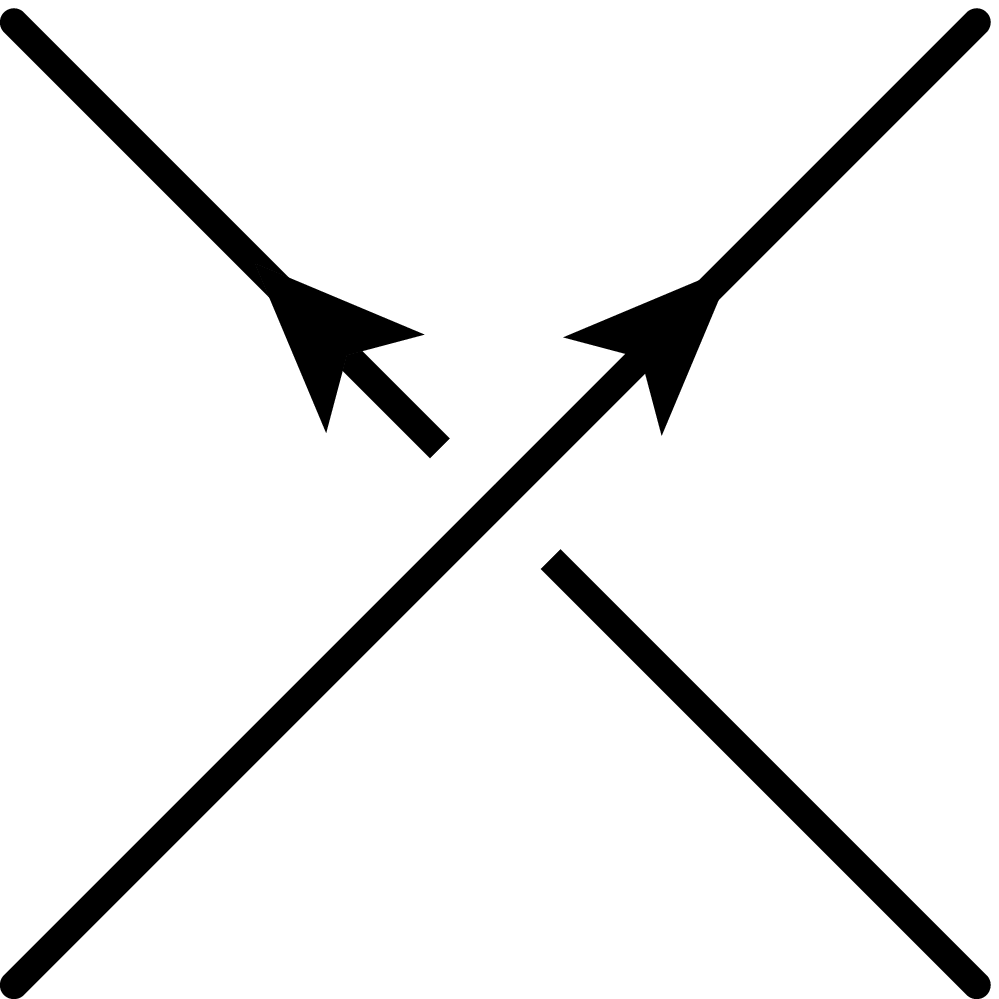} \\
    \pd{MM13f1} \ar@{|->}[r]^{\text{s}_L} & \pd{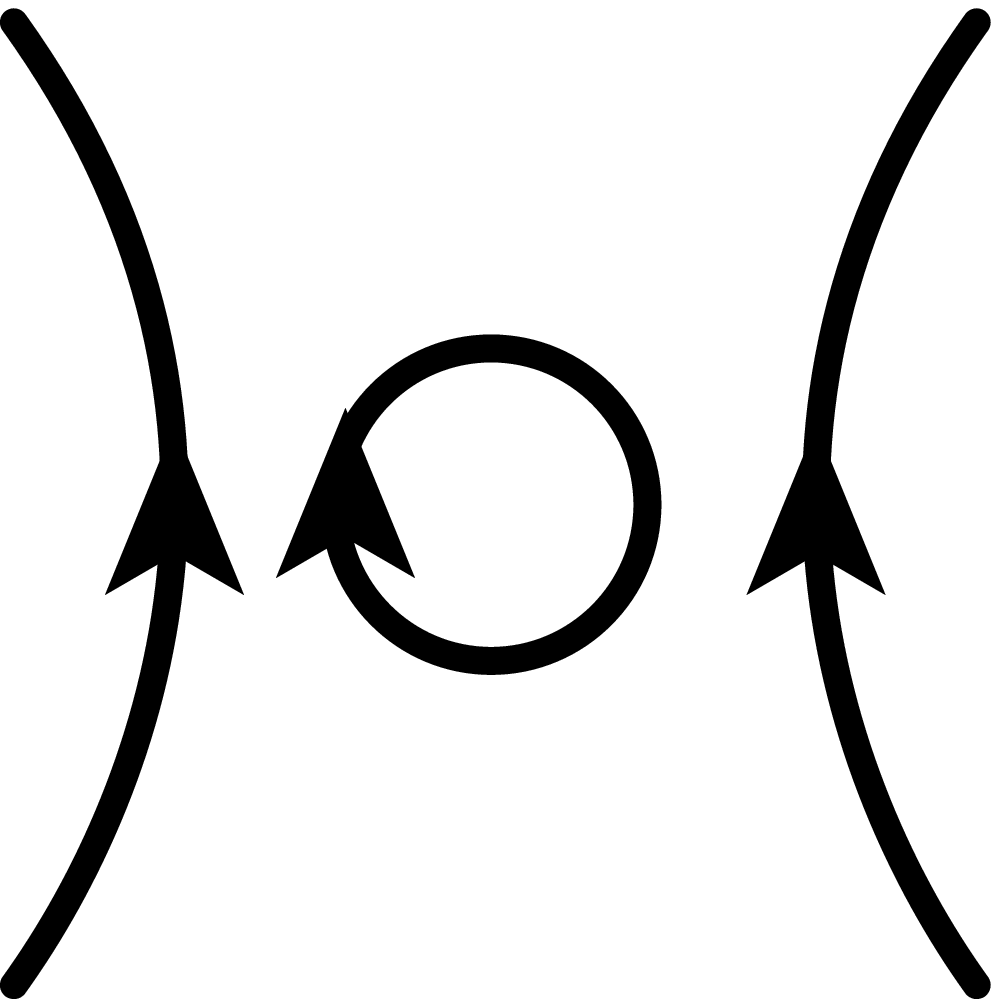} \ar@{|->}[r]^{\text{s}_R} & \pd{MM13f1},
}
\end{align*}
and on the right,
\begin{align*}
\xymatrix@R-5mm@C+10mm{
    \pd{MM13f1} \ar@{|->}[r]^{R_1a} & \pd{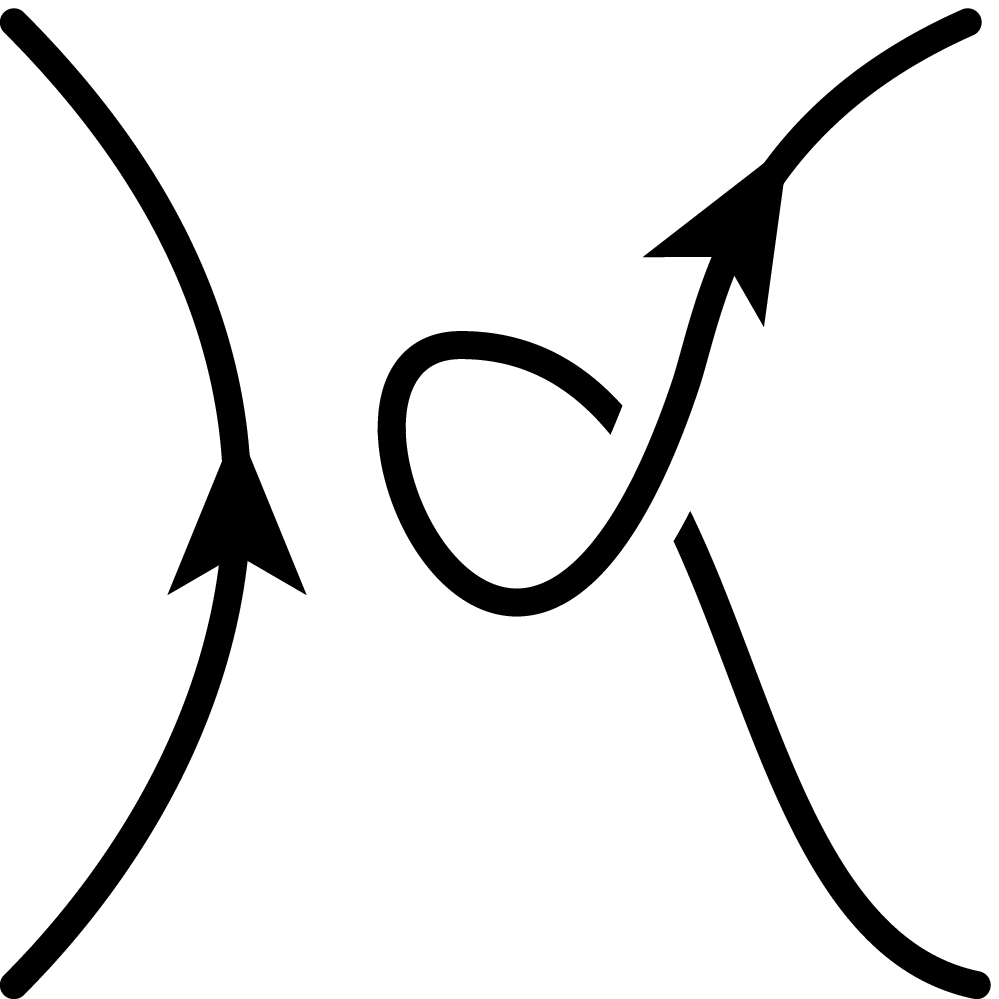} \ar@{|->}[r]^{\text{saddle}} & \pd{MM13f3} \\
    \pd{MM13f1} \ar@{|->}[r]^{\text{s}_R} & \pd{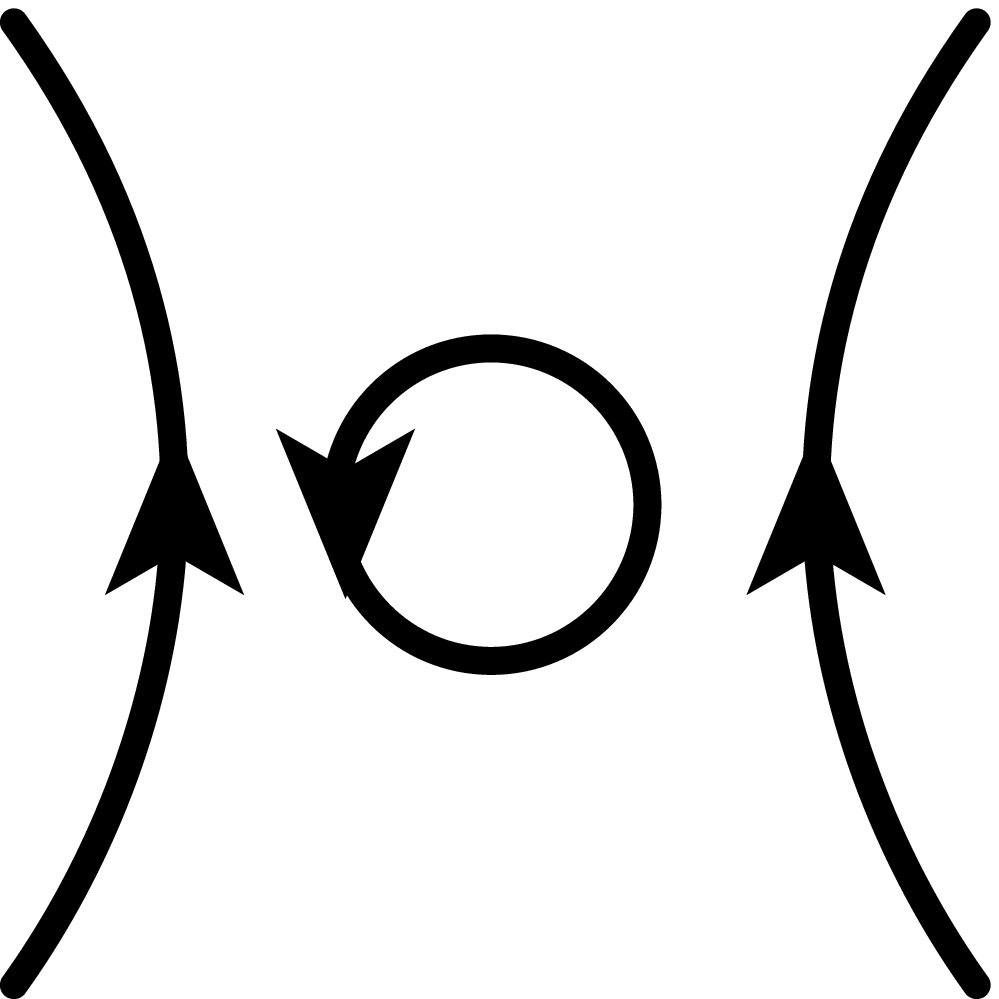} \ar@{|->}[r]^{\text{s}_L} & \pd{MM13f1},
}
\end{align*}
}
where $\text{s}_L$ is a saddle to left sheet, and $\text{s}_R$ is a saddle to the right sheet. Either way, our composition is a tube between the sheets.

In the reverse time direction, we have
{
\newcommand{\pd}[1]{\mathfig{0.08}{movie_moves/MM13/#1}}
\newcommand{\pds}[1]{\mathfig{0.06}{movie_moves/MM13/#1}}
\begin{align*}
\xymatrix@R-5mm@C+10mm{
    \pd{MM13f1} & \pd{MM13f2a} \ar@{|->}[l]^{R_1a^{-1}} & \pd{MM13f3} \ar@{|->}[l]^{\text{saddle}} \\
    \pd{MM13f1} & \pd{MM13f2a_trace} \ar@{|->}[l]^{\text{cap}} & \pd{MM13f1} \ar@{|->}[l]^{\text{s}_L}
}
\end{align*}
on the left, and
\begin{align*}
\xymatrix@R-5mm@C+10mm{
    \pd{MM13f1}  & \pd{MM13f2b} \ar@{|->}[l]^{R_1a^{-1}} & \pd{MM13f3} \ar@{|->}[l]^{\text{saddle}} \\
    \pd{MM13f1} & \pd{MM13f2b_trace} \ar@{|->}[l]^{\text{cap}} & \pd{MM13f1} \ar@{|->}[l]^{\text{s}_R}
}
\end{align*}
}
on the right, giving us the identity compositions.

Again, the mirror image move (with a negative twist) has the same morphisms, though swapped with respect to time direction.

\subsubsection*{MM14}

$$\mathfig{0.2}{movie_moves/MM14}$$

We have some orientation choices for this move, and the circle may end up above or below the vertical strand. Assume first the circle is oriented anti-clockwise, and lies above an upward-oriented vertical strand. In the forward time direction, on the left we have
{
\newcommand{\pd}[1]{\mathfig{0.06}{movie_moves/MM14/#1}}
\newcommand{\pds}[1]{\mathfig{0.07}{movie_moves/MM14/#1}}
\begin{align*}
\xymatrix@R-5mm@C+10mm{
    \mathfig{0.011}{movie_moves/MM14/MM14f1} \ar@{|->}[r]^{\text{birth}} & \pd{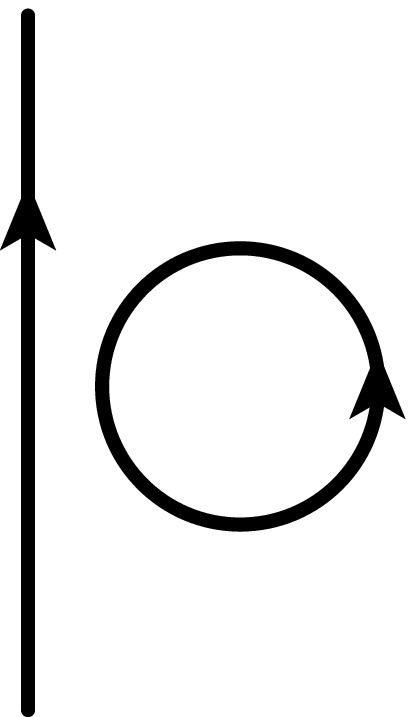} \ar@{|->}[r]^{R_2b} & \pd{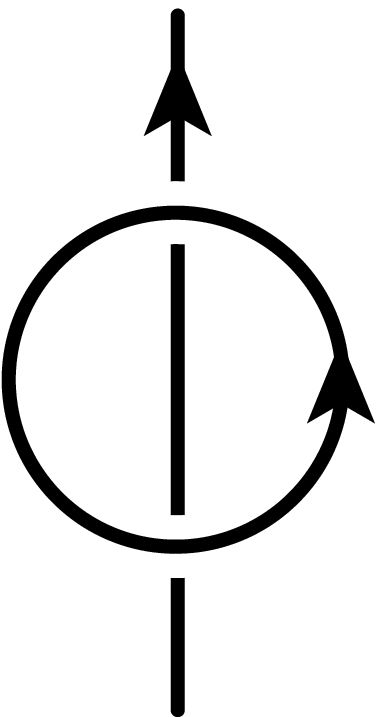} \\
    \mathfig{0.011}{movie_moves/MM14/MM14f1} \ar@{|->}[r]^{\pds{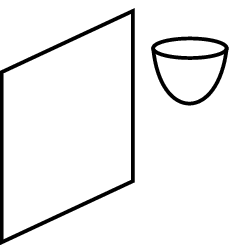}} & \pd{MM14f2a} \ar@{|->}[r]^{\pds{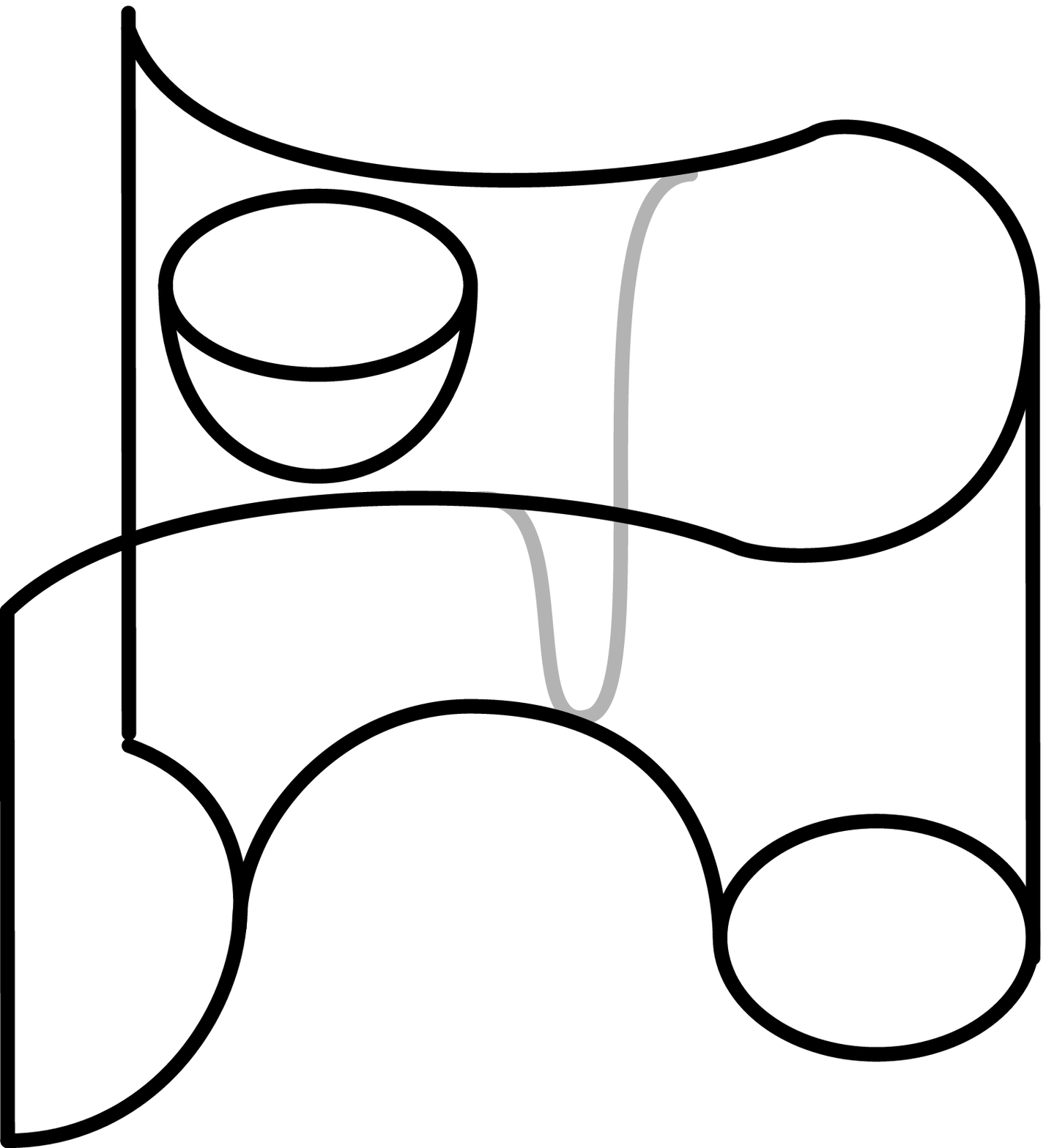}} \ar@{|->}[dr]_{\pds{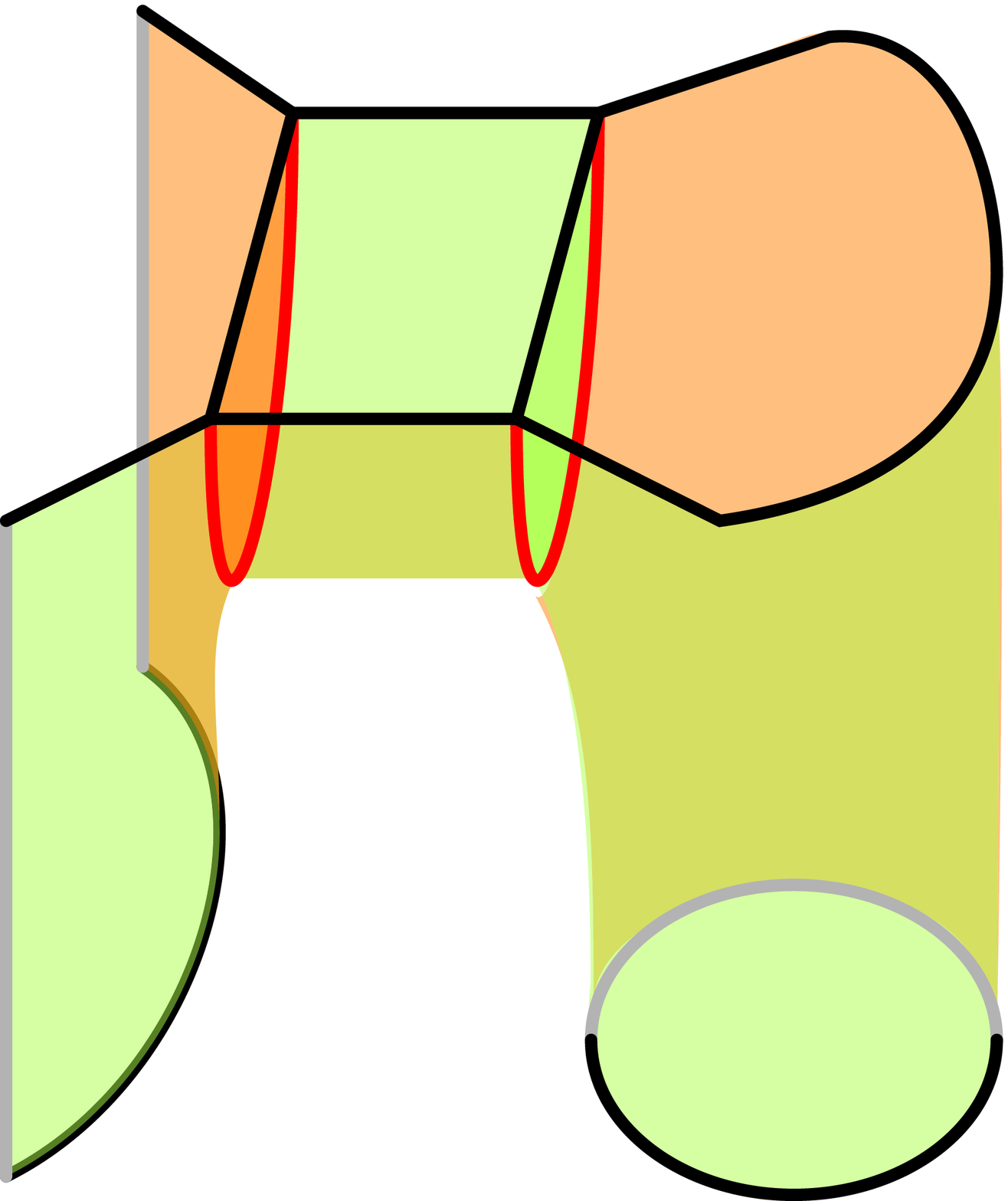}} & \pd{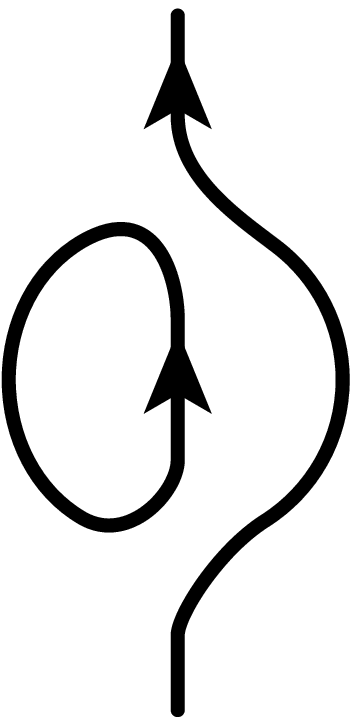} \\
    & & \pd{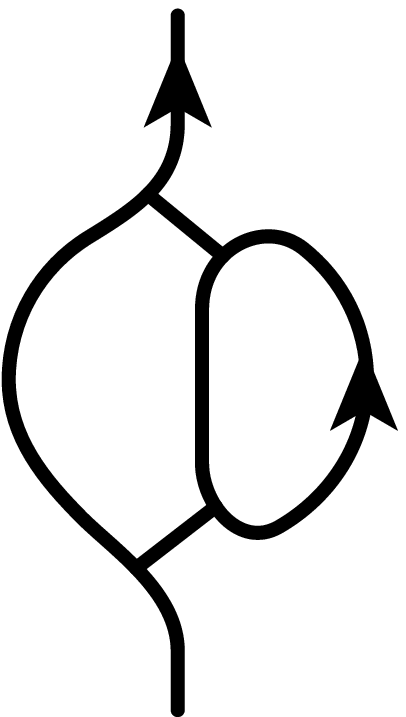}
}
\end{align*}
while on the right we have
\begin{align*}
\xymatrix@R-5mm@C+10mm{
    \mathfig{0.011}{movie_moves/MM14/MM14f1} \ar@{|->}[r]^{\text{birth}} & \pd{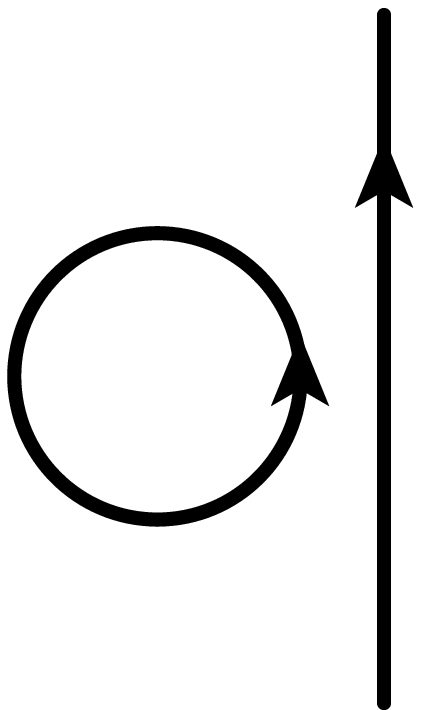} \ar@{|->}[r]^{R_2b} & \pd{MM14f3} \\
    \mathfig{0.011}{movie_moves/MM14/MM14f1} \ar@{|->}[r]^{\pds{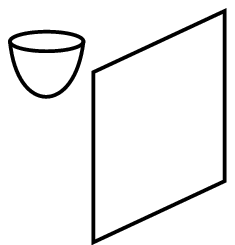}} & \pd{MM14f2b} \ar@{|->}[r]^{id} \ar@{|->}[dr]_{\pds{R2a_tuck_map}} & \pd{MM14f3cx1} \\
    & & \pd{MM14f3cx2}
}
\end{align*}
}
The top composition in each case is just
$$\mathfig{0.1}{movie_moves/MM14/curtain_and_ldisc_creation},$$
while the bottom composition is given by
$$\mathfig{0.1}{movie_moves/MM14/lower_composition}.$$

The reverse time direction is similar and easily described in words. First, reverse all arrows and turn all morphisms upside down in the diagrams above. Then add a negative sign to each of the two diagonal arrows, as dictated in our definitions of the $R2$ maps. Clearly, the left and right sides again yield the same compositions.

Changing orientations and strand height do not change the calculations.

\subsubsection*{MM15}

$$\mathfig{0.2}{movie_moves/MM15}$$

By symmetry we can assume that the middle strand is oriented left to right. The other strands must be oriented oppositely for the saddle to occur, leaving two possible orientations, and the middle strand can either pass above or below the other two. Thus there are a total of four variations, and we'll show calculations for one.

In forward time, we have on the left
\begin{align*}
\newcommand{\pd}[1]{\mathfig{0.1}{movie_moves/MM15/#1}}
\xymatrix@R-7.5mm@C+5mm{
    \pd{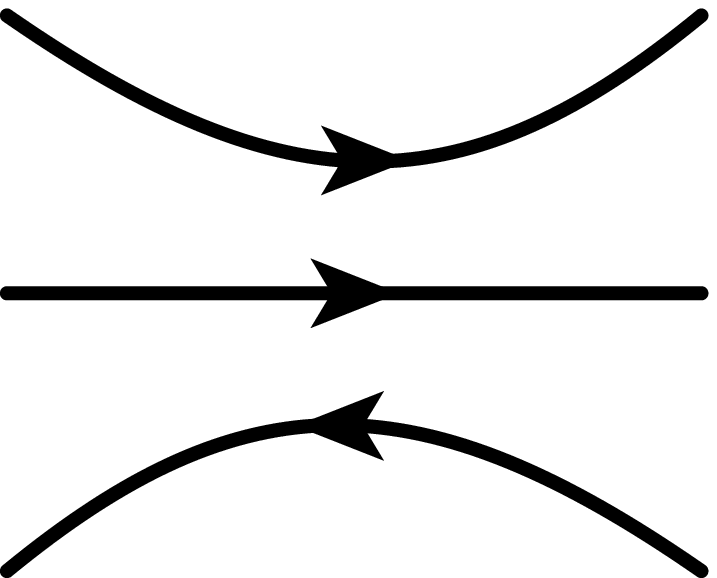} \ar[r]^{R_2a} & \pd{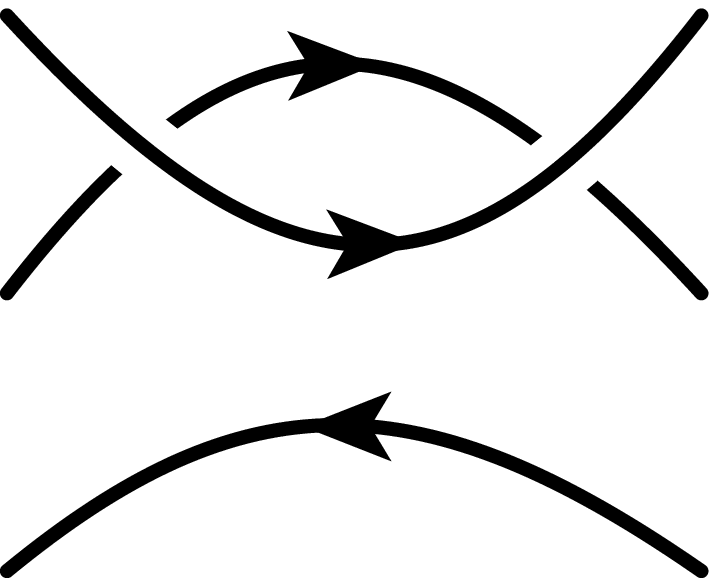} \ar[r]^{\text{saddle}} & \pd{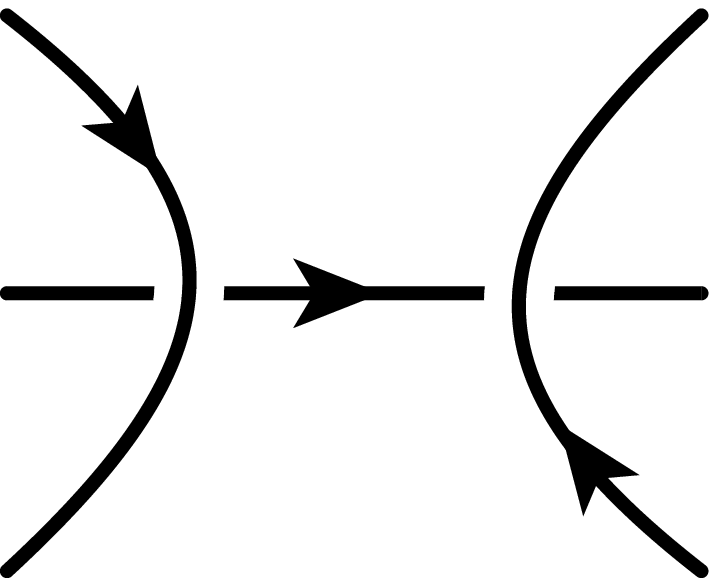} \\
    \\
     \pd{MM15f1} \ar@{|->}[r]^{1} \ar@{|->}[dr] & %
     \pd{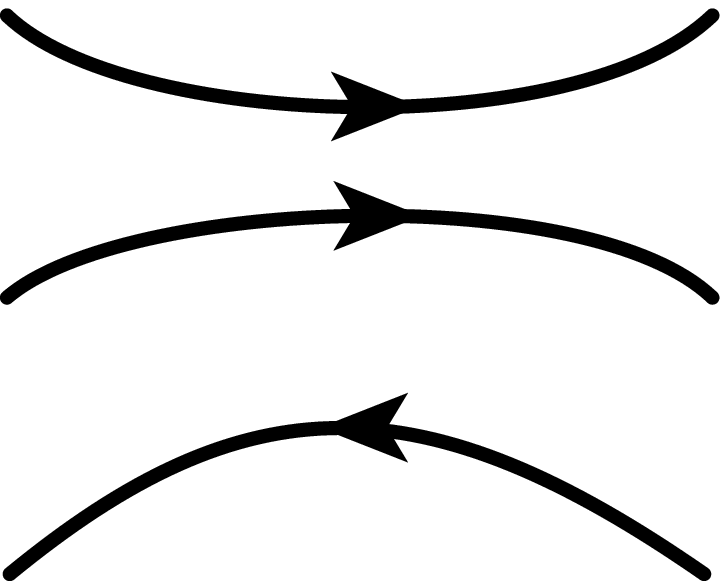} \ar@{|->}[r]^{\text{saddle}} \ar@{}[d]|{\directSum} & %
     \pd{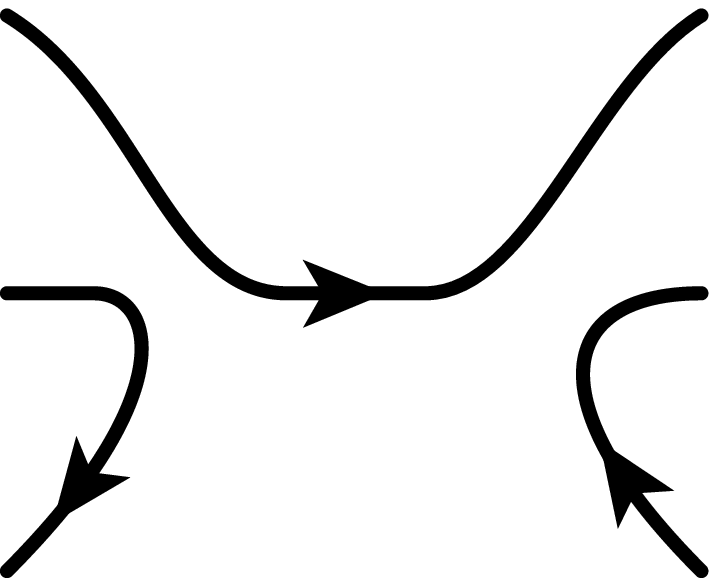} \ar@{}[d]|{\directSum} \\ %
     & \pd{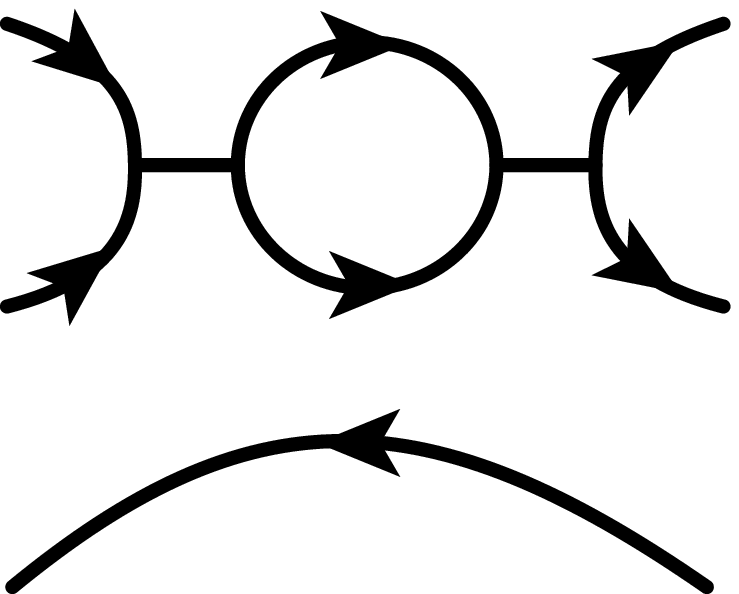} \ar@{|->}[r] & \pd{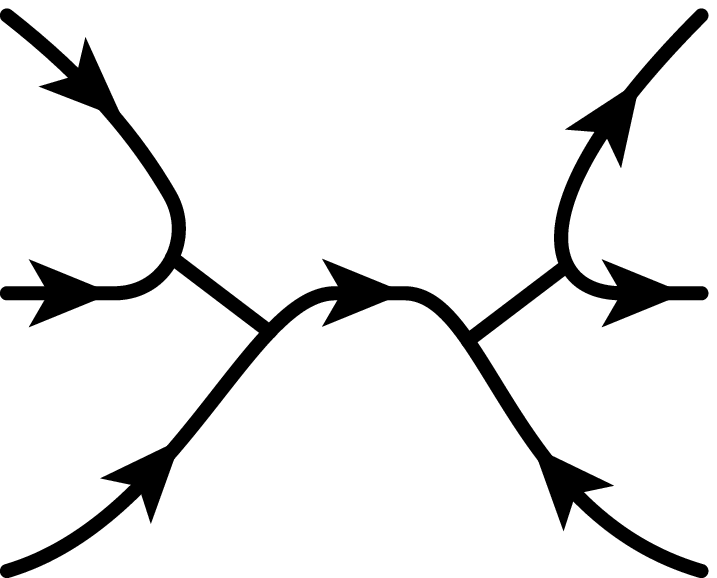} %
}
\end{align*}
and on the right
\begin{align*}
\newcommand{\pd}[1]{\mathfig{0.1}{movie_moves/MM15/#1}}
\newcommand{\pa}[1]{\mathfig{0.05}{movie_moves/MM15/#1}}
\xymatrix@R-7.5mm@C+5mm{
    \pd{MM15f1} \ar[r]^{R_2b} & \pd{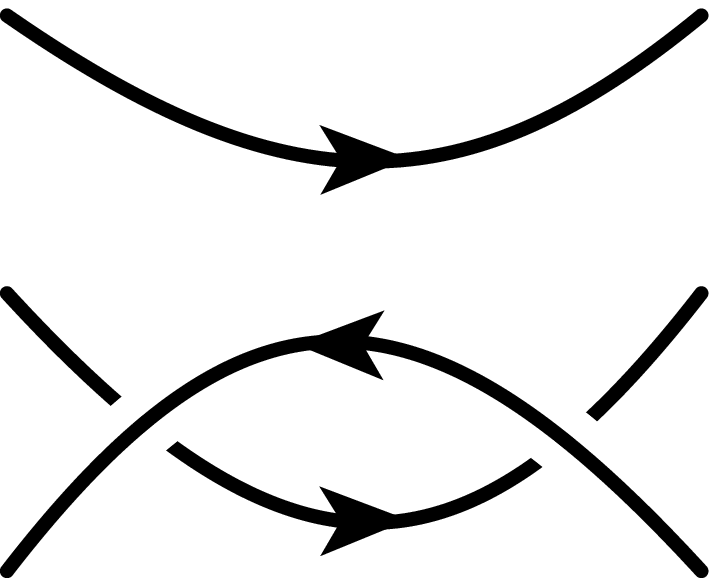} \ar[r]^{\text{saddle}} & \pd{MM15f3} \\
     & & \\
     \pd{MM15f1} \ar@{|->}[r]^{\pa{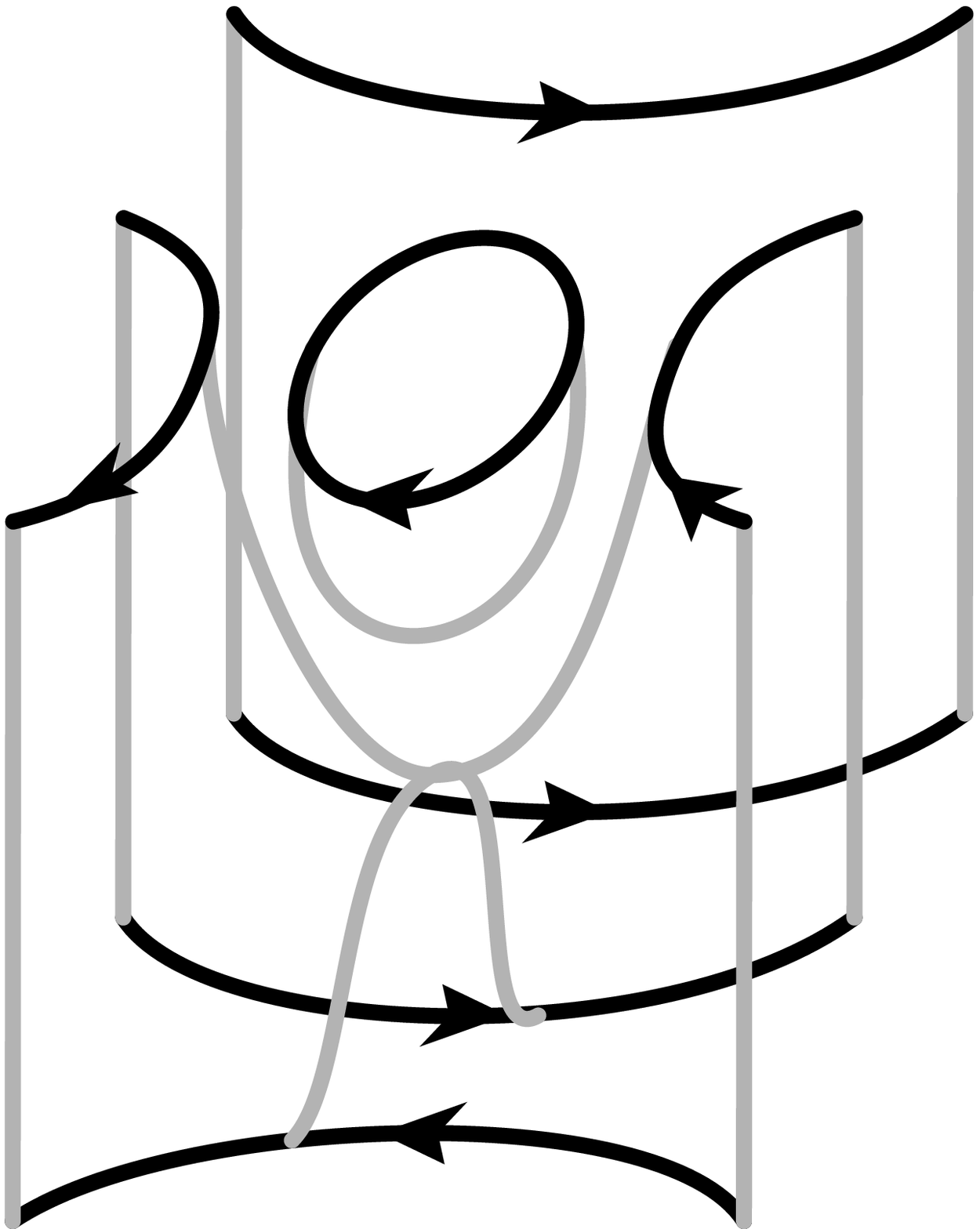}} \ar@{|->}[dr] & %
     \pd{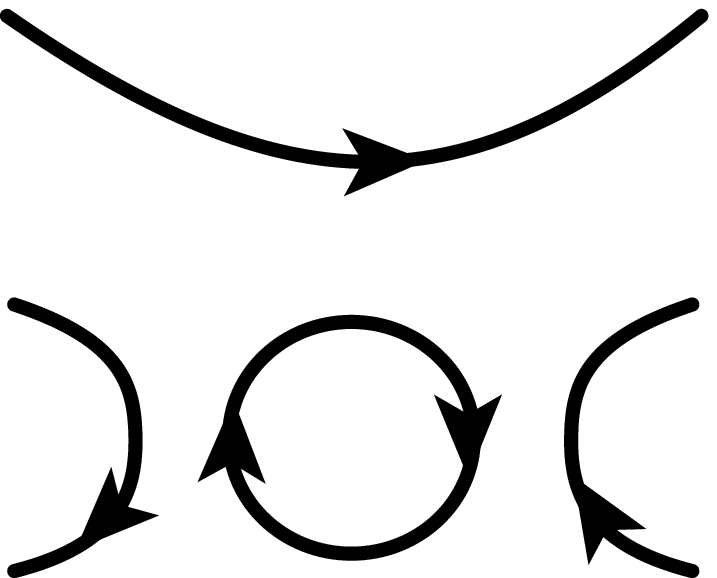} \ar@{|->}[r]^{\pa{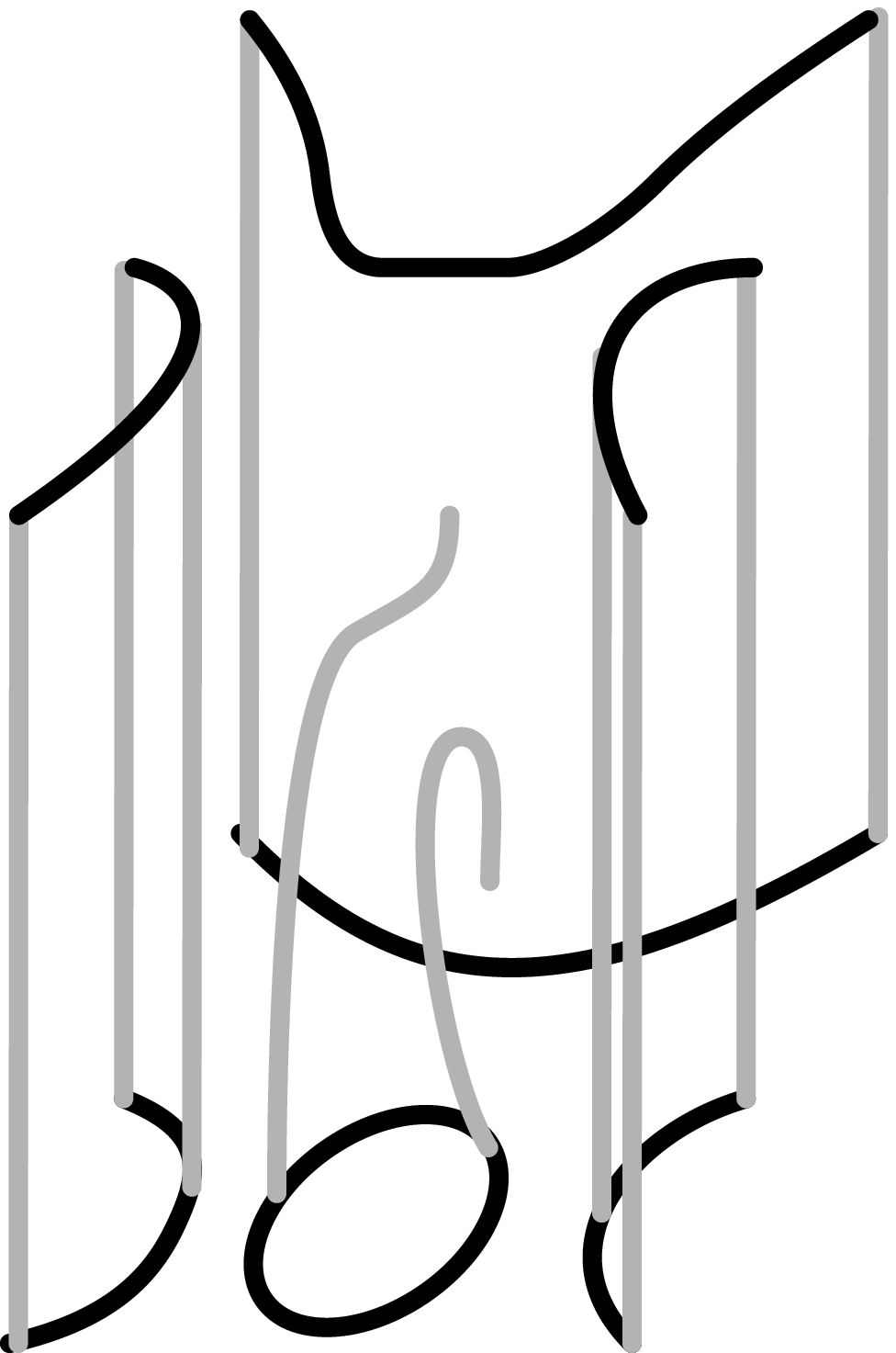}} \ar@{}[d]|{\directSum} & %
     \pd{MM15f3_trace1} \ar@{}[d]|{\directSum} \\ %
     & \pd{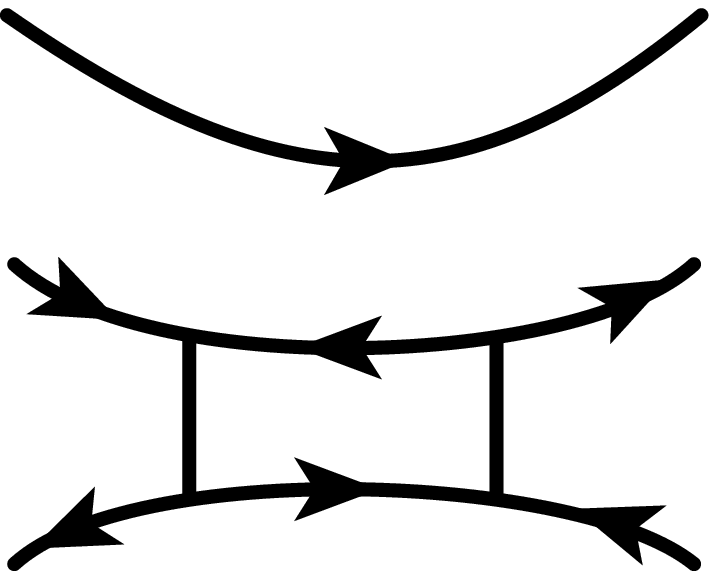} \ar@{|->}[r] & \pd{MM15f3_trace2}. %
}
\end{align*}

We need only concern ourselves with the component of the maps going to
$\mathfig{0.06}{movie_moves/MM15/MM15f3_trace1}$, and have left the other components unlabeled. Clearly, the relevant composition on each side is just a saddle between the lower two strands, and we have equality.

In the reverse time direction, we have on the left
\begin{align*}
\newcommand{\pd}[1]{\mathfig{0.1}{movie_moves/MM15/#1}}
\xymatrix@R-7.5mm@C+5mm{
    \pd{MM15f1}& \pd{MM15f2a} \ar[l]_{{R_2a}^{-1}} & \pd{MM15f3} \ar[l]_{\text{saddle}} \\
    \\
     \pd{MM15f1} & %
     \pd{MM15f2a_trace1} \ar@{|->}[l]_{1} & %
     \pd{MM15f3_trace1} \ar@{|->}[l]_{\text{saddle}} \\ %
}
\end{align*}
and on the right
\begin{align*}
\newcommand{\pd}[1]{\mathfig{0.1}{movie_moves/MM15/#1}}
\newcommand{\pa}[1]{\mathfig{0.05}{movie_moves/MM15/#1}}
\xymatrix@R-7.5mm@C+5mm{
    \pd{MM15f1}& \pd{MM15f2b} \ar[l]_{{R_2b}^{-1}} & \pd{MM15f3} \ar[l]_{\text{saddle}} \\
    \\
     \pd{MM15f1} & %
     \pd{MM15f2b_trace2} \ar@{|->}[l]_{\pa{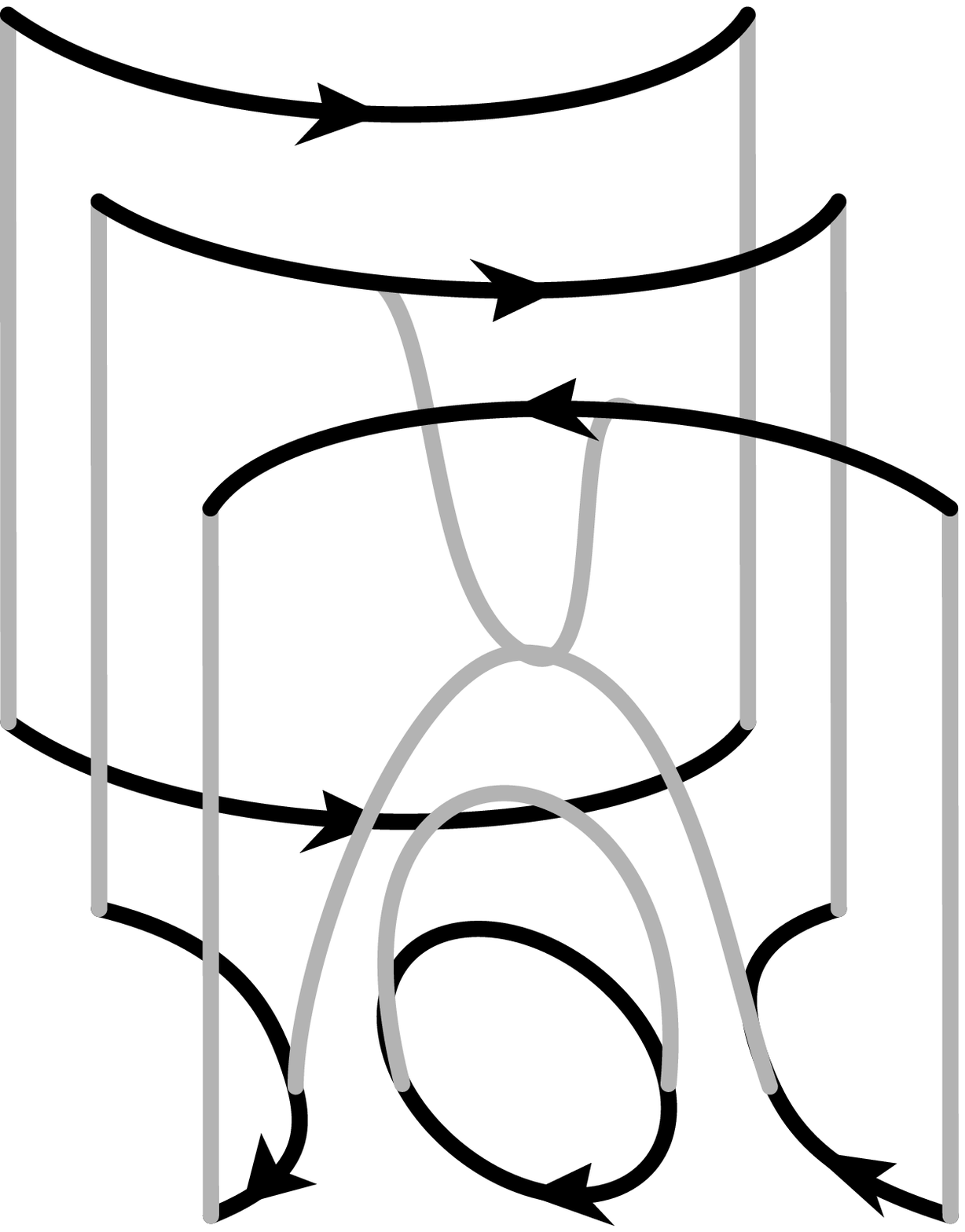}} & %
     \pd{MM15f3_trace1} \ar@{|->}[l]_{\pa{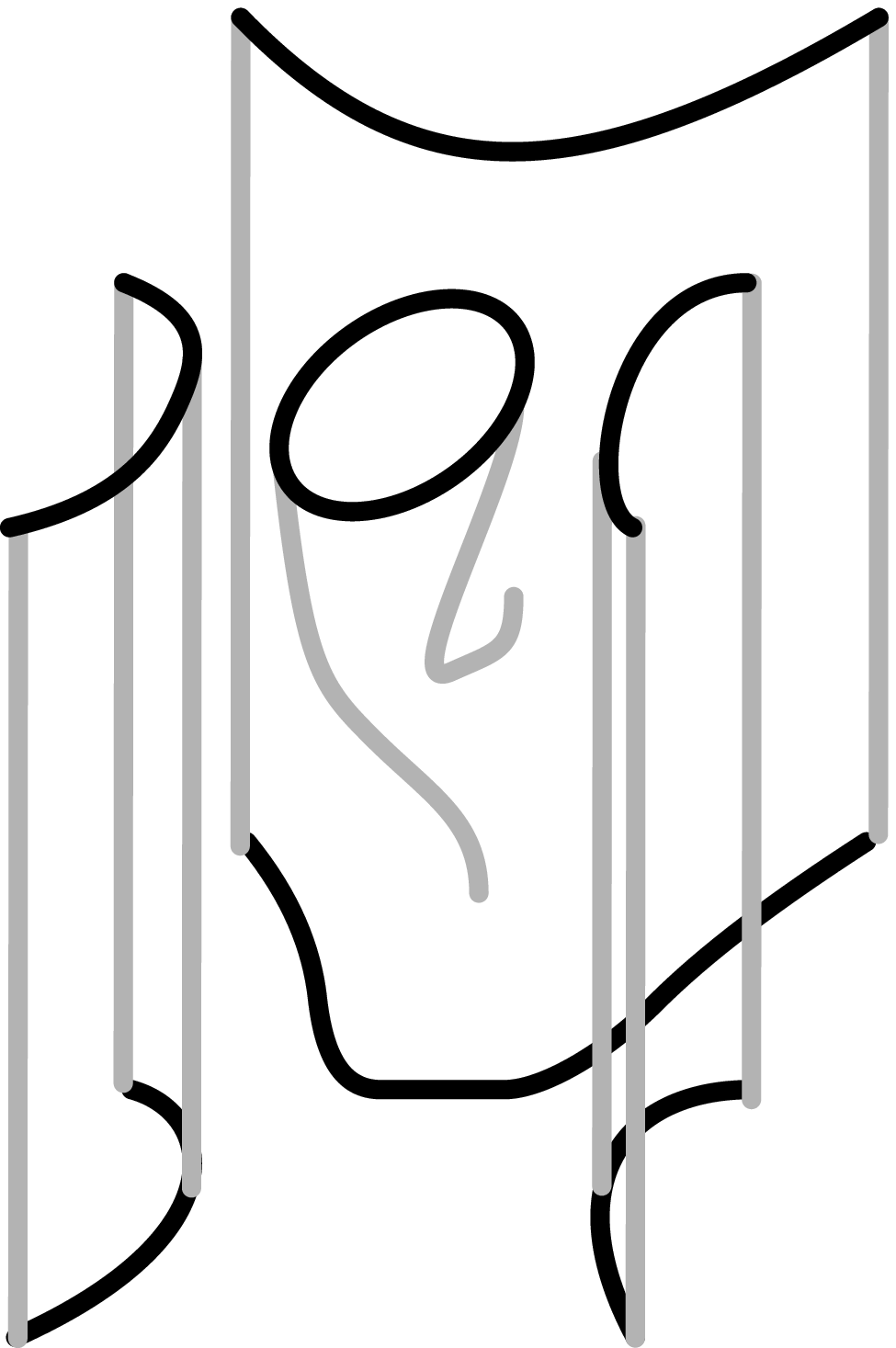}}. \\ %
}
\end{align*}
Again we see a saddle between the lower strands on each side.

The other three variations are almost exactly the same, and we leave them as an exercise.

\subsection{The End}
\label{ssec:theend}
Having now shown that movie moves do not change induced maps, the proof of Theorem \ref{thm:1} is complete. 

\appendix
\section{Simplifying complexes}
\label{sec:simplifying}

\subsection{Gaussian elimination for complexes}
\label{ssec:gaussian}

The following two lemmas provide a nice tool for simplifying chain complexes without changing their homotopy type. Note that, throughout complexes in Appendix \ref{sec:simplifying}, we'll write $\bullet$ for any maps that we don't need to know explicitly.

The first lemma comes from Bar-Natan \cite{BN07_MR2320156}.

\begin{lemma}[Single Gaussian elimination]%
\label{lem:single-gaussian}%
Consider the complex
\begin{equation}
\label{eq:complex}
 \xymatrix@C+55pt@R+20pt{
    A                     \ar[r]^{\psmallmatrix{\bullet \\ \alpha}}              &
    \directSumStack{B}{C} \ar[r]^{\psmallmatrix{\varphi & \lambda \\ \mu & \nu}} &
    \directSumStack{D}{E} \ar[r]^{\psmallmatrix{\bullet & \epsilon}}             &
    F
 }
\end{equation}
in any additive category, where $\varphi: B \IsoTo D$ is an
isomorphism, and all other morphisms are arbitrary (subject to
$d^2=0$, of course). Then there is a homotopy equivalence with a
much simpler complex, ``stripping off" $\varphi$.

\begin{equation*}
 \xymatrix@C+55pt@R+20pt{
    A                     \ar[r]^{\psmallmatrix{\bullet \\ \alpha}}              \ar@{<->}[d]^{\psmallmatrix{1}} &
    \directSumStack{B}{C} \ar[r]^{\psmallmatrix{\varphi & \lambda \\ \mu & \nu}} \ar@<-0.5ex>[d]_{\psmallmatrix{0 & 1}}  &
    \directSumStack{D}{E} \ar[r]^{\psmallmatrix{\bullet & \epsilon}}             \ar@<-0.5ex>[d]_{\psmallmatrix{-\mu \varphi^{-1} & 1}} &
    F                                                                            \ar@{<->}[d]^{\psmallmatrix{1}} \\
    A \ar[r]^{\psmallmatrix{\alpha}}                          &
    C \ar[r]^{\psmallmatrix{\nu - \mu \varphi^{-1} \lambda}} \ar@<-0.5ex>[u]_{\psmallmatrix{-\varphi^{-1} \lambda \\ 1}} &
    E \ar[r]^{\psmallmatrix{\epsilon}}                       \ar@<-0.5ex>[u]_{\psmallmatrix{0 \\ 1}} &
    F
 }
\end{equation*}
\end{lemma}
\begin{remark}
It's an easy check that Gaussian elimination is a strong deformation retract (Definition \ref{def:simple}).
\end{remark}

By applying Gaussian elimination twice on two adjacent isomorphisms (that aren't composable), we get the following corollary \cite{MorrN:sl(3),ClarMW:Fix}.

\begin{lemma}[Double Gaussian elimination]%
\label{lem:double-gaussian}%
When $\psi$ and $\varphi$ are isomorphisms, there's a homotopy
equivalence of complexes:
\begin{equation*}
 \xymatrix@C+45pt@R+20pt{
    A                                 \ar[r]^{\psmallmatrix{\bullet \\ \alpha}}                                    \ar@{<->}[d]^{\psmallmatrix{1}} &
    \directSumStack{B}{C}             \ar[r]^{\psmallmatrix{\psi & \beta \\ \bullet & \bullet \\ \gamma & \delta}} \ar@<-0.5ex>[d]_{\psmallmatrix{0 & 1}}  &
    \directSumStackThree{D_1}{D_2}{E} \ar[r]^{\psmallmatrix{\bullet & \varphi & \lambda \\ \bullet & \mu & \nu}}   \ar@<-0.5ex>[d]_{\psmallmatrix{-\gamma \psi^{-1} & 0 & 1}} &
    \directSumStack{F}{G}             \ar[r]^{\psmallmatrix{\bullet & \eta}}                                       \ar@<-0.5ex>[d]_<(0.2){\psmallmatrix{-\mu \varphi^{-1} & 1}} &
    H                                                                                                              \ar@{<->}[d]^{\psmallmatrix{1}} \\
    A \ar[r]^{\psmallmatrix{\alpha}}                          &
    C \ar[r]^{\psmallmatrix{\delta - \gamma \psi^{-1} \beta}} \ar@<-0.5ex>[u]_<(0.25){\psmallmatrix{-\psi^{-1} \beta \\ 1}} &
    E \ar[r]^{\psmallmatrix{\nu - \mu \varphi^{-1} \lambda}}  \ar@<-0.5ex>[u]_<(0.3){\psmallmatrix{0 \\ - \varphi^{-1} \lambda \\ 1}} &
    G \ar[r]^{\psmallmatrix{\eta}}                            \ar@<-0.5ex>[u]_<(0.35){\psmallmatrix{0 \\ 1}} &
    H
 }
\end{equation*}
\end{lemma}

\subsection{Three isomorphisms in $\Kob{3}$}
\label{ssec:isomorphisms}

To put Gaussian elimination to work for us, we'll use the following isomorphisms (Theorem 3.11 in \cite{MorrN:sl(3)}).\footnote{These isomorphisms categorify the relations in Kuperberg's spider (Equation \ref{eq:spider}), and are readily proven using the local relations in $\Cob{3}$ (Section \ref{ssec:categorifying})}
\begin{enumerate}
\item $\mathfig{0.045}{webs/clockwise_circle} \iso q^{-2} \, \emptyset \directSum q^0 \, \emptyset \directSum q^2 \, \emptyset$, a.k.a. ``delooping," is an isomorphism via the maps
\begin{align*}
\xymatrix@C+=20mm@R+=5mm{
    & q^{-2} \, \emptyset \ar@{.}[d]|{\textrm{\scalebox{1.5}{$\directSum$}}} \ar[dr]^{\frac{1}{3}\rotatemathfig{0.025}{90}{cobordisms/handle_bdy_right}} & \\
    \mathfig{0.05}{webs/clockwise_circle}
                    \ar[ur]_{\mathfig{0.025}{foams/death}}
                    \ar[r]_(0.65){\frac{1}{3} \rotatemathfig{0.025}{90}{cobordisms/handle_disc_bdy_left}}
                    \ar[dr]_{\frac{1}{3} \rotatemathfig{0.025}{90}{cobordisms/handle_bdy_left}} &
    q^0 \, \emptyset \ar@{.}[d]|{\textrm{\scalebox{1.5}{$\directSum$}}} \ar[r]^(0.3){-\frac{1}{3} \rotatemathfig{0.025}{90}{cobordisms/handle_disc_bdy_right}} &
    \mathfig{0.05}{webs/clockwise_circle} \\
 & q^2  \, \emptyset \ar[ur]^{\mathfig{0.025}{foams/birth}} & }
\end{align*}

\item $\mathfig{0.045}{webs/bubble} \iso q^{-1} \;\mathfig{0.009}{webs/tall_strand} \directSum q \;\mathfig{0.009}{webs/tall_strand}$, a.k.a. ``debubbling," is an isomorphism via the maps
\begin{align*}
\xymatrix@C+15mm@R-15mm{%
    &
    q^{-1} \;\mathfig{0.009}{webs/tall_strand}
        \ar@{}[dd]_{\textrm{\scalebox{2}{$\directSum$}}}
        \ar[dr]^{\frac{1}{2}\mathfig{0.1}{cobordisms/tube_relation/upper_bubble_kiss}} &
    \\
     \mathfig{0.045}{webs/bubble}
        \ar[ur]^{\mathfig{0.1}{cobordisms/tube_relation/lower_bubble}}
        \ar[dr]_{\frac{1}{2}\mathfig{0.1}{cobordisms/tube_relation/lower_bubble_kiss}} &
    &
     \mathfig{0.045}{webs/bubble} \\
    &
    q^{\phantom{-1}} \;\mathfig{0.009}{webs/tall_strand}
        \ar[ur]_{\mathfig{0.1}{cobordisms/tube_relation/upper_bubble}} &
    \\
}
\end{align*}

\item $\mathfig{0.08}{webs/oriented_square}
    \iso \mathfig{0.08}{webs/two_strands_horizontal} \directSum \mathfig{0.08}{webs/two_strands_vertical}$, a.k.a. ``desquaring," is an isomorphism via
\begin{align*}
\xymatrix@C+15mm@R-10mm{%
    &
    \mathfig{0.08}{webs/two_strands_horizontal}
        \ar[dr]^{-\mathfig{0.06}{cobordisms/rocket_relation/rocket_y_upper}} &
    \\
     \mathfig{0.08}{webs/oriented_square}
        \ar[ur]^{\mathfig{0.06}{cobordisms/rocket_relation/rocket_y_lower}}
        \ar[dr]_{\mathfig{0.06}{cobordisms/rocket_relation/rocket_x_lower}} &
    \textrm{\scalebox{1.6}{{$\directSum$}}} &
     \mathfig{0.08}{webs/oriented_square} \\
    &
    \mathfig{0.08}{webs/two_strands_vertical}
        \ar[ur]_{-\mathfig{0.06}{cobordisms/rocket_relation/rocket_x_upper}} &
    \\
}
\end{align*}
\end{enumerate}

\subsection{Proof of the strand-past-vertex moves}
\label{ssec:SPV-moves}

\begin{proof}[Proof of Lemma \ref{lem:spv}]
This requires only a slight modification of the argument given for Lemma \ref{lem:IBL} in \cite{MorrN:sl(3)}.

Consider the IBR variation. In step 1, we desquare and deloop the two obvious objects, giving an isomorphic complex:

{
\newcommand{\pd}[1]{\mathfig{0.08}{strand_past_vertex/IBR/#1}}
\newcommand{\pa}[1]{\mathfig{0.1}{strand_past_vertex/IBR/#1}}
\begin{align*}
\Kh{\pd{IBR_initial}} &=
  \xymatrix@C+=20mm{
  q^{-6} \pd{initial1}
    \ar[r]^{\psmallmatrix{\phantom{-}\text{u} \\ -\text{u}}} &
  \directSumStack{q^{-5} \pd{initial2}}{q^{-5} \pd{initial3}}
    \ar[r]^{\psmallmatrix{\text{u} & \text{u}}} &
  q^{-4} \pd{initial4}
  }\\
&\iso
  \xymatrix@C+=20mm{
  \directSumStack{q^{-6} \pd{initial2_prime}}{q^{-6} \pd{initial1_prime}}
    \ar[r]^{\psmallmatrix{\bullet & \bullet \\ \Id & 0 \\ \text{z} & \text{z}}} &
  \directSumStackThree{q^{-4} \pd{initial2_prime}}{q^{-6} \pd{initial2_prime}}{q^{-5} \pd{initial3}}
    \ar[r]^{\psmallmatrix{\Id & \bullet & \text{u}}} &
  q^{-4} \pd{initial4}
  },
\intertext{where the differentials were calculated using the blister and airlock relations. We now see two adjacent isomorphisms (identities, in fact), and proceed to step 2: apply Lemma \ref{lem:double-gaussian}.}
\phantom{\Kh{\pd{IBR_initial}}} &\htpy \xymatrix@C+=20mm{
    q^{-6} \pd{final1} \ar[r]^{\text{z}} &
    q^{-5} \pd{final2} &
  }
\end{align*}
}
Here the homotopy equivalence component at height $-2$ is given by $\psmallmatrix{0 \\ \Id}$ with inverse $\psmallmatrix{0 & \Id}$; at height $-1$, the component is $\psmallmatrix{\phantom{-}0 & -\text{z} & \phantom{-}\Id}$ with inverse $\psmallmatrix{-\text{u} \\ \phantom{-}0 \\ \phantom{-}\Id}$. Composing these components with the desquaring and delooping maps from step 1 (and Appendix \ref{ssec:isomorphisms}), we obtain the claimed $\s{IBR}$ and $\T{IBR}$.

Let's consider, instead, the OBR variation. Again, we desquare and deloop:

{
\newcommand{\pd}[1]{\mathfig{0.08}{strand_past_vertex/OBR/#1}}
\newcommand{\pa}[1]{\mathfig{0.1}{strand_past_vertex/OBR/#1}}
\begin{align*}
\Kh{\pd{OBR_initial}} &=
  \xymatrix@C+=20mm{
  q^{-6} \pd{initial1}
    \ar[r]^{\psmallmatrix{-\text{u} \\ \phantom{-}\text{u}}} &
  \directSumStack{q^{-5} \pd{initial2}}{q^{-5} \pd{initial3}}
    \ar[r]^{\psmallmatrix{\text{u} & \text{u}}} &
  q^{-4} \pd{initial4}
  }\\
&\iso
  \xymatrix@C+=20mm{
  \directSumStack{q^{-6} \pd{initial2_prime}}{q^{-6} \pd{initial1_prime}}
    \ar[r]^{\psmallmatrix{\phantom{-}\bullet & \phantom{-}\bullet \\ -\Id & \phantom{-}0 \\ -\text{z} & -\text{z}}} &
  \directSumStackThree{q^{-4} \pd{initial2_prime}}{q^{-6} \pd{initial2_prime}}{q^{-5} \pd{initial3}}
    \ar[r]^{\psmallmatrix{\Id & \bullet & \text{u}}} &
  q^{-4} \pd{initial4}
  },
\intertext{Applying Lemma \ref{lem:double-gaussian}, we get:}
\phantom{\Kh{\pd{OBR_initial}}} &\htpy \xymatrix@C+=20mm{
    q^{-6} \pd{final1} \ar[r]^{-\text{z}} &
    q^{-5} \pd{final2} &
  }
\intertext{with identical homotopy equivalence components to the ones we saw in IBR. Of course, this complex is isomorphic to}
\phantom{\Kh{\pd{OBR_initial}}} &\iso \xymatrix@C+=20mm{
    q^{-6} \pd{final1} \ar[r]^{\text{z}} &
    q^{-5} \pd{final2} &
  }
\end{align*}
via the identity at height $-2$ and minus the identity at height $-1$. Composing these three steps gives $\s{OBR}$ and $\T{OBR}$.
}

The other four variations are proven analogously. (Note that the IAR and OAR complexes will be horizontally reflected.)

\end{proof}


\section{Homological hoodoo}
\label{sec:homological}

\subsection{Planar compositions of complexes}
\label{ssec:planar-comps}

Let's consider the action of the planar arc diagram $\mathfig{0.14}{webs/operad1}$ on the following two complexes
{
\newcommand{\pd}[1]{\mathfig{0.1}{webs/twoholes/#1}}
\newcommand{\pa}[1]{\mathfig{0.14}{webs/twoholes/#1}}
\begin{align*}
\pd{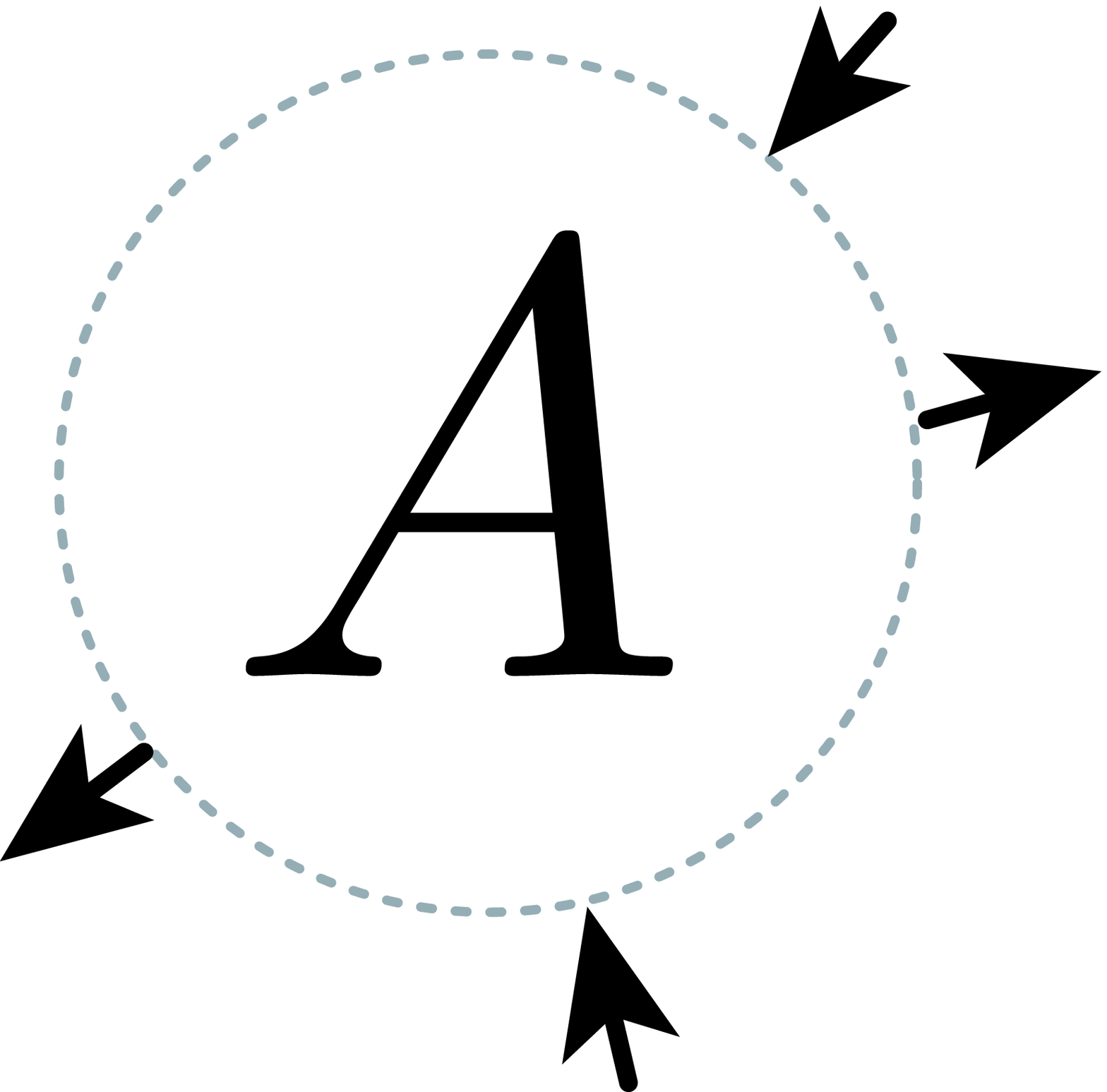} \ = \ \left( \pd{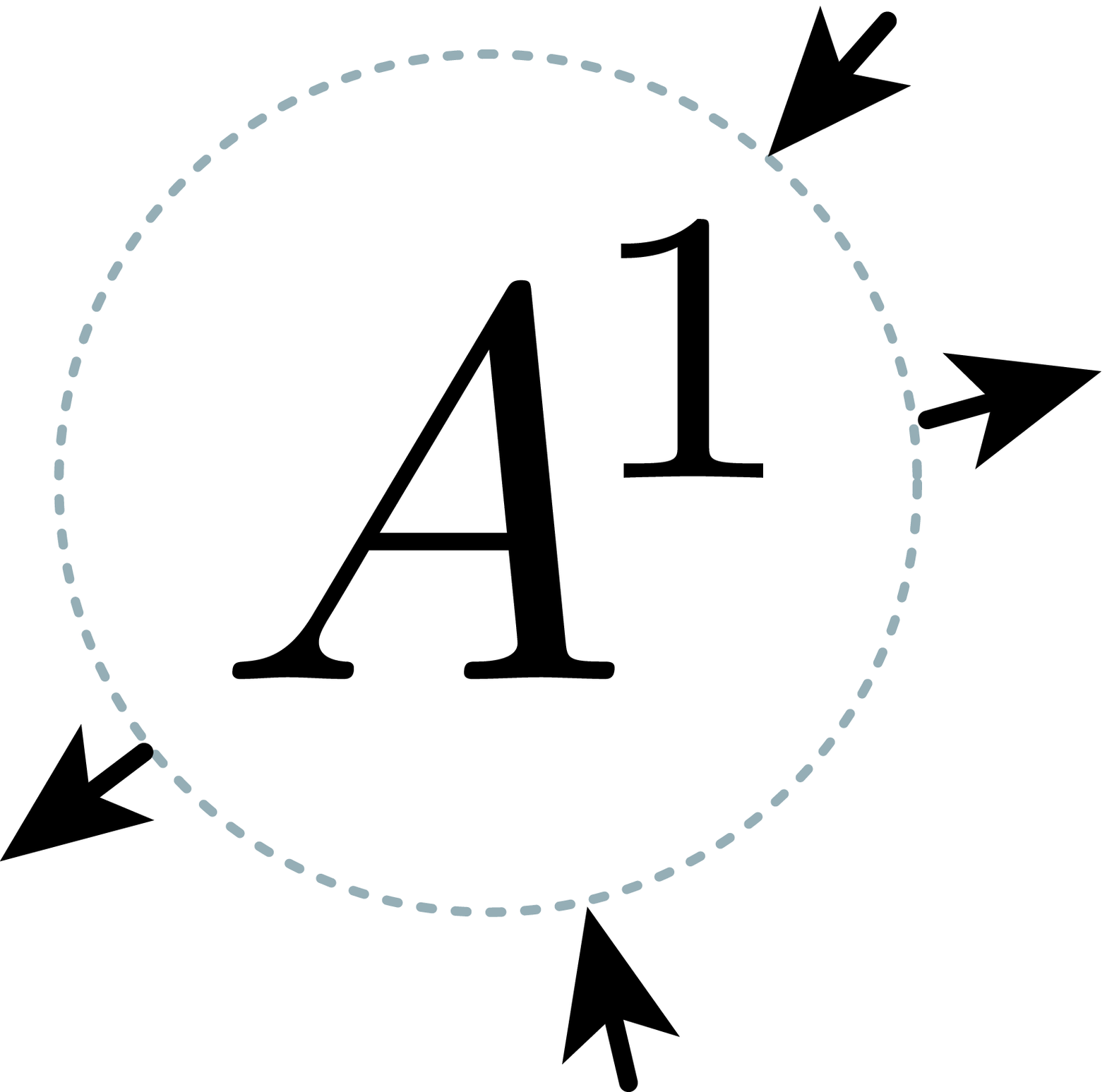} \xrightarrow{ \ \ d_A^1 \ \ } \pd{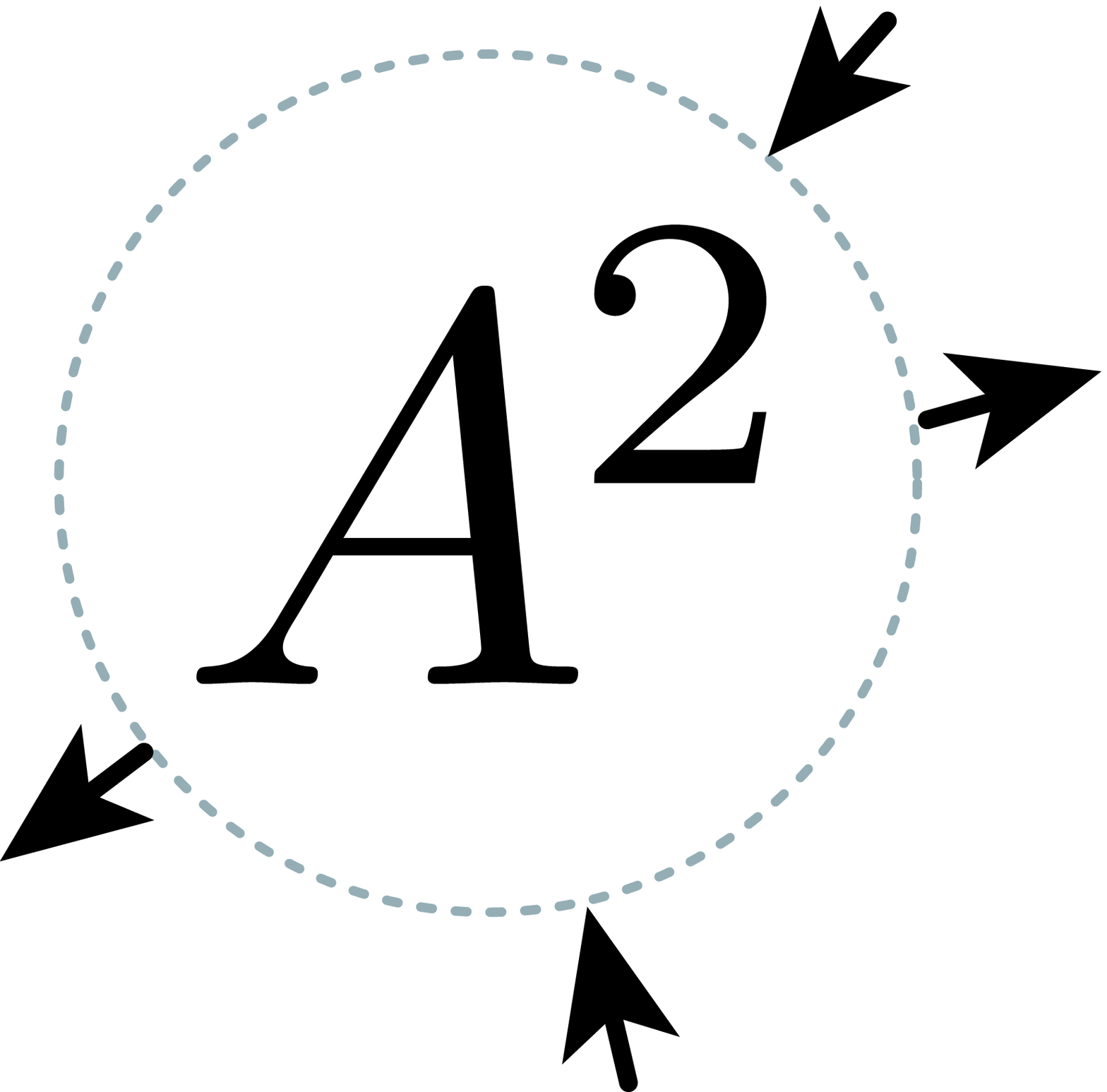} \xrightarrow{ \ \ d_A^2 \ \ } \pd{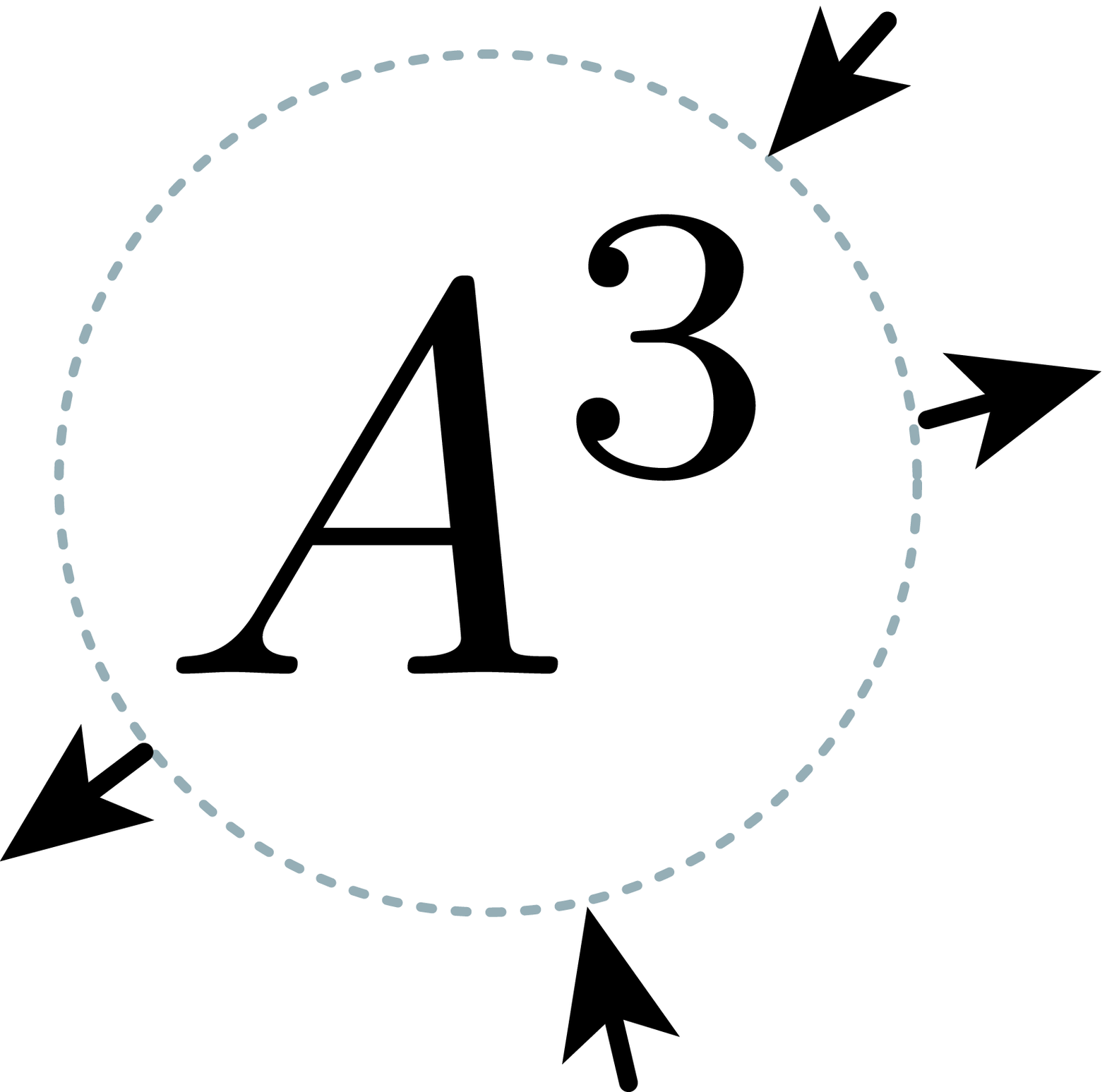} \right) \\ \\
\pd{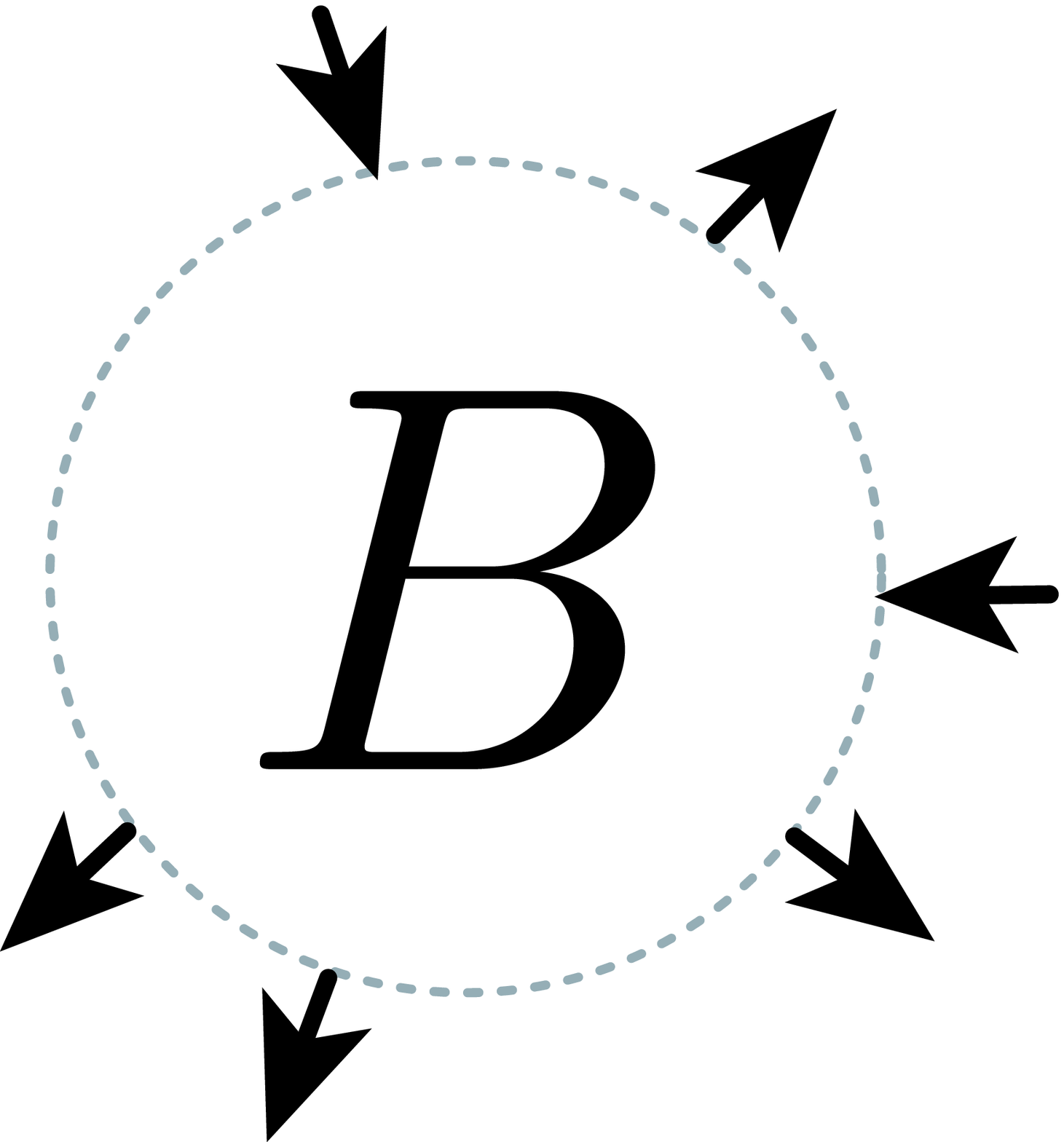} \ = \ \left( \pd{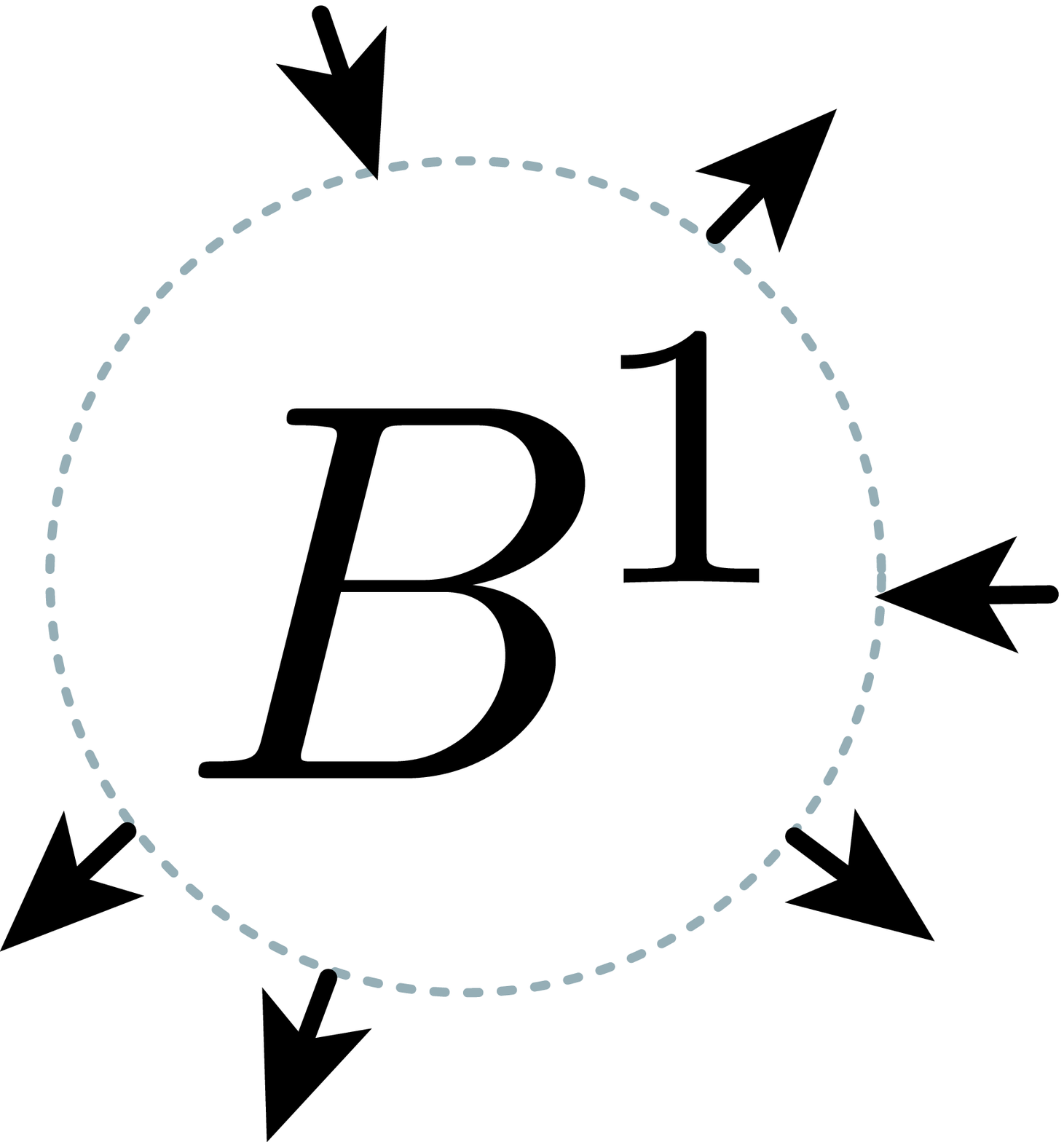} \xrightarrow{ \ \ d_B^1 \ \ } \pd{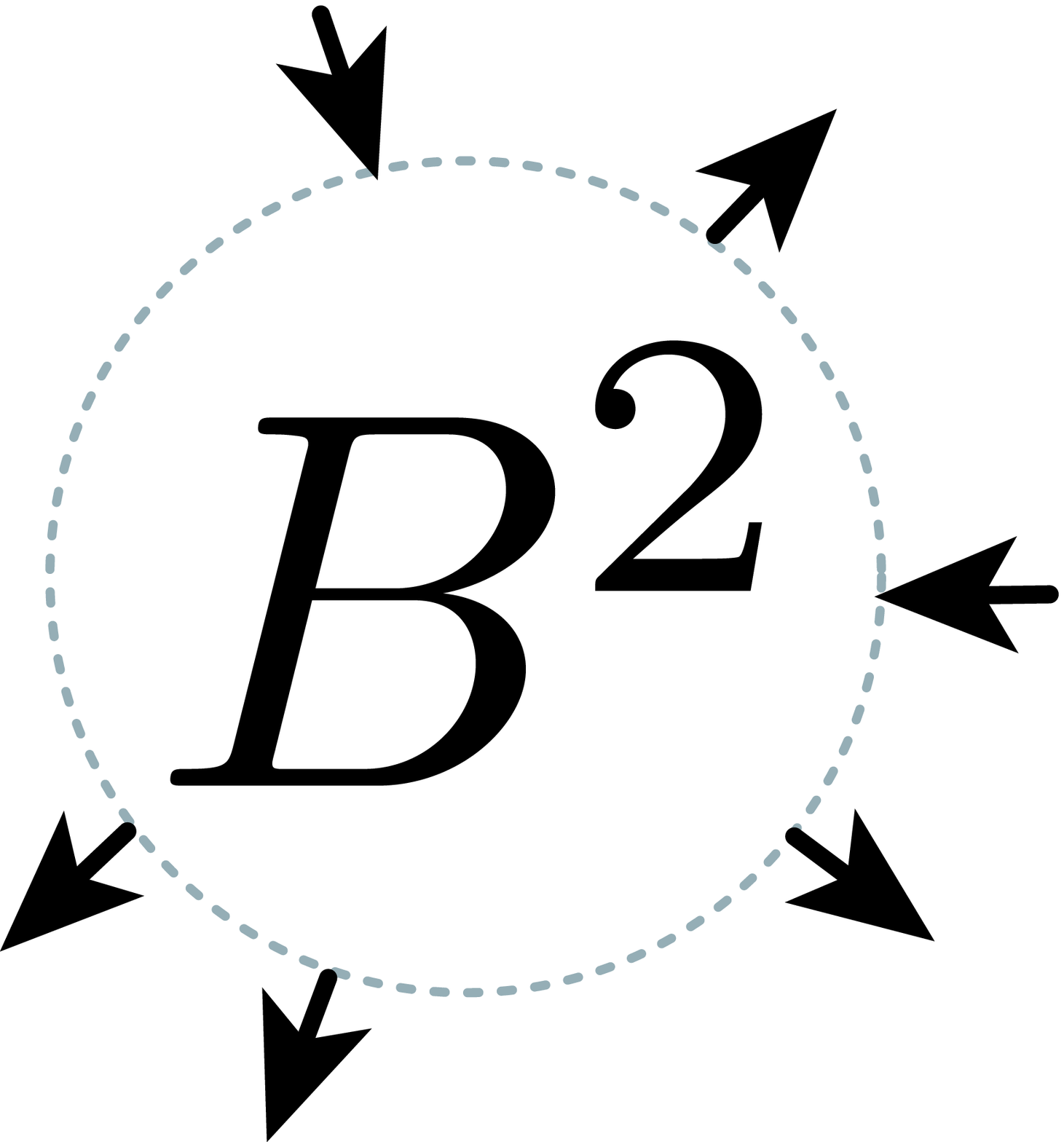} \xrightarrow{ \ \ d_B^2 \ \ } \pd{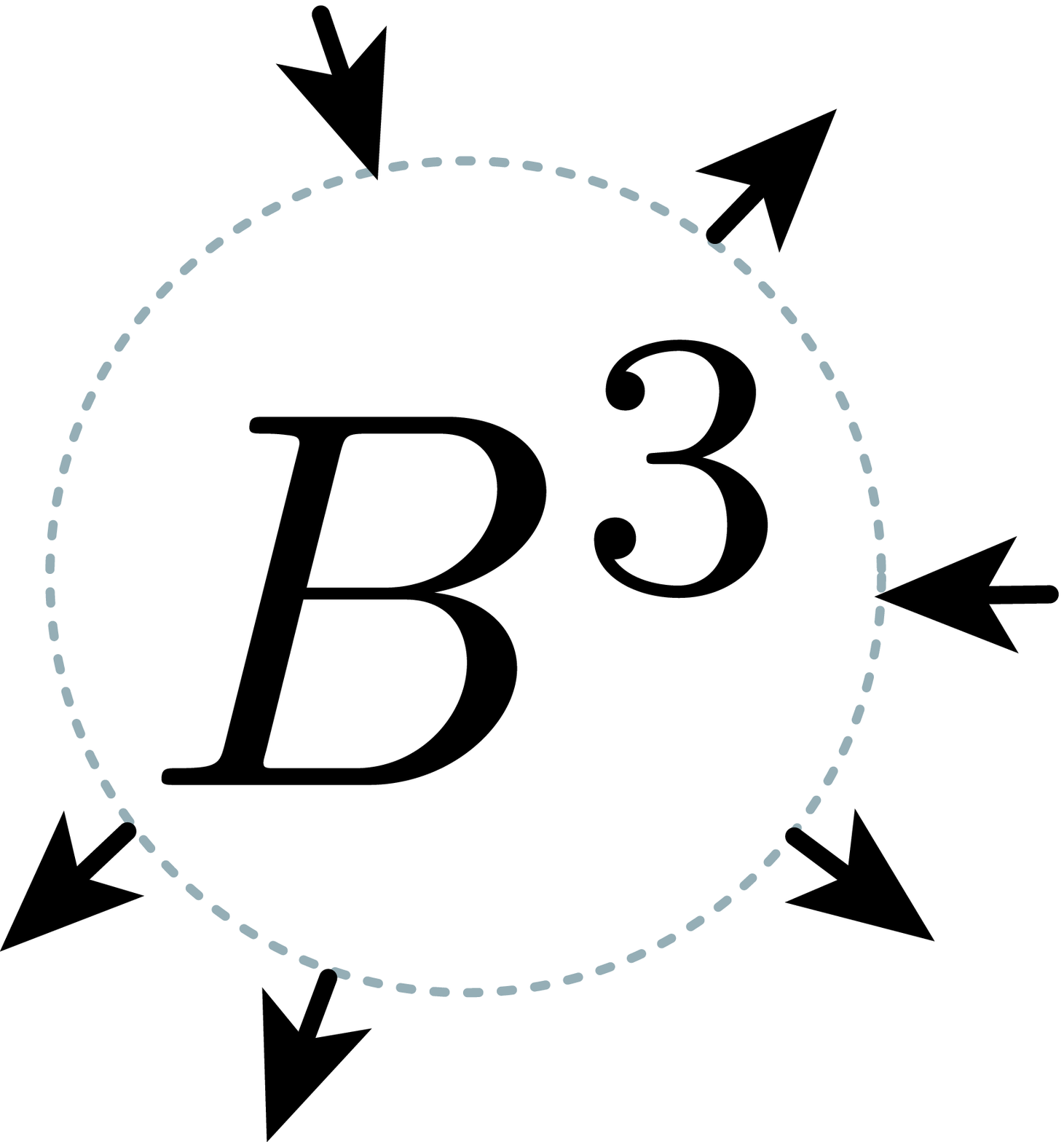} \right).
\end{align*}

This will give us the new complex
$$\mathfig{0.14}{webs/twoholes/AB} = \mathfig{0.14}{webs/operad1} \left( \pd{A} , \pd{B} \right)$$
that we construct by taking the double complex and direct summing along the line $y=-x$. (See Figure \ref{fig:twoholes}.)
\begin{figure}[h]
\begin{equation*}
\xymatrix@R20pt@C20pt{
 \pa{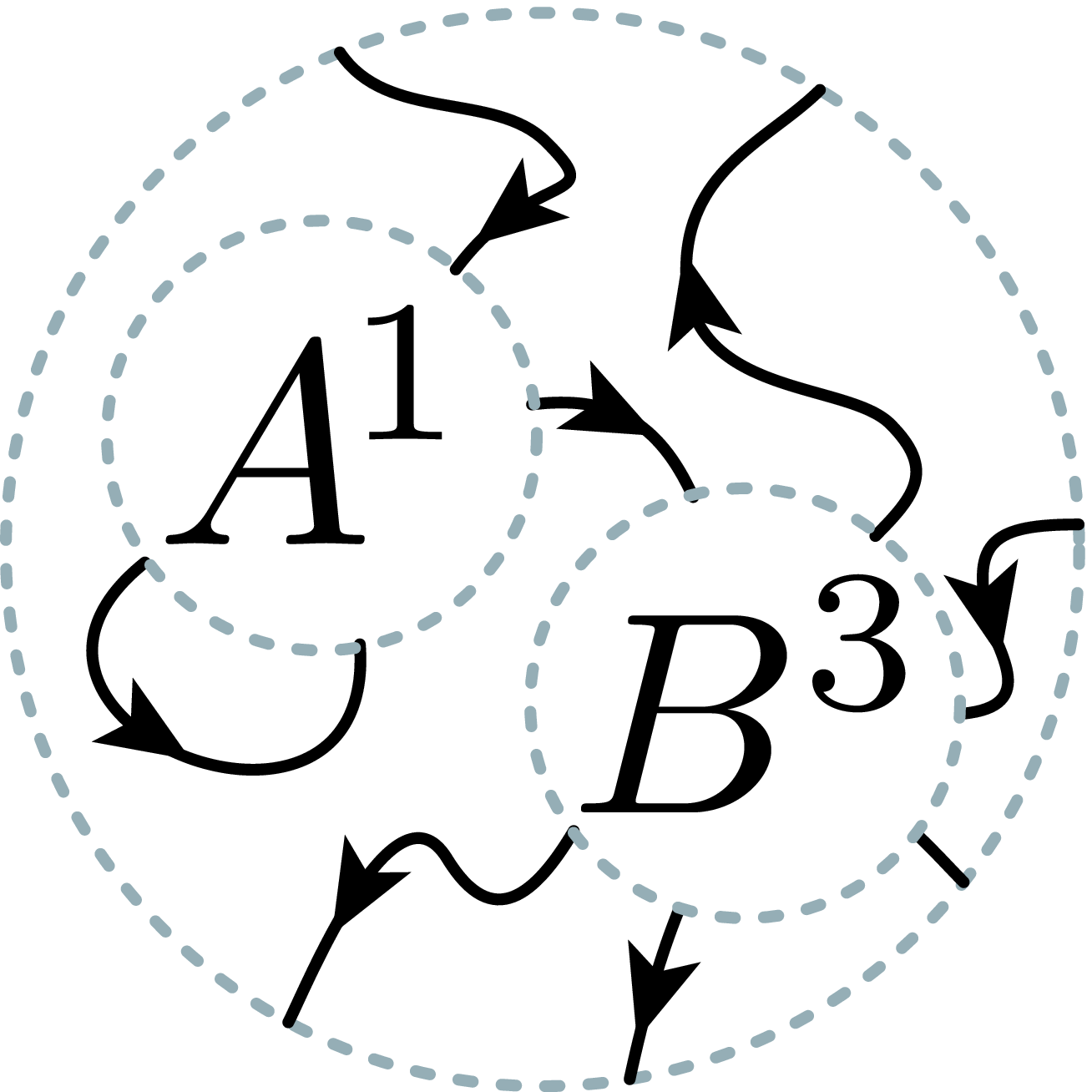} \ar[r]^{-} \ar@{.}[dr]|{\directSum} &  \pa{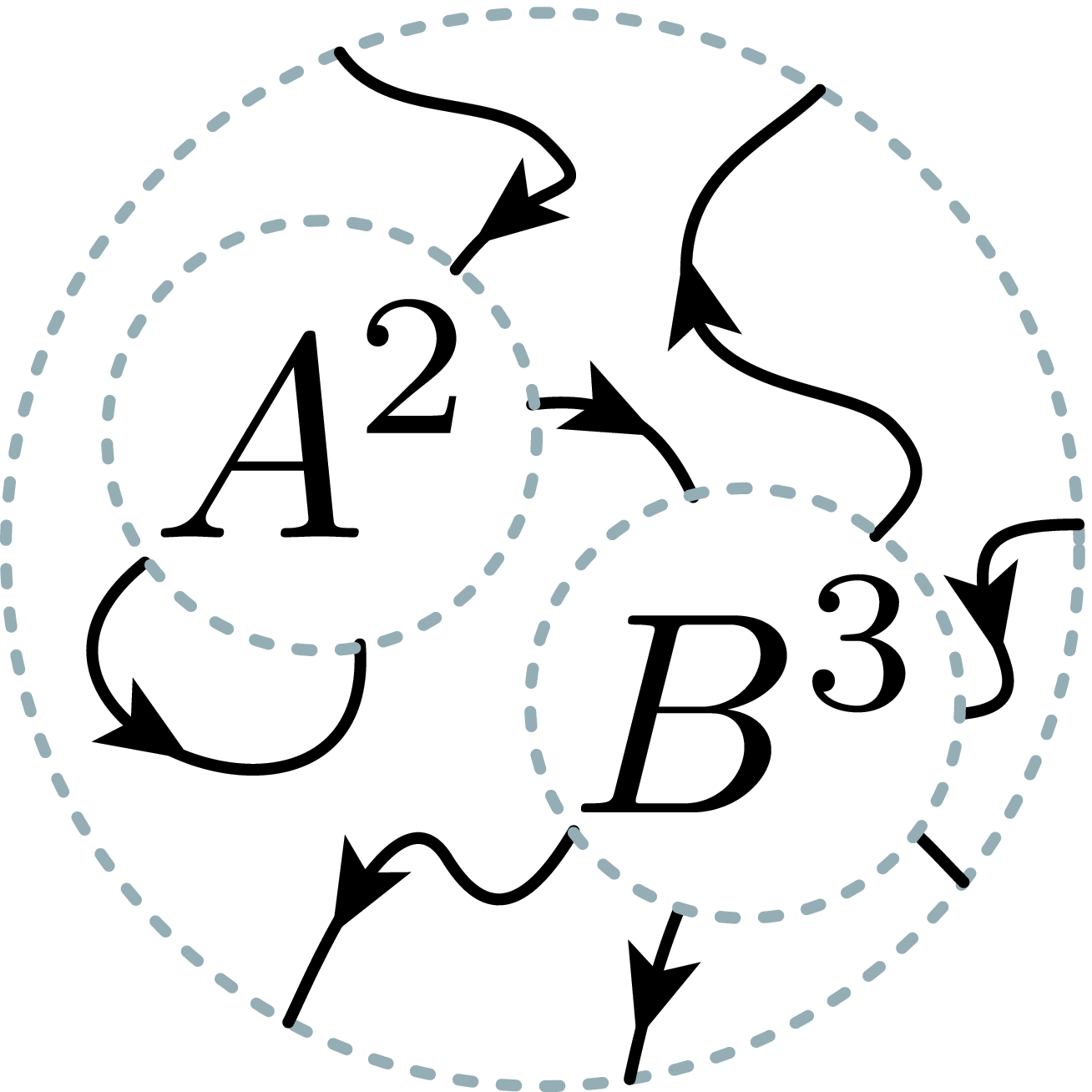} \ar[r]^{-} \ar@{.}[dr]|{\directSum} & \pa{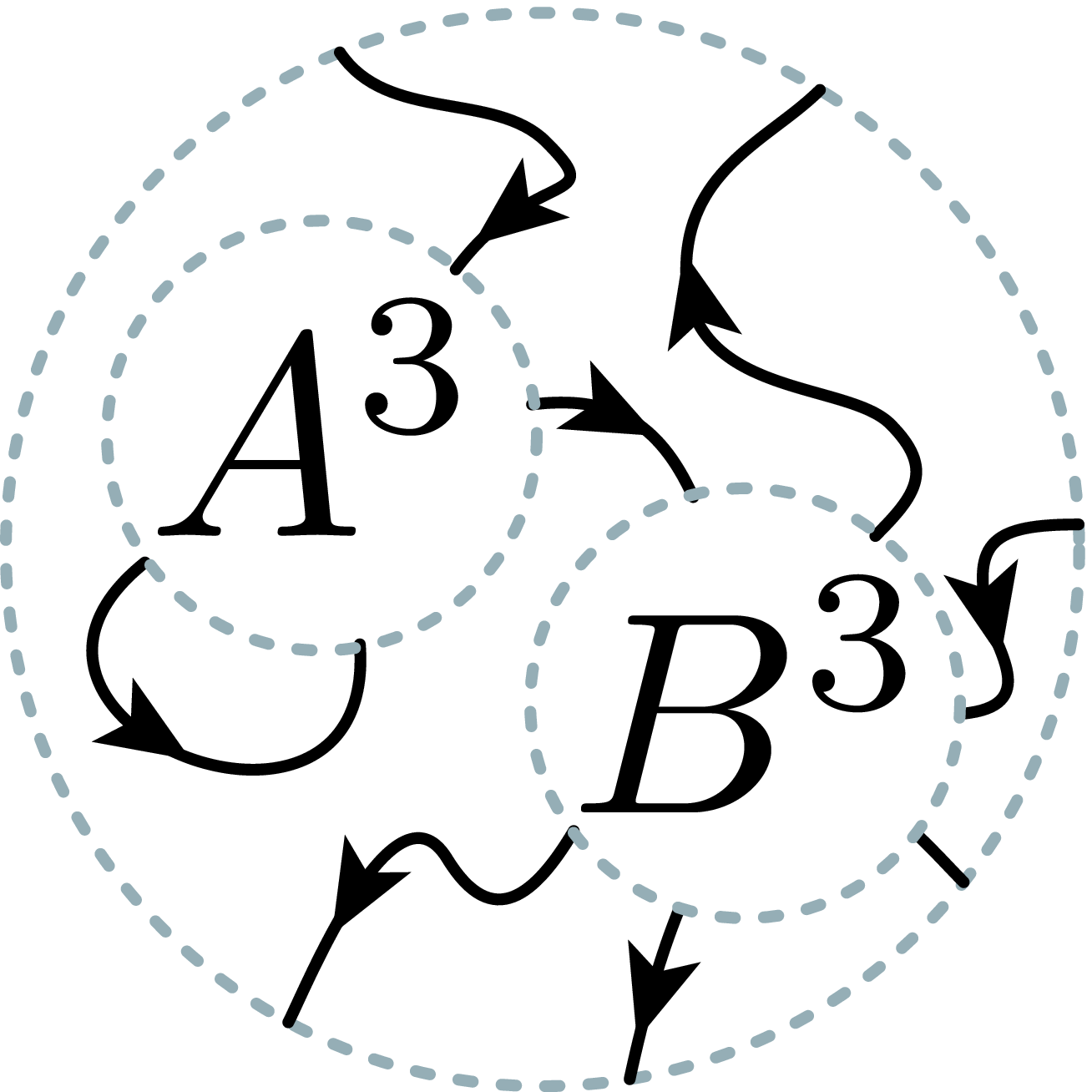} \\
 \pa{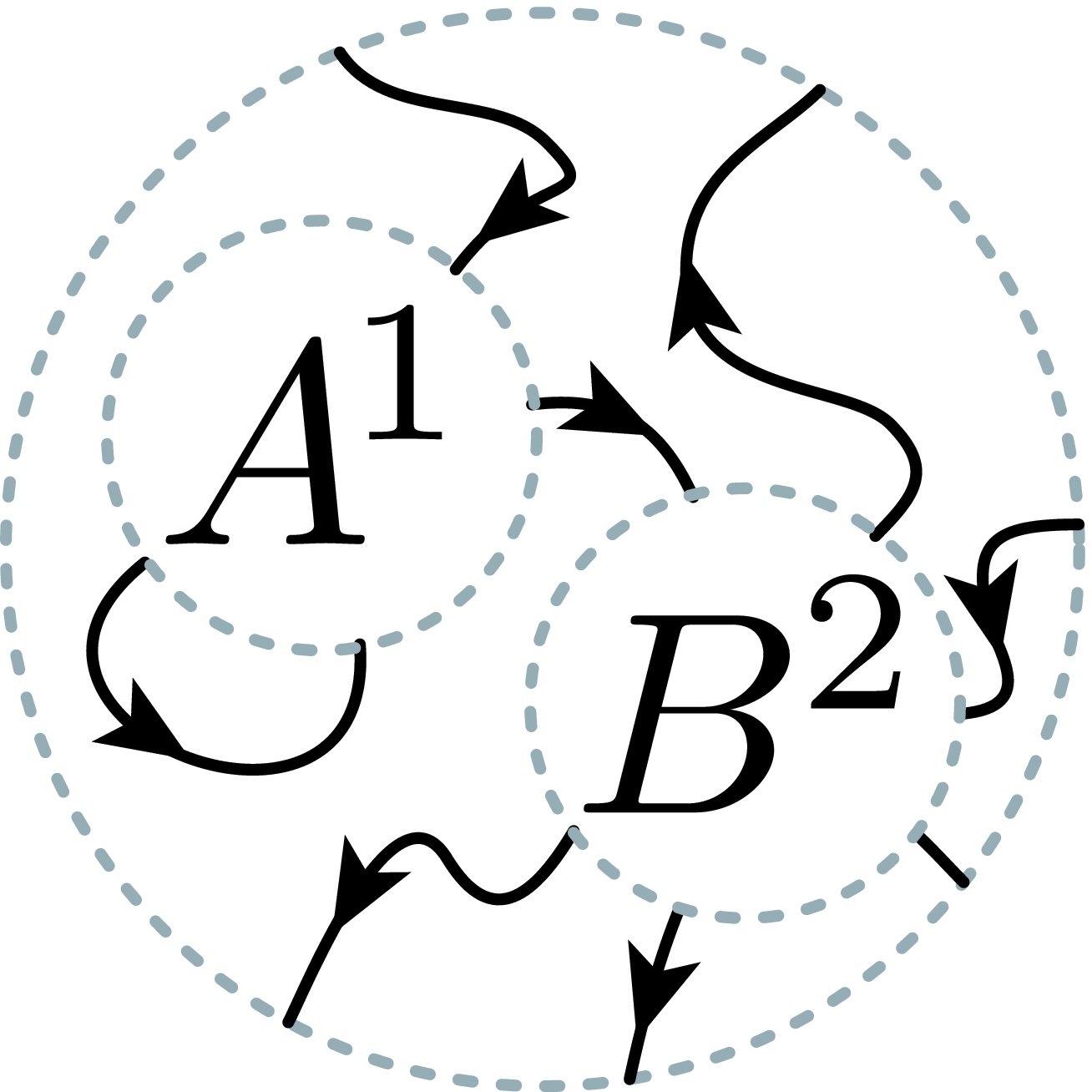} \ar[r] \ar[u] \ar@{.}[dr]|{\directSum} &  \pa{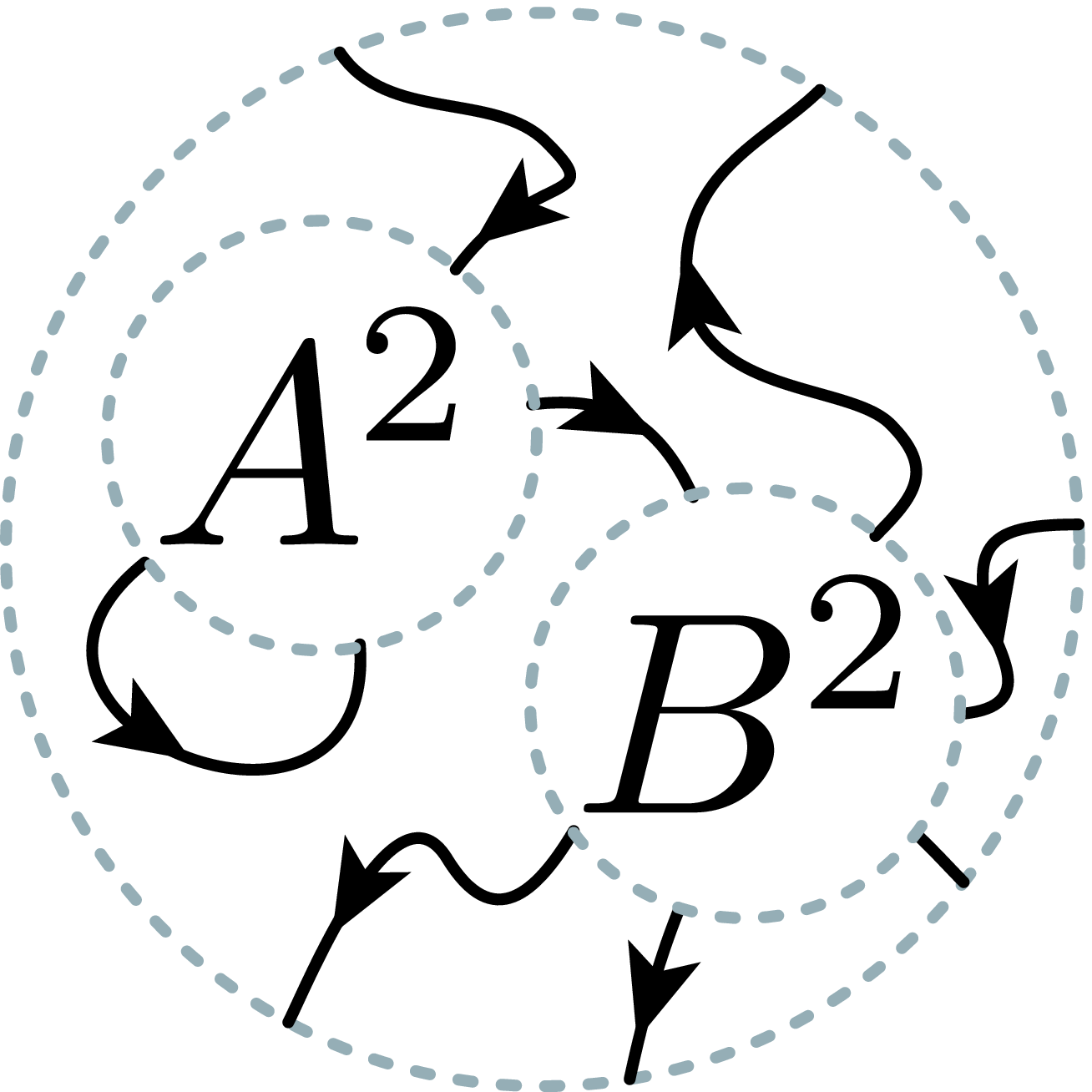} \ar[r] \ar[u] \ar@{.}[dr]|{\directSum} & \pa{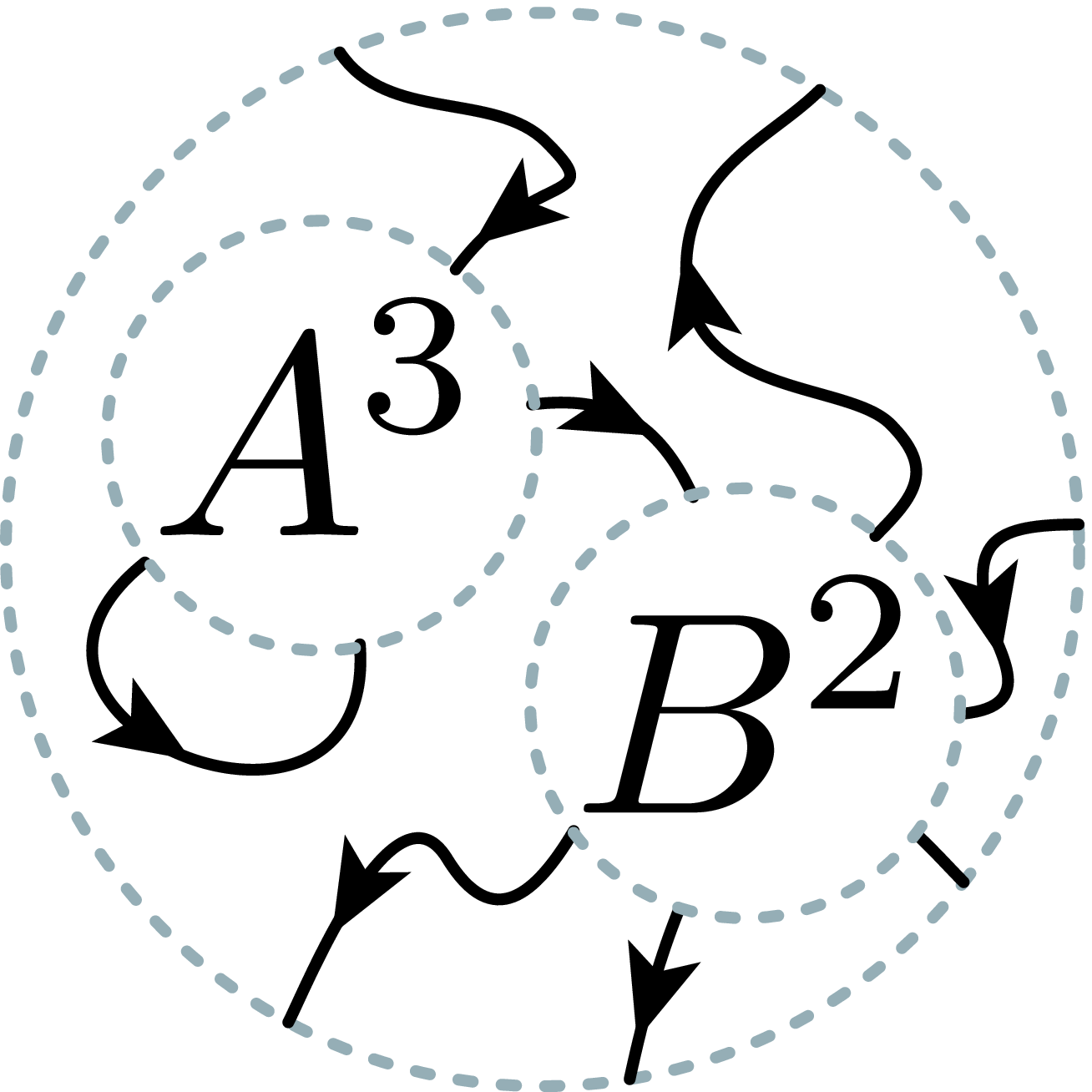} \ar[u] \\
 \pa{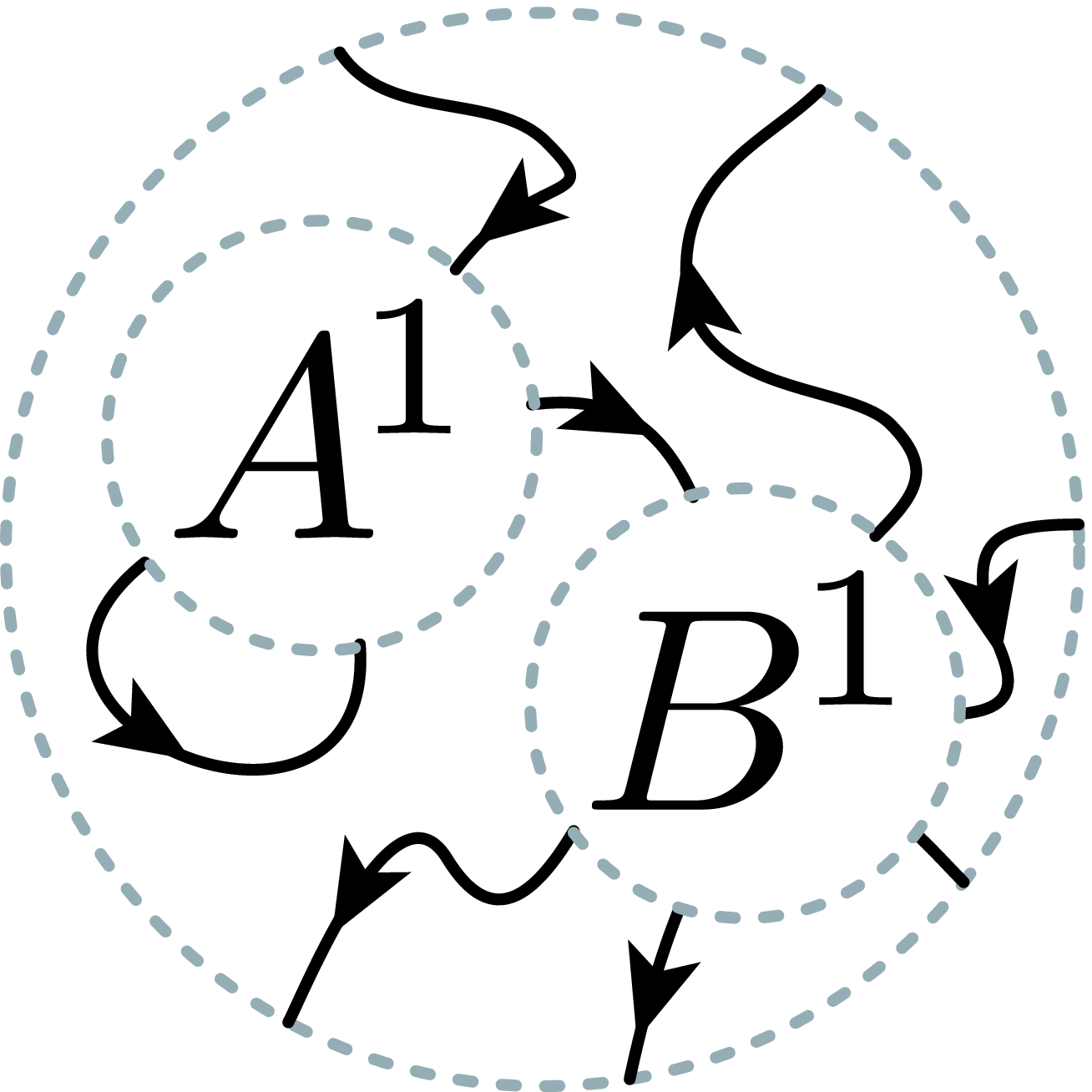} \ar[r]^{-} \ar[u] & \pa{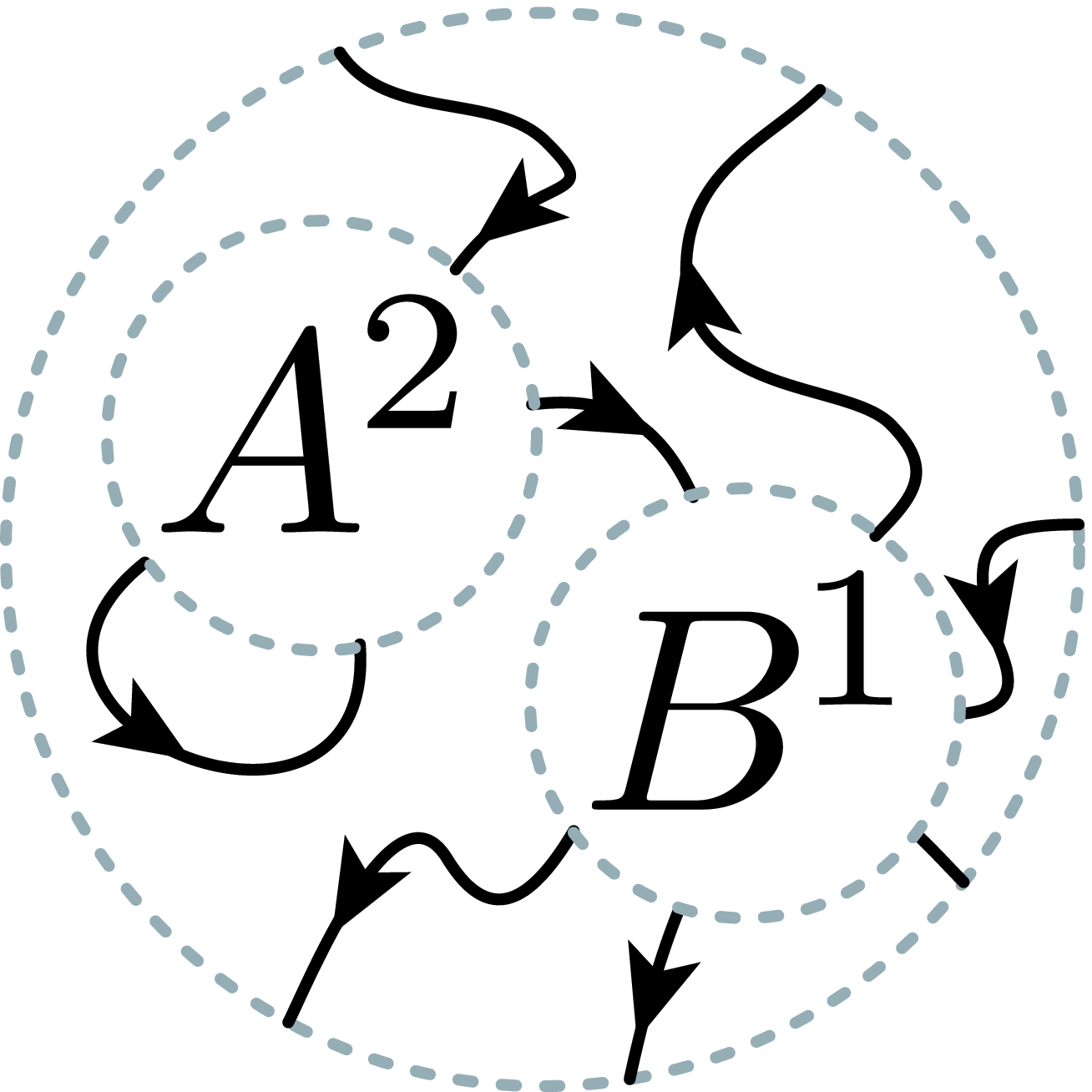} \ar[r]^{-} \ar[u] & \pa{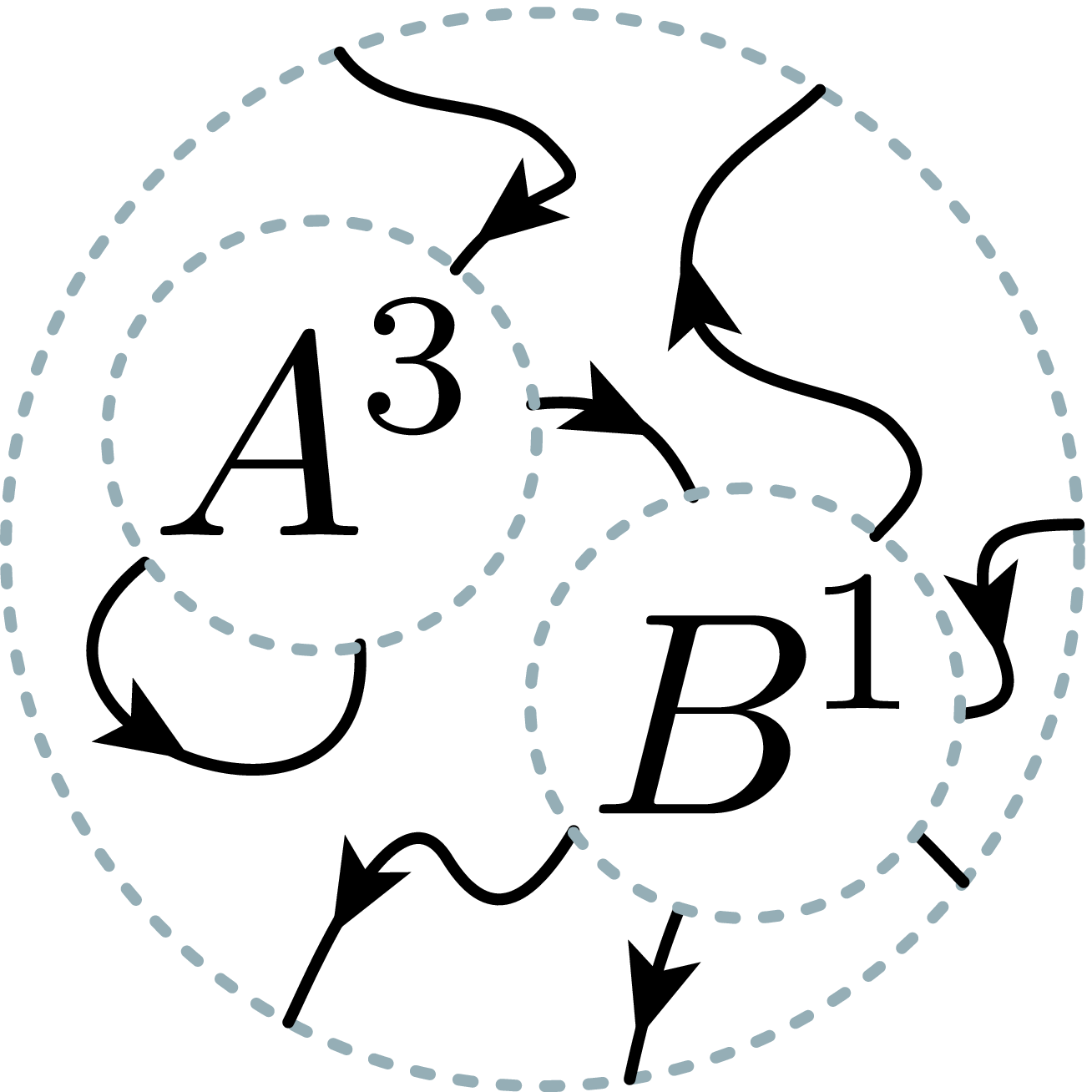} \ar[u]
}
\end{equation*}
\caption[The double complex being collapsed]{The double complex being collapsed.}
\label{fig:twoholes}
\end{figure}
In this picture, the horizontal arrow originating at the $(i,j)$th entry is the planar composition of $(-1)^j d_A^i$ (in hole 1) and $\Id_{B^j}$ (in hole 2). Similarly the vertical arrow at $(i,j)$ is the planar composition of $\Id_{A^i}$ and $d_B^j$.  (Please refer to the sign conventions in Appendix \ref{ssec:sign-conventions}.)

Given the associativity of planar composition, this rule easily generalizes (pairwise, if you like) to planar arc diagrams with $n$ holes.
}

\subsection{Sign conventions}
\label{ssec:sign-conventions}

We'll be using the following conventions for tensor products of complexes \cite{Gelf96_MR1438306}; these rules will translate directly to (ordered) planar compositions.

The tensor product of two complexes $(A^\bullet, d_A)$ and
$(B^\bullet, d_B)$ is defined to be
$$(A \tensor B)^\bullet = \DirectSum_{i+j=\bullet} A^i \tensor B^j,$$
and $$d_{(A \tensor B)^\bullet} = \sum_{i+j=\bullet} (-1)^j d_A^i
\tensor \Id_{B^j} + \Id_{A^i} \tensor d_B^j.$$

If $A^\bullet$ lies horizontally and $B^\bullet$ stands vertically in the double complex, this rule just says ``negate the differentials in every odd row."

As a consequence of these signs in the tensor product construction, the isomorphism $A^\bullet \tensor B^\bullet \iso B^\bullet \tensor A^\bullet$ is not quite the na\"{\i}ve permutation, which is not a chain map. Instead, to we'll need to define the map this way:
\begin{align*}
A^i \tensor B^j & \longrightarrow B^j \tensor A^i \\
(a,b) & \longmapsto (-1)^{ij} (b,a).
\end{align*}
Thus, performing a transposition in a tensor product will negate everything in ``doubly odd" degree. In the $\Kob{3}$ picture, this means that each time we alter the ordering of crossings by a transposition, we are really applying the isomorphism above. The ``doubly odd" objects here are the webs in which both crossings are I-resolved, and these ``doubly I-resolved" webs will pick up the additional minus signs. All other objects are mapped via the identity. 

\newpage
\newcommand{\urlprefix}{}
\bibliographystyle{gtart}
\bibliography{bibliography/bibliography}{}

\end{document}